\documentclass[10pt]{amsart}

\usepackage{amssymb}
\usepackage{amsmath}
\usepackage{pstricks}
\usepackage[all]{xy}
\usepackage{latexsym}
\usepackage{graphicx}
\usepackage{hyperref}

\newtheorem{theorem}{Theorem}[section]
\newtheorem{lem}[theorem]{Lemma}
\newtheorem{prop}[theorem]{Proposition}
\newtheorem{cor}[theorem]{Corollary}
\newtheorem*{theoremstar}{Theorem}
\theoremstyle{definition}
\newtheorem{definition}[theorem]{Definition}
\newtheorem{ex}[theorem]{Example}
\newtheorem{rem}[theorem]{Remark}
\newtheorem*{ack*}{Acknowledgments}

\newcommand{\AM}{{\mathcal{AM}}}

\newcommand{\N}{\mathbb N}
\newcommand{\R}{\mathbb R}
\newcommand{\C}{\mathbb C}
\newcommand{\hsset}{sSet_\infty}
\newcommand{\hssetp}{{sSet_*}_{\infty}}
\newcommand{\Top}{Top}
\newcommand{\hTop}{Top_\infty}
\newcommand{\hTopp}{{Top_*}_\infty}
\newcommand{\hcdga}{CDGA_\infty}
\newcommand{\sset}{sSet}
\newcommand{\cdga}{CDGA}
\newcommand{\hkmod}{k\text{-Mod}_\infty}

\newcommand{\Der}{\mathrm{Der}}
\newcommand{\Sym}{\mathrm{Sym}}
\newcommand{\com}{\bullet}
\newcommand{\colim}{\mathop{\lim\limits_{\textstyle\longrightarrow}}\limits}

\definecolor{M}{rgb}{1,0,0}
\definecolor{G}{rgb}{0,1,0}
\definecolor{T}{rgb}{0,0,1}

\newcommand{\CSS}{\mathcal{CS}e\mathcal{S}p}
\newcommand{\SeSp}{\mathcal{S}e\mathcal{S}p}

\newcommand{\noprint}[1]{ }   

\title[Higher Hochschild cohomology, Brane Topology and  centralizers]{Higher  Hochschild cohomology, Brane topology and centralizers of $E_n$-algebra maps}
\author[G.~Ginot]{Gr\'egory Ginot}
\address{Gr\'egory Ginot, UPMC - Sorbonne Universit\'e, Universit\'e Pierre et Marie Curie, Institut  Math\'ematiques de Jussieu Paris Rive Gauche, CNRS,
Case 247\\  4, place Jussieu, 75252 Paris Cedex 05, France}
\email{gregory.ginot@imj-prg.fr}

\author[T.~Tradler]{Thomas~Tradler}
\address{Thomas Tradler, Department of Mathematics, New York City College of Technology, City University of New York, 300 Jay Street, Brooklyn, NY 11201, USA}
\email{ttradler@citytech.cuny.edu}

\author[M.~Zeinalian]{Mahmoud~Zeinalian}
\address{Mahmoud Zeinalian, Department of Mathematics, Long Island University, LIU Post, 720 Northern Boulevard, Brookville, NY
11548, USA}
\email{mzeinalian@liu.edu}

\subjclass[2000]{16E40, 55P50, 55P35, 18D50, 18G55 (Primary), 55P99, 16E45, 16S80, 57P10 (Secondary)}

\keywords{$E_n$-algebras, factorization algebras, higher Hochschild homology, string topology, Poincar\'e duality, centralizers, Deligne conjecture, Bar construction, iterated loop spaces}

\begin{document}

\begin{abstract}
We use factorization homology and higher Hochschild (co)chains to study various problems in algebraic topology and homotopical algebra, notably brane topology, centralizers of $E_n$-algebras maps and iterated bar constructions. In particular, we obtain an $E_{n+1}$-algebra model on  the shifted integral chains $C_{\bullet+m}(Map(S^n, M))$ of the mapping space of the $n$-sphere into an $m$-dimensional orientable closed manifold $M$. 
We construct and use $E_\infty$-Poincar\'e duality to identify the higher Hochschild cochains, modeled over the $n$-sphere, with the chains on the above mapping space, and then relate the Hochschild cochains to the deformation complex of the $E_\infty$-algebra $C^*(M)$, thought of as an $E_n$-algebra. We  invoke (and prove) the higher Deligne conjecture to furnish $E_n$-Hochschild cohomology, and all that is naturally equivalent to it, with an $E_{n+1}$-algebra structure and further prove that this construction recovers the sphere product. In fact, our approach to the Deligne conjecture is based on an explicit description of the $E_n$-centralizers of a map of $E_\infty$-algebras $f:A\to B$ by relating it to the algebraic structure on Hochschild cochains modeled over spheres, which is of independent interest and explicit. More generally, we give a factorization algebra model/description of the centralizer of any $E_n$-algebra map and a solution of Deligne conjecture. We  also apply similar ideas to the iterated bar construction. We obtain factorization algebra models for (iterated) bar construction of augmented $E_m$-algebras together with their $E_n$-coalgebras and $E_{m-n}$-algebra structures, and discuss some of its features.  
For $E_\infty$-algebras we obtain a higher Hochschild chain model, which is an $E_n$-coalgebra.  In particular, considering the $E_\infty$-algebra structure of an $n$-connected topological space $Y$, we obtain a higher Hochschild cochain model of the natural $E_n$-algebra structure of the chains of the iterated loop space $C_*(\Omega^{n}Y)$. 
\end{abstract}
\maketitle

\tableofcontents

\section{Introduction}
The main goal of this paper is to apply the recent tools given by \emph{factorization algebras} and \emph{factorization homology} (or higher Hochschild (co)homology) to study various problems in algebraic topology and homological algebra, including the study of string and \emph{brane topology}, existence and explicit description of \emph{centralizers} of maps, which gives rise to a solution of higher Deligne conjecture, and the study of \emph{iterated bar constructions} for (homotopically commutative) algebras and \emph{iterated loop spaces}.  These applications are the core of the sections \ref{S:centralizers}, \ref{S:Brane} and \ref{S:Barmain}.

We start demonstrating these ideas by first explaining the starting point of our work. Our original motivation  was  the study of brane topology, as emphasized by D. Sullivan: the algebraic structure of the  chains on the mapping space of
the $n$-sphere into an orientable $m$-dimensional manifold $M$;  the coefficient of the chains being over a field of arbitrary characteristic, or over the integers.  
The algebraic structure of the chains on the mapping spaces of spheres into a manifold
has drawn considerable interest, following the work of Chas and Sullivan~\cite{CS} on the free loop space. It is now standard that the homology of the free loop space $LM = Map(S^1,M)$, shifted by the dimension of $M$,
has an intriguing structure of a BV-algebra, and in particular of a Gerstenhaber algebra\footnote{also called $1$-Poisson algebra},
 that is of a (graded) commutative algebra endowed with a degree $1$ Lie bracket satisfying the Leibniz rule.
This structure is in fact part of a $2$-dimensional homological conformal field theory (for instance see~\cite{Go, BGNX, L-TFT}));
the BV-algebra structure comes from the genus $0$ part of this topological conformal field theory.

\emph{Higher string topology}, also referred to as \emph{brane topology}, is a generalization of string topology
 in which the circle is replaced by the $n$-dimensional sphere.
Sullivan and Voronov (see~\cite{CV})
have stated\footnote{also see~\cite{CV, BGNX, Chataur} for rigorous explicit construction of the underlying
graded commutative multiplication, called the sphere product} that the (shifted) homology of the mapping sphere $Map(S^n, M)$ has the structure of a
 $\text{BV}_n$-algebra and in particular of an $n$-Poisson algebra
(or $n$-braid algebra in the terminology of~\cite{KM}).
The latter structure is the analogue of a Gerstenhaber algebra in which the Lie bracket is of degree $n$.
 A   $\text{BV}_n$-algebra is an algebra over the homology of the operad of framed $n$-dimensional little disks, while an $n$-Poisson algebra is an algebra
 over the homology of the \emph{little $n$-dimensional disks operad} (for instance see~\cite{CV, SW}); algebras over (the chains on) the little $n$-dimensional disks operad are usually called $E_n$-algebras.

 The \emph{$E_n$-algebras} form a hierarchy of homotopy commutative structures, whose commutativity increase with $n$, with $E_1$-algebras being essentially equivalent to dg-associative algebras. In particular, an $E_2$-algebra  is a dg-associative algebra, with product $\cup_0$, together with an homotopy operator $\cup_1$ for the commutativity of the product $\cup_0$. Similarly, in an $E_n$-algebra, $\cup_1$ is homotopy commutative, the homotopy being given by an operator $\cup_2$ which is homotopy  commutative and so on until an homotopy operator $\cup_{n-1}$ which is not  (required to be) homotopy commutative. It is well known that the homology of an $E_n$-algebra is an  $n$-Poisson algebra. These algebras are nowadays of fundamental importance in quantization (for instance see \cite{KS, PTVV}). 
 
 \smallskip
 
Sullivan-Voronov's work leads to the following:

\smallskip

\noindent \textbf{Question:} \emph{
Is it possible to lift the $n$-Poisson algebra structure on the homology of
$Map(S^n, M)$ to a structure of (framed) $E_n$-algebras  on the (suitably shifted) chains  of $Map(S^n, M)$ with coefficient in an arbitrary ring $k$?}

\smallskip

 For an $n$-connected closed and oriented manifold $M$, we give a positive answer
to this conjecture.
\setcounter{section}{7} \setcounter{theorem}{0} \begin{theorem}
Let $M$ be an $n$-connected Poincar\'e duality space whose homology groups are projective $k$-modules.
 The shifted chain complex $C_{\ast+\dim(M)} (Map(S^n, M))$
has a natural $E_{n+1}$-algebra structure which induces the Sullivan-Voronov sphere product in homology
$$H_{p}\big(Map(S^n, M)\big)\otimes H_{q}\big(Map(S^n, M)\big) \to H_{p+q-\dim(M)}\big(Map(S^n, M)\big),$$
when $M$ is an oriented closed manifold.
\end{theorem}
This $E_{n+1}$-algebra structure can be seen as a higher dimensional analogue of the genus $0$ part of a topological conformal field theory. 

\smallskip

Our approach is based on an \emph{algebraic model} of the chains on the mapping spaces generalizing
\emph{Hochschild cochains}, a fruitful model for string topology operations. This algebraic model is an instance of factorization homology for commutative or $E_\infty$-algebras which we develop in sections~\ref{S:HHforEinftyAlg} and~\ref{S:Operation}.

 Hochschild cohomology groups of an associative algebra $A$ with value in a bimodule $N$ are defined as
$$HH^{n}(A,N) \cong H^n\big(\mathbb{R}Hom_{A\otimes A^{op}}(A,N)\big)\cong Ext^n_{A\otimes A^{op}}(A,N),$$
where $\mathbb{R}Hom$ denotes the derived mapping space,
while the Hochschild homology groups $HH_\bullet(A,N)\cong Tor_n^{A\otimes A^{op}}(A,N)$
are defined similarly by derived tensor products. These (co)homology groups are given by standard (co)chain complexes~\cite{Ge, L}. 
Hochschild cohomology of any (dg-)associative  algebra has a natural Gerstenhaber algebra structure and further,  by the (solutions to the) Deligne conjecture, the latter is induced by an $E_2$-algebra structure on the  Hochschild cochains.
 
 Hochschild (co)chains have been used as models for free loop spaces since at least the 1980s.  
 Indeed,   there is an isomorphism (see \cite{CV, FTV}, for example)
\begin{equation}\label{eq:Hochschild=string} H_\bullet(LM)\cong HH^\bullet(C^\ast(M),C_\ast(M)) \cong HH^\bullet(C^\ast(M),C^\ast(M))[d]\end{equation}
if $M$ is  an oriented and simply connected manifold of dimension $d$ which, in characteristic zero is an isomorphism of Gerstenhaber algebras~\cite{FT}.  Further,  Hochschild chains of a Calabi-Yau algebra carries a topological conformal
field theory structure~\cite{L-TFT}. The above isomorphisms~\eqref{eq:Hochschild=string} make use of two ingredients. First, it uses the (dual of) an isomorphism $HH_\bullet(C^\ast(M), C_\ast(M)) \cong H_\bullet(LM)$ for any simply connected space $M$ (which can be described in geometric terms by Chen interated integrals when $M$ is a manifold) and, second, it uses a lift of the Poincar\'e duality quasi-isomorphism $C^{\ast}(M)\to C_{\ast}(M)[\dim(M)]$ to a bimodule map, when $M$ is further a closed manifold.

In this paper, we  generalize  these two facts from circles to $n$-dimensional spheres as well as the $E_2$-algebra structure on Hochschild cochains as we explain below. Combining these three ingredients will give us the desired $E_{n+1}$-algebra structure on $C_\ast(Map(S^n,M))$. Our technique should be related to those of Hu~\cite{Hu} and Hu-Kriz-Voronov~\cite{HKV}.

Bimodules over an associative algebra correspond to the \emph{operadic} notion of
$E_1$-modules. There is a notion of $E_n$-Hochschild cohomology where maps of $A$-bimodules are replaced by maps of $A$-$E_n$-modules for an $E_n$-algebra $A$ (\cite{L-HA, F, Fre-Mod}). The Kontsevich-Soibelman generalization of the Deligne conjecture, \emph{i.e.}, the higher Deligne conjecture, is that the  $E_n$-Hochschild cohomology of $A$, denoted $HH_{\mathcal{E}_n}(A,A)$ is an $E_{n+1}$-algebra. For $X$ a topological space, the cochains $C^\ast(X)$ are more than simply an associative algebra but are \emph{homotopy commutative}, that is, it carries a functorial structure of an $E_\infty$-algebra; in particular of an $E_n$-algebra for all $n$. In characteristic zero, one can use CDGAs models for the cochains, but this is not possible when working over the integers or a finite field.  Nevertheless, for $E_\infty$-algebras, $E_n$-Hochschild cohomology have extra functoriality (not shared by all $E_n$-algebras)  and actually identifies with higher Hochschild cohomology over the n-spheres.

The latter theories are the subject of Section~\ref{S:HHforEinftyAlg} and can be expressed in terms of \emph{factorization homology}, also referred to as \emph{topological chiral homology}~\cite{L-HA, F, CG, GTZ2}. Factorization homology is an invariant of  \emph{both} (framed) manifolds (and framed embeddings) and $E_n$-algebras based on (extended) topological field theories.
In fact, the factorization homology of  $E_\infty$-algebras becomes a homotopy invariant and can be applied to any space (and continuous maps) and not just to framed manifolds. This generalization is precisely computed by \emph{higher Hochschild homology},  introduced by Pirashvili  in~\cite{P}, which can be seen as a kind of \emph{limit} of these ideas when the dimension of the TFT goes to infinity~\cite{GTZ2}.  Indeed, by Theorem~\ref{T:CH=TCH} below (and~\cite{GTZ2, F, L-HA}) if $X$ is a  manifold  and $A$ is an $E_\infty$-algebra, then, the factorization homology $\int_X A$ of $X$ with coefficients in $A$   is naturally equivalent  to the Hochschild chains $CH_X(A)$ of $A$ over $X$.

 The restriction to $E_\infty$-algebras is not an issue in our case of interest since the cochain complex $C^\ast(X)$   is indeed an $E_\infty$-algebra. We  study the \emph{higher Hochschild (co)chains} for $E_\infty$-algebras and modules in Section \ref{SS:FactandHHforEinfty}, which are modeled over spaces in the same way the usual Hochschild (co)chains are modeled on circles. More precisely, this is  a rule that assigns to any space  $X$,  $E_\infty$-algebra $A$, and $A$-module $M$, a chain complex $CH_X(A,M)$, \emph{functorial in every argument}, such that for $X=S^1$, one recovers the usual Hochschild chains.  The functoriality with respect to spaces is a key feature which allows us to derive algebraic operations on the higher Hochschild (co)chain complexes from maps of topological spaces. 
 
 \smallskip
 
 Higher Hochschild chains have a good axiomatic characterization (similar to Eilenberg-Steenrod axioms) which formally follows from the fact that $E_\infty$-algebras are tensored over spaces, see Corollary \ref{C:properties} in Section \ref{SS:AxiomHH}. 
This allows to generalize the aforementioned relationship between free loop spaces and Hochschild chains  to every space. 
In fact, we prove (Theorem~\ref{T:Einftymapping}) that there is a natural map of $E_\infty$-algebras $CH_Y(C^\ast(X))\to C^\ast(Map(Y,X))$ which is a quasi-isomorphism when $X$ is $\dim(Y)$-connected. This is an $E_\infty$-analogue of our previous result~\cite{GTZ} for CDGA's in characteristic zero (using generalizations of Chen {\em iterated integrals}).

In Section~\ref{S:Operation}, we study the algebraic structure of higher Hochschild cochains. We first define, for any $E_\infty$-$A$-algebra $B,$ the (associative) \emph{wedge product}
$$CH^{X}(A,B) \otimes CH^{Y}(A,B)\to CH^{X\vee Y}(A,B) $$
and then we prove that when $X$ is a sphere $S^d$, the wedge product induces a structure of \emph{$E_d$-algebra on $CH^{S^d}(A,B)$}, 
generalizing the usual cup-product in Hochschild cohomology.
\setcounter{section}{4} \setcounter{theorem}{11} \begin{theorem}
Let  $A$ be an $E_\infty$-algebra and $B$ an $E_\infty$-$A$-algebra. The collection of maps $(pinch_{S^d,k}: \mathcal{C}_d(k) \times S^d \longrightarrow
 \bigvee_{i=1\dots k}\, S^d)_{k\geq 1}$ makes $CH^{S^d}(A,B)$ into an $E_d$-algebra, such that the underlying $E_1$-structure of $CH^{S^d}(A,B)$ agrees with the one given by the \emph{cup-product},
$$ \cup_{S^d}:\quad CH^{S^d}(A,B) \otimes CH^{S^d}(A,B)  \longrightarrow CH^{S^d\vee S^d}(A,B) \longrightarrow CH^{S^d}(A,B).$$
\end{theorem}
The CDGA version of this result goes back to the first author's note~\cite{G}. This result is in fact a relative version of the higher Deligne conjecture. 

\smallskip

In Section~\ref{S:centralizers}, we reinterpret and generalize the above results, to the case of all $E_n$-algebras in terms of \emph{centralizers} of $E_n$-algebra maps.
The latter are $E_n$-algebras satisfying a universal property  whose existence was established by Lurie~\cite{L-HA}. Their importance lies in the fact that their structure controls relative deformations of categories of $E_n$-modules. 
We prove 
\begin{theoremstar}[Theorem~\ref{T:EnAlgHoch} and Proposition~\ref{P:HHEn=z}]
Let $f:A\to B$ be an $E_n$-algebra map. The $\mathcal{E}_n$-Hochschild cohomology $HH_{\mathcal{E}_n}(A,B)\cong \mathbb{R}Hom^{\mathcal{E}_n}_{A} \big({A}, {B}  \big)$ has a natural $E_n$-algebra structure exhibiting it as the centralizer $\mathfrak{z}(f)$ of $f$. 
\end{theoremstar}
Our result gives another proof of existence of centralizers and also gives an \emph{explicit description} in terms of factorization algebras.  
Applying the universal property of centralizers when $f= id_A$, and  using an approach due to Lurie~\cite{L-HA}, we obtain that   $\mathfrak{z}(id_A)\cong HH_{\mathcal{E}_n}(A,A)$ inherits a canonical  $E_{n+1}$-algebra structure, giving a \emph{solution to the higher Deligne conjecture}, see Corollary~\ref{T:Deligne}.
We also prove that the Hochschild cochains $CH^{S^d}(A,A)$ of a commutative algebra $A$ 
are equivalent to  its $E_n$-Hochschild cohomology (Proposition~\ref{P:E1Pn}). 

\smallskip

As already mentioned, our approach is based on the relationship 
 between $E_n$-algebras and factorization algebras which we briefly explain, among other preliminaries, in Section~\ref{S:operad}. 
 
 \emph{Factorization algebras} originated from   quantum field theories  and the pioneering work of Beilinson-Drinfeld~\cite{BD} on chiral and vertex algebras. We follow an approach due to Lurie~\cite{L-HA} and Costello-Gwilliam~\cite{CG}.
They are algebraic structures which share many similarities with (co)sheaves and were introduced to describe quantum field theories, much in the same way the sheaf of functions describes the structure of a manifold or scheme~\cite{BD, CG}. 
Roughly speaking a factorization algebra $\mathcal{F}$  associate (covariantly) cochain complexes to open subsets   of   a (stratified) manifold $X$  together with multiplications $\mathcal{F}(U_1)\otimes \cdots \otimes \mathcal{F}(U_n)\to \mathcal{F}(V)$ for any family of pairwise disjoint open subsets of an open set $V$ in $X$. It is required to satisfy a \lq\lq{}cosheaf-like\rq\rq{} condition, meaning that $\mathcal{F}(V)$ can be computed by analogues of \v{C}ech complexes indexed on nice enough covers.  

\emph{Factorization homology} is a catchword to describe homology theories \emph{specific} to, say, oriented\footnote{there are also variants specific to many other classes of structured manifold of fixed dimension; for instance framed, spin or unoriented ones} manifolds of a fixed dimension $n$. It can be seen as the (derived) global section of (locally constant) factorization algebras much in the same way as singular cohomology can be seen as sheaf cohomology with value in a constant sheaf.

$E_n$-algebras can be identified with  factorization algebras on  $\mathbb{R}^n$ which are locally constant, that is for which the structure map $\mathcal{F}(U) \to \mathcal{F}(V)$ is an equivalence when $U$ is a subset of $V$ and both are homeomorphic to a disk. This provides a nice model for the category of $E_n$-algebras, which, in some sense can be thougt as a kind of mild \lq\lq{}strictification\rq\rq{} of $E_n$-algebras. This is the model we use in our approach to centralizers.  
In Section~\ref{S:FactandMod}, we recall the relationship between $E_n$-modules over an $E_n$-algebra $A$ and factorization homology over $S^{n-1}\times \mathbb{R}$. 
 Namely  that the category of $E_n$-$A$-modules is equivalent to the category of left modules over the (associative) algebra $\int_{S^{n-1}\times \mathbb{R}} A$. For $n=\infty$, one recovers that $E_\infty$-$A$-modules are the same as left modules over $A$ (\cite{L-HA, L-III, KM}).

\setcounter{section}{5} \setcounter{theorem}{12}\begin{theorem}
Let $A$ be an $E_\infty$-algebra. There is an equivalence of symmetric monoidal $\infty$-categories between the category $A\text{-}Mod^{E_\infty}$ of $E_\infty$ $A$-Modules and the category of left $A$-modules (where $A$ is viewed as an $E_1$-algebra).
\end{theorem}
We give a proof of this result using factorization homology in Section \ref{SS:lifttoEinfty}.
From this, we deduce in Section \ref{SS:E-inf-PD}, that,  
the Poincar\'e duality isomorphism can be \emph{uniquely lifted} into an $E_\infty$-quasi-isomorphism 
\setcounter{section}{5} \setcounter{theorem}{25}\begin{cor}
Let $(X, [X])$ be a Poincar\'e duality space.
The cap-product by $[X]$ induces a quasi-isomorphism of $E_\infty$-$C^{\ast}(X)$-modules
\begin{equation*}
C^{\ast}(X) \stackrel{\simeq}\longrightarrow  C_{\ast}(X)[\dim(X)]
\end{equation*}
realizing the (unique) $E_\infty$-lift of the Poincar\'e duality isomorphism.
\end{cor}

Putting together the above results on the Deligne conjecture, Poincar\'e duality and interpretation of higher Hochschild chains in terms of mapping spaces,
 we obtain in Section \ref{S:Brane} that the chains $C_\ast(Map(S^n,M))$ for an $n$-connected manifold $M$ inherits a natural $E_{n+1}$-algebra structure (Theorem~\ref{T:BraneChain}) lifting Sullivan-Voronov sphere product. To identify the sphere product, we use the fact that we have an \emph{explicit} description of the centralizers of $E_n$-algebra maps. 
 The above results  yield chain level constructions over any field or the ring of integers. 
 Results similar to Theorem~\ref{T:BraneChain} can be obtained using only bimodules maps (not necessarily quasi-isomorphisms) $C^\ast(M) \to C_\ast(M)[d]$. 
 This yield a \emph{functorial} construction of $E_{n+1}$-algebra structures on $C_\ast(Map(S^n,M))$, see  Theorem~\ref{THM:Ed-on-HC(AM)}.

Furthermore, ideas similar (in some sense dual) to the concept of the centralizer for $E_n$-algebra maps, lead to a description and construction of bar and iterated bar constructions for $E_n$-algebras. The bar construction of a (dg-)associative \emph{augmented} algebra is a standard functor in homological algebra and algebraic topology. The bar construction of a CDGA is itself an augmented algebra and thus can be iterated; this result was extended to $E_\infty$-algebras by Fresse~\cite{Fre2}. 

In topology, \emph{iterated bar constructions} arise as models for  \emph{iterated loop spaces} $\Omega^n (X)$, the space of pointed maps from the sphere $S^n$ to a pointed space $X$. The latter is an $E_n$-algebra in the category of spaces so that its singular cochains  becomes an $E_n$-coalgebra in $E_\infty$-algebras.
In Section \ref{S:Bar}, for an augmented $E_\infty$-algebra $A$ we apply the $E_n$-algebra structure on higher Hochschild chains $CH^{S^n}(A,k)$ (identified with the centralizer construction for the augmentation $A\to k$) to describe the \emph{iterated Bar construction} of an augmented $E_\infty$-algebra. We obtain that  $Bar(A)$ 
is naturally an $E_1$-coalgebra in augmented $E_\infty$-algebras so that we can iterate the construction. With this, we prove that \emph{the $n$-iterated bar construction $Bar^{(n)}$ is an  $E_n$-coalgebra inside the ($(\infty,1)$)-category of $E_\infty$-algebras}, see Theorem~\ref{P:EncoAlgBar}. We then relate this construction to iterated loop spaces by showing that there is a natural
map of $E_n$-coalgebras (and $E_\infty$-algebras) $Bar^{(n)}(C^\ast(X))\to C^\ast(\Omega^n(X))$ which is a quasi-isomorphism if $X$ is $n$-connected. 
\setcounter{section}{8} \setcounter{theorem}{9} \begin{cor}
Let $Y$ be a topological space.
The map $\mathcal{I}t^{\Omega^n}: Bar^{(n)}(C^{\ast}(Y)) \to C^{\ast}\big(\Omega^n(Y)\big) $ is an $E_n$-coalgebra morphism in the category of $E_\infty$-algebras, which is an equivalence if $Y$ is $n$-connected.
\end{cor}
We also give similar dual statements for \emph{chains on iterated loop spaces}  using that the dual of the Bar construction is precisely the centralizer $\mathfrak{z}(A\to k)\cong CH^{S^n}(A,k)$ of the augmentation $A\to k$.

In Section~\ref{SS:BarforEn}, we consider the \emph{bar construction of an $E_m$-algebra} $A$. Using its factorization homology interpretation due to Francis~\cite{F}, we prove that the bar construction $Bar^{(1)}(A)$  is naturally an $E_{m-1}$-algebra which allows us to iterate this construction up to $m$-times. Then, using the technique of Section~\ref{S:centralizers}, we prove that 
\setcounter{section}{8} \setcounter{theorem}{36} 
\begin{theorem}The $n$-iterated bar construction of an augmented $E_m$-algebra $(m\geq 1)$ has a natural structure of an $E_n$-coalgebra inside the $((\infty,1)-)$category $E_{m-n}$-algebras.
\end{theorem}
Similar result are stated in~\cite{F2}.
The existence of the iterated bar construction for $E_m$-algebras  as a chain complex was proved in~\cite{FresseBarEn}.
The idea behind the theorem is again to prove a similar statement for locally constant factorization algebras over $\mathbb{R}^m$. More precisely, we prove that an augmented locally constant factorization algebra over  $\mathbb{R}^m$ naturally gives rise to a locally constant \emph{stratified} factorization algebra on the pointed sphere $S^m$ whose global sections are precisely the iterated bar construction $Bar^{(m)}(A)$. This construction extends into a locally constant factorization \emph{coalgebra} over $\mathbb{R}^m$ which associates to any disk (the global sections of) a stratified factorization algebra on the one-point compactification of the disk.   

\smallskip

In this paper, we work in Lurie's framework of stable $\infty$-categories~\cite{Lu11, L-HA}, which is very well suited for doing homological algebra in the symmetric monoidal context. In particular, we will work over the (derived) $(\infty,1)$-category $\hkmod$ of chain complexes over a commutative unital ring $k$.
(In section \ref{S:operad}, we briefly recall notions of $(\infty, 1)$-categories, $\infty$-operads and in particular the $E_n$-operad and its algebras and their modules.)
It should be noted that in characteristic zero, one can use CDGA's instead of $E_\infty$-algebras which allows us to have (model) categories interpretation of all our results in the spirit of~\cite{G, GTZ, GTZ2}.

\medskip

\begin{ack*}
The authors\footnote{the first author was partially supported by the ANR Grant HOGT} would like to thank the Max-Planck-Institut
f\"ur Mathematik in Bonn, which hosted them while part of this project was being done.  The first author would also like to thank the Einstein Chair at CUNY for their invitation. The third author is partially supported by the NSF grant DMS-1309099. We also would like to thank Damien Calaque, Kevin Costello, John Francis and Hiro-Lee Tanaka for many helpful discussions, and Jim Stasheff for useful comments on an earlier draft of this paper.
\end{ack*}

\medskip

\noindent\textbf{Conventions and notations:}
\begin{enumerate}
\item We use \emph{homological grading}, emphasizing the geometric dimension of the chains on mapping spaces. In particular, unless otherwise stated, differential will lower the degree by one. We will write $\hkmod$ for the $(\infty,1)$-category of chain complexes of $k$-modules and $\otimes$ for tensor products over the ground ring $k$.

\item We will denote the \emph{Hochschild chain complex} of $A$, modeled over a space,
$X$ with values in an $A$-module $M$, by $CH_{X}(A, M)$ as an object in the  stable $(\infty,1)$-category
of chain complexes. This is a \emph{covariant} functor in $X$.
Similarly, we will also denote the \emph{Hochschild cochain complex} of $A$, modeled  over a space  $X$,
with values in an $A$-module $M$, by $CH^{X}(A, M)$, as an object in the stable $(\infty,1)$-category
of chain complexes.  This is a \emph{contravariant} functor of $X$, see \S~\ref{SS:coHHEinfty}. This is compatible with
the notation introduced in~\cite{GTZ2} but not with those in ~\cite{G, GTZ}.
We choose this notation in order to emphasize the variance of the functor with respect to $X$.

\item We will respectively denote $HH^{X, n}(A, M)$ and $HH_{X, n}(A,M)$ the degree $n$ homology groups of
$CH^{X}(A, M)$ and  $CH_{X}(A, M)$.
\item For $n\in \mathbb{N}\cup \{\infty\}$, we will write $E_n\text{-Alg}$ for the $(\infty,1)$-category of
 $E_n$-algebras in $\hkmod$ as studied in~\cite{L-VI, L-HA, F}. 
We will also denote by $\mathbb{E}^{\otimes}_n$ the $\infty$-operad governing $E_n$-algebras,
$HH_{\mathcal{E}_n}(A,M)$ for the $E_n$-Hochschild cohomology of an $E_n$-algebra
 with value in an $E_n$-$A$-module $M$ (Definition~\ref{D:EnHoch}) and $\int_{X} A$
for the factorization homology of $A$ on a framed manifold $X$ (see \S~\ref{SS:FactAlgebraFactHomology}).
 Also $\hcdga$ will be the $(\infty,1)$-category of commutative differential graded $k$-algebras (CDGA for short).
\item Given an $E_n$-algebra $A$, we will write $A\text{-}Mod^{E_n}$ for the  $(\infty,1)$-category of $E_n$-modules over $A$. Similarly, if $B$ is an $E_m$-algebra (with $m\geq n$), we will write  $B\text{-}Mod^{E_n}$ for the  $(\infty,1)$-category of $E_n$-modules over $B$ viewed as an $E_n$-algebra.
\item If $A$ is an $E_n$-algebra ($n\geq 1$) by a left or right module over $A$, we mean a left or right module over $A$ viewed as an $E_1$-algebra. We will denote $A-LMod$ and $A-RMod$ the respective $(\infty,1)$-categories of left and right modules over $A$.
\item If $A$ is a CDGA, and $V$ an $A$-module, we will write $\text{Sym}_A(V)$ for its differential graded symmetric algebra. When $A=k$, we will often simply write $S(V):=\text{Sym}_k(V)$.
\item Unless otherwise stated, we will work in the context of unital algebras.
\end{enumerate}

\setcounter{section}{1}
\section{Preliminaries on $E_n$-algebras and factorization homology}\label{S:operad}
In this section, we briefly recall notions of $(\infty, 1)$-categories, $\infty$-operads and in particular the $E_n$-operad, its algebras, and their modules. There are several equivalent notions of (symmetric monoidal) $(\infty,1)$-categories and the reader should feel free to use its favorite ones. Below, we recall very briefly the model given by the complete Segal spaces and give some examples.

\subsection{$\infty$-categories}\label{S:DKL}
Following~\cite{Re, L-TFT}, an $(\infty,1)$-category is a \textit{complete Segal space}. There is a simplicial closed model category structure, denoted $\SeSp$ on the category of simplicial spaces such that a fibrant object in the $\SeSp$ is precisely a Segal space. Rezk has shown that the category of simplicial spaces has another simplicial closed model structure, denoted $\CSS$, whose fibrant objects are precisely complete Segal spaces ~\cite[Theorem 7.2]{Re}.   Let $\mathbb{R}:\SeSp \to \SeSp$ be a fibrant replacement functor. Let $\widehat{\cdot}: SeSp \to \CSS $,  $X_\bullet \to \widehat{X_\bullet}$, be the completion functor that assigns to a Segal space an equivalent complete Segal space. The composition $X_\bullet \mapsto \widehat{\mathbb{R}(X_\bullet)}$ gives  a fibrant replacement functor $L_{\CSS}$ from simplicial spaces to complete Segal spaces.

Let us explain how to go from a model category to a simplicial space. The standard key idea is to use Dwyer-Kan localization. Let $\mathcal{M}$ be a model category and $\mathcal{W}$ be its subcategory of weak-equivalences. We denote $L^H(\mathcal{M},\mathcal{W})$ its \emph{hammock localization}, see \cite{DK}. One of the main properties of $L^H(\mathcal{M},\mathcal{W})$ is that it is a simplicial category and that the (usual) category $\pi_0(L^H(\mathcal{M},\mathcal{W}))$ is the homotopy category of $\mathcal{M}$. Furthermore,   every weak equivalence has a (weak) inverse in $L^H(\mathcal{M},\mathcal{W})$. When $\mathcal{M}$ is  further a simplicial model category, then for every pair $(x,y)$ of objects $\mathop{Hom}_{L^H(\mathcal{M},\mathcal{W})}(x,y)$ is naturally homotopy equivalent to the derived mapping space $ \mathbb{R}Hom (x,y)$.

It follows that any model category $\mathcal{M}$ gives rise functorially to the simplicial category
$L^H(\mathcal{M},\mathcal{W})$. Taking the nerve $N_\bullet(L^H(\mathcal{M},\mathcal{W}))$ we obtain a simplicial space. Composing with the complete Segal Space replacement functor we get a functor $\mathcal{M}\to L_\infty(\mathcal{M}):= L_{\CSS}(N_\bullet(L^H(\mathcal{M},\mathcal{W})))$ from model categories to $(\infty,1)$-categories (that is complete Segal spaces).

\begin{ex} \label{E:hsset}Applying the above procedure to the model category of simplicial sets $\sset$, we obtain the $(\infty,1)$-category $\hsset$. Similarly from the model category $\cdga$ of commutative differential graded algebras, which we refered to as  CDGAs for short, we obtain the $(\infty,1)$-category $\hcdga$.
 Note that a simplicial space is determined by its $(\infty, 0)$ path groupoid and therefore the category of simplicial sets should be thought of as the $(\infty, 1)$ category of all $(\infty, 0)$ groupoids. Further, the disjoint union of simplicial sets and the tensor products (over $k$) of algebras are monoidal functors which gives $\sset$ and $\cdga$ a structure of monoidal model category (see~\cite{Ho} for example). Thus $\hsset$ and $\hcdga$ also inherit the structure of symmetric monoidal $(\infty,1)$-categories in the sense of~\cite{Re, L-TFT}.

The model category of topological spaces yields the $(\infty,1)$-category $Top_\infty$. Since $\sset$ and $\Top$ are Quillen equivalent~\cite{GoJa, Ho}, their associated $(\infty,1)$-categories are equivalent (as $(\infty,1)$-categories): $\hsset \stackrel{\sim}{\underset{\sim}{\rightleftarrows}} Top_\infty$, where the left and right equivalences are respectively induced by the singular set and geometric realization functors.

One can also consider the pointed versions ${\hsset}_*$ and ${Top_\infty}_*$ of the above $(\infty,1)$-categories (using the model categories of these pointed versions~\cite{Ho}).
\end{ex}

\begin{ex} \label{E:CDGAMod} There are model categories
$A\text{-}Mod$ and $A\text{-}\cdga$ of modules and commutative algebras over a CDGA $A$, thus the above procedure
gives us $(\infty,1)$-categories $A\text{-}Mod_\infty$ and $A\text{-}\cdga_\infty$ and the base changed functor
lifts to an $(\infty,1)$-functor. Further, if $f: A\to B$ is a weak equivalence, the natural
functor $f_*:B\text{-}Mod \to A\text{-}Mod$ induces an equivalence $B\text{-}Mod_\infty \stackrel{\sim}\to A\text{-}Mod_\infty$
of $(\infty,1)$-categories since it is a Quillen equivalence.

Moreover, if $f:A\to B$ is a morphism of CDGAs, it induces a natural functor
$f^*:A\text{-}Mod \to B\text{-}Mod, M\mapsto M\otimes_A B$, which is an equivalence of $(\infty,1)$-categories
when $f$ is a quasi-isomorphism, and is a (weak) inverse of $f_*$ (see~\cite{ToVe} or~\cite{KM}).
Here, we also denote $f^*:A\text{-}Mod_\infty \to B\text{-}Mod_\infty$ and $f_*:B\text{-}Mod_\infty \to A\text{-}Mod_\infty$ the
(derived) functors of $(\infty,1)$-categories induced by $f$.
 Since we are working over a field of characteristic zero, the same results applies to monoids in
 $A\text{-}Mod$ and $B\text{-}Mod$, that is to the categories $A$-$\hcdga$ and $B$-$\hcdga$.

Also, note that if $A,B,C$ are CDGAs and $f:A\to B$, $g:A\to C$ are CDGAs morphisms, we can form
the (homotopy) pushout $D\cong B\otimes^{\mathbb{L}}_{A} C$; let us denote $p: B\to D$ and $q:C\to D$
the natural CDGAs maps. Then, we get the two natural based change $(\infty,1)$-functors
$C\text{-}Mod_\infty \underset{p_*\circ q^*}{\stackrel{f^*\circ g_*}\rightrightarrows} B\text{-}Mod_\infty$. Given any $M\in C\text{-}Mod$,
the natural map $ f^*\circ g_*(M) \to p_*\circ q^*(M)$ is an
equivalence~\cite[Proposition 1.1.0.8]{ToVe}.
\end{ex}

\subsection{$\infty$-operads, $E_n$-algebras} \label{S:EnAlg}

An operad is a special case of a colored operad which itself is a special case of an $\infty$-operad. An \emph{infinity operad} $\mathcal{O}$ is an $\infty$-category together with a functor $\mathcal{O}^\otimes \to N(Fin_\ast)$  satisfying a list of axioms, see \cite{L-HA}. An other (equivalent) approach to $\infty$-operads is given by the dendroidal sets~\cite{CiMo-Dendroidalasinftyoperad}. 

The simplest example of an $\infty$-operad is $N(Fin_\ast) \to N(Fin_\ast)$. This example is the $\infty$-operad associated to the operad $Comm$. In other words $Comm^\otimes =N(Fin_\ast)$. 
\begin{definition}\label{D:EnOperad} The configuration spaces of small $n$-dimensional cubes embedded in a bigger $n$-cube form an operad, $\mathbb E_n$, whose associated $\infty$-operad is denoted by $\mathbb E_n^\otimes$ see \cite{L-HA}.
This example has the same objects as $Fin_\ast$, and we will denote $\mathbb{E}_n^{\otimes}(I,J)$ its spaces of morphisms from $I$ to $J$.  

There is a standard model for this operad given by $\big(\mathcal{C}_n(r)\big)_{r\geq 1}$, the \emph{operad of little $n$-cubes}, see~\cite{Ma, L-HA}, where $\mathcal{C}_n(r)$ is the  configuration space of rectilinear embeddings of $r$-disjoint cubes inside an unit cube. Its singular chain $C_\ast(\mathcal{C}_n(r))$ is a model for the operad governing $E_n$-algebras in $\hkmod$.
\end{definition}
\smallskip

Recall that, for any integer $n\geq 0$,  $\mathbb E_n^\otimes$ denotes the $\infty$-operad of little $n$-cubes.
 By an $\mathbb E_n$-algebra we mean an algebra over the $\infty$-operad $\mathbb E_n^\otimes$.
We will denote $ E_n\text{-Alg}$  the symmetric monoidal
$\infty$-category of $\mathbb E_k$-algebras in the symmetric
monoidal $\infty$-category of differential graded $k$-modules (which is equivalent to the one given in Definition~\ref{D:EnasDisk} below). 

For
any $\mathbb E_n$-algebra $A$, let $A\text{-}Mod^{E_n}$ denote the
symmetric monoidal $\infty$-category of modules over the
 $\mathbb E_n$-algebra $A$. If $\mathcal{C}$ is a symmetric monoidal $(\infty,1)$-category (different from $\hkmod$), \emph{we denote $E_n\text{-Alg}(\mathcal{C})$ for the $(\infty,1)$-category of $E_n$-algebras in $\mathcal{C}$} and similarly $E_n\text{-coAlg}(\mathcal{C})$ for the category of \emph{$E_n$-coalgebras in $\mathcal{C}$}.

\smallskip

 There  are natural maps (sometimes called the stabilization functors)
\begin{equation}\label{eq:towerofEnoperad}
 \mathbb E_0^\otimes \longrightarrow \mathbb E_1^\otimes \longrightarrow \mathbb E_2^\otimes \longrightarrow \cdots
\end{equation} (induced by taking products of cubes with the interval $(-1,1)$).
It is a fact that the colimit of this diagram, denoted by $\mathbb E_\infty^\otimes$
can be identified with $Comm^\otimes$~\cite[Section 5.1]{L-HA}. In particular,
for any $n\in \mathbb{N}-\{0\} \cup \{+\infty\}$,
 the map $\mathbb{E}_1^\otimes \to \mathbb{E}_n^\otimes$ induces a functor $E_n\text{-Alg} \to E_1\text{-Alg}$
which associates to  an $E_n$-algebra its underlying $E_1$-algebra structure.

According to Lurie~\cite{L-HA} (also see~\cite{F,AFT}), we also have an alternative definition for $E_n$-algebras.
\begin{definition}\label{D:EnasDisk}
 The $(\infty,1)$-category of \emph{$E_n$-algebras},
is defined as the ($(\infty,1)$-)category of symmetric monoidal functors
$$
\text{\emph{Fun}}^\otimes(\text{\emph{Disk}}^{fr}_n,\,\hkmod)
$$
where $\text{\emph{Disk}}^{fr}_n$ is the category with objects the integers and morphism the spaces $\text{\emph{Disk}}^{fr}_n(k,\ell):= \text{\emph{Emb}}^{fr}(\coprod_{k} \R^n, \coprod_{\ell} \R^n)$  of framed embeddings of $k$ disjoint copies of a disk $\R^n$ into $\ell$ such copies; the monoidal structure is induced by disjoint union of copies of disks.

We will denote by $Map_{E_{n}\text{-Alg}}(A,B)$ the mapping space
of  $E_n$-algebras maps from $A$ to $B$, \emph{i.e.}, the space of
maps between the associated symmetric monoidal functors.
\end{definition}
In other words,  $E_n\text{-Alg}$ is equivalent to the
$(\infty,1)$-category of $\text{\emph{Disk}}^{fr}_n$-algebras (where
$\text{\emph{Disk}}^{fr}_n$ is equipped with its obvious
$\infty$-operad structure). The tensor products in $\hkmod$ induces
a symmetric monoidal structure on $E_n\text{-Alg}$ as well (which,
for instance,  can be represented by usual Hopf operads such as
those arising from the filtration of the Barratt-Eccles
operad~\cite{BF}).

\begin{ex}[Opposite of an $E_n$-algebra] \label{R:OppositeAlgebra}
 There is a canonical $\mathbb{Z}/2\mathbb{Z}$-action on $E_n\text{-Alg}$ induced by the antipodal map $\tau: \R^n \to \R^n$, $x\mapsto -x$ acting on the source of  $\text{\emph{Fun}}^\otimes(\text{\emph{Disk}}^{fr}_n,\,\hkmod)$. If $A$ is an $E_n$-algebra, then the result of this action $A^{op}:=\tau^*(A)$ is its opposite algebra. If $n=\infty$, the antipodal map is homotopical to the identity so that $A^{op}$ is equivalent to $A$ as an $E_\infty$-algebra.
\end{ex}

\begin{ex}[Singular (co)chains]\label{E:singularchainasEinfty} Let $X$ be a topological space.
Then its \emph{singular cochain complex $C^\ast(X)$ has a natural
structure of an $E_\infty$-algebra}, whose underlying
$E_1$-structure is given by the usual (strictly associative)
cup-product (for instance see~\cite{M2}). Similarly, the
\emph{singular chains $C_\ast(X)$ have a natural structure of
$E_\infty$-coalgebra} which is the predual of $(C^\ast(X),\cup)$.
There are similar explicit constructions for simplicial sets
$X_\com$ instead of spaces, see~\cite{BF}. 

We recall that
$C^\ast(X)$ is the linear dual of the singular chain complex
$C_\ast(X)$ which is the geometric realization (in $\hkmod$) of the
simplicial $k$-module $k[ \Delta_\bullet(X)]$ spanned by the
singular set $\Delta_\bullet(X):=Map(\Delta^\bullet,X)$. Here,
$\Delta^n$ is the standard $n$-dimensional simplex. 

Also note that, for
$E_\infty$-algebras $A$, $B$, the mapping space
$Map_{E_{\infty}\text{-Alg}}(A,B)$ is the (geometric realization of
the) simplicial set $[n]\mapsto Hom_{E_\infty\text{-Alg}}\big(A,
B\otimes C^\ast(\Delta^n)\big)$.
\end{ex}
\begin{rem}\label{R:Einftyistensored}
 The $(\infty,1)$-category $E_\infty\text{-Alg}$ is enriched over $\hsset\cong \hTop$ and has all ($\infty$-)colimits. In particular, it is \emph{tensored} over $\hsset$, see~\cite{Lu11, L-HA} for details on tensored $\infty$-categories (and~\cite{Ke} for the classical theory) or, for instance,~\cite{EKMM, MCSV} in the  $E_\infty$-case (in the context of topologically enriched model categories). We recall that it means that there is a functor $E_\infty\text{-Alg} \times \hsset \to E_\infty\text{-Alg}$, denoted $(A, X_\bullet)\mapsto A\boxtimes X_\bullet$, together with natural equivalences
 $$Map_{E_\infty\text{-Alg}}\big(A\boxtimes X_\bullet, B\big) \; \cong \; Map_{\hsset}\big(X_\com, Map_{E_\infty\text{-Alg}}\big(A, B\big)\big). $$
 Note that the tensor $A\boxtimes X_\bullet$ can be computed as the
colimit $\colim p_{X_\bullet}^A$ where $p_{X_\bullet}^A$ is the
constant map $X_\bullet \to E_\infty\text{-Alg}$ taking value $A$,
for instance see~\cite[Corollary 4.4.4.9]{Lu11}.  Similarly,
$\hcdga$ is tensored over $\hsset$ (and thus $\hTop$ as well).
\end{rem}

 We will use the following fact, which identifies the coproduct in $E_\infty\text{-Alg}$ with the tensor product, to show the Hochschild complex of an $E_\infty$-algebra model over a space $X$ has a natural $E_\infty$ structure.

 \begin{prop}\label{P:tensor=coprod} In the symmetric monoidal $(\infty, 1)$-category ${E_\infty}\text{-Alg}$, the tensor product is a coproduct.
 \end{prop}
 For a proof see Proposition 3.2.4.7 of \cite{L-HA} (or \cite[Part V, Corollary 3.4]{KM}); this essentially follows from the observation that an $E_\infty$-algebra is a commutative monoid in $(\hkmod,\otimes)$, see~\cite{L-HA} or \cite[Section 5.3]{KM}. In particular, Proposition~\ref{P:tensor=coprod} implies that,   for any finite set $I$,  $A^{\otimes I}$ has a natural structure of $E_\infty$-algebras which can be rephrased as
\begin{prop}\label{P:tensorEinftyisEinfty} A symmetric monoidal functor $N(Fin) \to \hkmod$ has a natural lift to an  $\infty$-functor $N(Fin) \to {E_\infty}\text{-Alg}$.
\end{prop}

It follows that, for a finite set $I$, we have  natural
multiplication maps $$A^{\otimes I}
\stackrel{m_A^{(I)}}\longrightarrow A $$ which is a map in
$E_\infty\text{-Alg}$ and is compatible with compositions. We will
simply write $m_A: A\otimes A\to A$ for the $E_\infty$-algebra map
obtained by taking $I=\{0,1\}$. Any $A$-module can be pulled back
along $m_A^{(I)}$ to inherit a canonical $A^{\otimes I}$-module
structure:
\begin{prop}\label{P:tensorEinftyisEinftyMod}
 The maps $m_A^{(I)}$ induced by the functor of Proposition~\ref{P:tensorEinftyisEinfty} yields a  natural (in $A$) ($\infty$-)functors $A\text{-Mod}^{E_\infty} \stackrel{(m_A^{(I)})^*}\to  {A^{\otimes I}}\text{-Mod}^{E_\infty}$ lifting the usual base-change functor for commutative algebras.
\end{prop}

Let $Fin$ and $Fin_\ast$ denote the categories of finite sets and pointed finite sets respectively. There is a forgetful functor $Fin_\ast \to Fin$ forgetting which point is the base point. There is also a functor $Fin \to Fin_\ast$ which adds an extra point called the base point.

 Further, since the $\infty$-operads $\mathbb{E}^{\otimes}_n$ are
coherent (see~\cite{L-HA, L-VI}), the \emph{categories
$A\text{-}Mod^{E_n}$ for $A\in E_n\text{-Alg}$ assembles to form an
$(\infty,1)$-category} of all $E_ n$-algebras and their modules,
denoted $Mod^{E_n}$ (or $Mod^{E_n}(\mathcal{C})$ when we want to
emphasize $\mathcal{C}$).
 The canonical functor $Fin \to Fin_*$ adding a base point yields a canonical functor \footnote{which essentially forget the underlying module}
$\iota: Mod^{E_n}(\mathcal{C})\to Alg_{E_n}(\mathcal C)$ which gives rise,
for any $E_n$-algebra $A$,  to a (homotopy) pullback square:
\begin{equation}\label{eq:pullbackModAlg}
 \xymatrix{  A\text{-}Mod^{E_n} \ar[r] \ar[d] & Mod^{E_n} \ar[d]^{\iota} \\ \{A\}\ar[r] & E_n\text{-Alg}}
\end{equation}
 We refer to~\cite{L-III, L-HA, F} for details. Note that the functor $\iota$ is monoidal. Further, if $A\stackrel{f}\to B$, $A\stackrel{g}\to C$ are two maps of $E_\infty$-algebras, and $M\in B\text{-}Mod^{E_\infty}$ and $N\in  C\text{-}Mod^{E_\infty}$, then
\begin{equation} \label{eq:iotapushout}
\iota\Big( M\mathop{\otimes}_{A}^{\mathbb{L}} N \Big) \quad \cong \quad B \mathop{\otimes}_{A}^{\mathbb{L}} C.
\end{equation}
\begin{ex}\label{E:E1andEinftyModules}
If $n=1$, $A\text{-}Mod^{E_1}$ is equivalent to the $(\infty,1)$-category of $A$-bimodules and if $n=\infty$,  $A\text{-}Mod^{E_\infty}$ is equivalent to the $(\infty,1)$-category of left $A$-modules, see~\cite{L-III, L-HA} (and Proposition~\ref{P:EnMod}, Theorem~\ref{T:lifttoEinfty} below). In general, $A\text{-}Mod^{E_n}$ can be described in terms of factorization homology of $A$, see~\S~\ref{S:FactandMod}.
\end{ex}

\subsection{Factorization algebras and factorization homology}\label{SS:FactAlgebraFactHomology}

\begin{definition}\label{D:FacAlg}Given a topological manifold $M$ of dimension $n$, one can define a colored operad whose objects
are open subsets of $M$ that are homeomorphic to $\mathbb R^n$ and whose morphisms from $\{U_1, \cdots, U_n\}$ to $V$
are empty except  when the $U_i$'s are mutually disjoint subsets of $V$, in which case they are singletons.
The $\infty$-operad associated to this colored operad is denoted by $N(\text{Disk}(M))$, see \cite{L-HA}, Remark 5.2.4.7.
 \end{definition}
 Unfolding the definition we find  that an $N(\text{Disk}(M))$-algebra  on a manifold $M$, with value in chain complexes,
  is a rule that assigns to any open disk\footnote{i.e. an open subset of $M$  homeomorphic to a Euclidean ball}
  $U$ a chain complex $\mathcal{F}(U)$ and, to any finite family of disjoint open disks $U_1,\dots, U_n \subset V$
   included in a disk $V$, a natural map $\mathcal{F}(U_1)\otimes\cdots\otimes \mathcal{F}(U_n) \to \mathcal{F}(V)$.
An $N(\text{Disk}(M))$-algebra is \textbf{locally constant} if for any inclusion of open disks $U\hookrightarrow V$
in $X$, the structure map $\mathcal{F}(U) \to \mathcal{F}(V)$ is a quasi-isomorphism (see~\cite{L-HA, L-VI}) .

Locally constant $N(\text{Disk}(M))$-algebra are actually (homotopy) \emph{locally constant factorization algebras}
 in the sense of Costello~\cite{CG, Co}, see Remark~\ref{R:FactAlgisFactAlg} below.

A \emph{locally constant factorization coalgebra} is an
$N(\text{Disk}(M))$-coalgebra such that for any inclusion of open
disks $U\hookrightarrow V$ in $X$, the structure map $\mathcal{F}(V)
\to \mathcal{F}(U)$ is a quasi-isomorphism.

\noindent \textbf{Notation: } we   denote \emph{$\text{Fac}^{lc}_M$
the  $(\infty,1)$-category of  locally constant
$N(\text{Disk}(M))$-algebras} (see~\cite{CG, G-Houches}). We will
also denote $N(\text{Disk}(M))\text{-Alg}$ the $(\infty,1)$-category
of $N(\text{Disk}(M))$-algebras (that is of prefactorization
algebras).

\medskip

If $\mathcal{A}$ is a locally constant $N(\text{Disk}(M))$-algebra, the rule which to an open disk $D$ associates
the chain complex $\mathcal{A}(D)$ can be extended to any open set of $M$. In fact, Lurie has proved~\cite{L-HA, L-VI}
that the functor $\text{Disk}(M)\stackrel{\mathcal{A}}\longrightarrow \hkmod$ has a left Kan extension along
the embedding $\text{Disk}(M)\hookrightarrow Op(M)$ where $Op(M)$ is the standard ($(\infty,1)$-)category of open subsets
of $M$, \emph{i.e.},  with objects the open subsets of $M$ and morphism  from $U$ to $V$ are empty unless
when $U \subset V$ in which case they are singletons.

\begin{definition}\label{D:FactHomology} Let $M$ be a topological manifold and $\mathcal{A}$ be  a locally constant
 factorization algebra.

\emph{Factorization homology} is the  ($\infty,1)$-functor $Op(M)\otimes \text{Fac}_{M}^{lc}\to \hkmod$,
denoted $(M,\mathcal{A})\mapsto \int_M \mathcal{A}$, given by the left Kan extension\footnote{the existence of this extension is a non-trivial Theorem of~\cite{L-HA}} of
 $\text{Disk}(M)\stackrel{\mathcal{A}}\longrightarrow \hkmod$. We say that $\int_M \mathcal{A}$ is the factorization
  homology of $M$ with values in $\mathcal{A}$.
\end{definition}
\begin{rem}
Factorization homology is also called topological chiral homology in~\cite{L-HA, L-TFT, L-VI}.
We prefer Francis terminology~\cite{F, AFT, F2} which is justified by the fact that factorization homology is actually
the homology (or said otherwise derived sections) of factorizations algebras in the sense of Costello~\cite{CG}
as we proved~\cite{GTZ2}, see Remark~\ref{R:FactAlgisFactAlg} below.
\end{rem}
\begin{ex}[Framed manifolds]\label{ex:framedisEn}
Let $M$ be a framed manifold, then any $E_n$-algebra determines a
locally constant factorization algebra on $M$ (for instance,
see~\cite{L-HA, F, GTZ2, G-Houches} or Theorem~\ref{T:Theorem6GTZ2}) so that one can define the
factorization homology $\int_M A$.   These locally constant
factorization algebras are essentially constant, in the sense that
they satisfy the property that there is a (globally defined)
$E_n$-algebra $A$
 together with  natural (with respect to the structure map of the factorization algebra) quasi-isomorphism
 $\mathcal{A}(D) \stackrel{\simeq}\to A$ for every disk $D$. Thus,  we call such a factorization algebra the \emph{constant factorization algebra}
 on $M$ associated to $A$. For instance,   for $n=0,1,3,7$,  there is a faithful embedding of $E_n$-algebras into   locally constant factorization algebras over the $n$-sphere $S^n$.

 For $M=\mathbb{R}^n$, one
actually gets an equivalence between all locally constant
factorization algebras over $M=\mathbb{R}^n$ and $E_n$-algebras, see
Theorem~\ref{P:En=Fact} below. 
\end{ex}
\begin{ex}\label{E:EinftygivesNDisk}
The canonical map $N(\text{Disk}(M))\to N(Fin_\ast)$ shows that any $E_\infty$-algebras has a canonical structure of
 prefactorization algebra on any topological manifold $M$ which turns out to actually be a locally constant factorization
 algebra. This construction is studied in detail (using the Hochschild chain models) in~\cite{GTZ2}
 and actually extends to define a factorization algebra on any
 $CW$-complex, \emph{a fortiori} to any manifold with corners.
\end{ex}

\begin{rem}\label{R:FactAlgisFactAlg}
Let us justify a bit more the terminology of locally constant
factorization algebras we are using (hoping it will avoid any
possible confusion, also see \cite[\S~4.2]{G-Houches}). The notion
of locally constant $N(\text{Disk}(X))$-algebra is actually
equivalent to the \lq\lq{}full\rq\rq{} notion of a  locally constant
factorization algebra on $X$ in the sense of Costello~\cite{CG, Co}
which are a similar construction where the $U_i$ are allowed to be
any open subsets, satisfying a kind of
\lq\lq{}\v{C}ech/cosheaf-like\rq\rq{} condition (and still being
locally constant in the above sense). Let us now be more precise.
\begin{definition}\label{D:PrefactAlgebras}
A \textbf{prefactorization} algebra is an algebra over the colored operad whose objects are open subsets of $X$
and whose morphisms from $\{U_1, \cdots, U_n\}$ to $V$ are empty unless  when $U_i$'s are mutually disjoint subsets of $U$,
 in which case they are singletons.
\end{definition}
The above definition makes sense for \emph{any}  topological space
$X$.
  Unfolding the definition, we find  that a prefactorization algebra  on $X$, with value in chain complexes,
  is a rule that assigns to any open set $U$ a chain complex $\mathcal{F}(U)$ and, to any finite family
  of pairwise disjoint open sets $U_1,\dots, U_n \subset V$ included in an open $V$, a natural map
  $\mathcal{F}(U_1)\otimes\cdots\otimes \mathcal{F}(U_n) \to \mathcal{F}(V)$.
  These structure maps are required to satisfy obvious associativity and symmetry conditions, see~\cite{CG}.
   They allow us to define  \lq\lq{}\v{C}ech-complexes\rq\rq{} associated to a cover $\mathcal{U}$ of $U$.
   Denoting $P\mathcal{U}$ the set of finite pairwise disjoint open subsets $\{U_1,\dots,U_n \, ,\, U_i\in \mathcal{U}\}$,
     it is, by definition the
chain (bi-)complex
$$\check{C}(\mathcal{U},\mathcal{F})= \bigoplus_{P\mathcal{U}} \mathcal{F}(U_1) \otimes \cdots \otimes \mathcal{F}(U_n)
\leftarrow \bigoplus_{P\mathcal{U} \times P\mathcal{U}} \mathcal{F}({U_1}\cap {V_1}) \otimes \cdots \otimes
\mathcal{F}({U_n}\cap {V_m})  \leftarrow \cdots$$
where the horizontal arrows are induced by the alternating sum of the natural inclusions as for the usual \v{C}ech complex
 of a cosheaf (see~\cite{CG} or \cite{GTZ2, G-Houches}).
The prefactorization algebra structure also induce a canonical map $\check{C}(\mathcal{U},\mathcal{F})\to \mathcal{F}(U)$.

\begin{definition}\label{D:FactAlgebraGlobal}\begin{itemize}
\item A prefactorization algebra $\mathcal{F}$ on $X$  is said to be a
\textbf{factorization algebra} (in the sense of~\cite{CG}) if, for
all open subsets $U\in Op(X)$ and for every factorizing
 cover\footnote{an open cover of $\mathcal{U}$  is {factorizing} if,  for all finite collections $x_1,\dots, x_n$
 of distinct points in $U$, there are pairwise disjoint open subsets $U_1,\dots, U_k$ in $\mathcal{U}$
 such that $\{x_1,\dots, x_n\} \subset \bigcup_{i=1}^k U_i$} $\mathcal{U}$ of $U$, the canonical map
$\check{C}(\mathcal{U},\mathcal{F})\to \mathcal{F}(U)$ is a quasi-isomorphism (see~\cite{Co,CG}).
Again, a factorization algebra is  \textbf{\emph{locally constant}} if for any inclusion of open disks
$U\hookrightarrow V$  in $X$, the structure map $\mathcal{F}(U) \to \mathcal{F}(V)$ is a quasi-isomorphism.

\item
In the above Definition, one can replace the symmetric monoidal
$\infty$-category of chain complexes by any symmetric monoidal
$\infty$-category $(\mathcal{C}, \otimes)$ which has all colimits,
sifted limits and such that geometric realization (see~\cite{Lu11})
distributes with respect to the monoidal structure, see~\cite{CG,
G-Houches}.

\item A (resp. locally constant) \emph{factorization coalgebra} on $X$ with value in
$(\mathcal{C}, \otimes)$ is a (resp. locally constant) factorization
algebra on $X$ with value in the opposite category
$(\mathcal{C}^{op}, \otimes)$, that is an object of
$\text{Fac}_X(\mathcal{C}^{op})$ (resp.
$\text{Fac}_X^{lc}(\mathcal{C}^{op})$).
\end{itemize}
\noindent \textbf{Notation: } we   denote
\emph{$\text{Fac}_X(\mathcal{C})$ the  $(\infty,1)$-category of
locally constant factorization algebras on $X$ with values in
$(\mathcal{C},\otimes)$} and $\text{Fac}^{lc}_X(\mathcal{C})$ its
$(\infty,1)$-subcategory of locally constant factorization algebras. We also denote
$\text{coFac}_X(\mathcal{C})$ (resp. $\text{coFac}^{lc}_X(\mathcal{C})$) the $(\infty,1)$-categories of (resp. locally constant) factorization coalgebras. 
\end{definition}

\smallskip

In~\cite{GTZ2}, we proved
\begin{theorem}[{\cite[Theorem 6]{GTZ2}}]\label{T:Theorem6GTZ2}
  The functor $(U,\mathcal{A})\mapsto \int_U \mathcal{A}$ induces an equivalence of $(\infty,1)$-categories
  between locally constant $N(\text{\emph{Disk}}(X))$-algebra and locally constant factorization algebras
  on the manifold $X$ in the sense of~\cite{CG}. Further this functor is (equivalent to) the identity functor
  when restricted to open disks.
\end{theorem} This justifies our terminology of locally constant factorization algebras
and factorization homology; further, the extension  on any open set $U$ of a
(locally constant) $N(\text{Disk}(X))$-algebra $\mathcal{A}$ is precisely given by the factorization homology
$\int_{U} \mathcal{A}$, see \emph{loc. cit.}. \end{rem}

\begin{ex}[Trivial example: the unit factorization algebra $k$] \label{E:kasFact}
The unit object of the symmetric monoidal category of factorization
algebras over a CW-complex or topological manifold with corners $X$
is the trivial factorization algebra associated to the ground ring
$k$. 
We review it here for latter use. We will simply denote it
by $k$.

 Its prefactorization
algebra structure is given by $k(U):=k$ for any open set $U\subset
X$ and the structure maps are given, for any pairwise disjoint open subsets
$U_1,\dots, U_r$ of an open $V\subset X$, by
$$k(U_1)\otimes \cdots \otimes k(U_r) = \bigotimes_{i=1}^{r} k
\stackrel{\text{multiply}}{\underset{\simeq}\longrightarrow}
k=k(V)$$ where the last map is the multiplication in the ring $k$.
\begin{lem}\label{L:kasFact}The prefactorization algebra $U\mapsto k(U)=k$ is a
factorization algebra. It is further naturally equivalent to the
Hochschild chains (Proposition~\ref{P-CHbifunctorTop}: $CH_U(k)
\stackrel{\simeq}\longrightarrow k(U)=k$
\end{lem}
$k$ is by definition locally constant on any
 manifold and actually a commutative constant
factorization algebra in the terminology of~\cite[\S~4.2]{GTZ2}.
\begin{proof}[Proof of Lemma~\ref{L:kasFact}] The symmetry and associativity axioms of a
prefactorization algebra follows respectively  from the
commutativity and associativity of the ring structure of $k$. Note
that if $X$ is a manifold, $k$ is locally constant by definition and
thus a factorization algebra. Note also that the Hochschild chain
functor (as described in Section~\ref{S:HHforEinftyAlg}) induces a
factorization algebra $U\mapsto CH_U(A)$ for any commutative algebra
$A$ by~\cite[Theorem 4]{GTZ2}. For $A=k$, and any simplicial model
$X_\bullet$ of a space $X$, one has that $CH_{X_\bullet}(k)$ is the
differential graded commutative algebra associated to the constant
simplicial $k$-module $n\mapsto CH_{X_n}^{simp}(k)\cong k$ see
Definition~\ref{D-CH-PROP}. Hence the projection
$CH_{X_\bullet}\twoheadrightarrow CH_{X_0}(k)=k$ is a natural (in
$X_{\bullet}$) quasi-isomorphism of CDGA's and we obtain this way a
natural (in $X$) equivalence $CH_{X}(k)\stackrel{\simeq}\to k$. The
commutative diagram (induced by \cite[Lemma 2]{GTZ2})
$$\xymatrix{ CH_{V}(k) \ar[rrr]^{\simeq} &&& k \\
\bigotimes_{i=1}^{r} CH_{U_i}(k) \ar[u]^{\mu_{U_1,\dots,
U_r,V}}\ar[rrr]^{\simeq} &&& \bigotimes_{i=1}^{r} k
\ar[u]_{\text{multiply}}}$$
 proves that $U\mapsto k(U)$ is
equivalent to $U\mapsto CH_{U}(k)$ and thus is a factorization
algebra since the latter is (of course, the latter can also be
checked by direct inspection rather easily).
\end{proof}
\end{ex}

\begin{rem}[\textbf{Alternative definition: parametrized prefactorization algebras}]\label{R:AlternativeFacAlg}
There is a slight variation of the notion of locally constant (pre)factorization algebras, \emph{i.e.},
locally constant $N(\text{Disk}(M))$-algebras.
Following Lurie~\cite[Remark 5.2.4.8]{L-HA}, we let
\begin{definition}\label{D:AlternativeFacAlg} $N(\text{Disk}(M)')$ be the $\infty$-operad
associated to the colored operad
 $\text{Disk}(M)'$ whose objects are open embeddings $\R^n \hookrightarrow M$ and whose morphisms
  $\text{Disk}(M)'(\phi, \psi)$ are commutative diagrams
$$\xymatrix{ \mathbb{R}^n \ar[rr]^{h} \ar[rd]_{\phi} & & \mathbb{R}^n \ar[ld]^{\psi} \\ & M &}$$
where $f$ is an \emph{open embedding}.

An $N(\text{Disk}(M)')$-algebra $\mathcal{A}$ is said to be \emph{locally constant} if the structure map
 $\mathcal{A}(h):\mathcal{A}(\phi) \to \mathcal{A}(\psi)$ is a quasi-isomorphism for any open embedding $h$ such that
 $\psi\circ h= \phi$  as in the above diagram.
\end{definition}
Note that the (forgetful) functor $\iota:\phi\mapsto \phi(\mathbb{R}^n)$ is an equivalence of categories.
Hence,
\begin{prop}\label{P:AlternativeFacAlg} The functors $\iota^*: {\text{Disk}(M)}^{lc}\text{-Alg}\to
{\text{Disk}(M)'}\text{-Alg}$ and  $\iota^*: \text{Fac}^{lc}_M\to {\text{Disk}(M)'}^{lc}\text{-Alg}$ are equivalences of
 $(\infty,1)$-categories.

 Similarly, the functors $\iota^*: {\text{Disk}(M)}^{lc}\text{-coAlg}\to
{\text{Disk}(M)'}\text{-Alg}$ and  $\iota^*: \text{Fac}^{lc}_M\to
{\text{Disk}(M)'}^{lc}\text{-coAlg}$ are equivalences of
 $(\infty,1)$-categories.
\end{prop}
We will refer to locally constant ${\text{Disk}(M)'}$-algebras as \emph{locally constant parametrized factorization
algebras}. By the above proposition~\ref{P:AlternativeFacAlg}, the two notions of factorizations algebras are essentially
the same.

Unfolding Definition~\ref{D:AlternativeFacAlg}, we see that a locally constant parametrized factorization algebra
$\mathcal{F}$  is thus a rule  which associates to each
embedding $\phi: \mathbb{R}^n \to M$ a chain complex $\mathcal{F}(\phi)$ with natural maps
$\mathcal{F}(\phi_1)\otimes \cdots\otimes \mathcal{F}(\phi_r) \to \mathcal{F}(\psi)$ associated to any open embedding
 $h: \coprod_{i=1}^r \R^n \to \R^n$ such that $\psi\circ h= \coprod_{i=1}^r \phi_i: \coprod_{i=1}^r \R^n \to M$
 (satisfying the obvious associativity and symmetry conditions). Further, for any $h:\phi \mapsto \psi $ (\emph{i.e.}
$\psi \circ h =\phi$), the structure map $\mathcal{F}(\phi)\to \mathcal{F}(\psi)$ is required to be a quasi-isomorphism.
\end{rem}
\medskip

\begin{definition}
Let $U$ be an open subset of $X$. By restricting to open subsets of $U$,  a (locally constant) factorization algebra
$\mathcal{A}$ on $X$ has a canonical restriction $\mathcal{A}_{|U}: Op(U)\ni V\mapsto \mathcal{A}(V)$ to a
(locally constant) factorization algebra on $U$.
\end{definition}

Factorization algebras satisfy a local-to-global property and can thus be defined out of their restriction to a basis
or descent/gluing data.
Indeed, let $\mathcal{U}$ be a \emph{basis} for the topology of a manifold $M$ which is stable by finite intersections
 and  is also a factorizing cover (as in Remark~\ref{R:FactAlgisFactAlg}).
\begin{definition}\label{D:UPrefactAlg}A \emph{$\mathcal{U}$-prefactorization algebra} is defined
similarly to a prefactorization algebra on $M$, except that we are
only considering opens that belongs to $\mathcal{U}$ in the
definition. In other words, it is an algebra over
$N\big(\text{Disk}_{\mathcal{U}}(M)\big)$ where
$\text{Disk}_{\mathcal{U}}(M)$ is the colored sub-operad of
$\text{Disk}(M)$ whose only colors are those consisting of opens in
$\mathcal{U}$.

Similarly, a \emph{$\mathcal{U}$-factorization algebra} is defined
similarly to a factorization algebra on $M$, except that we are only
considering opens that belongs to $\mathcal{U}$ in the definition
(in other words, it is a $\mathcal{U}$-prefactorization algebra
satisfying the descent condition of
Remark~\ref{R:FactAlgisFactAlg}).
\end{definition}
We refer to\cite{CG, G-Houches} for more details.
\begin{prop}[Costello-Gwilliam~\cite{CG}]\label{P:extensionfrombasis} Let $\mathcal{F}$ be a $\mathcal{U}$-factorization algebra.
There is an unique\footnote{up to natural equivalence}  factorization algebra $i_*^{\mathcal{U}}(\mathcal{F})$  on $M$ extending $\mathcal{F}$\footnote{More precisely, for an open set $V\subset X$, one has
$i_*^{\mathcal{U}}(\mathcal{F})(V) \,:=\, \check{C}(\mathcal{U}_V, \mathcal{F})$
where $\mathcal{U}_V$ is the open cover of $V$ consisting of all open subsets of $V$ which are in $\mathcal{U}$} (that is equipped with a quasi-isomorphism of $\mathcal{U}$-factorization algebras $i_*^{\mathcal{U}}(\mathcal{F})\to \mathcal{F}$). 

In fact the canonical functor $\text{Fac}_{\mathcal{U}}(\mathcal{C})\to \text{Fac}_{M}(\mathcal{C})$ is an equivalence of $\infty$-categories.
\end{prop}

More generally,  if $\mathcal{A}$ is a factorization algebra on $X$ and  $\mathcal{U}=(U_i)_{i\in I}$ is a cover of $X$,
then  $\mathcal{A}$  can be uniquely recovered by the data of the factorization algebras $\mathcal{A}_{|U_i}$ restricted to the $U_i$'s (thanks to the \v{C}ech condition applied to suitable covers).
In fact, any family of factorization algebras $\mathcal{F}_i$ on $U_i$, satisfying natural compatibility conditions on the intersections of the $U_i$'s, extends uniquely into a factorization algebra on $X$;
we refer to Costello-Gwilliam~\cite[Section 4]{CG} (and~\cite[\S~5]{G-Houches}) for details on this descent property of factorization algebras.

\subsection{$E_n$-algebras as factorization algebras}
\label{S:EnasFact}
Theorem 5.2.4.9 of \cite{L-HA} (also see~\cite{L-VI, GTZ2}) provides an equivalence  between $E_n$-algebras and
 \emph{locally constant factorization algebra} on the open disk $D^n$:
\begin{theorem}\label{P:En=Fact} Let $\mathcal{C}$ be a symmetric monoidal $(\infty,1)$-category.
There is a natural equivalence of $(\infty,1)$-categories
$$E_n\text{-Alg}(\mathcal{C}) \;\cong \;\text{Fac}^{lc}_{\R^n}(\mathcal{C}).$$
Similarly there is an equivalence between the
$(\infty,1)$-categories of locally constant factorization coalgebras
on $\R^n$ and the one of $E_n$-coalgebras.
\end{theorem}
In particular, an $E_n$-algebra can be seen as an $n$-dimensional
(topological) field theory (over the space-time manifold
$\mathbb{R}^n$), providing an invariant for framed $n$-manifolds
which is precisely computed by factorization homology.

\medskip

 Let $X$, $Y$ be topological spaces and $f:X\to Y$ be continuous. There is a \emph{pushforward functor}
 \begin{equation} \label{eq:DefPushforward} f_*: \text{Fac}_X(\mathcal{C})\to \text{Fac}_Y(\mathcal{C})\end{equation} that was constructed in~\cite{CG}. It is given, for $\mathcal{F}\in \text{Fac}_X(\mathcal{C})$ and $V\in Op(Y)$,  by $f_*(\mathcal{F})(V) =\mathcal{F}(f^{-1}(V))$.  
 In particular, let $p: M\to pt$  be the canonical map from $M$ to a point and $\mathcal{F}$ be a factorization algebra on $M$.
 
 \smallskip
 
\noindent \textbf{Notation:} we also denote   $p_*(\mathcal{F})$ the object $p_*(\mathcal{F})(pt)\cong \mathcal{F}(M)$. 
 If $\mathcal{F}$ is locally constant, we have natural equivalences
 \begin{equation} \label{eq:FacHomologyispushforward}
  p_*(\mathcal{F}) \, \cong \, \mathcal{F}(M) \, \cong \, \int_{M} \mathcal{F}.
 \end{equation}

\smallskip

Assume $X$, $Y$ are manifolds and let $\pi: X\times Y \to X$ be the canonical projection. 
Then, there is  a factorization of the  pushforward functor:
$$\xymatrix{\text{Fac}^{lc}_{X\times Y}(\mathcal{C}) \ar@{^{(}->}[d] \ar[rr]^{\pi_*} & &
\text{Fac}^{lc}_{X}\big(\text{Fac}^{lc}_Y(\mathcal{C})\big) \ar@{^{(}->}[d] \ar[rr]^{\quad (Y\to pt)_*} & &
\text{Fac}^{lc}_{X}(\mathcal{C}) \ar@{^{(}->}[d]\\  \text{Fac}_{X\times Y} (\mathcal{C}) \ar[rr]^{\pi_*} & &
\text{Fac}_{X}\big(\text{Fac}_{Y}(\mathcal{C})\big) \ar[rr]^{\quad(Y\to pt)_*} & &
\text{Fac}_{X}(\mathcal{C})}$$ where the vertical arrows are induced by the fully faithfull inclusion of the locally constant factorization algebras inside all factorization algebras.  The right horizontal arrows are induced by the pushforward along the canonical map $Y\to pt$. 
Here the fact that the pushforward of a
locally constant factorization algebra is locally constant follows
from the fact that the fibers are all naturally identified with the
same manifold $Y$. We refer to~\cite{G-Houches} for more details. 

\smallskip

The fact that locally constant factorization
algebras on $\R^n$ are $E_n$-algebras implies that, when $Y=\R^n$,
the pushforward factors through a functor $\pi_*:
\text{Fac}^{lc}_{X\times \R^n}\to
\text{Fac}^{lc}_{X}(E_n\text{-Alg})$ see~\cite{GTZ2}. In particular,
we can take $X=\R^m$. The following $\infty$-category version of
Dunn's Theorem was proved by Lurie~\cite{L-HA} (and~\cite{GTZ2} for
the pushforward interpretation):
\begin{theorem}[Dunn's Theorem]\label{T:Dunn}
There is an equivalence of $(\infty,1)$-categories
$E_{m+n}\text{Alg} \cong E_n\text{-Alg}(E_m\text{-Alg})$.

Under the
equivalence $E_n\text{-Alg} \cong Fac_{\R^n}^{lc}$
(Theorem~\ref{P:En=Fact}), the above equivalence is realized by the
pushforward $\pi_*:\text{Fac}_{\R^m\times \R^n}^{lc} \to
\text{Fac}_{\R^m}^{lc}(E_n\text{-Alg})$ associated to the canonical
projection $\pi:\R^m\times \R^n\to \R^m$.
\end{theorem}

\section{Higher Hochschild (co)chains for $E_\infty$-algebras} \label{S:HHforEinftyAlg}

In this section we define and study higher Hochschild (co)chains modeled over spaces for $E_\infty$-algebras with values in $E_\infty$-modules.

\subsection{Factorization homology of $E_\infty$-algebras and higher Hochschild chains}\label{SS:FactandHHforEinfty}

Factorization homology with values in $E_\infty$-algebras has special properties. It becomes an homotopy invariant and can be defined over any space, providing an homology theory for spaces.
Indeed, it identifies with the Hochschild chains modeled on a
 space and with values in an $E_\infty$-algebra, see Theorem~\ref{T:CH=TCH} below.
We will denote $(X,A)\mapsto CH_X(A)$ the latter construction which we explain further in this section.
The reader familiar with factorization homology for commutative algebras can skip it and keep in mind that  $ CH_X(A)$ means factorization homology extended to spaces.

\smallskip

 The Hochschild chains $CH_{X_\com, \com}(A,A)$ over a simplicial set $X_\bullet$ with value in a CDGA $A$ is defined in~\cite{P}. As explained in~\cite{GTZ2}, it can be defined using the PROPic definition of commutative (differential graded) k-algebras as follows. A CDGA over $k$ is a  strict symmetric monoidal functor $A: Fin \to k\text{-Mod} $ from the category of finite sets (with disjoint union for the monoidal structure) to the category of chain complexes (with tensor product as the monoidal structure). Given a finite simplicial set $\Delta^{op}\to Fin$ we can compose these two functors to get a simplicial $k$-module. The geometric realization of this simplicial $k$-module is the Hochschild complex modeled on $X$. In fact one can do better. A strictly symmetric monoidal functor $A: Fin \to k\text{-Mod} $ has a canonical lift to $A: Fin \to k\text{-Alg} $, along the forgetful functor $k\text{-Alg}\to k\text{-Mod}$. This is due the simple fact that for a commutative algebra $A$, the multiplication $m: A\otimes A\to A$ is a map of algebras or, said otherwise, the fact that the tensor product is a coproduct in CDGA. Thus, the above procedure gives rise to a simplicial CDGA whose geometric realization is the Hochschild complex modeled on $X_\bullet$, with a canonical CDGA structure.    In the classical example $X=S^1$, whereby we get the classical Hochschild complex, this CDGA structure is given by the shuffle product on $CH_\bullet(A, A)$. Another way of saying this is to say that   that $\cdga$ is tensored over simplicial sets and that Hochschild chains over $X_\bullet$ with values in $A$ is simply the tensor $A\boxtimes X_\bullet$ (see \cite{L-HA, Lu11, EKMM, Ke, MCSV} for details on tensored categories over spaces and Remark~\ref{R:Einftyistensored}).
 We now discuss how all of the above generalizes to the case of $E_\infty$-algebras.

\smallskip

Recall from Section~\ref{S:operad} that the $(\infty,1)$-category of $E_\infty$-algebras (with value in chain complexes) is equivalent to the  $(\infty,1)$-category $Fun^{\otimes}(N(Fin),\hkmod)$ of (lax) monoidal functors from  (the $\infty$-category associated to) $Fin$ to the $(\infty,1)$-category of chain complexes (see Lurie~\cite{L-III, L-HA, KM}). We denote  $$\Big(X\mapsto CH^{simp}_X(A) \Big) \, \in \, Fun^{\otimes}(N(Fin),\hkmod)$$ the\footnote{it is unique up to contractible choices} monoidal functor associated to   an $E_\infty$-algebra $A$. This functor extends naturally to a functor $Fun^{\otimes}(N(Set),\hkmod)$ (where $Set$ is the category of sets) by taking colimits; that is to say, $X\mapsto \colim_{Fin \ni K\to X}  CH^{simp}_K(A)$.

By Proposition~\ref{P:tensorEinftyisEinfty}, $CH^{simp}_K(A)$ has a natural structure of $E_\infty$-algebras; more precisely, the functor $\big(X\mapsto CH_X^{simp}(A)\big)$ factors as  $$N(Fin) \longrightarrow E_\infty\text{-Alg} \xrightarrow{forget} \hkmod.$$ 
We  simply denote $\big(X\mapsto CH_X^{simp}(A)\big)\in Fun^{\otimes}(N(Set), E_\infty\text{-Alg})$ the induced lift. 

\begin{rem}
Fixing a set $X$, $CH_X(A)$  is  (functorially) quasi-isomorphic to the tensor product $A^{\otimes X}$ (where $A$ is viewed as a chain complex). Note that this construction (of the underlying chain complex structure) is  the same as the one in\cite[Section 2.1]{GTZ} in the case of CDGAs. However the functorial structure involves higher homotopies and not only the multiplication and seems difficult to write explicitly on this particular choice of cochain complex.
\end{rem}

Let $DK: sk\text{-}Mod_\infty \to \hkmod$ be the Dold-Kan functor
from the $(\infty,1)$-category of simplicial $k$-modules to the
chain complexes. The Dold-Kan functor refines to a functor $s
E_\infty\text{-Alg}  \to E_\infty\text{-Alg} $ from simplicial
$E_\infty$-algebras to differential graded $E_\infty$-algebras which
preserves weak-equivalences  (see~\cite[Section 3]{M}).
 \begin{definition}\label{D-CH-PROP}
 The derived Hochschild chains of an $E_\infty$-algebra $A$ and a simplicial set $X_\bullet$ is
 $$CH_{X_\bullet}(A) := DK \Big( \colim_{Fin \ni K\to X_\bullet}  CH^{simp}_K(A)\Big). $$
  \end{definition}

\begin{rem}\label{R:Astrict}
In the case where the $E_\infty$-algebra $A$ is strict, \emph{i.e.} a CDGA, it follows from Corollary~\ref{C:cdgaE} below that $CH_{X_\bullet}(A)$ is quasi-isomorphic to the Hochschild chain complex over $X_\bullet$ described in details in\cite[Section 2.1]{GTZ} (also see\cite{P, G, GTZ2}).
\end{rem}

\begin{prop} \label{P-CHbifunctor}The derived Hochschild chain $(X_\bullet,A) \mapsto CH_{X_\bullet}(A)$ lifts as a functor of $(\infty,1)$-categories
 $$CH: \hsset \times E_\infty\text{-Alg} \longrightarrow E_\infty\text{-Alg} .$$ Further, it is the tensor of $A$ and $X_\bullet$ in $E_\infty\text{-Alg}$, \emph{i.e.}, there is a natural equivalence $CH_{X_\bullet}(A) \cong A\boxtimes X_{\bullet}$. In particular,
\begin{equation}\label{eq:CH=tensored}\text{Map}_{\hsset}\big({X_\bullet},\text{Map}_{E_\infty\text{-Alg}}(A,B)\big) \cong \text{Map}_{E_\infty\text{-Alg}}(CH_{X_\bullet}(A),B). \end{equation}
\end{prop}
Note that we could also just have used the tensor definition $A\boxtimes X_\bullet$  to define higher Hochschild chains.
\begin{proof}[Proof of Proposition~\ref{P-CHbifunctor}]
  Proposition~\ref{P-CH-coeq} below  implies that the derived Hochschild chain functor is invariant under (weak) equivalences of $E_\infty$-algebras and simplicial sets and thus lifts as an ($(\infty,1)$-)functor $\hsset\times E_\infty\text{-Alg} \to \hkmod$, $(X_\bullet,A)\mapsto CH_{X_\bullet}(A)$. Since the tensor products of $E_\infty$-algebras is an $E_\infty$-algebra,  $CH^{simp}_{K_\bullet}(A)$ is a simplicial $E_\infty$-algebra for any simplicial set $K_\bullet$. Since the (refined) Dold-Kan functor  $s E_\infty\text{-Alg}  \to E_\infty\text{-Alg} $ preserves weak-equivalences~\cite{M}, the derived Hochschild chain functor lifts as a functor of $(\infty,1)$-categories $CH: \hsset \times E_\infty\text{-Alg} \to E_\infty\text{-Alg}$. By Proposition~\ref{P:tensor=coprod}, $CH_{X_\bullet}(A) \cong A\boxtimes X_\bullet$ in $sE_\infty\text{-Alg}$ from which the second assertion of the Proposition follows after passing to geometric realization.
\end{proof}

\begin{rem}There is a derived functor interpretation of the above Definition~\ref{D-CH-PROP}. 
Recall that to any simplicial set $X_\bullet$, one can associate a canonical $E_\infty$-coalgebra structure on its chains~\cite{Ma, BF}, denoted $C_\ast(X_\bullet)$ (Example~\ref{E:singularchainasEinfty}). Dually to the case of algebras, an $E_\infty$-coalgebra $C$ defines a  \emph{contravariant} monoidal functor $X\mapsto {CH^{simp}}^{co}_X(C)$, \emph{i.e.}, an object of $Fun^{\otimes}(N(Fin)^{op},\hkmod)$.

\smallskip

In particular, an $E_\infty$-coalgebra $C$ defines a \emph{right} module over the $\infty$-operad $\mathbb{E}_\infty^{\otimes}$ and an  $E_\infty$-algebra a \emph{left} module over the $\infty$-operad $\mathbb{E}_\infty^{\otimes}$. We can thus form their (derived) tensor products $C\mathop{\otimes}\limits_{\mathbb{E}_\infty^{\otimes}}^\mathbb{L} A$ which is computed as a (homotopy) coequalizer:
$$C\mathop{\otimes}_{\mathbb{E}_\infty^{\otimes}}^\mathbb{L} A \cong  \mathop{hocoeq} \Big( \coprod_{f:\{1,\dots, q\} \to \{1,\dots, p\}} C^{\otimes p} \otimes \mathbb{E}_\infty^{\otimes}(q,p)\otimes A^{\otimes q} \rightrightarrows  \coprod_{n} C^{\otimes n}\otimes A^{\otimes n}\Big) $$
where the maps $f:\{1,\dots, q\} \to \{1,\dots, p\}$ are maps of sets. The upper map  in the coequalizer is induced by the maps $f^*:C^{\otimes p} \otimes \mathbb{E}_\infty^{\otimes}(q,p)\otimes A^{\otimes q} \to C^{\otimes q} \otimes A^{\otimes q}$ obtained from the coalgebra structure of $C$ and the lower map is induced by the maps $f_*:C^{\otimes p} \otimes \mathbb{E}_\infty^{\otimes}(q,p)\otimes A^{\otimes q} \to C^{\otimes p} \otimes A^{\otimes p}$ induced by the algebra structure.
\begin{prop}\label{P-CH-coeq} Let $X_\bullet$ be a simplicial set and $A$ be an $E_\infty$-algebra. There is a natural equivalence
 \begin{multline*} CH_{X_\bullet}(A) \cong \hspace{-0.2pc}  C_\ast(X_\com) \mathop{\otimes}_{\mathbb{E}_\infty^{\otimes}}^\mathbb{L} A \\ \cong  \hspace{-0.2pc}\mathop{hocoeq} \Big(\hspace{-0.3pc} \coprod_{f:\{1,\dots, q\} \to \{1,\dots, p\}}\hspace{-1.5pc} C_\ast(X_\bullet)^{\otimes p} \otimes \mathbb{E}_\infty^{\otimes}(q,p)\otimes A^{\otimes q} \rightrightarrows  \coprod_{n} C_\ast(X_\bullet)^{\otimes n}\otimes A^{\otimes n}\Big)\end{multline*}
\end{prop}
\begin{proof}
Note that the $E_\infty$-coalgebra structure on $C_\ast(X_\com)$ is given by the functor $N(Fin_\ast)^{op}\to \hkmod $ defined by $I\mapsto
 k\big[Hom_{Fin}(I, X_\bullet)\big]$. The rest of the proof now is the same as in~\cite[Proposition~4]{GTZ2}.
\end{proof}
\end{rem}

In~\cite{GTZ2}, a functor $CH^{cdga}: \hsset \times \hcdga \to
\hcdga $ was defined\footnote{it was simply denoted $CH$ in loc.
cit.}. There is a forgetful functor $\hcdga \to
E_\infty\text{-Alg}$. Proposition~\ref{P-CH-coeq}, Proposition~4
in~\cite{GTZ2} and the equivalence $\mathbb{E}_\infty^{\otimes}
\stackrel{\simeq}\to Comm^{\otimes}$ yield
\begin{cor} \label{C:cdgaE} If $A$ is a commutative differential graded algebra, the following diagram is commutative in the
$(\infty,1)$-category   $Fun(\hsset\times \hcdga,
E_\infty\text{-Alg})$:
 \begin{eqnarray*}
  \xymatrix{\hsset\times \hcdga  \ar[r]^{\quad CH^{cdga}} \ar[d]  & \hcdga \ar[d] \\ \hsset\times E_\infty\text{-Alg} \ar[r]^{\quad CH}&  E_\infty\text{-Alg}  }
 \end{eqnarray*}
In particular, $CH^{cdga}_{X_\bullet}(A)$ is naturally equivalent to $CH_{X_\bullet}(A)$.
\end{cor}
In other words, the corollary means that the functors $CH^{cdga}$ and $CH$ are equivalent (for a CDGA).
\begin{rem}
 In the sequel we will use the equivalence given by corollary~\ref{C:cdgaE} to identify the functors $CH$ and $CH^{cdga}$ without further notice.
\end{rem}

There is an equivalence of $(\infty,1)$-categories
$\hsset  \stackrel{\sim} {\underset{\sim}{\rightleftarrows}} Top_\infty$ induced by the underlying Quillen equivalence between $\sset$ and $\Top$~\cite{GoJa, Ho}. The left and right equivalences are respectively induced by the standard singular set functor $X\mapsto S_\bullet(X):=\mathop{Map}(\Delta^{\bullet}, X)$ and geometric realization $X_\bullet\mapsto |X_\bullet|$ functors.
In particular, we can replace simplicial sets by \emph{topological spaces} in Definition~\ref{D-CH-PROP} and Proposition~\ref{P-CH-coeq} to get the following analogue of Proposition~\ref{P-CHbifunctor}. Letting $C_\ast(X)$ be the natural $E_\infty$-coalgebra structure on the singular chains of $X$, we deduce from Proposition~\ref{P-CHbifunctor} and Proposition~\ref{P-CH-coeq}:
\begin{prop}\label{P-CHbifunctorTop}
 The derived \emph{Hochschild chain with value in an $E_\infty$-algebra $A$ modeled on a space $X$} given by
 \begin{eqnarray*}
  CH_X(A)&:=&  DK \Big( \colim_{Fin \ni K\to S_\bullet(X)}  CH^{simp}_K(A)\Big)\\
  &\cong & C_\ast(X) \mathop{\otimes}\limits_{\mathbb{E}_\infty^{\otimes}}^{\mathbb{L}} A
 \end{eqnarray*}
 induces a $(\infty,1)$-functor    $CH: (X,A) \mapsto CH_{X}(A)$ from $\hTop \times E_\infty\text{-Alg} $ to $E_\infty\text{-Alg}$. Further, one has a natural equivalence $A\boxtimes X \cong CH_X(A)$; in particular
 \begin{equation}\label{eq:CH=tensoredTop}\text{Map}_{\hTop}\big({X},\text{Map}_{E_\infty\text{-Alg}}(A,B)\big) \cong \text{Map}_{E_\infty\text{-Alg}}(CH_{X}(A),B). \end{equation}
\end{prop}
\begin{rem}\label{R:SpaceActiononCH}
 Since $(X,A)\mapsto CH_X(A)\cong A\boxtimes X$ is a functor of both variables, $CH_X(A)$ has a natural action of the topological monoid $Map_{\hTop}(X,X)$ (and in particular of the group $Homeo(X)$). 
 This means that there is a monoid map $Map_{\hTop}(X,X)\to Map_{E_\infty\text{-Alg}}(CH_X(A), CH_X(A))$; in other words  a chain map $C_\ast\big(Map_{\hTop}(X,X)\big)\otimes CH_X(A)\to CH_X(A)$ which makes $CH_X(A)$ a $Map_{\hTop}(X,X)$-algebra in $E_\infty\text{-Alg}$ (for the monad associated to the monoid $Map_{\hTop}(X,X)$).

 \smallskip

 Similarly, given any map $f:X\times K\to Y$ of topological spaces, we get  a
 canonical map $\tilde{f}:K\to Map_{E_\infty\text{-Alg}}(CH_X(A),CH_Y(A))$ in $\hTop$ as follows.
 By Proposition~\ref{P-CHbifunctorTop}, the map $f_*:CH_{X\times K}(A)\to CH_{Y}(A)$ yields a natural map of mapping spaces:
 \begin{multline}\label{eq:DefofTransffromspacetoCH}
  \tau_f: Map_{E_\infty\text{-Alg}}\big(CH_{Y}(A), CH_Y(A)\big) \\ \stackrel{(f_*)^*}\longrightarrow Map_{E_\infty\text{-Alg}}\big(CH_{X\times K}(A), CH_Y(A)\big)\\ \cong Map_{\hTop}\big(K, Map_{E_\infty\text{-Alg}}(CH_X(A), CH_Y(A))\big)
 \end{multline}
where the last equivalence follows from Proposition~\ref{P-CHbifunctorTop}~\eqref{eq:CH=tensoredTop} and Corollary~\ref{C:properties}.(4) below.  Applying the map~\eqref{eq:DefofTransffromspacetoCH} to the identity $id_{CH_Y(A)}$ we get the map
\begin{equation}\label{eq:deftilde} \tilde{f}:= \tau_f(id_{CH_Y(A)}) \in Map_{\hTop}\big(K, Map_{E_\infty\text{-Alg}}(CH_X(A), CH_Y(A))\big).\end{equation}
In particular, the map $\tilde{f}$ yields a map \begin{equation}\label{eq:deftilde*}\tilde{f}_*: C_\ast(K)\otimes CH_X(A)\to CH_Y(A)\end{equation}                                                                                                                                    in $\hkmod$.  

\medskip

The map $\tilde{f}$ (respectively $\tilde{f}_*$) are functorial in the obvious sense.
 Indeed, let $f:K\times X\to Y$, $g:K\times Y\to Z$, $j:L\to K$ be continuous maps. Out of $f$ and $g$ we can form  the map $$K\times X \stackrel{(p_K, f)} \longrightarrow K\times Y\stackrel{g}\longrightarrow Z$$ where $p_K: K\times X\to K$ is the canonical projection while out of $j$ and $f$, we can form the composition $$L\times X \stackrel{j\times id_X}\longrightarrow K\times X \stackrel{f}\longrightarrow Y.$$
 We thus  get  the maps $\tilde{f}$, $\tilde{g}$ as well as $$\widetilde{(f\circ (j\times id_X))}: L\to Map_{E_\infty\text{-Alg}}(CH_X(A), CH_Y(A))\big)$$ and $\widetilde{(g\circ(p_K,f))}: K\to  Map_{E_\infty\text{-Alg}}(CH_X(A), CH_Z(A))\big)$.
 The functoriality relation are given by:
\begin{prop}\label{P:Functorialitytildemap}
  The following two diagrams
 $$ \xymatrix{K \ar[rr]^{(\tilde{f}, \tilde{g})\hspace{12pc}} \ar@/_/[rrd]_{\widetilde{(g\circ (p_K,f))}\hspace{4pc}} &&  Map_{E_\infty\text{-Alg}}(CH_X(A), CH_Y(A)) \times Map_{E_\infty\text{-Alg}}(CH_Y(A), CH_Z(A)) \ar[d]^{-\circ -}  \\  && Map_{E_\infty\text{-Alg}}(CH_X(A), CH_Z(A)) },$$ where the vertical arrow is the composition of morphisms, and
 $$ \xymatrix{L \ar[d]_{j} \ar[rr]^{\widetilde{(f\circ (j\times id_X))}\hspace{5pc}}    & & Map_{E_\infty\text{-Alg}}(CH_X(A), CH_Y(A)) \\ K \ar[rru]_{\tilde{f}} &&  }$$
 are commutative.
\end{prop}
\begin{proof}The result follows from the following two factorizations
$$(f\circ (j\times id_X))_*:CH_{L\times X}(A)\stackrel{(j\times id_X)_*}\longrightarrow CH_{K\times X}(A) \stackrel{f_*}\longrightarrow CH_{Y}(A), $$
$$(g\circ (p_K,f))_*: CH_{K\times X}(A)\stackrel{(p_K\times f)_*}\longrightarrow CH_{K\times Y}(A) \stackrel{g_*}\longrightarrow CH_{Z}(A) $$
which in turn follow from Proposition~\ref{P-CHbifunctorTop}.
\end{proof}
\begin{ex}\label{ex:squarecommutestildemap}
Consider a commutative diagram of spaces
\begin{equation*}
\xymatrix{L\times K\times X \ar[rr]^{p_L\times f} \ar[d]_{(q\times id_X)}&&  L\times Y \ar[d]^{g} \\ R\times X \ar[rr]^{h} && Z }
\end{equation*} where $p_L:L\times K\times X\to L$ is the canonical projection, $f:K\times X\to Y$ and  $q:L\times K\to R$ are continuous maps.
Then by Proposition~\ref{P:Functorialitytildemap}, we get that the following diagram is commutative
\begin{equation*}
\xymatrix{C_\ast(L\times K) \otimes CH_{X}(A) \ar[rr]^{ \tilde{f}_*}
\ar[d]_{q_*\otimes id} && C_\ast(L)\otimes CH_{Y}(A) \ar[d]^{\tilde{g}_*} \\
C_\ast(R) \otimes CH_{X}(A) \ar[rr]^{\tilde{h}_*} && CH_Z(A).}
\end{equation*}
\end{ex}
\end{rem}

As we previously mentioned, the higher Hochschild functor (modeled on spaces) agrees with factorization homology (see~\cite{L-HA, F} and Definition~\ref{D:FactHomology}) for $E_\infty$-algebras.
Indeed the following result (whose CDGA version was proved in~\cite{GTZ2}) was proved by  Francis~\cite{F}.
\begin{theorem} \label{T:CH=TCH} Let $M$ be a  manifold of dimension $m$  and $A$ be an $E_\infty$-algebra viewed as an $N(\text{Disk}(M))$-algebra (by restriction of structure, Example~\ref{E:EinftygivesNDisk}).
 Then, the factorization homology $\int_M A$ of $M$ with coefficients in $A$   is naturally
equivalent  to $CH_M(A)$.
\end{theorem}
\begin{proof} The proof is the same as the ones for CDGA's in~\cite{GTZ2} (see Theorem 6 and Corollary 9 in \emph{loc. cit.}) using the axioms of Theorem~\ref{T:derivedfunctor}. Further, as pointed out by John Francis~\cite{F}, the proof also applies to topological manifolds.
\end{proof}
In particular, it follows that the factorization homology of an $E_\infty$-algebra and framed manifold $M$ is canonically an $E_\infty$-algebra which is independent of the choices of framing, and further, can be extended functorially with respect to all continuous maps $h:N\to M$.

\begin{rem}
There is also a nice interpretation of Hochschild chains over spaces  in terms of  derived algebraic geometry.
Let $\mathbf{dSt}_k$ be the (model) category of derived stacks over the ground ring $k$ described in details
in~\cite[Section 2.2]{ToVe}. This category, which admits internal Hom's denoted by $\mathbb{R}\mathop{Map}(F,G)$
following~\cite{ToVe, ToVe2}, is an enrichment of the homotopy category of spaces. Indeed, any
simplicial set $X_\com$  yields a constant simplicial presheaf
$E_\infty\text{-Alg} \to \sset$ defined by $R\mapsto X_\com$ which,
in turn, can be stackified.  We denote $\mathfrak{X}$ the associated
stack, \emph{i.e.} the stackification of $R\mapsto X_\com$, which
depends only on the (weak) homotopy type of $X_\bullet$. It is sometimes called the \emph{Betti stack} of $X_\bullet$.

For a
(derived) stack $\mathfrak{Y}\in \mathbf{dSt}_k$, we denote
$\mathcal{O}_{\mathfrak{Y}}$ its functions, \emph{i.e.},
$\mathcal{O}_{\mathfrak{Y}}:=\mathbb{R}\underline{Hom}(\mathfrak{Y},\mathbb{A}^1)$,
(see~\cite{ToVe}).

\begin{cor}\label{C:mappingstack} Let $\mathfrak{R}=\mathbb{R}\mathop{Spec}(R)$ be an affine derived stack (for instance an
affine stack)~\cite{ToVe} and $\mathfrak{X}$ be the stack associated to a space $X$. Then the Hochschild chains over $X$ with
coefficients in $R$ represent the mapping stack $\mathbb{R}\mathop{Map}(\mathfrak{X}, \mathfrak{R})$. That is, there are canonical
equivalences $$\mathcal{O}_{\mathbb{R}\mathop{Map}(\mathfrak{X},\mathfrak{R})}\; \cong \; CH_{X}(R),
\qquad \mathbb{R}\mathop{Map}(\mathfrak{X},\mathfrak{R}) \;\cong \; \mathbb{R}\mathop{Spec}\big( CH_{X}(R) \big)$$
\end{cor}
\begin{proof}
The proof is analogous to the one of~\cite[Corollary 6.4.4]{GTZ2}.
\end{proof}
Note that if a group  $G$ acts on $X$, the natural action of $G$ on $CH_{X}(A)$ (see Remark~\ref{R:SpaceActiononCH}) identifies with the natural one of $G$ on $\mathbb{R}\mathop{Map}(\mathfrak{X},\mathfrak{R})$ under the equivalence given by Corollary~\ref{C:mappingstack}.
\end{rem}

\subsection{Higher Hochschild (co)chains with values in $E_\infty$-modules} \label{SS:coHHEinfty}
We now consider a dual notion of the Hochschild chain functor, which is well defined in the $E_\infty$-case.

  Let $\epsilon:pt\to X_\bullet$ be a base point of $X_\bullet$.
The map $\epsilon$ yields a map of $E_\infty$-algebras $ A\cong CH_{pt}(A) \stackrel{\epsilon_*}\to CH_{X_\bullet}(A)$
and thus makes   $CH_{X_\bullet}(A)$ an $A$-module.  Let $M$  be another $E_\infty$-$A$-module.

\begin{definition} \label{D-coHoch} The (derived) \emph{Hochschild cochains} of an $E_\infty$-algebra $A$ with value in $M$ over
(the pointed simplicial set) $X_\bullet$ is given by $$CH^{X_\bullet}(A,M)= Hom_{A}(CH_{X_\bullet}(A), M),$$
the (derived) chain complex of the underlying  left $E_1$-$A$-module homomorphisms.
\end{definition}
The definition above depends on the choice of the base point even though we do not write it explicitly in the definition.
We define similarly $CH^{X}(A,M)$ for any \emph{pointed topological space} $X$.
\begin{rem}\label{R:D-coHochE1Version}
 According to Theorem~\ref{T:lifttoEinfty}, one can also alternatively consider the chain complex of $E_\infty$-$A$-modules in Definition~\ref{D-coHoch}.
\end{rem}

\smallskip

\begin{definition}\label{D:HochChainMod}
The \emph{Hochschild chains} of an $E_\infty$-algebra $A$ with values in $M$ over (the pointed simplicial set)
$X_\bullet$ is defined as $$CH_{X_\bullet}(A,M)= M\mathop{\otimes}^{\mathbb{L}}_{A} CH_{X_\bullet}(A)$$ the relative tensor product of $E_\infty$-$A$-modules (as defined, for instance, in~\cite[Section 3.3.3]{L-HA} or~\cite{KM}).
\end{definition}
\begin{rem}
Any $E_\infty$-$A$-module has an underlying  $E_1$-module structure given by the forgetful functor $A\text{-}Mod^{E_\infty} \to A\text{-}Mod^{E_1}$ hence both a left and right $A$-module structure. Thus, given two $E_\infty$-$A$-modules $M, N$, one can form their relative tensor product
$ M \mathop{\otimes}^{\mathbb{L}}_{A} N$ where $M$ is viewed as a right $A$-module, $N$ as a left $A$-module and $A$ as an $E_1$-algebra.
According to Theorem~\ref{T:lifttoEinfty} and~\cite[Section 4.4.1]{L-HA} or~\cite[Section 5]{KM}, this  tensor is equivalent (as an object of $\hkmod$) to the relative tensor product  computed in $E_\infty$-$A$-modules. Hence, the tensor product of Definition~\ref{D:HochChainMod} can be computed  using this alternative definition.
\end{rem}

Since the based point map $\epsilon_*: A\to CH_{X_\bullet}(A)$ is a map of $E_\infty$-algebras, the canonical module structure of $CH_{X_\bullet}(A)$ over itself induces a module structure on $CH_{X_\bullet}(A,M)$ over $CH_{X_\bullet}(A)$ after tensoring by $A$ (also see \cite[Part V]{KM}, \cite{L-HA}):
\begin{lem}\label{L:HochModstructure}
 Let $M$ be in $A\text{-}Mod^{E_\infty}$, that is, $M$ is an $E_\infty$-$A$-module. Then $CH_{X_\bullet}(A,M)$ is canonically a  $CH_{X_\bullet}(A)$-$E_\infty$-module.
\end{lem}
\begin{rem}
By definition, if $A$ is endowed with its canonical $A$-$E_\infty$-module structure, the natural map
$CH_{X_\bullet}(A,A)\cong A\otimes^{\mathbb{L}}_{A} CH_{X_\bullet}(A)\to  CH_{X_\bullet}(A) $ is an equivalence of
$CH_{X_\bullet}(A)$-modules. Hence, tensoring by $M\otimes_A -$, we get a canonical lift of the relative tensor products $M\otimes_A^{\mathbb{L}} CH_{X_\bullet}(A)$, computed as a relative tensor product of left and right modules over $A$ seen as an $E_1$-algebra, to a  $CH_{X_\bullet}(A)$-$E_\infty$-module as well.
\end{rem}

\begin{prop}\label{P-CHfunctorMod}
The derived Hochschild chain $CH_{X_\bullet}(A,M)$ with value in an $E_\infty$-algebra $A$ and an $A$-module $M$
over a space $X_\bullet$ given by Definition~\ref{D:HochChainMod}
 induces a functor of $(\infty,1)$-categories
 $CH: (X_\bullet,M) \mapsto CH_{X_\bullet}(\iota(M),M)$ from $\hssetp \times Mod^{E_\infty}$ to $Mod^{E_\infty}$.

\smallskip

The derived Hochschild cochains $CH^{X_\bullet}(A,M)$ with value in an $E_\infty$-algebra $A$ and an $A$-module $M$
over a space $X_\bullet$ given by Definition~\ref{D-coHoch}
 induces a functor of $(\infty,1)$-categories
 $(X_\bullet,M) \mapsto CH^{X_\bullet}(A,M)$ from $\hssetp^{op} \times A\text{-}Mod^{E_\infty}$ to $A\text{-}Mod^{E_\infty}$,
which is further contravariant\footnote{using the canonical functor (similar to the one of
Example~\ref{E:CDGAMod})  $f_*: B\text{-}Mod^{E_\infty}
\to A\text{-}Mod^{E_\infty}$ associated to any $E_\infty$-algebras map $f: A\to B$} with respect to $A$.
\end{prop}
\begin{proof} It follows from Lemma~\ref{L:HochModstructure} and \S~\ref{SS:FactandHHforEinfty}. The fact that homomorphisms of $A$-$E_\infty$-modules have a canonical structure of $A$-$E_\infty$-modules follows from the same argument as for the tensor product above or from~\cite[Theorem V.8.1]{KM}.
\end{proof}

\begin{rem}
As usual, one obtains a similar version of the above Definition~\ref{D:HochChainMod} and Lemma~\ref{L:HochModstructure}
for pointed topological space $X$.
\end{rem}

\begin{rem} \label{R:cdgaEMod} If $A$ is a CDGA and $M$ a left $A$-module, similarly to Corollary~\ref{C:cdgaE},  there are
 natural equivalences $$CH^{cdga}_{X_\bullet}(A,M) \cong CH_{X_\bullet}(A,M), \qquad CH_{cdga}^{X_\bullet}(A,M)\cong  CH^{X_\bullet}(A,M)$$
where $CH^{cdga}_{X_\bullet}(A,M)$ and $CH_{cdga}^{X_\bullet}(A,M)$ are the usual higher Hochschild chain and cochain functors for CDGA's and their modules defined respectively in~\cite{P} and~\cite{G}.
\end{rem}

\subsection{Axiomatic characterization}\label{SS:AxiomHH}
The axiomatic approach to Hochschild functors over spaces for CDGA's
studied in the authors' previous work~\cite{GTZ2} extends formally
to $E_\infty$-algebras as well. It is actually an immediate
corollary of the fact that $E_\infty\text{-Alg}$ (as well as any
presentable $(\infty,1)$-category) is tensored over simplicial sets
in a unique way (up to homotopy).  We now recall quickly the
axiomatic characterization (similar to the Eilenberg-Steenrod
axioms) and some consequences for Hochschild theory over spaces with
value in $Mod^{E_\infty}$. A similar story for factorization
homology of $E_n$-algebras has been developed recently by
Francis~\cite{F2, AFT}.

We first collect the axioms characterizing the (derived) Hochschild chain theory over spaces into the following definition. Let $Forget: \hTopp \to \hTop$ be the functor that forget the base point.
\begin{definition} \label{D:axioms} An $E_\infty$-homology theory\footnote{with values in the symmetric monoidal $(\infty,1)$-category $(\hkmod, \otimes)$} is a pair of $\infty$-functors
$\mathcal{CA}: \hTop\times E_\infty\text{-Alg} \to
E_\infty\text{-Alg}$, denoted $(X,A)\mapsto \mathcal{CA}_X(A)$, and
$\mathcal{CM}: \hTopp\times Mod^{E_\infty}\to Mod^{E_\infty}$,
denoted $(X,M)\mapsto \mathcal{CM}_X(M)$, fitting in a commutative
diagram
\begin{equation}\label{eq:CMCA}
\xymatrix{ \hTopp \times Mod^{E_\infty} \ar[rr]^{\mathcal{CM}} \ar[d]_{Forget\times \iota}& & Mod^{E_\infty} \ar[d]^{\iota} \\
 \hTop \times {E_\infty}\text{-Alg} \ar[rr]^{\mathcal{CA}} & & {E_\infty}\text{-Alg} }
\end{equation}
 satisfying the following axioms:
\begin{description}
\item[i) value on a point] there is a natural equivalence $\mathcal{CM}_{pt}(M)\cong M$ in $Mod^{E_\infty}$;
\item[ii) monoidal] the natural map $$\mathcal{CM}_{X}(M) \otimes \mathcal{CA}_{Y}(\iota(M)) \stackrel{\simeq}\longrightarrow \mathcal{CM}_{X\coprod Y}(M)$$ (where $X\in \hTopp$ and $Y\in \hTop$)  is an equivalence.
\item[iii) excision] $\mathcal{CM}$ commutes with homotopy pushout of spaces, \emph{i.e.}, there is a canonical equivalence
$$\mathcal{CM}_{X\cup^h_{Z}Y}(M)\cong  \mathcal{CM}_{X}(M)\mathop{\otimes}\limits_{\mathcal{CA}_{Z}(\iota(M))}^{\mathbb{L}} \mathcal{CA}_{Y}(\iota(M))$$ where $X\in \hTopp$, $Y,Z\in \hTop$.
\end{description}
\end{definition}
\begin{rem}
Since any $E_\infty$-algebra is canonically a module over itself,
there is also a canonical functor $\phi:E_\infty\text{-Alg} \to
Mod^{E_\infty}$, hence a functor $\big(-\coprod \{*\}\big) \times
\phi: \hTop\times E_\infty\text{-Alg}\to \hTopp \times
Mod^{E_\infty}$ giving rise, by composition with $\iota\circ
\mathcal{CM}$ to a functor $\psi:\hTop\times E_\infty\text{-Alg}\to
E_\infty\text{-Alg}$. By axioms $i)$ and $ii)$ in
Definition~\ref{D:axioms} and commutativity of the
diagram~\eqref{eq:CMCA}, we get a natural equivalence
$$\psi_X(A) \;\cong\; \phi(A)\otimes \mathcal{CA}_X(A) . $$ Hence, the functor $\mathcal{CA}$ is actually completely defined by the functor $\mathcal{CM}$.
\end{rem}
\begin{rem}\label{R-GenHomTheory} We also define
a generalized  $E_\infty$-homology theory to be a triple of functors
$F:Mod^{E_\infty}\to Mod^{E_\infty}$, $\mathcal{CA}: \hTop\times
E_\infty\text{-Alg} \to E_\infty\text{-Alg}$ and $\mathcal{CM}:
\hTopp\times Mod^{E_\infty}\to Mod^{E_\infty}$ satisfying all
properties as in Definition~\ref{D:axioms} except that the value on
a point axiom is modified by requiring a natural equivalence
$\mathcal{CM}_{pt}(M)\cong F(M)$ in $Mod^{E_\infty}$.
\end{rem}

The next theorem shows that higher Hochschild homology theory is the unique functor satisfying the assumptions of Definition~\ref{D:axioms}.
\begin{theorem}\label{T:derivedfunctor} \begin{enumerate}
\item The derived Hochschild chains functors $CH:\hTop \times E_\infty\text{-Alg}\to E_\infty\text{-Alg}$ (see Proposition~\ref{P-CHbifunctorTop}) and the derived Hochschild chains with value in a module $ CH_{X}: \hTopp\times Mod^{E_\infty} \to Mod^{E_\infty}$ (see Proposition~\ref{P-CHfunctorMod}) form a  $E_\infty$-homology theory in the sense of Definition~\ref{D:axioms}.
\item Any  $E_\infty$-homology theory (in the sense of Definition~\ref{D:axioms}) is naturally equivalent to derived Hochschild chains, \emph{i.e.},   there are natural equivalences
$\mathcal{CA}_X(A)\cong CH_X(A)$ and $\mathcal{CM}_X(M)\cong CH_X(\iota(M),M)$.
\end{enumerate}
\end{theorem}
\begin{proof} This is essentially implied by the fact that $CH_{X}(A) \cong A\boxtimes X$ is the tensor of $A$ with the space $X$ and that such a tensor is defined uniquely, see~\cite[Corollary 4.4.4.9]{Lu11}. Note that the first assertion follows from Proposition~\ref{P-CHfunctorMod} and Proposition~\ref{P-CHbifunctorTop}. The proof of the uniqueness follows from the proofs of Theorem 4.2.7 and Theorem 4.3.1 in~\cite{GTZ2}. The excision and the value on a point axioms applied to $X=Z=pt$ show that there is a natural equivalence $$\mathcal{CM}_Y(M)\cong M\mathop{\otimes}\limits_{\iota(M)}^{\mathbb{L}} \mathcal{CA}_Y(\iota(M))$$ which reduces to proving the assertion for $\mathcal{CA}$. Since $\iota: Mod^{E_\infty}\to E_\infty\text{-Alg}$ is monoidal, $\mathcal{CA}$ is monoidal. Similarly, the natural equivalence\eqref{eq:iotapushout} implies that $\mathcal{CA}$ satisfies the excision axiom (in the category of $E_\infty$-algebras).
Now the proof of~\cite[Theorem 2]{GTZ2} applies verbatim. The argument boils down to the fact that $\hTop$ is generated by a point using coproducts and homotopy pushouts.
\end{proof}

We now list a few easy properties derived from the above Theorem~\ref{T:derivedfunctor}.
\begin{cor}\label{C:properties}
\begin{enumerate} \item The derived Hochschild chain functor is the unique functor $\hTop\times E_\infty\text{-Alg}\to E_\infty\text{-Alg}$ satisfying the following three axioms
\begin{enumerate}
\item {\bf value on a point:} \label{A:point} There is a natural equivalence of $E_\infty$-algebras $CH_{pt}(A)\cong A$.
\item {\bf coproduct:} \label{A:coproduct}There are canonical equivalences
$$CH_{\coprod\limits_{I} (X_i)}(A)\stackrel{\simeq}\longleftarrow 
\colim_{\text{\scriptsize $\begin{array}{l}K \subset I \\ K \mbox{ finite}\end{array}$}} \bigotimes_{k\in K} CH_{X_k}(A) $$
\item {\bf homotopy glueing/pushout:} \label{A:pushout} there are natural equivalences
$$CH_{X\cup^{h}_{Z} Y}(A)\stackrel{\simeq}\longleftarrow  CH_{X}(A)\otimes_{CH_{Z}(A)}^{\mathbb{L}} CH_{Y}(A).$$
\end{enumerate}
\item \textbf{(generalized uniqueness)} Let $F:Mod^{E_\infty}\to Mod^{E_\infty}$, $\mathcal{CA}: \hTop\times E_\infty\text{-Alg} \to E_\infty\text{-Alg}$ and $\mathcal{CM}: \hTopp\times Mod^{E_\infty}\to Mod^{E_\infty}$ be a generalized  $E_\infty$-homology theory. Then there is a natural equivalence
$$ \mathcal{CM}_X(M)\; \cong \; CH_{X}\big(\iota(F(M)),F(M)\big). $$
\item \textbf{(commutations with colimits)} The derived Hochschild chains functors $CH:\hTop\times E_\infty\text{-Alg} \to E_\infty\text{-Alg}$ and  $CH:\hTopp\times Mod^{E_\infty} \to Mod^{E_\infty}$ commutes with finite colimits in $\hTop$ and all colimits in $Mod^{E_\infty}$, that is  there are natural equivalences
$$CH_{\colim_{\mathcal{F}} {X_i}}(\iota(M),M)\cong \colim_{\mathcal{F}} CH_{{X_i}}(\iota(M),M) \quad (\text{for a finite category } \mathcal{F}),$$
$$CH_{X}(\colim {A_i}) \cong \colim CH_{X}(A_i). $$
\item {\bf (product)} Let $X$, $Y$ be pointed spaces, $M\in Mod^{E_\infty}$ and $A=\iota(M)\in E_\infty\text{-Alg}$. There is a natural equivalence
$$CH_{X\times Y}(A,M) \stackrel{\sim}\to CH_{X}\left( CH_{Y}(A), CH_{Y}(A,M))\right)$$ in $Mod^{E_\infty}$.
\end{enumerate}
\end{cor}
\begin{proof}
 The proof of the first assertion follows directly from Theorem~\ref{T:derivedfunctor} by applying the monoidal functor $\iota$. The proof of the other assertions are  the same as the analogous result for CDGA's proved in~\cite{GTZ2}.
\end{proof}

\subsection{Higher Hochschild (co)chains models for mapping spaces} \label{S:EinftyModel}
This section is devoted to the relationship in between higher Hochschild chains and mapping spaces. In particular,
we prove an
\emph{$E_\infty$-algebra version of the Chen iterated integral} morphism studied in~\cite{GTZ}.

\smallskip

Let $A$ be an $E_\infty$-algebra. Recall that
by the coproduct axiom and functoriality of Hochschild chains (see Theorem~\ref{T:derivedfunctor}, Corollary~\ref{C:properties}),
there is a natural equivalence $A\otimes A \cong CH_{S^0} (A)$ of $E_\infty$-algebras as well as a natural
$E_\infty$-algebras map $CH_{S^0} (A) \to CH_{pt}(A)\cong A$.

\begin{lem}\label{L:AWEinfty} Let $X,Y$ be topological spaces and $C^\ast(X)$, $C^\ast(Y)$ be their $E_\infty$-algebras of cochains. Denote $\pi_X:X\times Y\to X$ and $\pi_Y: X\times Y\to Y$ the projections onto the first and second factors. The composition,

\begin{multline}\label{eq:AWEinfty}
C^\ast(X)\otimes C^\ast(Y) \stackrel{\pi_X^*\otimes \pi_Y^*}\longrightarrow C^\ast(X\times Y)\otimes C^\ast(X\times Y)
\stackrel{\simeq}\rightarrow CH_{S^0}\big(C^{\ast}(X\times Y)\big) \\
\longrightarrow CH_{pt} \big(C^{\ast}(X\times Y)\big) \cong C^{\ast}(X\times Y)\end{multline}
is a natural morphism of $E_\infty$-algebras. It is further  an equivalence under the assumption that $H_\ast(Y)$ (or $H_\ast(X)$) is    finitely generated in each degree.
\end{lem}
\begin{proof}
 That the maps involved are natural (in $X,Y\in \hTop$) maps of $E_\infty$-algebras follows from the functoriality of
 $X\mapsto C^\ast(X)$ and the functorial and monoidal properties of the higher Hochschild derived functor
 (see Theorem~\ref{T:derivedfunctor}).

We now prove that the map~\eqref{eq:AWEinfty} is an equivalence under the assumption that  that $H_\ast(Y)$  is   projective, finitely generated in each degree.  The idea is to prove that the map~\eqref{eq:AWEinfty} is homotopy equivalent to the cross product.

Note that if the ground ring $k$ is a field of characteristic zero,  the map~\eqref{eq:AWEinfty} induces a map $H^\ast(X)\otimes H^{\ast}(Y)\to H^{\ast}(X\times Y)$
 which is easily identified with the K\"unneth morphism since for a graded commutative algebra, the map $A\otimes A \cong CH_{S^0} (A) \to CH_{pt}(A)\cong A$ is
 given by the multiplication in $A$ (by Corollary~\ref{C:cdgaE}).

For a general ground ring of coefficients, note that as a mere
$E_1$-algebra (via the forgetful functor $E_\infty\text{-Alg}
\hookrightarrow E_1\text{-Alg}$),
 the singular cochain complex $C^\ast(X)$ is endowed with the (strictly) associative algebra structure given by the cup-product.
Let $D^1_+, D^1_{-}$ be two open disjoint sub-intervals of $D^1$ and $i: D^1_{-}\coprod D^1_+ \hookrightarrow D^1$ be the inclusion
map. By definition (see~\cite{L-HA, L-VI, F}), for any differential graded associative algebra $(A,m)$, the canonical map of
\emph{chain complexes} (and not $E_1$-algebras)
$$ A\otimes A\; \cong \; \int_{D^1_{-}\coprod D^1_{+}} A \stackrel{i_*}\longrightarrow \int_{D^1} A \;\cong\; A$$
is the multiplication map $m:A\otimes A\to A$ defining the $E_1$-structure of $A$.
 If furthermore $(A,m)$ is actually an $E_\infty$-algebra,  by Theorem~\ref{T:CH=TCH} and functoriality of derived Hochschild
functor, there is a (homotopy) commutative diagram of chain complexes
$$\xymatrix{A\otimes A \ar[ddd]_{m} \ar[rd]_{\simeq} \ar@/^/[rrd]^{\simeq} \ar@/^1pc/[rrrd]^{\simeq}& & & \\
&\int_{D^1_{-}\coprod D^1_{+}} A  \ar[r]^{\simeq} \ar[d]_{i}    & CH_{D^1_{-}\coprod D^1_{+}}(A)\ar[r]^{\simeq} \ar[d]^{i_*}
& CH_{S^0}(A) \ar[d]  \\
 &\int_{D^1} A \ar[ld]_{\simeq}\ar[r]^{\simeq} & CH_{D^1}(A) \ar[r]^{\simeq} \ar@/^/[lld]_{\simeq} & CH_{pt}(A)  \ar@/^1pc/[llld]^{\simeq}\\ A & & &} $$
and thus, the map~\eqref{eq:AWEinfty} is homotopy equivalent, as a map of chain complexes, to
 \begin{equation}\label{eq:AWEinfty2}
  C^\ast(X)\otimes C^\ast(Y) \stackrel{\pi_X^*\otimes \pi_Y^*}\longrightarrow C^\ast(X\times Y)\otimes C^\ast(X\times Y)
 \stackrel{\cup}\longrightarrow
 C^{\ast}(X\times Y).\end{equation} The cochain complex structure of $C^\ast(X)$ is the normalization of the cosimplicial
$k$-module $n\mapsto C^n(X)$ so that
 the above map~\eqref{eq:AWEinfty2} is the (dual of the) Alexander-Whitney diagonal  (in $\hkmod$):
\begin{equation} \label{eq:AWdual} AW: C^\ast(X)\otimes C^\ast(Y)\hookrightarrow  \big(C_\ast(X)\otimes C_\ast(Y)\big)^{\vee} \stackrel{AW^{\vee}}\longrightarrow C^{\ast}(X\times Y).\end{equation} Here the first arrow is the canonical inclusion and the second one the dual of the Alexander-Whitney quasi-isomorphism: $AW: C_\ast(X\times Y)\stackrel{\simeq}\to C_\ast(X)\otimes C_\ast(Y)$. Since $C_\ast(X)$, $C_\ast(Y)$ are complexes of  modules and $C_\ast(Y)$ has finitely generated  homology in each degree, both maps in the composition~\eqref{eq:AWdual} are quasi-isomorphisms; the lemma follows.
\end{proof}

\begin{rem}\label{R:AWEinfty} The map of $E_\infty$-algebra $C^\ast(X)\otimes C^\ast(Y)\to C^\ast(X\times Y)$ given by Lemma~\ref{L:AWEinfty} is  in particular a map of chain complexes. From the last part of the proof of Lemma~\ref{L:AWEinfty}, it follows that this map is equivalent   in $k\text{-}Mod_\infty$ to the dual of the Alexander-Whitney diagonal (see the maps~\eqref{eq:AWEinfty2},~\eqref{eq:AWdual}), \emph{i.e.} the map given by Lemma~\ref{L:AWEinfty} is an $E_\infty$-lifting of the Alexander-Whitney diagonal (also called the cross product).
\end{rem}

\smallskip

Let $X_\bullet$ be a simplicial set  and $Y$ be a topological space. We define a map $$ev: Y^{|X_\bullet|} \times \Delta^n \to Y^{X_n} $$ by $ev(f, (t_0, \cdots, t_n)) =g$, where for $$f: \big(\coprod (X_n\times \Delta^n)/\sim\big) \to Y \text{ and } (t_0, \cdots, t_n)\in \Delta^n,$$ we have,  $$g(\sigma_n) = f([\sigma_n, (t_0, \cdots, t_n)]), \quad\quad \text{for }\sigma_n\in X_n. $$

Note that this is a well defined map of cosimplicial topological spaces. In fact, $ev$ is induced by the canonical map $X_n \to \mathop{Map}(\Delta^n, |X_\bullet|)$  given by the unit of the adjunction between simplicial sets and topological spaces.

Applying the $E_\infty$ cochain functor $C^\ast(-)$ (Example~\ref{E:singularchainasEinfty}) yields a natural map
\begin{equation}\label{eq:Einftyeta}ev^*: \big(C^{\ast}( Y^{X_i}) \big)_{(i\in \mathbb{N})} \to \big(C^{\ast}(Y^{|X_\bullet|} \times \Delta^i)\big)_{(i\in \mathbb{N})}\end{equation}
of simplicial $E_\infty$-algebras.

\begin{lem}\label{L:realizationmapping}
The geometric realization of the simplicial $E_\infty$-algebra $\big(C^{\ast}(Z \times \Delta^i)\big)_{(i\in \mathbb{N})}$ is naturally equivalent to $C^{\ast}(Z)$, as an $E_\infty$-algebra.
\end{lem}
\begin{proof}
By Lemma~\ref{L:AWEinfty}, there is a natural equivalence
$C^{\ast}(Z \times \Delta^i) \cong C^{\ast}(Z)\otimes
C^\ast(\Delta^i)$ in $E_\infty\text{-}Alg $. This induces an
equivalence, $$C^{\ast}(Z) \otimes
\big(C^{\ast}(\Delta^i)\big)_{(i\in \mathbb{N})}  \stackrel{\simeq}
\longrightarrow  \big(C^{\ast}(Z) \times \Delta^i)\big)_{(i\in
\mathbb{N})} $$ of simplicial $E_\infty$-algebras. Since the
constant map $\Delta^i \to pt$ is a homotopy equivalence, the
canonical map $C^\ast(pt) \to C^\ast(\Delta^i)$, where $C^\ast(pt)$
is viewed as a constant simplicial $E_\infty$-algebra, is an
equivalence. Composing the above with the equivalence,
$$C^{\ast}(Z) \otimes \big(C^{\ast}(pt)\big)_{(i\in \mathbb{N})}  \stackrel{\simeq} \longrightarrow C^{\ast}(Z) \otimes \big(C^{\ast}(\Delta^i)\big)_{(i\in \mathbb{N})} $$
gives rise to an equivalence between $C^{\ast}(Z)$ and the constant simplicial $E_\infty$-algebra $C^{\ast}(Z \times \Delta^i)$.
\end{proof}

Let $X_\bullet$ be a simplicial set.  Iterating Lemma~\ref{L:AWEinfty}, we get, for any $n\in \mathbb{N}$, a natural map  of $E_\infty$-algebras
\begin{equation} \label{eq:simpHochmapping}CH_{X_n}(C^{\ast}(Y)) \longrightarrow C^{\ast}\big( Y^{X_{n}}\big)
 \end{equation}
Composing the map~\eqref{eq:simpHochmapping} with the $ev^\ast$ map in ~\eqref{eq:Einftyeta}, we get a natural morphism of simplicial  $E_\infty$-algebras,

\begin{equation}\label{eq:ItEinfty} \mathcal{I}t:CH^{simp}_{X_\bullet}(C^{\ast}(Y)) \longrightarrow C^{\ast}\big( Y^{X_{\bullet}}\big)  \stackrel{ev^\ast}\longrightarrow C^{\ast}\big(Y^{|X_\bullet|} \times \Delta^{\bullet}\big).  \end{equation}

The following result is an integral, $E_\infty$-lifting of the iterated integrals~\cite{GTZ2}. 
\begin{theorem}\label{T:Einftymapping}
The geometric realization of the map $$\mathcal{I}t: CH^{simp}_{X_\bullet}(C^{\ast}(Y))\longrightarrow C^{\ast}\big(Y^{|X_\bullet|} \times \Delta^{\bullet}\big)$$ yields a natural (in $X_\bullet$ and $Y$) morphism of $E_\infty$-algebras $$\mathcal{I}t: CH_{X_\bullet}(C^\ast(Y)) \longrightarrow C^{\ast}\big( Y^{|X_\bullet|} \big).$$
Further, if $|X_\bullet|$ is $n$-dimensional (\emph{i.e.} the highest degree of any non-degenerate simplex is $n$) and $Y$  is $n$-connected, then the map $\mathcal{I}t$ is an equivalence.
\end{theorem}
\begin{proof}
Since the natural map~\eqref{eq:simpHochmapping}, $CH^{simp}_{X_\bullet}(C^{\ast}(Y))\to C^{\ast}\big( Y^{X_{\bullet}}\big)$, and the map~\eqref{eq:Einftyeta}, $ C^{\ast}\big( Y^{X_{\bullet}}\big) \stackrel{ev^\ast}\to C^{\ast}\big( Y^{|X_\bullet|} \times \Delta^\bullet\big)$, are simplicial, their realization yields a map of $E_\infty$-algebras $$CH_{X_\bullet}(C^\ast(Y)) \to |C^{\ast}\big( Y^{|X_\bullet|} \times \Delta^\bullet \big) |\;\cong\; C^{\ast}\big( Y^{|X_\bullet|} \big)  $$ where the last equivalence follows from  Lemma~\ref{L:realizationmapping}.  This defines the map $\mathcal{I}t$ which is natural by construction.

\smallskip

Now, we assume  $|X_\bullet|$ is $n$-dimensional and  $Y$ is $n$-connected. We only need to check that the underlying map of \emph{cochain complexes} $CH_{X_\bullet}(C^{\ast}(Y))\to C^{\ast}\big(Y^{|X_\bullet|} \times \Delta^{\bullet}\big)$ is an equivalence in the $(\infty,1)$-category of cochain complexes. The proof of Lemma~\ref{L:AWEinfty} (see Remark~\eqref{R:AWEinfty}) implies that the cochain complex morphism
$$CH_{X_\bullet}(C^{\ast}(Y))\longrightarrow  C^{\ast}(Y^{X_\bullet}) $$ is the map induced by the iterated Alexander-Whitney diagonal.
  Since the geometric realization commutes with the forgetful functor $E_\infty\text{-Alg}  \to k\text{-}Mod_\infty$, the geometric realization of the map $ C^{\ast}(Y^{X_\bullet}) \to  C^{\ast}\big(Y^{|X_\bullet|} \times \Delta^{\bullet}\big)$ is equivalent in $k\text{-}Mod_\infty$ to the
map induced  by the slant products $$C^{\ell}(Y^{X_n}) \to C^{\ell}\Big(Y^{|X_\bullet|} \times \Delta^{\ell}\Big) \stackrel{ \slash [\Delta^n]}\longrightarrow C^{\ell-n}\Big(Y^{|X_{\bullet}|}\Big)$$
by the fundamental chain $[\Delta^n]$ given by the unique non-degenerate $n$-simplex of $\Delta^n$.

 Hence we have proved that $\mathcal{I}t$ is equivalent in $k\text{-}Mod_\infty$ to the composition
\begin{multline*}
\bigoplus_{n \geq 0} CH_{X_n}(C^{\ast}(Y))\longrightarrow \bigoplus_{n\geq 0} C^{\ast}(Y^{X_n}) \\ \stackrel{\bigoplus ev^\ast}\longrightarrow  \bigoplus_{n\geq 0} C^{\ast}\Big(Y^{|X_\bullet|} \times \Delta^{n}\Big)
\stackrel{\bigoplus_{n\geq 0} \slash [\Delta^{n}]}\longrightarrow  C^{\ast}\Big(Y^{|X_{\bullet}|}\Big).
\end{multline*}
This last map is a quasi-isomorphism under the appropriate assumptions on $X_\bullet$ and $Y$ using the same argument as in~\cite{GTZ, PT}.
\end{proof}
\begin{rem}[Relationship with Chen integrals] Let $Y=M$ be a manifold and $k$ a field of characteristic zero. Then, by Corollary~\eqref{C:cdgaE} and homotopy invariance of higher Hochschild cochains, there is a natural equivalence of $E_\infty$-algebras $$CH_{X_\bullet}(C^\ast(M))\; \cong \; CH_{X_\bullet}^{cdga}(\Omega(M)).$$
Recall that the slant product is a model for integration over $\Delta^n$. Unfolding the proof of Theorem~\ref{T:Einftymapping} and the construction of the map in Lemma~\ref{L:AWEinfty},  one can check that the map $\mathcal{I}t: CH_{X_\bullet}(C^\ast(M))\to C^{\ast}(M^{|X_\bullet|})$, given by Theorem~\ref{T:Einftymapping}, coincides\footnote{as natural transformations of $\infty$-functors $\hsset \to E_\infty\text{-Alg}$} with  the generalized Chen's iterated integral map defined in~\cite[Section 2]{GTZ}. In particular, when $X_\bullet$ is the standard simplicial set model of the compact interval or the circle, we recover the original Chen iterated integral construction~\cite{Ch}. This justifies our notation $\mathcal{I}t$ for the map defined in Theorem~\ref{T:Einftymapping}.

\smallskip

Similarly, the argument of the proofs of Theorem~\ref{T:Einftymapping} and Lemma~\ref{L:AWEinfty} as well as Theorem~\ref{T:CH=TCH} (applied to the forgetful functor from $E_\infty$-algebras to $E_1$-algebras) show that the iterated integral map $\mathcal{I}t:CH_{X_\bullet}(C^{\ast}(Y)) \rightarrow C^{\ast}\Big(Y^{|X_\bullet|}\Big)$ given by Theorem~\ref{T:Einftymapping} is homotopy equivalent to the map of differential graded algebras described in~\cite{PT}. In particular, for $X=S^1_\bullet$, we recover an $E_\infty$-algebra lift of Jones  quasi-isomorphism~\cite{Jo}.
\end{rem}

 Similarly, if $X$ is a topological space, by choosing a   simplicial model $X_\bullet$ for $X$ (that is a simplicial set with an equivalence $|X_\bullet| \to X$), we  get a natural equivalence $CH_{X}(C^{\ast}(Y)) \stackrel{\simeq}\longrightarrow CH_{X_\bullet}(C^{\ast}(Y))$ and thus Theorem~\ref{T:Einftymapping} yields the following corollary.  Note that an independent proof was obtained by Francis~\cite{F2} in the case where $X$ is a manifold.
\begin{cor}\label{C:Einftymapping} The map $$\mathcal{I}t:CH_{X}(C^{\ast}(Y))\stackrel{\simeq}\longrightarrow CH_{X_\bullet}(C^{\ast}(Y)) \longrightarrow C^\ast(Y^{X}) $$ is a natural (in $X$, $Y$) morphism of $E_\infty$-algebras and an equivalence if $Y$ is $\dim(X)$- connected.
\end{cor}

We will give a cohomological version of Theorem~\ref{T:Einftymapping}.  Assume now that $X$ is pointed (by a map $\epsilon: pt\to X$) and choose  a pointed simplicial set model $X_\bullet$ of $X$. By naturality of the map $\mathcal{I}t$ in Theorem~\ref{T:Einftymapping}, there is a commutative diagram of $E_\infty$-algebras maps:
\begin{equation}
\xymatrix{ CH_{X_\bullet}(C^\ast(Y)) \ar[rr]^{\mathcal{I}t} && C^{\ast}\big( Y^{|X_\bullet|} \big) \\ C^{\ast}(Y) \cong CH_{pt}(C^{\ast}(Y)) \ar[u]^{\epsilon_*} \ar[rr]^{\mathcal{I}t} && C^{\ast}\big(Y^{pt}\big)\cong C^{\ast}(Y) \ar[u]^{C^{\ast}(\epsilon^*)}.}
\end{equation}
in which the lower map is seen to be the identity map by construction. It follows that $\mathcal{I}t$ is a $C^\ast(Y)$-$E_\infty$-module map. Denoting $M^{\vee}=Hom_{k}(M,k)$ the linear dual of $M$ (equipped with its canonical $A$-$E_\infty$-structure if $M$ is an $A^{op}$-$E_\infty$-module), we thus get a map
\begin{multline}\label{eq:Itdual}
 \mathcal{I}t^*: C_{\ast}\Big( Y^{X}\Big)\cong C_{\ast}\Big( Y^{|X_\bullet|}\Big) \longrightarrow Hom_{k}\left(\big(C^{\ast}(Y^{|X_\bullet|})\big), k \right) \\ \stackrel{\simeq}\longrightarrow  Hom_{C^{\ast}(Y)}\left(\big(C^{\ast}(Y^{|X_\bullet|})\big), \Big(C^{\ast}(Y)\Big)^{\vee}\right)  \\ \stackrel{-\circ \mathcal{I}t} \longrightarrow Hom_{C^{\ast}(Y)}\left(CH_{X_\bullet}\big(C^{\ast}(Y)\big), \Big(C^{\ast}(Y)\Big)^{\vee}\right)\\ \cong CH^{X_\bullet}\left(C^{\ast}(Y), \big(C^{\ast}(Y)\big)^{\vee} \right) \cong CH^{X}\left(C^{\ast}(Y), \big(C^{\ast}(Y)\big)^{\vee} \right)
\end{multline}
where the first map is biduality morphism, the second map is the canonical isomorphism and the last two  isomorphisms are from Definition~\ref{D:HochChainMod}.

 \begin{cor}\label{C:Itdual}
  The morphism $\mathcal{I}t^*:C_{\ast}\Big( Y^{X}\Big) \longrightarrow CH^{X}\left(C^{\ast}(Y), \big(C^{\ast}(Y)\big)^{\vee} \right) $ in $k\text{-}Mod_\infty$ is natural in $X$ and $Y$.

  Further, if $Y$ is $\dim(X)$-connected, $X$ is compact and the homology groups of $Y$ are  finitely generated in each degree, then $\mathcal{I}t^*$ is a quasi-isomorphism.
 \end{cor}
\begin{proof}
 That $\mathcal{I}t^*$ is natural in $X$ and $Y$ is immediate since all maps involved in its definition are natural in their two arguments.  

 The assumption $Y$ is $\dim(X)$-connected ensures that $\mathcal{I}t$ is a quasi-isomorphism. Further, for a model $X=|X_\bullet|$ where $X_k$ is finite in every degree,   the above assumption together with the assumption on the homology groups of $Y$ ensures that the biduality map $C_{\ast}\Big( Y^{|X_\bullet|}\Big) \longrightarrow Hom_{k}\left(\big(C^{\ast}(Y^{|X_\bullet|})\big), k \right)$ is a quasi-isomorphism as well. Indeed, the connectivity assumption ensures that $H_{\ast}\Big( Y^{|X_\bullet|}\Big) $ is the abutment of the (first quadrant hence) converging  spectral sequence given by the simplicial filtration of $X_\bullet$ (and so is $Hom_{k}\left(\big(C^{\ast}(Y^{|X_\bullet|})\big), k \right)$). Its $E_1$-term is given by the (reduced) homology of $\bigoplus_{k} C_\ast(Y^{X_k})$.  The finiteness of $X_k$   ensures that each $C_\ast(Y^{X_k}) \stackrel{\simeq}\to \bigotimes_{X_k} C_\ast(Y)$   has  finite type homology groups in every degree (since $Y$ has), hence is  quasi-isomorphic to its bidual from which we deduce that the biduality map is already an isomorphism at the $E_1$-page.
\end{proof}
\begin{rem}[\emph{Weakening the connectivity condition}]
The assumption of $Y$ being $\dim(X)$-connected in Theorem~\ref{T:Einftymapping} and Corollary~\ref{C:Itdual} is merely there to ensure the convergence of a spectral sequence (introduced in the proof of Corollary~\ref{C:Itdual}, see~\cite{GTZ, PT} for more details), which boils down to the convergence of an Eilenberg-Moore spectral sequence (as explained in~\cite{PT, BoSe}). 
When $X_\bullet$ is a finite simplicial set,   the convergence   is ensured under the weaker assumption that  $Y$  is connected, nilpotent, with finite homotopy groups in degree less or equal to $n$ as is proved in~\cite{F2}. 
It follows that we have the following proposition.
\begin{prop}\label{P:weakenconnectivity}
Assume $X\cong |X_\bullet|$ is compact and $n$-dimensional. Then, Theorem~\ref{T:Einftymapping} and Corollaries~\ref{C:Einftymapping} and~\ref{C:Itdual} hold true if $Y$ is only connected, nilpotent, with finite homotopy groups in degree less or equal to $n$.
\end{prop}
\end{rem}

\section{Algebraic structure of higher Hochschild cochains} \label{S:Operation}

\subsection{Wedge and cup products}\label{S:wedge}

Let $A$ be an $E_\infty$-algebra and assume $B$ is an $A$-algebra, \emph{i.e.},
an $E_\infty$-algebra object in the symmetric monoidal $(\infty,1)$-category $A\text{-}Mod^{E_\infty}$ of $A$-modules, see~\cite{L-HA, KM} for details.
\begin{ex}\label{ex:EinftyAMod}
 A map $f:A\to B$ of $E_\infty$-algebras induces a natural
$E_\infty$-$A$-algebra structure on $B$.

Note further that, if $B$ is a unital $E_\infty$-$A$-algebra, then the map $a\mapsto a\cdot 1_B$ lifts to
a map $f:A\to B$ of $E_\infty$-algebras such that the induced $E_\infty$-$A$-algebra structure
on $B$ is equivalent to the original one.
\end{ex}

Since there is a canonical map $m_A: A\otimes A\to A$ of $E_\infty$-algebras (Proposition~\ref{P:tensorEinftyisEinfty}), any $A$-module inherits  a canonical structure of $A\otimes A$-module (Proposition~\ref{P:tensorEinftyisEinftyMod}).

\begin{lem} \label{L:wedge} Let $M\in A\text{-}Mod^{E_\infty}$ be an $A$-module and $X,Y$ be pointed topological spaces.
There is a natural equivalence $$\mu:\, Hom_{A\otimes A}\left( CH_{X}(A)\otimes CH_{Y}(A), M\right)\stackrel{\simeq}\longrightarrow CH^{X\vee Y}(A,M)$$
\end{lem}
\begin{proof}
The excision property yields a natural equivalence $$CH_{X\vee Y}(A) \cong A\mathop{\otimes}^{\mathbb{L}}_{A\otimes A} \Big( CH_{X}(A) \otimes CH_{Y}(A)\Big) $$ It follows that we have an equivalence
$$Hom_{A\otimes A}\left( CH_{X}(A)\otimes CH_{Y}(A), M\right)\;\cong \; Hom_{A}\left(CH_{X\vee Y}(A), M\right) $$
and the result now follows by Definition~\ref{D-coHoch}.
\end{proof}
Using the above Lemma~\ref{L:wedge}, for pointed spaces $X,Y$ and $B$ an $A$-algebra, we can define the following map
\begin{multline}\label{eq:muvee}
\mu_{\vee}: CH^{X}(A,B) \otimes CH^{Y}(A,B) \longrightarrow Hom_{A\otimes A}\Big( CH_{X}(A) \otimes CH_{Y}(A), B\otimes B\Big) \\
\stackrel{(m_B)_*}\longrightarrow Hom_{A\otimes A}\Big( CH_{X}(A) \otimes CH_{Y}(A), B\Big) \cong CH^{X\vee Y}(A,B)
\end{multline}
where the first map is given by the tensor products $(f,g)\mapsto f\otimes g$ of functions.
\begin{definition}\label{D:wedgeproduct} We call $\mu_{\vee}: CH^{X}(A,B) \otimes CH^{Y}(A,B)\to CH^{X\vee Y}(A,B)$
the \emph{wedge product} of Hochschild cochains (here we do not require that $B$ is unital).
 \end{definition}
 Note that this construction was already studied in some particular cases in our previous papers~\cite{G,GTZ}.
\begin{ex}[Small model for CDGA's]
If $A,B$ are actually CDGA\rq{}s and given finite pointed set models $X_\bullet, Y_\bullet$ of $X,Y$, the map $\mu_{\vee}$ can be combinatorially described as follows.  We have two cosimplicial chain complexes  $CH^{X_\bullet}(A,B)\otimes CH^{Y_\bullet}(A,B)$ (with the diagonal cosimplicial
structure) and  $CH^{X_\bullet \vee Y_\bullet}(A,B)$. There is a cosimplicial
map $\tilde{\mu}:CH^{X_\bullet}(A,B)\otimes CH^{Y_\bullet}(A,B)\to CH^{X_\bullet \vee Y_\bullet}(A,B)$
given, for any $f\in CH^{X_n}(A,B)\cong Hom_A(A^{\otimes \# X_n},B)$, $g\in CH^{X_n}(A,B)\cong Hom_A(A^{\otimes
  \# Y_n},B))$ by $$\mu(f,g)(a_0,a_2,\dots a_{\# X_n},b_2,\dots,
b_{\#Y_n})=\pm a_0.f(1,a_2,\dots a_{\# X_n}).g(1,b_2,\dots, b_{\#Y_n})$$ where $a_0$ corresponds to the element indexed by the base point of $X_n\vee Y_n$ (the sign is given by the usual Koszul-Quillen sign convention). Composing the map $\tilde{\mu}$ with the dual of the Eilenberg-Zilber quasi-isomorphism realizes the wedge map~\eqref{eq:muvee}: $$\mu_{\vee}:CH^{X}(A,B) \otimes CH^{Y}(A,B)\to CH^{X\vee Y}(A,B).$$
\end{ex}

\begin{prop}\label{P:wedge}
The map $\mu_{\vee}$ is associative, \emph{i.e.}, there is a commutative diagram
$$\xymatrix{CH^{X}(A,B) \otimes CH^{Y}(A,B) \otimes CH^{Z}(A,B)   \ar[r]^{\quad\quad  \mu_{\vee} \otimes id }  \ar[d]_{id\otimes \mu_{\vee}} & CH^{X\vee Y}(A,B) \otimes CH^{Z}(A,B) \ar[d]^{\mu_{\vee}}  \\
CH^{X}(A,B) \otimes CH^{Y\vee Z}(A,B) \ar[r]^{\mu_{\vee}} & CH^{X\vee Y\vee Z}(A,B)  } $$ in $k\text{-}Mod_\infty$.
\end{prop}
\begin{proof}
It follows from the associativity of the wedge product of spaces and tensor products
of $E_\infty$-algebras as used in Lemma~\ref{L:wedge} and Proposition~\ref{P:tensorEinftyisEinfty}.
\end{proof}

Let $X$ be a homotopy coassociative  co-$H$-space, \emph{i.e.}, a topological space $X$ endowed with a continuous map
$\delta_X: X\to X\vee X$ which is co-associative (up to homotopy). Note that all suspension spaces has this structure, even though they are rarely manifolds.
Then, by functoriality,  we get a morphism  $\delta_X^*: CH^{X\vee X}(A,B) \to CH^{X}(A,B) $.
\begin{cor}\label{C:cupHoch} Assume $X$ is a homotopy coassociative  co-$H$-space. The composition
$$ \cup_X:\quad
CH^{X}(A,B) \otimes CH^{X}(A,B) \stackrel{\mu_{\vee}} \longrightarrow CH^{X\vee X}(A,B) \stackrel{\delta_X^*}\longrightarrow CH^{X}(A,B)
,$$ called the \emph{cup-product}, induces a structure of graded associative algebra on the cohomology groups $HH^{X}(A,B)$.
 It is further unital if $B$ is unital and $X$ counital.
\end{cor}
\begin{proof}
 The associativity follows from Proposition~\ref{P:wedge} and Proposition~\ref{P-CHfunctorMod}. When $B$ has an unit
$1_B$ and $X$ is counital, then it follows from the contravariance of Hochschild cochains with respect to maps of pointed
spaces that
the unit of $\cup_X$ is given by the canonical map
\begin{equation}\label{eq:defunitwedge} k\stackrel{1_B}\longrightarrow B \cong CH^{pt}(A,B)\stackrel{(X\to pt)^*}\longrightarrow CH^{X}(A,B).
 \end{equation}
Indeed,  the two compositions $\Big(id\vee (X\to pt) \Big)\circ \delta_X$ and
$\Big((X\to pt)\vee id \Big)\circ \delta_X$ are homotopical to the identity. Further, the composition
$$CH^{X}(A,B)\otimes k \stackrel{id\otimes 1_B}\to CH^{X}(A,B)\otimes CH^{pt}(A,B) \stackrel{\mu_\vee}\to
CH^{X}(A,B)$$ is the identity map of $CH^{X}(A,B)$ (which can be checked on any simplicial set model of $X$).
\end{proof}

In particular, the pinching map $S^d\to S^d\vee S^d$ obtained by collapsing the equator to a point induces a cup product
$\cup_{S^d}:CH^{S^d}(A,B) \otimes CH^{S^d}(A,B) \to CH^{S^d}(A,B) $ for Hochschild cohomology over spheres for any $E_\infty$-algebra $A$ and $A$-algebra $B$. For CDGA\rq{}s, this cup-product agrees by definition and Remark~\ref{R:cdgaEMod} with the one introduced by the first author in~\cite{G}.

\begin{ex}[\textbf{Cup-product on the standard simplicial model for spheres}]\label{E:cupSphere}
In the case of spheres and \emph{CDGA\rq{}s}, there is an explicit description of the cup product if one uses the standard model of the dimension $d$ sphere.
Recall that the standard simplicial set model of the circle $S^1$ is the simplicial set, denoted $(S^1_{st})_\bullet$, generated by a unique non-degenerate simplex of dimension $1$. Thus $(S^1_{st})_n:=n_+$ where $n_+=\{0,\cdots,n\}$ has $\{0\}$ for its base point see~\cite{G, GTZ, P}.
 The standard simplicial set $(S^d_{st})_\bullet$ is the iterated smash product
$(S^d_{st})_\bullet=(S^1_{st})_\bullet\wedge \dots \wedge (S^1_{st})_\bullet$
so that $(S^d_{st})_n=({n^d})_+$.
 Using this standard simplicial set model, we have an equivalence
$$CH^{S^d}(A,M)\cong CH^{(S^d_{st})_\bullet}(A,M) \cong
Hom_{k}\big(A^{\otimes (\bullet)^d},M\big) $$ see~\cite{G} (in particular, for the description
of the differential on the right hand side).
Note that we do not know any simplicial map $(S^d_{st})_\bullet \to
(S^d_{st})_\bullet\vee (S^d_{st})_\bullet$ modeling the pinching map.
However, there is a simplicial map $q:sd_2((S^d_{st})_\bullet )\to (S^d_{st})_\bullet\vee (S^d_{st})_\bullet$ modeling
it.
Here $sd_2((S^d_{st})_\bullet)$ is the edgewise subdivision~\cite{McC} (also see~\cite[\S~3.3.2]{GTZ} for examples of applications in the context of higher Hochschild complexes) of $(S^d_{st})_\bullet$;
it can be seen
 as the simplicial model of the circle obtained by gluing two intervals at their endpoints. In other words $sd_2((S^d_{st})_n:=(2n+1)_{+}=\{0,\dots,2n+1\}$ pointed in $0$. The map $q:sd_2((S^d_{st})_n)\to (S^d_{st})_n\vee (S^d_{st})_n\cong \{1,\dots,n\}\cup \{0\} \cup \{n+2,\dots 2n+1\}$  identifies $n+1$ with $0$.
 The cup-product is thus realized by the  induced map $$\Big(CH^{(S^d_{st})_\bullet}(A,B)\Big)^{\otimes 2} \stackrel{\mu_\vee}\longrightarrow CH^{(S^d_{st})_\bullet\vee (S^d_{st})_\bullet}(A,B) \stackrel{q^*}\longrightarrow CH^{sd_2((S^d_{st})_\bullet)}(A,B).$$

\smallskip

There is also  a cochain complex map (\emph{not} induced by a map simplicial sets) making $CH^{(S^d_{st})_\bullet}(A,B)$ a differential graded associative algebra on the nose described in~\cite{G}.
Let $f\in C^{(S^d_{st})_p}(A,B)\cong Hom_k(A^{\otimes \big(p^d\big)},B)$ and  $g\in
 C^{(S^d_{st})_q}(A,B)\cong Hom_k(A^{\otimes \big(q^d\big)},B)$.  Define
$f\cup_0 g \in  C^{(S^d_{st})_{p+q}}(A,B)\cong Hom_k\big(A^{\otimes\big((p+q)^d\big)},B\big)$ by
\begin{multline}f\cup_0 g\Big((a_{i_1,\dots,i_d})_{1\leq i_1,\dots,i_d\leq
  p+q}\Big)\\ =f((a_{i_1,\dots,i_d})_{1\leq i_1,\dots,i_d\leq
  p}\big)g((a_{i_1,\dots,i_d})_{p+1\leq i_1,\dots,i_d\leq
  p+q}\big)\prod a_{j_1,\dots,j_d} \end{multline} where the last product is over all
indices which are not in the argument of $f$ or $g$. Note that for $d=1$, this is the formula of the usual cup-product for Hochschild cochains as in~\cite{Ge} and for $n=2$, this is the Riemann sphere product as defined in~\cite{GTZ}.

The following lemma is proved using a straightforward computation
\begin{lem}\label{L:cupstandard}
Let $A$ be a CDGA and $B$ a commutative differential  graded $A$-algebra.
 Then $\big(CH^{(S^d_{st})_{\bullet}}(A,B), d, \cup_0\big)$
(where $d$ is the total differential as in \cite{G,GTZ}) is an associative differential graded algebra (and unital if $B$ is unital).
\end{lem}

The above explicit formula for the cup-product realizes the cup-product induced
 by the co-$H$ space structures of the spheres.
\begin{prop}\label{P:stagrees}
The natural equivalence $CH^{{S^{d}_{st}}_\bullet}(A,B) \cong CH^{S^d}(A,B)$ is an equivalence of $E_1$-algebras (with $E_1$-structures induced by Lemma~\ref{L:cupstandard} and Corollary~\ref{C:cupHoch}).
\end{prop}
\begin{proof} The proof  in the case $d=2$ is given in the proof of \cite[Proposition 3.3.17]{GTZ}. The argument for general $d$ is the same.
\end{proof}
\end{ex}

We finish this section by giving a more structured version of the wedge product. The wedge product~\eqref{eq:muvee} is the degree $0$-component  of a higher homotopical tower of wedge products. Since $B$ is an $E_\infty$-algebra, it is in particular, by restriction of structure, an $E_n$-algebra for any positive integer $n$.

Assume, that $A$ and $B$ are algebras over the (chains on the) linear isometry operad (\cite{KM}) viewed as algebras over $C_\ast(\mathcal{C}_n)$, the (chains on the) little dimension $n$-cubes operad.
For any $c \in C_\ast(\mathcal{C}_n(r))$, we have a map
$m_B(c): B^{\otimes r} \longrightarrow B$. Similarly to the wedge product, we can thus define the composition
\begin{multline}\label{eq:muveeEn}
\mu_{\vee}(c): \bigotimes_{i=1}^r CH^{X_i}(A,B) \longrightarrow Hom_{A^{\otimes r}}\Big( \bigotimes_{i=1}^r CH_{X_i}(A) , B^{\otimes r}\Big) \\
\stackrel{(m_B(c))_*}\longrightarrow Hom_{A^{\otimes r}}\Big( \bigotimes_{i=1}^r CH_{X_i}(A), B\Big) \cong CH^{\bigvee_{i=1}^r X_i}(A,B).
\end{multline}
where the first map is given by the tensor products of morphisms.
\begin{rem}
When $B$ is a CDGA, then all operations $\mu_{\vee}(c)$ vanishes if $c$ is not of degree $0$.
\end{rem}

\begin{ex}[Strict chain model for algebras over an $E_\infty$ Hopf-operad]
 The map $\mu_{\vee}(c)$ can be defined similarly at the chain level whenever  $A$, $B$ are algebras over  an $E_\infty$-operad $(\mathcal{E}(n))_{n}$ which is further an Hopf-operad, that is, is  equiped with a diagonal of operads $\mathcal{E}(n)\to \mathcal{E}(n)\otimes \mathcal{E}(n)$). In that case one gets map $\mu_{\vee}(c)$ for any $c\in \mathcal{E}(n)$.  A nice model of such an  $E_\infty$ Hopf operad is the Barratt-Eccles operad~\cite{BF}.
\end{ex}

\subsection{A natural $E_d$-algebra structure on Hochschild cochains modeled on
$d$-dimensional spheres}\label{S:Edcochains}
We have already seen  the definition of the cup product for Hochschild cochains modeled on spheres for $E_\infty$-algebras,
 see Corollary~\ref{C:cupHoch}.
We now turn to the full $E_d$-structure on $CH^{S^d}(A,B)$. In ~\cite{G}, the first
author proved that if $A$ is a CDGA and  $B$ is a commutative $A$-algebra (for example $B=A$),
there is a natural $E_n$-algebra structure on $CH^{S^n}(A,B)$. In this section,
we recall this construction in the context of $\infty$-categories of $E_\infty$-algebras. We
will relate this construction to centralizers in the sense of Lurie~\cite{L-HA, L-VI} in
Section~\ref{S:centralizers}.

Recall that we denote $\mathcal{C}_d$ the usual $d$-dimensional little cubes operad
(as an operad of topological spaces)
whose associated $\infty$-operad is a model for $\mathbb{E}^{\otimes}_d$, see\cite{L-HA, L-VI}.
$\mathcal{C}_d(r)$ is the configuration space  of $r$ many $d$-dimensional open cubes in
$I^d$.
Any element $c\in \mathcal{C}_d(r)$ defines a map $pinch_c:S^d\to \bigvee_{i=1 \dots r}S^d$ by collapsing
the complement of the interiors of
the $r$ cubes to the base point. The maps $pinch_c$ assemble
together
to give a continuous map
\begin{equation}\label{eq:pinchcube} pinch: \mathcal{C}_d(r) \times S^d \longrightarrow
 \bigvee_{i=1\dots r}\, S^d.\end{equation}
Note that the map $pinch$ preserves the base point of $S^d$, hence passes to the pointed category.

For any topological space $X$, the singular set functor $X\mapsto \Delta_\bullet(X):=Map(\Delta^\bullet,X)$ defines a (fibrant) simplicial set model of $X$. Hence,
applying  the singular set functor to the above map $pinch$,
the contravariance\footnote{with respect to maps of topological spaces} of Hochschild cochains (see Proposition~\ref{P-CHbifunctorTop} and Proposition~\ref{P-CHfunctorMod})
and the wedge
 product~\eqref{eq:muveeEn} $\mu_{\vee}$, we get, for all $r\geq 1$, a morphism
\begin{multline}\label{eq:pinchSr}
pinch_{S^d,r}^*: C_{\ast}\big(\mathcal{C}_d(r)\big) \otimes \left(CH^{S^d}(A,B)\right)^{\otimes r}\\ \stackrel{diag\otimes id}\longrightarrow C_{\ast}\big(\mathcal{C}_d(r)\big)^{\otimes 2} \otimes \left(CH^{S^d}(A,B)\right)^{\otimes k} \\
\stackrel{\mu_{\vee}(diag^{(2)})_*}\longrightarrow  C_{\ast}\big(\mathcal{C}_d(r)\big) \otimes CH^{\bigvee_{i=1}^r S^d}\big(A, B \big)\\
\stackrel{pinch^*}\longrightarrow CH^{S^d}(A,B)
\end{multline}
in $k\text{-}Mod_\infty$. Here $diag:  C_{\ast}\big(\mathcal{C}_d(r)\big)\to C_{\ast}\big(\big(\mathcal{C}_d(r)\big)^2\big)\stackrel{AW}\to  \Big(C_{\ast}\big(\mathcal{C}_d(r)\Big)^{\otimes 2}$ is the diagonal and $diag^{(1)}$, $diag^{(2)}$ its components.

\begin{theorem}\label{T:EdHoch}
Let  $A$ be an $E_\infty$-algebra and $B$ an $E_\infty$-$A$-algebra (not necesssarily unital).
 The collection of maps $(pinch_{S^d,k})_{k\geq 1}$ makes
$CH^{S^d}(A,B)$ an $E_d$-algebra (naturally in $A$, $B$), which is unital if $B$ is unital.
 Further, the underlying $E_1$-structure of $CH^{S^d}(A,B)$ agrees with the one given by Corollary~\ref{C:cupHoch}.
\end{theorem}
Note that the last  map $C_{\ast}\big(\mathcal{C}_d(r)\big) \otimes
CH^{\bigvee_{i=1}^r S^d}\big(A, B \big)
\stackrel{pinch^*}\longrightarrow CH^{S^d}(A,B)$ in the definition
of the composition~\eqref{eq:pinchSr} is just the map dual to the
one associated to $$ \widetilde{pinch}:
\mathcal{C}_d(r)\longrightarrow
Map_{E_\infty\text{-Alg}}(CH_{S^d}(A), CH_{\bigvee_{i=1}^r
S^d}(A))\big) $$ in Remark~\ref{R:SpaceActiononCH} (see
formula~\eqref{eq:DefofTransffromspacetoCH}).
\begin{proof}
To prove the first statement we need to prove that the morphisms $pinch^*_{S^d,r}$ are compatible with the operadic composition in   $C_{\ast}\big(\mathcal{C}_d(r)\big)$, the  singular chains on the little $d$-dimensional cubes.
Since the diagonal  $diag:  C_{\ast}\big(\mathcal{C}_d(r)\to  \Big(C_{\ast}\big(\mathcal{C}_d(r)\Big)^{\otimes 2}$ is a map of $\infty$-operads, by Proposition~\ref{P:wedge} and Proposition~\ref{P:Functorialitytildemap} (as in Example~\ref{ex:squarecommutestildemap}), the statement reduces to the commutativity of the following diagram for every $j\in \{1,\dots, k\}$
$$\xymatrix{ \mathcal{C}_d(k) \times \mathcal{C}_d(\ell) \times S^d \ar[r]^{pinch} \ar[d]_{\circ_j \times id_{S^d}} &  \mathcal{C}_d(\ell) \times \bigvee_{i=1\dots k}\, S^d \ar[d]^{id_{\bigvee_{i=1\dots j-1} S^d}\times pinch\times id_{\bigvee_{i=j+1\dots k} S^d}} \\  \mathcal{C}_d(k+\ell)\times   S^d \ar[r]^{pinch} &  \bigvee_{i=1\dots k+\ell}\, S^d .} $$ In other words, it reduces to the fact that the pointed sphere $S^d$ is a $\mathcal{C}_d$-coalgebra in the category of pointed topological spaces endowed with the monoidal structure given by the wedge product.

\smallskip

The underlying $E_1$-structure is given by any element in $\mathcal{C}_d(2)$ generating
the homology group $H_0(\mathcal{C}_d(2),\mathbb{Z})\cong \mathbb{Z}$. We can, for instance, take   the
configuration of the two open cubes  $(-1,0)^d$ and $(0,1)^d$ in $(-1,1)^d$. It follows immediately with this choice,
  that the associated $E_1$-structure is given by the cup-product $\cup_{S^d}$
of Corollary~\ref{C:cupHoch} up to equivalences of $E_1$-algebras. The unit is given by the
map~\eqref{eq:defunitwedge} as in Corollary~\ref{C:cupHoch}.

\end{proof}
This theorem will be generalized in Theorem \ref{THM:Ed-on-HC(AM)}
 below to also include generalized sphere topology operations. The naturality in $A$ and $B$ means that if $C$ is
a $B$-$E_\infty$-algebra map, then, there is an $E_d$-algebra homomorphism
$$CH^{S^d}(A,B) \otimes CH^{S^d}(B,C) \longrightarrow CH^{S^d}(A,C)$$
see Proposition~\ref{P:HH(A,B)=CH(A,B)} and Theorem~\ref{T:EnAlgHoch}.

\begin{rem}
For $d>1$, Theorem~\ref{T:EdHoch}  implies that the cup-product makes the  Hochschild cohomology groups $HH^{S^d}(A,B)$ a graded \emph{commutative} algebra (and not only associative as in the case $d=1$).

Further,  when $B=A$  (endowed with its canonical $A$-algebra structure), the $E_d$-structure can actually be lifted naturally to an $E_{d+1}$-structure; see Theorem~\ref{T:Deligne}.(3).
\end{rem}
\begin{rem}
Similarly to Example~\ref{E:cupSphere}, it is possible (but a bit tedious) to give explicit description of the higher $\cup_i$-products on the standard models of the spheres. Details are left to the interested reader.
\end{rem}

The core of the proof of Theorem~\ref{T:EdHoch} is the $E_d$-co-$H$-space structure of the sphere. We say that a pointed topological space $X$  is an $E_d$-co-$H$-space if it is an $E_d$-coalgebra in the category of pointed spaces  with monoidal structure given by the wedge product.

 In other words, there are continuous maps $\mathcal{C}_d(k)\times X\to \bigvee_{i=1\dots k} X$ which are compatible with the operadic composition in    $\mathcal{C}_d$. Mimicking the proof of Theorem~\ref{T:EdHoch} gives the following enhancement of Corollary~\ref{C:cupHoch}:
\begin{cor} Let $X$ be an $E_d$-co-$H$-space, $A$ an $E_\infty$-algebra and $B$ an $E_\infty$-$A$-algebra. Then there is a natural (in $X$ in $E_d$-co-$H$-space, $A$ and $B$) $E_d$-algebra structure on $CH^{X}(A,B)$ refining the cup-product of Corollary~\ref{C:cupHoch}
\end{cor}

\begin{ex}[Smooth CDGA] In characteristic zero,  there is an equivalence $E_n\text{-Alg}\cong H_\ast(\mathcal{C}_n)\text{-Alg}$ between the $\infty$-categories of $E_n$-algebras and (homotopy) $H_\ast(\mathcal{C}_n)$-algebras induced by any choice of formality of the little $n$-disks-operad. 
Note that for $n=1$ the latter operad $H_\ast(\mathcal{C}_1)$ is the operad of associative algebras while for $n\geq 2$, $H_\ast(\mathcal{C}_n)$ is the operad governing $P_n$-algebras.

The next proposition shows that for  \emph{free} graded 
commutative algebras, the homotopy $P_n$-structure given by Theorem~\ref{T:EdHoch} is trivial.

Here a $P_n$-algebra stands for a differential graded commutative unital algebra $(B,d,\cdot)$ equipped with a (homological) degree $n-1$ bracket which makes the iterated suspension $A[1-n]$ a differential graded Lie algebra. 
The bracket and product are further required to satisfy the graded Leibniz identity, see paragraph~\ref{SS:Pn} below.

If $P$ is a $P_n$-algebra and $C$ a $P_n$-coalgebra, we can form the convolution $P_n$-algebra $Hom(C, P)$ (as in \cite{Ta-Defofdalgebra}).

\begin{prop}\label{P:HKRrelative}
 Let $A=(\mathrm{Sym}(V), d)$ and $B=(\mathrm{Sym}(W),b)$ be  differential free graded commutative algebras and assume $n\geq 2$.
 \begin{itemize}
  \item There is an natural quasi-isomorphism $$CH_{S^n}(A) \stackrel{\simeq}\longrightarrow (\mathrm{Sym}(V\otimes H_\bullet(S^n)),\partial)\cong (\mathrm{Sym}(V\oplus V[-n], \partial)$$ of CDGA. Here the right hand side is equipped with the unique differential such that for any $v\in V$,  $\partial (v)=d(v)$ and $\partial v[-n]= (-1)^n s_n(d(v))$ where $s_n$ is the unique derivation satisfying $s_n(w)=w[-n]$, $s_n(w[-n])=0$ for $w\in V$.
  \item There is an natural equivalence of (homotopy) $P_n$-algebras 
  $$CH^{S^n}(A,B) \cong Hom_{\mathrm{Sym}(V)}\Big(\mathrm{Sym}(V\otimes H_\bullet(S^n)),\mathrm{Sym}(W)\Big) $$ where the right hand side is endowed with  the convolution $P_n$-algebra structure\footnote{which has a zero bracket} given by the linear isomorphism
  $$ Hom_{\mathrm{Sym}(V)}\Big(\mathrm{Sym}(V\otimes H_\bullet(S^n)),\mathrm{Sym}(W)\Big)\cong Hom(\mathrm{Sym}(V[-n]),\mathrm{Sym}(W)\Big) .$$ where $\mathrm{Sym}(V[-n])$ is the cofree coproartinian graded cocommutative coalgebra seen as a $P_n$-coalgebra with trivial bracket.  
 \end{itemize}
\end{prop}
 \begin{proof}
  The first claim is (a special case of) the Hochschild-Kostant-Rosenberg Theorem for higher Hoschild homology proved in~\cite{P} and Remark~\ref{R:Astrict}. Note that the quasi-isomorphism is obtained by the degeneration of a spectral sequence which is natural in both $A$ and maps of topological spaces~\cite[\S 2]{P}. In fact, in our case we can use Proposition~\ref{P-CHbifunctorTop}: one has natural equivalences
  $CH_{S^n}(A)\cong C_\ast(S^n) \mathop{\otimes}\limits_{\mathbb{E}_\infty^{\otimes}}^{\mathbb{L}} A$ and 
  $$(\mathrm{Sym}(V\otimes H_\bullet(S^n)),\partial)\cong  H_\ast(S^n) \mathop{\otimes}\limits_{\mathbb{E}_\infty^{\otimes}}^{\mathbb{L}} A.$$
  The equivalence is then induced by the fact that $S^n$ is formal (this is essentially the approach in \cite{P}); alternatively, one can use Corollary~\ref{C:mappingstack} if $V$ is negatively graded.

  The first claim thus also implies that $$CH^{S^n}(A,B) \cong Hom_{\mathrm{Sym}(V)}\Big(\mathrm{Sym}(V\otimes H_\ast(S^n)),\mathrm{Sym}(W)\Big) $$ as cochain complexes and that this equivalence is an equivalence of (homotopy) $P_n$-algebras, where the $P_n$-structures are given as algebras over the operad $\big(H_\ast(\mathcal{C}_n(r))\big)$.  The right hand side is equipped with a $\big(H_\ast(\mathcal{C}_n(r))\big)$-algebra structure given by the action of $\big(H_\ast(\mathcal{C}_n(r))\big)$ on $H_\ast(S^n)$ and thus on $H_\ast(S^n) \mathop{\otimes}\limits_{\mathbb{E}_\infty^{\otimes}}^{\mathbb{L}} A$; this action being similar to the one on $CH_{S^n}(A)\cong C_\ast(S^n) \mathop{\otimes}\limits_{\mathbb{E}_\infty^{\otimes}}^{\mathbb{L}} A$ given by Theorem~\ref{T:EdHoch}. Namely it is given by the composition:  
  \begin{multline}\label{eq:pinchSrbis}
pinch_{S^d,r}^*: H_{\ast}\big(\mathcal{C}_n(r))\big) \otimes \left(Hom_{\mathrm{Sym}(V)}\Big(\mathrm{Sym}(V\otimes H_\ast(S^n)),\mathrm{Sym}(W)\Big)\right)^{\otimes r}\\ 
\stackrel{\mu_B}\longrightarrow  H_{\ast}\big(\mathcal{C}_n(r))\big) \otimes Hom_{\mathrm{Sym}(V)}\Big(\mathrm{Sym}\Big(V\otimes H_\ast\big(\bigvee_{i=1}^r S^n\big)\Big),\mathrm{Sym}(W)\Big)\\
\stackrel{pinch^*}\longrightarrow Hom_{\mathrm{Sym}(V)}\Big(\mathrm{Sym}(V\otimes H_\ast(S^n)),\mathrm{Sym}(W)\Big).
\end{multline}
 Here $\mu_B$ stands for the multiplication in $B=\mathrm{Sym}(W)$. For degree reason, this $H_\ast\big(\mathcal{C}_n(r))\big)$-algebra structure is the one of a CDGA endowed with zero bracket. In particular it corresponds to the convolution $P_n$-algebra  mentioned in the Proposition.
 \end{proof}

\end{ex}

\section{Factorization homology and $E_n$-modules}

In this section,  for $n=\{1,2,\dots,\}\cup \{+\infty\}$, we collect some results on the category of $E_n$-modules over an $E_n$-algebra $A$. In particular we identify it with the category of left modules over the factorization homology $\int_{S^{n-1}}A$ in \S~\ref{S:FactandMod}. Then we apply this to $E_\infty$-modules to show the existence and uniqueness of the lift of Poincar\'e duality in the category of $E_\infty$-modules in \S~\ref{SS:E-inf-PD}. These results are latter used in \S~\ref{S:centralizers}  and \S~\ref{S:Brane}.

We first start by presenting a Factorization algebra point of view on $E_n$-modules.

\subsection{Stratified factorization algebras and $E_n$-modules}\label{SS:EnMod}

One can define a notion of locally constant factorization algebra for stratified manifolds as well as factorization homology for such spaces. We refer to~\cite{AFT} and \cite{G-Houches} for details. 

In this paper, we will essentially only need very special and easy cases: the disk and the sphere  with a marked point.

Let $X$ be a  Hausdorff paracompact topological space. By a \emph{stratification of $X$}, we mean  an union of a sequence of closed subspaces $\emptyset= X_{-1} \subset X_0\subset X_1\subset \cdots \subset X_n=X$   such that, for any point $x_i\in X_i$, there  is a neighborhood $U_{x_i}$ and a filtration preserving homeomorphism $U_{x_i}\stackrel{\phi}{\simeq} \R^i \times C(L)$ in $X$ where $C(L)$ is the (open) cone on a stratified space of dimension $n-i-1$.

 Note that  $X_{i+1}\setminus X_i$ is not necessarily connected nor non-empty. 
In all the examples considered in this paper, it will nevertheless be  a smooth manifold  of dimension $i+1$.
 We call the connected components of $X_i\setminus X_{i-1}$ the  strata of dimension $i$ of $X$. 

We define the \emph{\textbf{index}} of an open subset $U\subset X$ to be the \emph{smallest integer} $j$ such that $U\cap X_j\neq \emptyset$. 
\begin{definition}  \label{D:Stratified} An an open subset $U$ of $X$ is called a (stratified) \emph{disk} if there is a filtration preserving homeomorphism  $U \cong R^i\times C(L)$ with $L$ stratified of dimension $n-i-1$, and $i$ is the index of $U$. 

\smallskip

A  \emph{factorization algebra} $\mathcal{F}$ over a \emph{stratified manifold} $\emptyset \subset X_0\subset X_1\subset \cdots \subset X_n=X$  is  called \emph{locally constant} if the following condition is satisfied: 

{ }if $U\hookrightarrow V$ is an inclusion of  (stratified) disks of \emph{same index}\footnote{that is, for all $j=0,\dots, n-1$, either  $V\cap X_j =\emptyset$ or both $U\cap X_j$ and $V\cap X_j$ are non empty} and further $V$ intersects only one stratum of $X_i \setminus X_{i-1}$ where $i$ is the index of $V$, 
 then we require that the structure map $\mathcal{F}(U) \to \mathcal{F}(V)$ is a quasi-isomorphism.

We will also say that $\mathcal{F}$ is \emph{stratified locally constant} when we want to insist on the stratification. 
\end{definition}

\begin{ex}[Pointed disk] We write $D^n_*$ for the pointed 
 Euclidean open disk viewed as a stratified manifold. It has  only two  strata: a dimension $0$ stratum given by its center
and the dimension $n$-stratum given by the complement of the center. In other words $D^n_0=\{0\}=D^n_1\dots=D^n_{n-1}$ and $D^n_n=\R^n$.

Thus a factorization algebra (or an $N(\text{Disk}(D^n)$-algebra) on $D^n_*$ is locally constant if the
structure map $\mathcal{F}(U) \to \mathcal{F}(V)$ is an equivalence  when  $U\subset V$ are open disks such that either  $U$
contains the base point or $V$ is included in the $n$-stratum $D^n-\{0\}$. In other words, we do \emph{not}
require $\mathcal{F}(U) \to \mathcal{F}(V)$ to be an equivalence if $V$ contains the base point while $U$ does not.

\smallskip

We let $Fac^{lc}_{D^n_*}$ be ($(\infty,1)$)-category of stratified locally constant factorization algebras on the pointed disk.
\end{ex}

\begin{definition}\label{D:EnModCat} \begin{itemize}
\item We denote  $Fac^{lc, res}_{D^n_*}$, the $(\infty,1)$-category 
$$Fac^{lc, res}_{D^n_*} :=  \text{Fac}^{lc}_{D^n_*}\times_{\text{Fac}^{lc}_{\R^n\setminus\{0\}}}\text{Fac}^{lc}_{\R^n}$$  of
 pairs $(\mathcal{M}, \mathcal{A})$ of locally constant factorization algebras on respectively $D^n_*$ and $D^n$ together with an equivalence of factorization algebras  $\mathcal{M}_{|D^n-\{0\}} \stackrel{\simeq}\to \mathcal{A}_{|D^n-\{0\}}$  between the restrictions of $\mathcal{M}$ and  $\mathcal{A}$ to  $D^n-\{0\}$.
 \item
 Fix $\mathcal{B} \in \text{Fac}^{lc}_{D^n}$. 
 We denote $Fac^{lc}_{D^n_*|\mathcal{B}}$ the $(\infty,1)$-category
 $$Fac^{lc}_{D^n_*|\mathcal{B}} :=  \text{Fac}^{lc}_{D^n_*}\times_{\text{Fac}^{lc}_{\R^n\setminus\{0\}}}\{\mathcal{B}\},$$ that is the $(\infty,1)$-category  of  pairs $(\mathcal{M}, \mathcal{A})$ as above such that $\mathcal{A}$ is (equipped with an equivalence with) $\mathcal{B}$. \end{itemize} 
\end{definition}

\begin{rem}\label{R:EnModCat} 
Given an object $(\mathcal{M}, \mathcal{A})\in Fac^{lc, res}_{D^n_*}$, the factorization algebra $\mathcal{A}$ is essentially  determined by $\mathcal{M}$, since, by the locally constant condition, it is essentially defined by its restriction to any open ball in $D^n$, thus to any open ball included in $D^n-\{0\}$.

\smallskip

The same is \emph{locally} true for any object $\mathcal{N}$ of $Fac^{lc}_{D^n_*}$. Indeed, restricting to any disk $D$ of  $D^n-\{0\}$ yields a locally constant factorization algebra on the disk and thus an $E_n$-algebra $A_D\cong \mathcal{N}(D)$. For any two disks $D_1$, $D_2$, the  $E_n$-algebra $A_{D_1}$ and $A_{D_2}$ are equivalent, but such an equivalence depends on a choice of a bigger disk containing both of them. Thus, the main difference  between $Fac^{lc, res}_{D^n_*}$ and $Fac^{lc}_{D^n_*}$ is that we assume that these equivalences can be made canonically, which amout to the fact that $\mathcal{N}_{|D^\setminus\{0\}}$ is equivalent ot the restriction of a fixed factorization algebra on $\R^n$.
\end{rem}
\begin{ex}
 For $n=1$, the category $Fac^{lc, res}_{D^1_*}$ is equivalent to the category of all bimodules over an $E_1$-algebra.  
 However, $Fac^{lc}_{D^1_*}$  is equivalent to the category of all bimodules, that is the category whose objects are $(A,B)$-bimodules for  some  $E_1$-algebras $A,B$ which may be non-equivalent.
\end{ex}

Theorem~\ref{P:En=Fact} has an analogue for modules:
\begin{prop}[\cite{G-Houches}]\label{P:EnMod=Fact}
\begin{itemize} \item There is an equivalence between the $(\infty,1)$-categories $Mod^{E_n}$ of all $E_n$-modules (\S~\ref{S:EnAlg}) and $Fac^{lc, res}_{D^n_*}$, the
locally constant factorization algebras on the pointed disk as in Definition~\ref{D:EnModCat}.
\item Let $A$ be an $E_n$-algebra corresponding to a factorization algebra $\mathcal{A}\in \text{Fac}^{lc}_{\R^n}$ under Theorem~\ref{P:En=Fact}. 
Then the above equivalence restricts to an equivalence 
$$A\text{-Mod}^{E_n} \, \cong \, Fac^{lc}_{D^n_*|\mathcal{A}}.$$
\end{itemize}
\end{prop}
Note that the pushforward $D^n_*\to pt$ realizes the forgetful
functor $Mod^{E_n}(\mathcal{C}) \to \mathcal{C}$ of \S~\ref{S:EnAlg}. 
Further, as
noted in Remark~\ref{R:EnModCat}, fixing any Euclidean sub-disk
$D\subset D^n_*-\{0\}$ we get a functor $Fac^{lc, res}_{D^n_*} \to
Fac^{lc}_{D}$ which is equivalent to the functor $\iota:
Mod^{E_n}\to E_n\text{-Alg}$, \emph{i.e.}, to the forgetful functor
$(\mathcal{M}, \mathcal{A})\mapsto \mathcal{A}$. 

Forgetting the
stratification yields a canonical functor $Fac^{lc}_{D^n} \to Fac^{lc, res}_{D^n_*}$
realizing the canonical functor $E_n\text{-Alg}\to Mod^{E_n}$ (which
views an $E_n$-algebra as a module over itself in a canonical way).

\begin{rem}[Induced $E_n$-module structure associated to an $E_n$-algebra map]\label{R:Enmodulestructure}
Let $A$ be an $E_n$-algebra  and $f:A\to B$ an $E_n$-algebra map and let $B$ be endowed with the induced $A$-$E_n$-module structure. This module structure has an easy description in terms of factorization algebras. Denote $\mathcal{A}: U\mapsto \int_U A$ and $\mathcal{B}: V\mapsto \int_V B$ be the associated factorization algebras
on $\mathbb{R}^n$ (see Theorem~\ref{T:Theorem6GTZ2}). By Proposition~\ref{P:EnMod=Fact} and Proposition~\ref{P:AlternativeFacAlg}, the data of the $A$-$E_n$-module structure on $B$ is equivalent  to the data of a parametrized factorization algebra. 
Thus,
to  any embedding $\coprod_{i=0}^r \phi_i :  \coprod_{i=0}^r \mathbb{R}^n \to \mathbb{R}^n$ (with $\phi_0(0)=0$)
and commutative diagram
\begin{equation*}\xymatrix{ \coprod_{i=0}^r \mathbb{R}^n \ar[rr]^{h} \ar[rd]_{\coprod_{i=0}^r \phi_i} & & \mathbb{R}^n \ar[ld]^{\psi} \\ &  \mathbb{R}^n &}\end{equation*}
 of embeddings, one can  associate a natural\footnote{with respect to composition of embeddings, that is satisfies the usual associativity condition of the structure maps of a prefactorization algebra in the sense of~\cite{CG}} map
\begin{equation}\label{eq:dataEnmodulestructure}
\mathcal{A}(\phi_1(\R^n))  \otimes \cdots \otimes \mathcal{A}(\phi_r(\R^n))  \otimes\mathcal{B}(\phi_0(\R^n))   \longrightarrow \mathcal{B}(\psi(\R^n)).
\end{equation}
This map~\eqref{eq:dataEnmodulestructure} is very simple to describe, it is the composition
\begin{multline*}
\int_{\phi_1(\R^n)} A \otimes \cdots \otimes \int_{\phi_r(\R^n)} A \otimes \int_{\phi_0(\R^n)} B  \\
\stackrel{\big(\bigotimes_{i=1}^r \int_{\phi_i(\R^n)} f \big)\otimes id}\longrightarrow \int_{\phi_1(\R^n)} B \otimes \cdots \otimes \int_{\phi_r(\R^n)} B \otimes \int_{\phi_0(\R^n)} B
\\ \longrightarrow\int_{\psi(\R^n)} B \cong \mathcal{B}(\psi(\R^n)).
\end{multline*} where the last map is given by the factorization algebra structure of $\mathcal{B}$, \emph{i.e.}, the $E_n$-algebra structure of $B$.

 Now, let $g: B\to C$ be another $E_n$-algebra map endowing $C$ with an $A$-$E_n$-module structure; let $\mathcal{C}: U\mapsto \int_U C$ be the associated factorization algebra. Then a map $A$-$E_n$-modules $h: B\to C$ is equivalent to the data of a stratified parametrized factorization algebra map $\int_{U} h: \mathcal{B}(U)\cong \int_U B\to \int_U C \cong \mathcal{C}(U)$ such that, for all $\phi_0,\dots, \phi_r$ and $\psi$ as above,  the following diagram
\begin{equation*} \xymatrix{
\Big(\bigotimes_{i=1}^r \int_{\phi_i(\R^n)} A\Big)\otimes \int_{\phi_0(\R^n)} B
\ar[d]_{\big(\bigotimes id\big)\otimes \int_{\phi_0(R^n)}h}  \ar[rr]^{\qquad(\bigotimes \int_{\phi_i(\R^n)} f \big)\otimes id}& &
\bigotimes_{i=0}^r \int_{\phi_i(\R^n)}B \ar[r] & \int_{\psi(\R^n)} B \ar[d]^{\int_{\psi(\R^n)} h} \\
\Big(\bigotimes_{i=1}^r \int_{\phi_i(\R^n)} A\Big)\otimes \int_{\phi_0(\R^n)} C
 \ar[rr]_{\qquad(\bigotimes \int_{\phi_i(\R^n)} g\circ f \big)\otimes id}& &
\bigotimes_{i=0}^r \int_{\phi_i(\R^n)}C \ar[r] & \int_{\psi(\R^n)} C}
\end{equation*} is commutative.
\end{rem}

\subsection{Universal enveloping algebra of an $E_n$-algebra}\label{S:FactandMod}
In this section, we will recall some general results that are needed,
among other places, in the proof of Proposition~\ref{P:coHH=coTCH}.
We start with the following very useful result describing the
universal enveloping algebra of an $E_n$-algebra in terms of
factorization homology. Note that  universal enveloping algebras of
$E_n$-algebras are given by the left adjoint of the forgetful
functor $E_n\text{-Alg} \to \hkmod$ (for instance see~\cite{F}).

\begin{prop}[Francis, Lurie]\label{P:EnMod}
Let $A$ be an $E_n$-algebra ($n\in \mathbb{N}$). The category  $A\text{-}Mod^{E_n}$ is equivalent as a symmetric monoidal $(\infty, 1)$-category to the category of left modules over the factorization homology  $\int_{S^{n-1}}(A)$, with respect to the canonical outward n-framing on $S^{n-1}\subset \mathbb R^n$.
\end{prop}
\begin{proof} This is proved in~\cite{F} and can also be found in~\cite{L-VI, L-HA}.
 Note that by the $\infty$-version of the Barr-Beck theorem \cite{L-II, L-HA} for any $E_n$-algebra $A$,
there is an $E_n$-enveloping algebra $U^{(n)}_A\in E_1\text{-Alg}$
with a natural equivalence $U^{(n)}_A\text{-}LMod \cong
A\text{-}Mod^{E_n}$, see \emph{loc. cit} and also~\cite{Fre-Mod}.
Now the result follows from the natural equivalence
$U^{(n)}_A\stackrel{\simeq}\to\int_{S^{n-1}}A$ see~\cite[Proposition
3.19]{F}.
 \end{proof}

 This Lemma extends to the case $n=\infty$, see Lemma~\ref{L:EinftyMod} and more importantly Theorem~\ref{T:lifttoEinfty} below.

 \begin{rem}
  In terms of factorization algebra, the equivalence in Proposition~\ref{P:EnMod} can be thought
of as the pushforward of factorization algebras.
A Euclidean norm of a vector defines a canonical map $N:D^n_*\to [0,1)_*$, where $[0,1)_*$
is the half open interval with a unique closed stratum given by the point $0$.
The $(\infty,1)$-category of locally constant factorization algebra on the stratified manifold $[0,1)_*$ is
equivalent to the $(\infty,1)$-category $LMod$. The equivalence of Proposition~\ref{P:EnMod} is then just  induced by the pushforward
$N_*: Fac^{lc}_{D^n_*}\to Fac^{lc}_{[0,1)_*}$ by $N$. See~\cite{G-Houches} for details.
 \end{rem}

 \smallskip

 We will later need the following lemma, which expresses the compatibility of the equivalence of categories given
 by Proposition~\ref{P:EnMod} with the inclusions of $E_{n+1}$-algebras inside $E_n$-algebras. We feel this lemma is of independent interest anyhow.
 Suppose $X$ is a codimension $1$  submanifold of an $n$-framed manifold and $Y$ endowed with a
trivialization $\psi:X\times \mathbb{R}\hookrightarrow Y$ of a tubular neighborhood in $Y$.
Then, for any $E_n$-algebra $A$, there is a canonical map $\psi:\int_X A \to \int_Y A$
(which depends on the trivialization).

 \begin{lem}\label{L:EnModseq} Let $A$ be an $E_{n+1}$-algebra
and $\phi_n: S^{n-1}\times \mathbb{R} \hookrightarrow S^{n}$ the inclusion of an open (tubular) neighborhood
of the equatorial sphere
 $S^{n-1}=S^n\cap \big(\mathbb{R}^n\times \{0\}\big)$ inside $S^n$.
The following diagram, in which the vertical arrows are given by Proposition~\ref{P:EnMod}, is commutative,
  $$\xymatrix{  A\text{-}Mod^{E_{n+1}} \ar@{^{}->}[r]   \ar[d]^{\simeq}&  A\text{-}Mod^{E_n}  \ar[d]^{\simeq}  \\
  \ar[r]^{\phi_n^*}  (\int_{S^{n}}A)\text{-}LMod &  (\int_{S^{n-1}}A)\text{-}LMod     }.$$
 \end{lem}
\begin{proof}
 The universal property of the  $E_n$-enveloping algebra $U^{(n)}_A$ implies that the map of
$\infty$-operad $\mathbb{E}^{\otimes}_{n}\to \mathbb{E}^{\otimes}_{n+1}$ (see \S~\ref{S:EnAlg})
yields a canonical map of
$E_1$-algebras $U^{(n)}_A \to U^{(n+1)}_A$. It remains to identify the composition
 $\theta_n:\int_{S^{n-1}} A\cong U^{(n)}_A \to U^{(n+1)}_A\cong \int_{S^n} A$ with $\phi_n$ to prove the lemma.
 From the proof of~\cite[Proposition 3.19]{F}, we know that
 $U^{(n)}_A$ is computed by the colimit of a (simplicial) diagram,
 \begin{equation}\label{eq-UA}
  \coprod_{K\in Fin} \mathbb{E}_{n}^{\otimes}\Big(K\coprod\{pt\}\Big)\otimes A^{\otimes K} \leftleftarrows
\coprod_{{\mathbb{E}_n^{\otimes}}(J,I)} \mathbb{E}_{n}^{\otimes}\Big(I\coprod\{pt\}\Big)\otimes A^{\otimes J} \cdots.
 \end{equation}
Similarly,  $\int_{S^{n-1}}A$ can be computed as the colimit of a similar diagram,
\begin{equation}\label{eq-UA2}
  \coprod_{K\in Fin}\hspace{-0.5pc}\mathop{Emb^{fr}}\Big(\coprod_{K}D^n, S^{n-1}\times \mathbb{R}\Big)\otimes A^{\otimes K}
\leftleftarrows \coprod_{{\mathbb{E}_n^{\otimes}}(J,I)} \hspace{-0.5pc}\mathop{Emb^{fr}}\Big(\coprod_{I} D^n, S^{n-1}\times \mathbb{R}\Big)
\otimes A^{\otimes J} \cdots,
 \end{equation}
 where $\mathop{Emb^{fr}}$ denotes the space of framed embeddings.

 Furthermore, the equivalence $U^{(n)}_A \stackrel{\simeq}\to \int_{S^{n-1}} A$ is induced by the canonical maps
$\mathbb{E}_{n}^{\otimes}\Big(K\coprod\{pt\}\Big)\to \mathop{Emb^{fr}}\Big(\coprod_{K}D^n, S^{n-1}\times \mathbb{R}\Big)$
obtained by translating the disk labeled by the distinguished point to the origin;
see the proof of ~\cite[Proposition 3.19]{F}.

The natural map $U^{(n)}_A \to U^{(n+1)}_A$ is induced by the natural maps
 $\mathbb{E}_{n}^{\otimes}\Big(K\coprod\{pt\}\Big)\to \mathbb{E}^{\otimes}_{n+1}\Big(K\coprod\{pt\}\Big)$
in diagram~\eqref{eq-UA}.
 Since the natural map of $\infty$-operads $\mathbb{E}_{n}^{\otimes}\to \mathbb{E}^{\otimes}_{n+1}$
is given by sending  $n$-dimensional disks $D$ to $D\times \mathbb{R}$, we get commutative diagrams
 $$\xymatrix{  \mathbb{E}_{n}^{\otimes}\Big(K\coprod\{pt\}\Big)\ar[rr] \ar[d]&&
\mathbb{E}^{\otimes}_{n+1}\Big(K\coprod\{pt\}\Big) \ar[d] \\
 \mathop{Emb^{fr}}\Big(\coprod_{K}D^n, S^{n-1}\times \mathbb{R}\Big) \ar[rr] &&
\mathop{Emb^{fr}}\Big(\coprod_{K}D^{n+1}, S^{n}\times \mathbb{R}\Big)} $$
 where the lower map is induced by the embedding $$\phi_n\times \mathbb{R}: (S^{n-1}\times \mathbb{R})\times \mathbb{R} \hookrightarrow S^n\times \mathbb{R}$$ prescribed in the assumptions of the lemma. It follows that $\theta_n: \int_{S^{n-1}} A\cong U^{(n)}_A \to U^{(n+1)}_A\cong \int_{S^n} A$ is obtained by taking the colimit of  these lower maps $$\mathop{Emb^{fr}}\Big(\coprod_{K}D^n, S^{n-1}\times \mathbb{R}\Big) \stackrel{\phi_n\times \mathbb{R}}\longrightarrow \mathop{Emb^{fr}}\Big(\coprod_{K}D^{n+1}, S^{n}\times \mathbb{R}\Big)$$ applied to diagram~\eqref{eq-UA2}, which, by definition, is the map  $\phi_n:\int_{S^{n-1}} A\to \int_{S^n} A$.
\end{proof}

\begin{rem}[The trivial $E_n$-$A$-module structure on $A$]\label{R:A=intDA}It follows from the axioms of factorization homology
(see~\cite{L-VI, L-HA, F} or \cite[Section 6]{GTZ2}) that for any $E_n$-algebra $A$, there is a natural
equivalence $A\cong \int_{D^n} A$ in $k-Mod_{\infty}$. 
Note that  $D^n$ has an immediate trivialization ${S^{n-1}\times D^1}\cong D^n\setminus\{0\}$ of a complement of a point (and in fact of any complement of a closed disk). Hence, there is a natural left
$(\int_{S^{n-1}}A)$-module structure  on $\int_{D^n} A$ see~\cite{L-VI, L-HA},
\cite[Section 3]{F} or \cite[Section 6.3]{GTZ2} for details. Note that this left $\big(\int_{S^{n-1}}A\big)$-module structure is given by a map
$$\int_{S^{n-1}}A\otimes \int_{D^n} A\cong \int_{\big(S^{n-1}\times D^1\big)\coprod D^n} A
\, \longrightarrow \, \int_{D^{n}}A$$
induced by any embedding $\big(S^{n-1}\times D^1\big)\coprod D^n\hookrightarrow D^n$
mapping $D^n$ onto a subdisk
$D(0,r)\subset D^n$ (for some radius $r>0$) and $S^{n-1}\times D^1$ onto a sub-annulus included in
$D^n \setminus D(0,r)$.

By Proposition~\ref{P:EnMod}, we get a natural $A$-$E_n$-module
structure on $A$ which relates to the canonical $A$-$E_n$-module structure of $A$ as follows.
\begin{lem}\label{L:A=intDA}
 The natural equivalence $A\cong \int_{D^n} A$ is an equivalence of $A$-$E_n$-modules.
\end{lem}
\end{rem}
\begin{proof} Consider a framed embedding of $D^n\hookrightarrow \mathbb{R}^n$.
 Since $D^n\setminus \{0\}$ is framed, the result follows from \cite[Remark 3.26]{F}. In fact, the
proof of \cite[Proposition 3.19]{F} applied to $A$ and not the unit object of
$\mathcal{C}=k\text{-}Mod_\infty$ gives an equivalence of left $\int_{S^{n-1}}A$-modules between $A$ viewed
as an  $(\int_{S^{n-1}}A)$-module and $\int_{D^n}A$.
\end{proof}

\subsection{Application of higher Hochschild chains to prove Theorem~\ref{T:lifttoEinfty}} \label{SS:lifttoEinfty}

For $E_\infty$-algebras, Proposition~\ref{P:EnMod} has a simpler and well-known (see~\cite{L-HA, KM, Fre-Mod}) form, see Theorem~\ref{T:lifttoEinfty} below. In this section, we recall this result and then give an independent proof using the formalism of factorization homology/higher Hochschild chains.

\smallskip

 The following theorem is due to Lurie~\cite[Proposition 4.4.1.4]{L-HA},~\cite{L-III} and also appeared independently in the work of Fresse~\cite{Fre-Mod}.
\begin{theorem}\label{T:lifttoEinfty} Let $A$ be an $E_\infty$-algebra. There is an equivalence of symmetric monoidal $\infty$-categories between the category $A\text{-}Mod^{E_\infty}$ of $E_\infty$ $A$-Modules and the category of left $A$-modules (where $A$ is viewed as an $E_1$-algebra). In particular:
\begin{itemize}
\item Any left $A$-module can be promoted into an $E_\infty$-$A$-module (up to quasi-isomorphisms)
\item Any map $f:M\to N$ of left $A$-modules can be lifted to a map of $E_\infty$-modules (up to a contractible family of choices)
\end{itemize}
\end{theorem}
 The theorem allows us to reduce the study of $E_\infty$-modules on $C^\ast(X)$ to the study of left modules on the (differential graded) associative algebra $(C^\ast(X),\cup)$, for instance see~\S~\ref{SS:E-inf-PD}. Also, see Example~\ref{E:LiftModstructure} and Remark~\ref{R:LiftMod} for a more explicit description of the lifts of left modules into $E_\infty$-ones.
 \begin{rem}When $A$ is an $E_\infty$-algebra the categories of left and right modules over $A$ (viewed as an $E_1$-algebra) are equivalent. Hence, one can replace left modules by right modules in Theorem~\ref{T:lifttoEinfty}.
 \end{rem}

The rest of this section is devoted to an alternative proof of Theorem~\ref{T:lifttoEinfty} using \S~\ref{S:FactandMod} and higher Hochschild theory. We first start with the following analogue of Proposition~\ref{P:EnMod}.

\begin{lem}\label{L:EinftyMod}
Let $A$ be an $E_\infty$-algebra. The category $A\text{-}Mod^{E_\infty}$ of $E_\infty$-$A$-Modules is equivalent as a symmetric monoidal $(\infty, 1)$-category to the category of left modules over the derived Hochschild  chains $CH_{S^{\infty}}(A)$, viewed as an $E_1$-algebra by forgetting extra structure.
\end{lem}
\begin{proof} By Theorem~\ref{T:CH=TCH}, there is a canonical equivalence $\int_{S^{n}} A \cong CH_{S^{n}}(A)$ for any $n\in \mathbb{N}$.

The maps of operad $\mathbb{E}^{\otimes}_{i}\to \mathbb{E}^{\otimes}_{i+1}$ are induced by the maps $\mathbb{R}^{i}\cong \mathbb{R}^{i}\times\{0\}\hookrightarrow \mathbb{R}^{i+1}$ which, by restriction induces canonical maps $S^{i-1}\cong S^i\cap \big(\mathbb{R}^{i}\times\{0\}\big) \hookrightarrow S^{i}$, and, by functoriality, maps $\phi_i: CH_{S^{i-1}}(A) \to CH_{S^{i}}(A)$.

By Lemma~\ref{L:EnModseq} (and Theorem~\ref{T:CH=TCH}), we   get a commutative diagram
{\small $$\xymatrix{ \dots   \ar@{^{(}->}[r]&  \ar@{^{(}->}[r]  A\text{-}Mod^{E_{n+1}} \ar[d]^{\simeq}&  A\text{-}Mod^{E_n}  \ar[d]^{\simeq} \ar@{^{(}->}[r] & \dots   \ar@{^{(}->}[r] &A\text{-}Mod^{E_1} \ar[d]^{\simeq} \\
\dots   \ar[r]&  \ar[r]^{\phi_n^*}  CH_{S^{n}}(A)\text{-}LMod &  CH_{S^{n-1}}(A)\text{-}LMod   \ar[r] & \dots   \ar[r]^{\hspace{-2pc}\phi_0^*} &CH_{S^0}(A)\text{-}LMod   }.$$}
From Lemma~\ref{L:SinftyMod}, we deduce a natural equivalence $$A\text{-}Mod^{E_\infty}\;\cong\; \mathop{\lim}\limits_{n\geq 1} CH_{S^n}(A)\text{-}LMod.$$
Mimicking the proof of~\cite[Lemma 5.1.3]{GTZ2}, we get a natural equivalence $$\mathop{\colim}\Big( CH_{S^0}(A)\to CH_{S^1}(A)\to \dots \to CH_{S^n}(A)\to \dots\Big) \stackrel{\simeq}\longrightarrow CH_{S^\infty}(A).$$ It follows that we have an equivalence $CH_{S^\infty}(A)\text{-}LMod \to \mathop{\lim}\limits_{n\geq 1} CH_{S^n}(A)\text{-}LMod$ and the lemma follows.
\end{proof}

\begin{lem} \label{L:SinftyMod} Let $A$ be an $E_\infty$-algebra, then $CH_{S^\infty}(A)$ is canonically  equivalent to $A$ as an $E_\infty$-algebra. In particular, there is a canonical equivalence $$A\text{-}Mod^{E_n}\;\cong\; CH_{S^\infty}(A)-Mod^{E_n}$$ for any  $n\in \{0,1,\dots, \infty\}$.
\end{lem}
\begin{proof} This follows from Theorem~\ref{T:derivedfunctor} since $S^\infty$ has a deformation retraction to a point.
\end{proof}

The  canonical map $\mathbb{E}^{\otimes}_{n-1}\to \mathbb{E}^{\otimes}_{n}$ yields a natural functor
 $A\text{-}Mod^{E_n} \to A\text{-}Mod^{E_{n-1}}$ for any $E_n$-algebra $A$.
\begin {lem}\label{L:EinftyModLim}
Let $A$ be an $E_\infty$-algebra. Then $A\text{-}Mod^{E_\infty}$ is the (homotopy) limit
$$A\text{-}Mod^{E_\infty}\cong \mathop{lim} \Big( \dots \to A\text{-}Mod^{E_n} \to A\text{-}Mod^{E_{n-1}} \to \dots \to A\text{-}Mod^{E_1}\Big). $$
\end{lem}
\begin{proof}
 Recall that $\mathbb{E}^{\otimes}_\infty\cong  \mathop{colim}\limits_{n\geq 1} \mathbb{E}_n^{\otimes}$~\cite{L-VI, L-HA}.
Since we have commuting restriction maps $A\text{-}Mod^{E_\infty}\to A\text{-}Mod^{E_n}$ ($n\in \mathbb{N}$),
there is a canonical map $$\tau: A\text{-}Mod^{E_\infty}\longrightarrow \mathop{lim}\limits_{n\geq 1} A\text{-}Mod^{E_n}.$$
We want to prove that this map $\tau$ is an equivalence.
 Given any $E_n$-algebra $A$ and an $E_n$-$A$-module $M$, the trivial extension
 $A\oplus M$ has a natural structure of $E_n$-algebra.
 The trivial extension functor $M\mapsto A\oplus M$ is a (natural in $A$) equivalence
 of $\infty$-categories between $A\text{-}Mod^{E_n}$
 and $E_n\text{-Alg}_{/A}$ which, by naturality, commutes with the restriction of structure functors
$A\text{-}Mod^{E_n} \to A\text{-}Mod^{E_{n-1}}$ and
$E_n\text{-Alg}_{/A}\to E_{n-1}\text{-Alg}_{/A}$.
 It follows that any object of  $ \mathop{\lim}\limits_{n\geq 1} A\text{-}Mod^{E_n}$ is equivalent to an object of
 $ \mathop{lim}\limits_{n\geq 1}E_n\text{-Alg}_{/A}$.
 Such an object is a (homotopy type of) chain complex equipped with compatible $E_n$-structures for all
$n\geq 1$, thus is an $E_\infty$-algebra. It is also endowed with compatible augmentations of $E_n$-algebras to $A$.
 Hence we get a map $$\varphi:\mathop{lim}\limits_{n\geq 1} E_n\text{-Alg}_{/A} \longrightarrow E_\infty\text{-Alg}_{/A}$$
which is a quasi-inverse of the canonical map $$ \tau:
E_\infty\text{-Alg}_{/A}\longrightarrow
  \mathop{lim}\limits_{n\geq 1} E_n\text{-Alg}_{/A}$$ induced by the restrictions functors.
 The result now follows by applying the (quasi-inverse of the) trivial extension functor.
\end{proof}

\begin{proof}[Proof of Theorem~\ref{T:lifttoEinfty}] The first statement follows from Lemmas~\ref{L:EinftyMod} and~\ref{L:SinftyMod}
 and the last two statements are consequences of the first one.
\end{proof}

\begin{ex}\label{E:LiftModstructure} Let $A$ be an $E_\infty$-algebra and  $M$ be a left $CH_{S^\infty}(A)$-module (here $CH_{S^\infty}(A)$ is equipped with its canonical $E_1$-structure by restriction of structure along the operad maps~\eqref{eq:towerofEnoperad}).
Since  $CH_{S^\infty}(A)$ is in fact an $E_\infty$-algebra, for any $n\in \{1,\dots,\infty\}$, it is  canonically  equivalent to its opposite $E_n$-algebra $CH_{S^\infty}(A)^{op}$. The equivalence is explicitly given by the antipodal map $ S^{\infty}\stackrel{\text{ant}}\to S^{\infty}$ (and functoriality of Hochschild chains). Thus, there is a canonical structure of $CH_{S^\infty}(A)\otimes CH_{S^\infty}(A)$-$E_\infty$-modules on $CH_{S^\infty}(A)$. By restriction of structure, the map $$CH_{S^\infty}(A)\stackrel{1\otimes id}\longrightarrow CH_{S^\infty}(A)\otimes CH_{S^\infty}(A)$$ endows $CH_{S^\infty}(A)$  with a right module structure over itself (viewed as an $E_1$-algebra)  which commutes with the $CH_{S^\infty}(A)$-module structure induced by the map  $$CH_{S^\infty}(A)\stackrel{id\otimes 1}\longrightarrow CH_{S^\infty}(A)\otimes CH_{S^\infty}(A).$$
This extra structure of  $CH_{S^\infty}(A)$ endows the tensor product (of a right and left module over $CH_{S^\infty}(A)$ viewed as an $E_1$-algebra)
$$ CH_{S^\infty}(A)\otimes_{CH_{S^\infty}(A)} M$$  with a structure of $E_\infty$-module.
\end{ex}

\begin{rem}[Iterative liftings]\label{R:LiftMod} One can lift any left $A$-module to an $A$-$E_\infty$-module in the same way as in Example~\ref{E:LiftModstructure}.

By restriction of structure, any left $A$-module map between $E_\infty$-$A$-modules can be lifted  to a map of $E_n$-modules (for $n\in \mathbb{N}\cup \{\infty\}$).
 For the sake of explicit computations, we now explain how to realize this concretely using the higher Hochschild functor. Let $M$, $N$ be $E_\infty$-modules over $A$. By restriction of structure we get in particular left $A$-modules structure on $M$ and $N$.
 Let $f:M\to N$ be a map of left $A$-modules.

 The natural structure of $A \otimes A^{op}$-$E_\infty$-module structure on $A$ also yields, by restriction of structure, Proposition~\ref{P:EnMod} and Lemma~\ref{T:CH=TCH}, a natural structure of left $A\otimes \Big(CH_{S^{n-1}}(A)\Big)^{op}$-module on $A$ where the left factor $A$ is viewed as an $E_1$-algebra  only.

 It follows that, viewing $N$ as left $A$-module only by restriction, $Hom_{A}(A,N)$ is endowed with a natural left $CH_{S^{n-1}}(A)$-module structure and further that we have a natural isomorphism of left $CH_{S^{n-1}}(A)$-modules $Hom_{A}(A,N)\stackrel{\simeq}\to N$ (given by $f\mapsto f(1)$).    We get similarly a  left $CH_{S^{n-1}}(A)\otimes A^{op}$-module structure on $A$ and a natural equivalence of   left $CH_{S^{n-1}}(A)$-modules
 $A\otimes_{A}M \stackrel{\simeq}\to M$
(where the tensor product is over $A$ viewed as an $E_1$-algebra only).

We now explain how to lift $f$ to an $E_n$-module map (here $n\in \{1,\dots, \infty\}$).
The canonical map $D^n\to pt$ being a homotopy equivalence, we get a natural
quasi-isomorphism $ CH_{D^n}(A)\stackrel{\sim}\to A$ with quasi-inverse induced by the map
sending a point to the center of $D^n$.
 The canonical map $S^{n-1}\hookrightarrow D^n$ given by the boundary of $D^n$ gives a map
of $E_\infty$-algebra $CH_{S^{n-1}}(A)\to CH_{D^n}(A)$ which,
together with the previous morphism, endow $CH_{D^n}(A)$ with a structure
of left $CH_{S^{n-1}}(A)\otimes A ^{op}$-module.
 We thus have a natural quasi-isomorphism (of chain complexes)
\begin{multline*}
Hom_{CH_{S^{n-1}}(A)}(M,N)\cong  Hom_{CH_{S^{n-1}}(A)}\big(A\otimes_{A} M,Hom_A(A,N)\big)\\
\stackrel{\sim}\longrightarrow Hom_{CH_{S^{n-1}}(A)}\big(CH_{D^n}(A)\otimes_{A} M,Hom_A(A,N)\big)\\
\stackrel{\simeq}\longrightarrow  Hom_{A}\big(A\otimes_{CH_{S^{n-1}}(A)}CH_{D^n}(A)\otimes_{A} M,N)\big)
\end{multline*} where the last map is the canonical isomorphism $$\psi\mapsto \Big( x\otimes_{CH_{S^{n-1}}(A)} y\otimes_A m\mapsto \pm \psi(y\otimes_A m)(x)\Big)$$ where the sign $\pm$ is given by the Koszul-Quillen signs rule.

Note that there is an equivalence of $E_\infty$-algebras $A\otimes_{CH_{S^{n-1}}(A)}CH_{D^n}(A) \stackrel{\simeq}\leftarrow CH_{S^n} A$ which induces, by restriction, a quasi-isomorphism  of left $A\otimes A^{op}$-modules (induced by the choice of two antipodal points on $S^n$).    We thus get a quasi-isomorphism
\begin{multline*}
Hom_{A}\big(A\otimes_{CH_{S^{n-1}}(A)}CH_{D^n}(A)\otimes_{A} M,N)\big) \stackrel{\simeq}\longrightarrow Hom_{A}\big(CH_{S^{n}}(A)\otimes_{A} M,N)\big)
\end{multline*} hence an explicit quasi-isomorphism
\begin{equation}
\label{eq:liftEnMod}  Hom_{CH_{S^{n-1}}(A)}(M,N)\stackrel{\simeq}\longrightarrow Hom_{A}\big(CH_{S^{n}}(A)\otimes_{A} M,N)\big).
\end{equation}
The canonical map $S^n\to pt$ also yield a map of $E_\infty$-algebras $CH_ {S^n}(A)\to A$, which, by restriction of structures is also a map  of left $A\otimes A^{op}$-modules. Hence; we have a natural morphism
\begin{equation}\label{eq:liftEnMod2}
Hom_A(M,N)\cong Hom_A(A\otimes_A M,N)\longrightarrow Hom_{A}\big(CH_{S^{n}}(A)\otimes_{A} M,N)\big).
\end{equation}
Thus, for any $n$, we can lift the left module map $f \in Hom_A(M,N)$ to a map of left $CH_{S^{n-1}}(A)$-module hence a map of $A$-$E_n$-module (by Proposition~\ref{P:EnMod} or Lemma~\ref{L:EinftyMod}). Note that by Lemma~\ref{L:SinftyMod} the map $CH_{S^\infty}(A)\to A$ is a quasi-isomorphism, hence the map \eqref{eq:liftEnMod2} is a quasi-isomorphism for $n=\infty$ and the lift of $f$ is unique in that case. However, lift of $f$ to $E_n$-module maps are not unique in general for finite $n$.
\end{rem}
\begin{rem}[CDGA case]

When $A$ is a CDGA (over a field), the Hochschild chain complex $CH_{D^n}(A)$ is a semi-free module over $CH_{S^{n-1}}(A)$ (provided we choose a simplicial model $D^n_\bullet$ for $D^n$ and take $\partial D^n_\bullet$ as a model for $S^{n-1}$), and therefore all equivalences involved in the maps \eqref{eq:liftEnMod} and \eqref{eq:liftEnMod2} can be (quasi-)inverted by standard homological algebra techniques. Note that when $A= C^{\ast}(X)$ is the algebra of cochains for a topological space $X$,  the map of $E_\infty$-algebras $CH_{S^n}(A) \to A$ can be factorized as a map

$$CH_{S^n}( C^{\ast}(X)) \to C^{\ast}(\mathop{Map}(S^n, X)) \to C^\ast(X) $$
where the last map is induced by the map $X\to \mathop{Map}(S^n,X)$ that sends every point in $p\in X$ to a constant map $C_p:S^n \to X$ defined as $C_p(a)=p$.  Hence, in the special case $n=1$, we recover the construction of~\cite{FTV}, which was done for $M= C^{\ast}(X)$ and $N=C_\ast(X)$ only.
\end{rem}

\subsection{Poincar\'e duality as a map of $E_\infty$-modules}\label{SS:E-inf-PD}

 We apply the results of the previous sections to achieve an $E_\infty$-lift of the Poincar\'e duality isomorphism for a closed manifold.

\medskip

Let $C$ be an $E_\infty$-coalgebra and let $C^{\vee}=Hom_{k}(C,k)$ be its linear dual
endowed with its canonical $E_\infty$-algebra structure; in particular,  $C^{\vee}$ is naturally a $E_\infty\text{-}C^{\vee}$-module.
Similarly, the dual space $(C^{\vee})^\vee$ is $E_\infty\text{-}C^\vee$-module.
Note that $C \subset (C^{\vee})^\vee$ has an induced $E_\infty\text{-}C^\vee$-module structure. If $C$ is an $E_1$-coalgebra, then $C^\vee$ is an $E_1$-algebra has well.

\smallskip

We recall the following standard definition of the cap-product
\begin{definition} \label{D-cap-product} Let $C$ be an $E_1$-coalgebra. The \emph{cap-product} is
the composition
\[\cap : C^{\vee}\otimes C\stackrel{id\otimes \Delta}\longrightarrow C^{\vee}\otimes C\otimes C
\stackrel{\langle-,-\rangle\otimes id}\longrightarrow C
  \]
where $\Delta: C\to C\otimes C$ is the coproduct (given by the $E_1$-structure of $C$) and
$\langle-,-\rangle: C^\vee \otimes C\to k$ is the duality pairing. The cap-product of $x\in C^\vee$, $y\in C$
will be denoted $x\cap y$ as usual.
\end{definition}
The cap-product map $\cap : C^{\vee}\otimes C\to C$ allows us to associate to any cycle $c$ in $C$,  a
map of \emph{left $C^{\vee}$-modules} $\cap c: C^{\vee} \to C$, $x\mapsto x\cap c$, called \emph{the cap-product by $c$}. Note that this construction only
uses the underlying $E_1$-coalgebra structure of $C$ (even if $C$ is an $E_\infty$-algebra).

\begin{cor}\label{C:PD}Let $C$ be an $E_\infty$-coalgebra. The cap product by $c$, $C^{\vee} \stackrel{\cap c  } \longrightarrow C$, lifts uniquely to a map of $E_\infty$-modules $\rho_c:C^\vee \to C$
which is an equivalence if $\cap c$ is a quasi-isomorphism.
\end{cor}
\begin{proof} The cap-product by $c$, denoted $\cap c: C^\vee\to C$,
is a map of left modules over $C^\vee$ (seen as an $E_1$-algebra) because $\Delta: C\to C\otimes C$
is an $E_1$-coalgebra structure.
It follows from Theorem~\ref{T:lifttoEinfty} that the unique lift exists.
 If $\cap c$ is a quasi-isomorphism,
 then it is an invertible element in $Hom_{C^{\vee}}(C^{\vee}, C)$
and thus its lift is invertible in $Hom_{CH_{S^\infty}(C^{\vee})}(C^{\vee},C)$
(see Remark~\ref{R:LiftMod} for an explicit description of the equivalence).
\end{proof}

\smallskip

We now specialize to the case where $C$ is the singular cochain of a space.
Let us recall the following definition.
\begin{definition}\label{D:PDspace}
By a \emph{Poincar\'e duality space}, we mean a topological space $X$ together
with a choice of cycle $[X]\in C_d(X)$ (for some integer $d$) such that that cap-product
$C^{\ast}(X) \stackrel{\cap [X]} \longrightarrow C_{d-\ast}(X)$
by $[X]$ is a quasi-isomorphism.
The integer $d$ is called the \emph{dimension} of $X$ and denoted $d=\dim(X)$.
\end{definition}
\begin{ex}
 An oriented\footnote{with respect to the homology with coefficients in the ground ring $k$}
closed manifold $M$ of dimension $\dim(M)$ (in the usual manifold sense of dimension)
 is a Poincar\'e duality space of dimension $\dim(M)$.
\end{ex}
\begin{rem}\label{R:PDimpliesFinGen}
 By definition, the cap product by a class $[X]$ is given by
$f\mapsto \sum f\big([X]^{(1)}\big)\, [X]^{(2)}$ (where we denote
$\Delta([X]):= \sum [X]^{(1)}\otimes [X]^{(2)}$ the coproduct). It follows that the image
 $\chi_X\big(H^{\ast}(X) \big)$ is a finitely generated sub $k$-module of $H_{\ast}(X)$. Thus,
\emph{if $X$ is a Poincar\'e duality space, its (co)homology groups are finitely generated}
(as $k$-modules).
\end{rem}

Let $X$ be a Poincar\'e duality space (for instance, an oriented closed manifold) with fundamental class $[X]$.
 Recall that $C_{\ast}(X)$ is the singular cochains of $X$ with its natural structure of $E_\infty$-coalgebra (Example~\ref{E:singularchainasEinfty}).
Its linear dual $C^{\ast}(X)$ is endowed with the dual $E_\infty$-algebra structure.
Then, by Corollary~\ref{C:PD} we have
\begin{cor}\label{C:PDmap} Let $(X, [X])$ be a Poincar\'e duality space.
The cap-product by $[X]$ induces a quasi-isomorphism of $E_\infty$-$C^{\ast}(X)$-modules
\begin{equation}\label{eq:PDmap}
\chi_X: C^{\ast}(X) \stackrel{\simeq}\longrightarrow  C_{\ast}(X)[\dim(X)]
\end{equation}
realizing the (unique) $E_\infty$-lift of the Poincar\'e duality isomorphism.
\end{cor}In other words, a Poincar\'e duality space $X$
(in the sense of Definition~\ref{D:PDspace}) gives rise to a canonical equivalence of $E_\infty$-modules between
its singular chains and cochains.
\begin{definition}\label{D:PDmap}
 Let $(X,[X])$, $(Y,[Y])$ be Poincar\'e duality space (of  same dimension $d=\dim(X)=\dim(Y)$).
A map of Poincar\'e duality space $f:(X,[X])\to (Y,[Y])$ is a map of topological spaces $f:X\to Y$
such that the following diagram is commutative
\[\xymatrix{ C^\ast(X) \ar[r]^{\cap [X]} & C_\ast(X)[d] \ar[d]^{f_*}  \\C^\ast(Y) \ar[r]^{\cap [Y]} \ar[u]^{f^*}
 & C_\ast(Y)[d] }  \]
in $C^\ast(Y)\text{-Mod}^{E_\infty}$.
\end{definition}
\begin{ex}
 Let $f:M\to N$ be a continuous map between oriented smooth manifolds such that $f_*([M])=[N]$.
Then $f$ induces a map of Poincar\'e duality spaces.
\end{ex}

\section{Centralizers} \label{S:centralizers}

Given any map $f:A\to B$ of $E_\infty$-algebras, by Theorem~\ref{T:EdHoch}, there is a natural $E_n$-algebra
structure on $CH^{S^n}(A,B)$.  On the other hand,
for a map $f:A\to B$ of $E_n$-algebras, Lurie~\cite{L-HA, L-VI} constructs an $E_n$-algebra $\mathfrak{z}(f)$.
We prove in \S~\ref{S:maincentralizers} that $CH^{S^n}(A,B)$ is equivalent to $\mathfrak{z}(f)$ as an $E_n$-algebra.
This will be a corollary of a more general construction for $E_n$-Hochschild cohomology. Indeed, Hochschild cochains
modeled on spheres $CH{S^n}(A,B)$ is a special case of $E_n$-Hochschild cohomology $HH_{\mathcal{E}_n}(A,B)$ of $A,B$
viewed as $E_n$-algebras, see \S~\ref{SS:EnHoch=Hochoverspheres}.
In Section~\ref{S:Prelimcentralizers}, in the general case of a map $f:A\to B$ between $E_n$-algebras, we will
give an explicit $E_n$-algebra structure on
$HH_{\mathcal{E}_n}(A,B)$ ,
 similar
to the one obtained in Section~\ref{S:Edcochains}. We then prove that $HH_{\mathcal{E}_n}(A,B)$ is equivalent to $\mathfrak{z}(f)$. We will apply these results to the case $A=B$, \emph{i.e.}, to get solutions of the (higher) Deligne conjecture in \S~\ref{S:Deligne}.

\subsection{$E_n$-Hochschild cohomology and  Hochschild cohomology over $S^n$} \label{SS:EnHoch=Hochoverspheres}
There is an (operadic) notion of cohomology for $E_n$-algebras closely related to their deformation complexes,
see~\cite{F, KS}. We start with the following definition.
\begin{definition}\label{D:EnHoch}
Let $M$ be an $E_n$-$A$-module over an $E_{n}$-algebra $A$.
The $E_n$-Hochschild complex of $A$ with values $M$, denoted
by $HH_{\mathcal E_n}(A, M)$, is by definition (see~\cite{F})
$RHom_{A}^{\mathcal E_n}(A, M)$. Here $RHom_{A}^{\mathcal E_n}$ denotes the hom space in the ($\infty$-)category $A\text{-}Mod^{E_n}$ of $E_n$-$A$-modules.
\end{definition}
In particular, if $A$ is an $E_m$-algebra
with $m\in \{n,n+1\,\dots, \infty\}$ (for instance a CDGA), we can define the $E_n$-Hochschild complex of $A$
$HH_{\mathcal E_n}(A,A)$.

In the case where $A$ is an $E_\infty$-algebra, its $E_n$-Hochschild complex can be described by higher Hochschild
cochains over the $n$-dimensional sphere $S^n$:
\begin{prop}\label{P:coHH=coTCH} If $A$ is an $E_\infty$-algebra and $M$ an $E_\infty$-$A$-module,
there is a natural equivalence $$HH_{\mathcal E_n}(A, M) \cong CH^{S^n}(A, M),$$
where $CH^{S^n}$ denotes the derived  higher Hochschild cochain functor.
\end{prop}
\begin{proof} Given left modules $M,N$ over an $E_1$-algebra $R$, we write $ RHom_{R}^{left}\left(M,N\right)$
for the hom space in the ($\infty$,1)-category $R\text{-}LMod$ of left $R$-modules.
By Proposition~\ref{P:EnMod}, there is an equivalence of $\infty$-categories
$A\text{-}Mod^{E_n} \cong \big(\int_{S^{n-1}}A\big)\text{-}LMod$ where $\int_{S^{n-1}}A$
is the factorization homology of $
S^{n-1}$ with value in $A$. Here, $S^{n-1}$ is endowed with the  $n$-framing induced by the
natural embedding  $S^{n-1}\hookrightarrow \mathbb R^{n}$.  Thus we have a sequence of natural
 equivalences

\begin{eqnarray*}
HH_{\mathcal E_n}(A, M)&\cong & RHom_A^{\mathcal E_n}(A, M)\\
&\cong & RHom_{\int_{S^{n-1}}A}^{left}(A, M)\\
&\cong & RHom_{\int_{S^{n-1}}A}^{left}\left({\int_{D^{n}}A}, M\right)\\
&\cong & RHom_{CH_{S^{n-1}}(A)}^{left}\left({CH_{D^n}(A)}, M\right)\\
&\cong & RHom_A^{left}\left( CH_{D^n}(A)\otimes_{CH_{S^{n-1}}(A)}^{\mathbb{L}} A   , M\right)\\
&\cong & RHom_A^{left}\left( CH_{S^n}(A), M\right)\\
&\cong & CH^{S^n}(A,M).
\end{eqnarray*}
Here we are using the natural equivalence of $E_n$-$A$-modules
$\int_{D^n} A\stackrel{\simeq}\to A$ (Lemma~\ref{L:A=intDA}).

Note that, by Theorem \ref{T:CH=TCH}, when $A$ is further an
$E_\infty$-algebra, we  get  a natural equivalence of $E_1$-algebras
$\int_{S^{n-1}}(A)\cong CH_{S^{n-1}}(A)$ and by
Theorem~\ref{T:derivedfunctor} a natural equivalence of $E_\infty$-algebras
 $CH_{S^n}(A)\cong CH_{D^n}(A)\otimes_{CH_{S^{n-1}}(A)}^{\mathbb{L}} A$ .
\end{proof}

\begin{rem}\label{R:equivHHRHomTCH}
 Let $A, B$ be $E_n$-algebras and $f:A\to B$ an $E_n$-algebra map so that $B$ inherits an $A$-$E_n$-module structure.
 By Definition~\ref{D:EnHoch},  Proposition~\ref{P:EnMod} and Lemma~\ref{L:A=intDA},
we have natural
equivalences $$HH_{\mathcal{E}_n}(A,B)\;\cong\;  RHom^{\mathcal{E}_n}_{A} \Big({A}, {B}  \Big)
\;\cong\; RHom_{\int_{S^{n-1}}A}^{left}\Big( \int_{D^n} A, \int_{D^n }B\Big).$$
\end{rem}

\subsection{The $E_n$-algebra structure on $\mathcal{E}_n$-Hochschild cohomology
$HH_{\mathcal{E}_n}$(A,B)}
\label{S:Prelimcentralizers}
In this section, we construct an explicit $E_n$-algebra structure on the
$\mathcal{E}_n$-Hochschild cohomology $HH_{\mathcal{E}_n}(A,B)$ of an $E_n$-algebra $A$ with
value in an $E_n$-algebra $B$ endowed with an $A\text{-}E_n$-module structure given by a map $A\to B$
of $E_n$-algebras.

\smallskip

We fix a map $f:A\to B$ of $E_n$-algebras and
we endow $B$ with the induced $A$-$E_n$-module structure so that we have  $E_n$-Hochschild cohomology\footnote{which depends on the map $f:A\to B$ even though it is not explicitly written in the notation} $HH_{\mathcal{E}_n}(A,B)$.

\smallskip

 Recall from Section~\ref{S:EnasFact} (and~\cite{L-VI, L-HA}) that
 giving an $E_n$-algebra structure
to $HH_{\mathcal{E}_n}(A,B)\cong  RHom^{\mathcal{E}_n}_{A} \Big({A}, {B}  \Big) $ is equivalent to
giving a structure of
locally constant factorization algebra on $D^n$ whose global
section\footnote{,\emph{i.e.}, its factorization homology over the whole disk $D^n$} are
$ RHom^{\mathcal{E}_n}_{A} \Big({A}, {B}  \Big)$. That is, we need to associate to any disk $U\subset D^n$
 a chain complex
$ HH_{\mathcal{E}_n} \Big({A}, {B}  \Big)\big(U\big)$ naturally quasi-isomorphic to
$ RHom^{\mathcal{E}_n}_{A} \Big({A}, {B}  \Big)$ equipped with natural chain maps from
\begin{equation}\label{eq:HHenABstructure} \rho_{U_1,\dots, U_\ell, V} HH_{\mathcal{E}_n} \Big({A}, {B}  \Big)\big(U_1\big)\otimes\cdots
\otimes HH_{\mathcal{E}_n} \Big({A}, {B}  \Big)\big(U_\ell\big)\to
HH_{\mathcal{E}_n} \Big({A}, {B}  \Big)\big(V\big)\end{equation} for any pairwise disjoint (embedded) sub-disks $U_i$ in a bigger disk $V$.

\smallskip

Let  $\mathcal{A}$, $\mathcal{B}$ be the  underlying
 locally constant factorization algebras on $D^n$ associated to $A$ and $B$ given by Theorem~\ref{P:En=Fact}
and still denote
$f:\mathcal{A}\to \mathcal{B}$ the induced map of factorization algebras. In other words:

\emph{we assume from now on that $A$, $B$ and $f$ are given by locally constant factorization algebras as in Section~\ref{S:EnasFact}.}

Similarly, given any map of $A$-$E_n$-modules
$g: A\to B$, by Proposition~\ref{P:EnMod=Fact}, we can \emph{assume that $g$ is given by  a map $g:\mathcal{A}\to \mathcal{B}$
of (stratified) factorization algebras}, as well as, by Proposition~\ref{P:EnMod}
a map of (left)
$\int_{S^{n-1}} A$-modules $g:\int_{D^n}A\to \int_{D^n} B$.

\begin{rem}[\emph{Sketch of the construction}]\label{R:sketchofEnHochstructure} We first sketch the idea of the construction. For any sub-disk $U_i$, we can think of $HH_{\mathcal{E}_n}(A,B)\cong RHom_{A}^{\mathcal{E}_n}(A,B)$ as the space of stratified factorization algebras maps on the disk $U_i$ (with a distinguished point $*_i$, see Proposition~\ref{P:EnMod=Fact}). Hence,   given $g_1,\dots, g_\ell\in HH_{\mathcal{E}_n}(A,B)$, we define the structure map~\eqref{eq:HHenABstructure} $\rho_{U_1,\dots, U_\ell, V}(g_1,\dots, g_{\ell})$ to be the factorization algebra map which, to any sub-disk $D$ inside a given $U_i$ associates $g_i(D)$ and, to any disk $D$ inside (a small neighborhood of) the complement of the $U_i$'s associates $f(D)$. The family of those disks is a basis of all disks inside $V$, so that such a rule does define a factorization algebra map, which underlies a map of $A$-$E$-modules (see Remark~\ref{R:Enmodulestructure}).  This is roughly described in Figure~\ref{fig:algebramap}.
\begin{figure}
\includegraphics[scale=0.4]{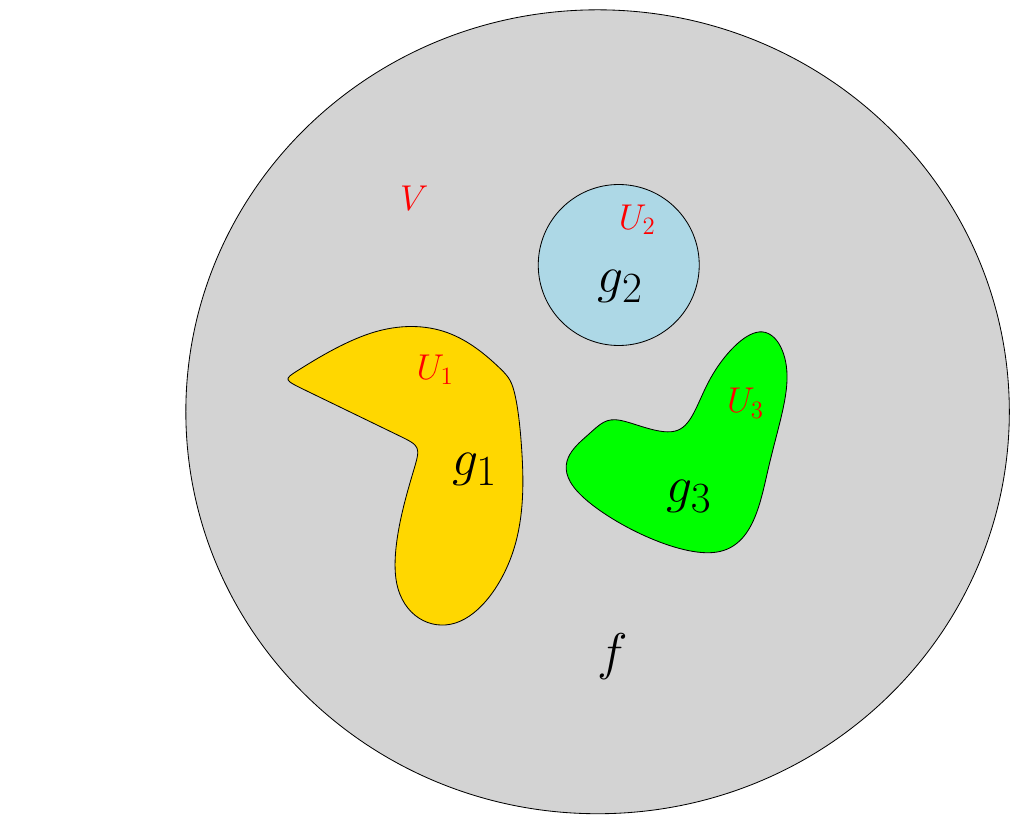}
\caption{The factorization algebra map $\mathcal{A}_{|V} \to \mathcal{B}_{|V}$ obtained by applying the relevant maps of  modules $g_1, g_2, g_3$ (viewed as maps of factorizations algebras) and the $E_n$-algebra map $f:A\to B$ on the respective regions}\label{fig:algebramap}
\end{figure}
\end{rem}

We now construct the locally constant factorization algebra on $\R^n$ we are looking for.

\subsubsection{Step 1: the underlying chain complexes.}\label{SS:step1} For any open
subset $U$, the restrictions $\mathcal{A}_{|U}$, $\mathcal{B}_{|U}$
are locally constant factorization
algebras\footnote{quasi-isomorphic to $A$ and $B$ by definition} on
$U$, and $f_{|U}: \mathcal{A}_{|U}\to \mathcal{B}_{|U}$ a
factorization algebra morphism. 
Thus, if $U$ is a disk\footnote{that
is an open set homeomorphic to a disk}, $\mathcal{A}(U)$ ($\cong
\int_{U} A$) is an $E_n$-algebra and $f_{|U}=\int_{U}f$ makes
$\mathcal{B}(U) \cong \int_{U}B$ an
$\mathcal{A}(U)$-$E_n$-module\footnote{in this section we will write
$\int_U A$ for $\mathcal{A}(U)$ viewed as an $E_n$-module over
itself and reserve the notation $\mathcal{A}(U)$ when we think of it
as an $E_n$-algebra}. In particular, given any point $*_U$ in $U$, the restrictions 
$\mathcal{A}_{|U}$, $\mathcal{B}_{|U}$ are canonically stratified factorization algebras on the pointed disk $U$ and further define canonical objects in $\text{Fac}^{lc}_{U_{*_U}|\mathcal{A}_{U}}\hookrightarrow \text{Fac}^{lc,res}_{U_{*_U}}$, see Definition~\ref{D:EnModCat}.

Thus to any open disk $U$, we  can associate the following object of $\hkmod$:
\begin{equation}\label{eq:Rhom(U)}
 RHom^{\mathcal{E}_n}_{A} \Big({A}, {B}  \Big)\big(U\big):=
RHom_{\mathcal{A}(U)}^{\mathcal{E}_n}\Big( \int_{U} A, \int_{U}B\Big)
\end{equation}
Note that $RHom^{\mathcal{E}_n} \Big({A}, {B}  \Big)\big(U\big)$ is pointed since
 our starting map of $E_n$-algebras $f:A\to B$ induces a canonical element
$\int_{U} f \in RHom^{\mathcal{E}_n}_{A} \Big({A}, {B}  \Big)\big(U\big)$.

\smallskip

\subsubsection{Step 2: the structure maps.}\label{SS:step2} By
Proposition~\ref{P:extensionfrombasis}, we only need to construct
the factorization algebra  $RHom^{\mathcal{E}_n}_{A} \Big({A}, {B}
\Big)\big(U\big)$ on the basis of opens subsets $\mathcal{CVX}$
consisting of all bounded open convex subsets of $D^n$.  The basis
$\mathcal{CVX}$ is stable by finite intersection and a factorizing
cover. Note that if $U\in \mathcal{CVX}$ with center $*_U$, then
(see~Remark~\ref{R:equivHHRHomTCH})  $$RHom^{\mathcal{E}_n}_{A}
\Big({A}, {B} \Big)\big(U\big)=RHom^{left}_{\int_{U\setminus\{*_U\}}
A} \Big(\int_U A, \int_U B  \Big)$$ which is \emph{the mapping space
between the associated stratified (in $*_U$) factorization algebras
$\mathcal{A}_{|U}$, $\mathcal{B}_{|U}$} corresponding to the module
structures of $\int_U A$, $\int_U B$ as given by
Proposition~\ref{P:EnMod=Fact}.

For pairwise disjoints disks $U_1,\dots, U_r \in \mathcal{CVX}$ included in a larger disk $D\in \mathcal{CVX}$,
we define the structure map
\begin{multline}\label{eq:structMapsCVX}
 RHom^{\mathcal{E}_n}_{A} \Big({A}, {B}  \Big)\big(U_1\big)\otimes \cdots
\otimes RHom^{\mathcal{E}_n}_{A} \Big({A}, {B}  \Big)\big(U_r\big) \stackrel{\rho_{U_1,\dots,U_r,D}}
\longrightarrow  RHom^{\mathcal{E}_n}_{A} \Big({A}, {B}  \Big)\big(D\big)
\end{multline}
as follows. Denote $*_1,\dots, *_r$, the respective centers of the $U_i$'s. First we use $U_1,\dots, U_r$ to define the cover $\mathcal{U}_{U_1,\dots, U_r,V}$ consisting of all opens $V$ in $D$ which
\begin{itemize}
 \item either do not contain any $*_i$: $V\subset D\setminus\{*_1,\dots, *_r\}$,
 \item or else is included in one of the $U_i$ and is a neighborhood of  $*_i$.
\end{itemize}
 Maps of factorization algebras over $D$ are uniquely determined by their value on $\mathcal{U}_{U_1,\dots, U_r,V}$ since it is a factorizing cover of $D$.
Let be given maps $g_i:\int_{|D_i} {A} \to \int_{|D_i}{B}$ of (left)
$\mathcal{A}(U_i)$-modules ($i=1\dots r$) and also denotes
$g_i:\mathcal{A}_{|D_i} \to \mathcal{B}_{|D_i}$ the induced maps of stratified (at the point $*_i$)
factorization algebras. We define
$ \rho_{U_1,\dots,U_r,D}\,(g_1,\dots, g_r)$ on an open $V\in \mathcal{U}_{U_1,\dots, U_r,V}$ by:
\begin{equation}\label{eq:FormularhoonV}
  \rho_{U_1,\dots,U_r,D}\,(g_1,\dots, g_r)_{|V} = \left\{\begin{array}{ll}
   f_{|V} & \mbox{ if $V\subset D\setminus\{*_1,\dots, *_r\}$, }\\
    {g_i}_{|V}      & \mbox{ if $*_i\in V \subset U_i$}.                                               \end{array}
 \right.
\end{equation}

\begin{lem}\label{L:CVXFactmaps}
 The rule $V\mapsto \rho_{U_1,\dots,U_r,D}\,(g_1,\dots, g_r)_{|V}$ (given by Formula~\eqref{eq:FormularhoonV}), defines a map of factorization algebras\footnote{where $\text{Fac}^{lc, res}_{D_*}$ is given by Definition~\ref{D:EnModCat}}: $$\rho_{U_1,\dots,U_r,D}\,(g_1,\dots, g_r) \in \text{Map}_{\text{Fac}^{lc, res}_{D_*}}(\mathcal{A}_{|D}, \mathcal{B}_{|D}).$$ 
\end{lem}
\begin{proof}  First we check that $\rho_{U_1,\dots,U_r,D}\,(g_1,\dots, g_r)$ defines a factorization algebra map.
 Since $\mathcal{U}_{U_1,\dots, U_r,V}$  is a factorizing cover of $D$, we only need to check that it is compatible with the structure maps of $\mathcal{A}$ and $\mathcal{B}$. If all opens involved lies either in $D\setminus \{*_1,\dots\}$ or contains a same point $*_i$, then the result is immediate since $f$ and $g_i$ are maps of factorization algebras.
 Now, assume  $*_i\in V\subset U_i$ and that there are pairwise disjoint opens $V_1,\dots, V_\ell \subset V$ (at most one of them can contain $*_i$). Since $g_i$ comes from a map of  $E_n$-$A$-modules (with module structure induced by $f$), ${g_i}_{|V_k} =f_{|V_k}$ whenever $V_k$ does not contain $*_i$.  It follows that the following diagram
 $$\xymatrix{ \bigotimes_{i=1\dots \ell} \mathcal{A}(V_k)\ar[d]_{\bigotimes_{i=1\dots \ell} \rho_{U_1,\dots,U_r,D}\,(g_1,\dots, g_r)_{|V_k}} \ar[rrr]^{\rho_{V_1,\dots, V_\ell,V}} &&& \mathcal{A}(V) \ar[d]^{g_i} \\ \bigotimes_{i=1\dots \ell} \mathcal{B}(V_k) \ar[rrr]^{\rho_{V_1,\dots, V_\ell,V}}
 &&& \mathcal{B}(V)} $$
 is commutative, hence $\rho_{U_1,\dots,U_r,D}\,(g_1,\dots, g_r)$ is a map of (pre-)factorization algebras.
\end{proof}

\begin{lem}\label{L:CVXFactmapsisStratified}                                                                                                                                                                                                                     The induced (by Lemma~\ref{L:CVXFactmaps}) map
$$\int_{D}\rho_{U_1,\dots,U_r,D}\,(g_1,\dots, g_r) \in RHom_{\mathcal{A}(D)}^{\mathcal{E}_n}\Big( \int_{D} A, \int_{D}B\Big)$$
is a map of $\mathcal{A}(D)$-$E_n$-modules.
\end{lem}
In particular, we define the map~\eqref{eq:structMapsCVX} to be the global section $\int_{D}\rho_{U_1,\dots,U_r,D}\,(g_1,\dots, g_r)$ of the maps defined by formula~\eqref{eq:FormularhoonV}.
\begin{proof}[Proof of Lemma~\ref{L:CVXFactmapsisStratified}]      Passing to the global section in  Lemma~\ref{L:CVXFactmaps}, we have  the map $\rho_{U_1,\dots,U_r,D}\,(g_1,\dots, g_r)\big( D\big): A\to B$ and we need to prove that it is a   map of $\mathcal{A}(D)$-$E_n$-module. By   Proposition~\ref{P:EnMod}, it is equivalent to  prove that  the induced map  $\int_{D}\rho_{U_1,\dots,U_r,D}\,(g_1,\dots, g_r): \int_{D} A \to \int_{D} B$  is a morphism of left
$\int_{S^{n-1}} A$-modules (where the module structure is induced by $f$).
Let $\tilde{D}$ be a closed sub-disk of $D$ containing $U_1\coprod \cdots\coprod U_r$.  The open sub-set $D\setminus \tilde{D}\cong S^{n-1}\times (0,\epsilon\rq{})$
lies in the complement (in $D$) of the $U_i$\rq{}s.
Since  $\int_{S^{n-1}}A$ is the  section
 $\int_{S^{n-1}}A=\mathcal{A}\big(S^{n-1}\times (0,\epsilon\rq{})\big)$ (Theorem~\ref{T:Theorem6GTZ2}),
we are left to prove that  the map
$\rho_{U_1,\dots,U_r,D}\,(g_1,\dots, g_r)$ restricted to
 $D\setminus \tilde{D}\cong S^{n-1}\times (0,\epsilon\rq{})$ is equivalent to $f$. This is an immediate consequence of the fact that $D\setminus \tilde{D}\subset D\setminus\{*_1,\dots,*_i\}$ and
$\big(\rho_{U_1,\dots,U_r,D}\,(g_1,\dots, g_r)\big)_{|\tilde{D}\subset D\setminus\{*_1,\dots,*_i\}}=f$ as given by construction~\eqref{eq:FormularhoonV}.
\end{proof}

\begin{rem}\label{R:StructMapemptytoD}
Let us consider the case of the inclusion of the empty set $\emptyset$ inside a disk $D$.
 Unwinding the definition of the structure map
$$\rho_{\emptyset, D}: k\cong RHom^{\mathcal{E}_n}_{A} \Big({A}, {B}  \Big)\big(\emptyset\big)
\longrightarrow RHom^{\mathcal{E}_n}_{A}\Big({A}, {B}  \Big)\big(D\big)$$ we see immediately
that $\rho_{\emptyset, D}(1)= \int_D f$, in other words $1$ is mapped to the base point of
$RHom^{\mathcal{E}_n} \Big({A}, {B}  \Big)\big(D\big)$.

A straightforward computation also shows that
$$\rho_{U_1,\dots, U_r,D}\Big(\int_{U_1}f,\dots,\int_{U_r}f\Big) = \int_{D} f .$$
\end{rem}

\medskip

\subsubsection{Step 3: the global structure.}\label{SS:step3} The cochain complexes
$U\mapsto RHom^{\mathcal{E}_n}_{A} \Big({A}, {B}  \Big)\big(U\big)
\cong Hom_{\mathcal{A}(U)}^{\mathcal{E}_n}\Big( \int_{U} A,
\int_{U}B\Big)$ are equipped with the structure
maps~\eqref{eq:structMapsCVX} (given by
formula~\eqref{eq:FormularhoonV} and
Lemma~\ref{L:CVXFactmapsisStratified}). These maps  assemble to form
a locally constant factorization algebra over $\mathbb{R}^n$,
yielding an $E_n$-algebra structure to $RHom^{\mathcal{E}_n}_A
\big({A}, {B}  \big)$. This is the content of the following result:
\begin{theorem}\label{T:EnAlgHoch} Let $f:A\to B$ be a map of $E_n$-algebras\footnote{which we may assume to be given by a map $f:\mathcal{A}\to \mathcal{B}$ of factorization algebras, see \S~\ref{S:EnasFact}}.
\begin{enumerate}
\item The structure maps~\eqref{eq:StructFactcenter} $\rho_{U_1,\dots, U_r,V}$ (given by step~2 above) make
 $U\mapsto  RHom^{\mathcal{E}_n}_{A} \big({A}, {B}  \big)\big(U\big)$ a locally constant
factorization algebra on $\mathbb{R}^n$ whose global section are naturally equivalent to
$RHom^{\mathcal{E}_n}_{A} \Big({A}, {B}  \Big)$.
\item In particular
 $HH_{\mathcal{E}_n}(A,B) \cong RHom^{\mathcal{E}_n}_A \big({A}, {B}  \big)$
inherits a natural $E_n$-algebra structure (with unit given by $f$).
\item Let $g:B\to C$ be another map of $E_n$-algebrass\footnote{which we may assume to be given by a map $g:\mathcal{B}\to \mathcal{C}$ of factorization algebras, see \S~\ref{S:EnasFact}}. The (derived) functor of composition of
 $E_n$-modules homomorphisms
$$RHom^{\mathcal{E}_n}_A \big({A}, {B}  \big)
\otimes RHom^{\mathcal{E}_n}_B \big({B}, {C}  \big) \stackrel{\circ}\longrightarrow
RHom^{\mathcal{E}_n}_A \big({A}, {C}  \big)$$ is a homomorphism of $E_n$-algebras\footnote{the
left hand side being endowed with the  $E_n$-algebra structure induced on the tensor products of
$E_n$-algebras and the $A$-module structure on $C$ being given by the $E_n$-algebra map $g\circ f: A\to C$}.
\item Let $h:C\to D$ be an $E_n$-algebra map.
The canonical map $$RHom^{\mathcal{E}_n}_A \big({A}, {B}  \big)
\otimes RHom^{\mathcal{E}_n}_C \big({C}, {D}  \big)
\longrightarrow RHom^{\mathcal{E}_n}_{A\otimes C} \big({A\otimes C}, {B\otimes D}  \big) $$
is a homomorphism of $E_n$-algebras.
\end{enumerate}
\end{theorem}
The naturality (in $B$) of the $E_n$-algebra structure of $RHom^{\mathcal{E}_n}_A \big({A}, {B}  \big)$ means
that, given a morphism $\phi: B\to B'$ of $E_n$-algebras, the induced map
$$\phi_{*}: RHom^{\mathcal{E}_n}_A \big({A}, {B}  \big) \to
RHom^{\mathcal{E}_n}_A \big({A}, {B'}  \big) \quad \mbox{(given by $g\mapsto \phi\circ g$)}  $$
is an $E_n$-algebra morphism. Here, the $A$-module structure of $B$ is of course given by
the $E_n$-algebra morphism $\phi\circ f: A\to B'$.
 Similarly, the naturality in $A$ means that, given a morphism
$\psi: A'\to A$ of $E_n$-algebras, the induced map
$$\psi_{*}: RHom^{\mathcal{E}_n}_A \big({A}, {B}  \big) \to
RHom^{\mathcal{E}_n}_{A'} \big({A'}, {B}  \big) \quad \mbox{(given by $g\mapsto g\circ \psi$)}  $$
is an $E_n$-algebra morphism.
 \begin{proof}[Proof of Theorem~\ref{T:EnAlgHoch}]
Since the global section $\mathcal{F}(\mathbb{R}^n)$ of a locally constant factorization algebra
 $\mathcal{F}$ on $\mathbb{R}^n$ is an $E_n$-algebra (Theorem~\ref{P:En=Fact}), the second statement is an immediate consequence of the first one.

We now prove the first one. Proposition~\ref{P:extensionfrombasis} implies that we need only to check the axioms of a locally constant factorization algebra on the basis of opens $\mathcal{CVX}$.

First we prove the naturality of the structure maps~\eqref{eq:structMapsCVX}  with respect to the inclusion of
open convex disks (in other words we check the prefactorization algebra axiom). That is we need to check that for a family of pairwise disjoints disks $U_1,\dots U_r \in \mathcal{CVX}$ inside a disk
$V\in \mathcal{CVX}$ and families $W^j_{1}\dots W^{j}_{i_j}$ of pairwise disjoints convex disks inside $U_j$ (for $j=1\dots r$) we have
\begin{equation} \label{eq:idnatFact}\rho_{U_1,\dots, U_r, V}
\Big( \rho_{W_{1}^1,\dots, W_{i_1}^1, U_1},\dots, \rho_{W^r_{1},\dots, W^r_{i_r}, U_r} \Big) =
\rho_{W_1^1,\dots, W_{i_1}^1,\dots, W_1^r,\dots, W_{i_r}^r,V  }.\end{equation}  We write respectively $*_i$ and $*^i_{i_j}$ for the centers of the $U_i$\rq{}s and $W^i_{i_j}$\rq{}s.
Recall that the structure maps in the above identity~\eqref{eq:idnatFact} are obtained by applying the construction~\eqref{eq:FormularhoonV} on the relevant opens subsets.
Thus the right hand side of~\eqref{eq:idnatFact} is  the global section of the map of factorization algebras which is equal to $g_{i_k}^j$ on  a open $W_{i_k}^j$ which contains $*_{i_k}^j$ and $f$ on opens lying in the complement of $*_{i_\ell}^s$.

To evaluate the left hand side, we first define a specific cover of $V$ as follows. For each $U_i$ ($i=1\dots r$), we choose a convex closed sub-disk of $U_i$ which contains each $W^i_j$ and $*_i$. We write $col_i$ for its complement in $U_i$. Then, we have a cover of $V$ given by the $U_i$\rq{}s and $U_\partial=V\setminus \big(\coprod_{i=1}^{r} (U_i-col_i)\big) $. The left hand side of identity~\eqref{eq:idnatFact} is determined by its restriction on the cover (see~\S~\ref{S:EnasFact} and \cite{CG, G-Houches}).
By construction~\eqref{eq:FormularhoonV}, the map ${\rho_{U_1,\dots, U_r, V}}_{|U_\partial}$ is equal to $f_{|U_\partial}$. Let $i\in\{1,\dots, r\}$ and $D$ be an  opens in $U_i$.  If $*_j^i\in D\subset W^i_{j}$, by construction~\eqref{eq:FormularhoonV}, the composition $$\rho_{U_1,\dots, U_r, V}\Big( \rho_{W_{1}^1,\dots, W_{i_1}^1, U_1}(g_1^1,\dots, g_{i_1}^1),\dots, \rho_{W^r_{1},\dots, W^r_{i_r}, U_r}(g_1^r,\dots, g_{i_r}^r) \Big)_{|D}(V)  $$
on the open $V$ is given by $g_j^i: \mathcal{A}_{|W_{j}^k}\to \mathcal{B}_{|W_j^k}$, that is, is equal to $\int_V g_{j}^i$. While if $D\subset U_i \setminus\{*_1^i,\dots,*^i_{j_i}\}$, then this composition is equal to $\int_V f$.
Note that the composition agrees with the map induced by $f$ on the intersection of the $U_i$'s with $U_\partial$.
It follows that the left hand side of identity~\eqref{eq:idnatFact} is the unique factorization algebra map which coincides with $g_j^i$ on each open subset of  $W^i_j$ containing $*_j^i$  and  coincides with $f$ on opens which do not contains any $*_j^i$.
It is thus equal to the right hand side of~\eqref{eq:idnatFact}. We have proved that the structure maps $\rho_{U_1,\dots,U_r,V}$ satisfies the associativity condition of a prefactorization algebra.  They also satisfy the symmetry condition since they are independent  of any  ordering of the opens $U_1,\dots, U_r$.

\smallskip

It remains to check that $ U\mapsto RHom^{\mathcal{E}_n}_{A} \Big({A}, {B}  \Big)\big(U\big)$ is locally constant.
Since the factorization algebras $\mathcal{A}$ and $\mathcal{B}$ are locally constant,
 the natural maps $\int_U A \, \to \int_V A $
and $\int_U B \, \to \int_V B $ are equivalences for any embedding
 $U\hookrightarrow V$ of a disk $U$  inside a bigger disk $V$. By definition we have
\begin{eqnarray*} RHom^{\mathcal{E}_n}_{A} \Big({A}, {B}  \Big)\big(U\big) & \cong &
Hom_{\mathcal{A}(U)}^{\mathcal{E}_n}\Big( \int_{U} A, \int_{U}B\Big), \\
  RHom^{\mathcal{E}_n}_{A} \Big({A}, {B}  \Big)\big(V\big)& \cong
 & Hom_{\mathcal{A}(V)}^{\mathcal{E}_n}\Big( \int_{V} A, \int_{V}B\Big).\end{eqnarray*}
 By definition, for any $g\in  RHom^{\mathcal{E}_n}_{A} \big({A}, {B}  \big)\big(U\big)$,  the map
 $$\rho_{U,V}:  Hom_{\mathcal{A}(U)}^{\mathcal{E}_n}\Big( \int_{U} A, \int_{U}B\Big)\longrightarrow
 Hom_{\mathcal{A}(V)}^{\mathcal{E}_n}\Big( \int_{V} A, \int_{V}B\Big)$$ applied to $g$ is induced
 by a map of factorization algebras $\rho_{U,V}(g):\mathcal{A}_{|V} \to \mathcal{B}_{|V}$ whose restriction
 to  $U$ is just $g$. It follows that the following diagram is commutative
$$\xymatrix{ \int_{V} A \ar[rr]^{\rho_{U,V}(g)}&  & \int_{V} B \\
 \int_{U} A \ar[rr]^{g} \ar[u]^{\simeq} & & \int_{U} B \ar[u]_{\simeq} } $$
for all $g\in  RHom^{\mathcal{E}_n}_{A} \big({A}, {B}  \big)\big(U\big)$.
Since the vertical maps are equivalences and independent of $g$, it follows that
 $\rho_{U,V}:  Hom_{\mathcal{A}(U)}^{\mathcal{E}_n}\Big( \int_{U} A, \int_{U}B\Big)\to
 Hom_{\mathcal{A}(V)}^{\mathcal{E}_n}\Big( \int_{V} A, \int_{V}B\Big)$
is an equivalence.
Note in particular that, taking $V=\mathbb{R}^n$, we have  canonical equivalences
\begin{multline*}
RHom^{\mathcal{E}_n}_{A} \Big({A}, {B}  \Big)\big(U\big)
\cong Hom_{\mathcal{A}(U)}^{\mathcal{E}_n}\Big( \int_{U} A, \int_{U}B\Big)  \\
\cong Hom_{\mathcal{A}(\mathbb{R}^n)}^{\mathcal{E}_n}\Big( \int_{\mathbb{R}^n} A, \int_{\mathbb{R}^n}B\Big)
\cong RHom^{\mathcal{E}_n}_{A} \Big({A}, {B}  \Big) \end{multline*}
 for any disk $U$ in $\mathbb{R}^n$.

\smallskip

A map of $E_n$-algebras $g:B\to C$ induces a canonical object in
$RHom_{B}^{\mathcal{E}_n}(B,C)$ given by $g$ itself. Thus the naturality of the $E_n$-algebra structure (claimed in
assertion~\textbf{(2)}) is in fact a
consequence of the assertion~\textbf{(3)} in the Theorem (that we will prove below). To finish the proof of
claims~\textbf{(1)}, \textbf{(2)} in the Theorem we need to see that
 the canonical element $f\in RHom_{A}^{\mathcal{E}_n}(A,B)$ is a unit.
Indeed, let $U_1,\dots, U_r, V$ be a finite family of pairwise disjoints convex disks inside a bigger bounded convex open set $D$,
and $g_i \in RHom_{A}^{\mathcal{E}_n}(A,B)(U_i)$ ($i=1\dots r$). Denote $*_1,\dots, *_r, *_V$ the respective centers of $U_i$'s and $V$. Let also
$f \in RHom_{A}^{\mathcal{E}_n}(A,B)(V)$ be  the canonical element induced by $f$.
By definition (see construction~\eqref{eq:FormularhoonV}) $\rho_{U_1,\dots, U_r,V, D}(g_1,\dots, g_r, f)$ is the factorization algebra map whose values on any open subset $W \subset V\setminus\{*_1,\dots, *_r,*_V\}$ is given by (the restriction to $V$ of) $f$, whose value on any open subset $*_i\in W\subset U_i$ is given by $g_i$ and its value on $*_V\in W \subset V$ is again given by $f$. It follows that this map is equal to $f$ on all $V$ and thus we get
$$\rho_{U_1,\dots, U_r,V, D}(g_1,\dots, g_r, f) = \rho_{U_1,\dots, U_r, D}(g_1,\dots, g_r). $$
This proves that $f$ is a unit for the $E_n$-algebra structure of $RHom_{A}^{\mathcal{E}_n}(A,B)$.

\smallskip

We now prove statement~\textbf{(3)}.
Since
the $B$-module structure of $C$ is given by the $E_n$-algebra map $g:B\to C$, the (derived) composition of maps
$RHom^{\mathcal{E}_n}_A \big({A}, {B}  \big)
\otimes RHom^{\mathcal{E}_n}_B \big({B}, {C}  \big) \stackrel{\circ}\longrightarrow
Hom_{\hkmod} \big({A}, {C}  \big)$
naturally  lands  in $RHom_{A}^{\mathcal{E}_n}(A,C)$, where $C$ is endowed with
the $A$-module structure induced by the $E_n$-algebra morphism $g\circ f: A\to C$.
Since the tensor product of $E_n$-algebras is induced by the tensor products of (locally constant)
factorization algebras, it remains to prove that, for any family $U_1,\dots, U_r$  of pairwise
disjoint open disks included inside a bigger disk $D$, the following diagram
\begin{equation}\label{eq:EnAlgHochcomp}
\xymatrix{\mathop{\bigotimes}\limits_{i=1}^r \Big( RHom^{\mathcal{E}_n}_{A} \big({A}, {B}  \big)(U_i)
\otimes RHom^{\mathcal{E}_n}_{B} \big({B}, {C}  \big)(U_i)\Big)  \ar[d]_{\rho_{U_1,\dots,U_r,D}^{\otimes 2}}
\ar[r]^{\hspace{4pc}\mathop{\bigotimes}\limits_{i=1}^r \circ} & \mathop{\bigotimes}\limits_{i=1}^r
 RHom^{\mathcal{E}_n}_{A} \big({A}, {C}  \big)(U_i)
\ar[d]^{\rho_{U_1,\dots,U_r,D}} \\
 RHom^{\mathcal{E}_n}_{A} \big({A}, {B}  \big)(D)\otimes
RHom^{\mathcal{E}_n}_{B} \big({B}, {C}  \big)(D) \ar[r]^{\hspace{4pc}\circ }
 & RHom^{\mathcal{E}_n}_{A} \big({A}, {C}  \big)(D)}
\end{equation}
is commutative in $\hkmod$. Let be given $\phi_i\in
RHom^{\mathcal{E}_n}_{A} \big({A}, {B}  \big)(U_i)$ and $\psi_i\in
RHom^{\mathcal{E}_n}_{B} \big( {B}, {C}  \big)(U_i)$. We keep
denoting $\phi_i:\mathcal{A}_{|U_i}\to \mathcal{B}_{|U_i}$ and
 $\psi_i:\mathcal{B}_{|U_i}\to \mathcal{C}_{|U_i}$ the induced maps of factorization algebras.
 The result of the two compositions in diagram~\eqref{eq:EnAlgHochcomp}, namely
 $\rho_{U_1,\dots,U_r,D}(\psi_1,\dots,\psi_r) \circ \rho_{U_1,\dots,U_r,D}(\phi_1,\dots,\phi_r)$ and
 $\rho_{U_1,\dots,U_r,D}\circ \big(\otimes_{i=1}^{r} \psi_i\circ \phi_i \big)$    are both global
 sections over $D$ of factorization algebras morphisms. It is thus
 enough to prove that the underlying diagram of factorizations
 algebras
\begin{equation}\label{eq:FactAlgHochcomp}
\xymatrix{\mathcal{A}_{|D}
\ar@/^1pc/[rrrrr]^{{\rho_{U_1,\dots,U_r,D}}
\big(\otimes_{i=1}^{r}\psi_i\circ \phi_i  \big)_{|D}}
\ar@/_1pc/[rrrrr]_{{\rho_{U_1,\dots,U_r,D}(\psi_1,\dots,\psi_r)}_{|D}
\circ {\rho_{U_1,\dots,U_r,D}(\phi_1,\dots,\phi_r)}_{|D}} &&&&&
\mathcal{C}_{|D}}
\end{equation}
 is commutative (here the structure maps are given by construction~\eqref{eq:FormularhoonV}).

It is sufficient to prove the result on the stable under finite
intersection factorizing basis $\mathcal{U}_{U_1,\dots,U_r,D}$ given
in step \textbf{2} (\S~\ref{SS:step2}). The upper arrow in
diagram~\eqref{eq:FactAlgHochcomp} is simply the factorization
algebra map which is equal to $\psi_i\circ phi_i$ on any open subset
$*_i\subset W\subset U_i$ and is equal to $g\circ f$ for any other
$W\in\mathcal{U}_{U_1,\dots,U_r,D}$.  On the other hand the lower
map in diagram~\eqref{eq:FactAlgHochcomp} is the composition of two
factorizations algebras maps: one of which being given by $\phi_i$
on any open subset $*_i\subset W\subset U_i$ and $f$ on any other
$W\in\mathcal{U}_{U_1,\dots,U_r,D}$; while the other one is given by
$\psi_i$ on any open subset $*_i\subset W\subset U_i$ and $g$ on any
other $W\in\mathcal{U}_{U_1,\dots,U_r,D}$.
 The commutativity of
diagram~\eqref{eq:FactAlgHochcomp} follows on
$\mathcal{U}_{U_1,\dots,U_r,D}$ and thus
diagram~\eqref{eq:EnAlgHochcomp} also commutes.

\smallskip

It remains to prove statement (4) in Theorem~\ref{T:EnAlgHoch} which is almost trivial:
 let $h:C\to D$ be an
 $E_n$-algebra map and denote $\mathcal{C}$, $\mathcal{D}$ the associated factorization algebras
on $\mathbb{R}^n$. By~\cite[Theorem 5.3.3.1]{L-HA},
 $$\int_{U\cup V} A\; \cong\; \int_{U} A \otimes \int_{V} A$$ for any $E_n$-algebra $A$
and disjoint open sets $U,V$. Thus, the factorization algebra
associated (in Proposition~\ref{P:EnMod}) to $A\otimes C$  is given
by $U\mapsto \mathcal{A}(U)\otimes \mathcal{C}(U)$.
For any pairwise
disjoints open convex disks $U_1,\dots, U_r$ included in a bigger
convex disk $V\subset D^n$, and maps $g_i\in
RHom_{A}^{\mathcal{E}_n}(A,B)(U_i)$, $g'_i\in
RHom_{C}^{\mathcal{E}_n}(C,D)(U_i)$ ($i=1\dots r$), the map
\begin{multline*}\rho_{U_1,\dots,
U_r,V}\Big(\big(g_1\otimes\cdots\otimes g_r \big) \otimes
\big(g'_1\otimes \cdots\otimes g'_r\big)\Big)
\in \Big(RHom_{A}^{\mathcal{E}_n}(A,B)\otimes RHom_{C}^{\mathcal{E}_n}(C,D)\Big)(V)\\
\cong RHom_{A}^{\mathcal{E}_n}(A,B)(V)\otimes
RHom_{C}^{\mathcal{E}_n}(C,D)(V)\end{multline*} is the map obtained
as the global section of a map of factorization algebras
$$\rho_{U_1,\dots, U_r,V}\Big(\big(g_1\otimes\cdots\otimes g_r \big)
\otimes \big(g'_1\otimes \cdots\otimes g'_r\big)\Big)_{|V}:
\mathcal{A}_{|V}\otimes \mathcal{C}_{|V}\to \mathcal{B}_{|V}\otimes
\mathcal{D}_{|V}$$ as constructed in~\S~\ref{SS:step2}. In
particular, for any $i\in \{1,\dots,r\}$, its value in any open
subset $*_i\in W\subset U_i$ is given by  the map $g_i\otimes
g_i':\mathcal{A}_{|U_i}\otimes \mathcal{C}_{|U_i}
\to\mathcal{B}_{|U_i}\otimes \mathcal{D}_{|U_i}$. Further, its value
on any other open subset $W\in \mathcal{U}_{U_1,\dots,U_r,V}$ is
given by $f_{|V}\otimes g_{|V}$ (here we use freely the notations
introduced in \S~\ref{SS:step2} to define the structure maps
$\rho_{U_1,\dots, U_r,V}$).

Hence, the  map $\rho_{U_1,\dots,
U_r,V}\Big(\big(g_1\otimes\cdots\otimes g_r \big) \otimes
\big(g'_1\otimes \cdots\otimes g'_r\big)\Big)_{|V}$ identifies, on
(the factorizing and stable by finite intersection) cover
$\mathcal{U}_{U_1,\dots,U_r,V}$, with the map obtained by evaluating
the composition
\begin{multline*}
 \Big( \bigotimes_{i=1}^{r} RHom_{A}^{\mathcal{E}_n}(A,B)(U_i)\Big) \otimes
 \Big( \bigotimes_{i=1}^{r} RHom_{C}^{\mathcal{E}_n}(C,D)(U_i)\Big)\\
 \longrightarrow
\bigotimes_{i=1}^{r} RHom_{A\otimes C}^{\mathcal{E}_n}(A\otimes C,B\otimes D)(U_i)\\
\stackrel{\rho_{U_1,\dots,U_r,V}}\longrightarrow RHom_{A\otimes C}^{\mathcal{E}_n}(A\otimes C,B\otimes D)(V)
\end{multline*}
at the tensor product $ \big(g_1\otimes\cdots\otimes g_r
\big) \otimes \big(g'_1\otimes \cdots\otimes g'_r\big)$.
This proves that the canonical maps
$$RHom^{\mathcal{E}_n}_A \big({A}, {B}  \big)(V)
\otimes RHom^{\mathcal{E}_n}_C \big({C}, {D}  \big) (V)
\longrightarrow RHom^{\mathcal{E}_n}_{A\otimes C} \big({A\otimes C}, {B\otimes D}  \big)(V) $$
assembles into a map of factorization algebras and, consequently,
$$RHom^{\mathcal{E}_n}_A \big({A}, {B}  \big)
\otimes RHom^{\mathcal{E}_n}_C \big({C}, {D}  \big)
\longrightarrow RHom^{\mathcal{E}_n}_{A\otimes C} \big({A\otimes C}, {B\otimes D}  \big) $$
is a homomorphism of $E_n$-algebras.
\end{proof}

\begin{rem} The $E_n$-algebra structure given by Theorem~\ref{T:EnAlgHoch} is in fact
the solution of a universal property as will be given by Proposition~\ref{P:HHEn=z} below
which identifies $HH_{\mathcal{E}_n}(A,B)$ with the centralizer of the map $f:A\to B$.
\end{rem}

\begin{ex}
Assume $n=1$, then
$$HH_{\mathcal{E}_1}(A,B)\cong RHom_{A}^{\mathcal{E}_1}\big(A,B\big)
\cong RHom_{\int_{S^0} A} \big(A,B\big)\cong RHom_{A\otimes A^{op}}\big(A,B\big)$$
 is the standard Hochschild cohomology of the algebra $A$ with value
in the algebra $B$. It is straightforward that the $E_1$-structure given by Theorem~\ref{T:EnAlgHoch} is induced
on the standard Hochschild complex
 by the usual cup-product~\cite{Ge}.
\end{ex}

\begin{ex}\label{E:A=k}
Assume $A=k$ the ground ring and let  $f:k\to B$ be the unit map.
 We have a canonical equivalence $RHom_{k}^{\mathcal{E}_n}(k,B) \cong B$ in $\hkmod$.
This equivalence is in fact an \emph{equivalence of $E_n$-algebras}\footnote{where the left hand side
is endowed with the $E_n$-algebra structure given by Theorem~\ref{T:EnAlgHoch}}.
Since $f:k\to B$ is the unit map, it is immediate from the definition
of the structure maps~\eqref{eq:StructFactcenter} to check that
the locally constant factorization algebra structure of $RHom_{k}^{\mathcal{E}_n}(k,B) $
 is the one of $\mathcal{B}$, the locally constant factorization algebra on $\mathbb{R}^n$ associated to $B$
(in \S~\ref{S:EnasFact}).
\end{ex}

\begin{rem}\label{R:FactAlgCollar} Since the factorization algebra constructed by Theorem~\ref{T:EnAlgHoch} is locally constant, in particular,
 \emph{its value $RHom^{\mathcal{E}_n}_{A} \big({A}, {B}
\big)(U)$ on any (non-necessarily convex) disk $U$  is
$RHom_{\mathcal{A}(U)}^{\mathcal{E}_n}\Big( \int_{U} A,
\int_{U}B\Big)$} as asserted in step 1 (\S~\ref{SS:step1}).

\smallskip

One can describe directly the structure maps
$\rho_{U_1,\dots,U_r,D}$ associated to pairwise disjoint subdisks
$U_1,\dots, U_r$ of a disk $D$ (that is without further covering
them by convex subsets). This can be done as follows.
First, we use
$U_1,\dots, U_n$ to see the factorization algebra $\mathcal{A}$
restricted to $D$, denoted $\mathcal{A}_{|D}$, as obtained by gluing
together $r+1$-factorization algebras on $D$ (see
\S~\ref{SS:FactAlgebraFactHomology} and~\cite{CG}). Note that the
$\mathcal{A}(U)$-$E_n$-module structure on $\int_U A$ allows us to
see $\mathcal{A}_{|U}$ as a stratified factorization algebra on $U$
with a closed strata given by a point $*_i$ (or a sub-disk);
different choices of points leads to canonically equivalent
$\mathcal{A}_{|U}$-modules structures. We can choose a  collar $c_i$
in the neighborhood of the boundary (in $D$) of each $U_i$ such that
$c_i\cong S^{n-1}\times (0,\epsilon)$ for a homeomorphism  induced
by a homeomorphism $U_i\cong D^{n}$; for instance we just choose
$c_i$ to be the complement of the point $*_i\in U_i$ (or a suitable
sub-disk). This way we get a connected open set
 $$U_{\partial}:= D \setminus \big(\coprod_{i=1}^r (U_i-col_i)\big)$$ (the notation $\partial$
is meant to suggest that $U_{\partial}$ is the boundary of $\coprod U_i$ in $D$; below we will
sometimes refer to it using this terminology).
By definition, the $U_i$'s and $U_{\partial}$ cover $D$, hence the factorization algebra
 $\mathcal{A}_{|D}$ is obtained as the gluing of the restricted factorizations algebras
 $\mathcal{A}_{|U_1}, \dots, \mathcal{A}_{|U_r} $ and $\mathcal{A}_{|U_{\partial}}$.

 Let be given maps $g_i:\int_{|D_i} {A} \to \int_{|D_i}{B}$ of (left)
$\mathcal{A}(U_i)$-modules ($i=1\dots r$) and also denotes
$g_i:\mathcal{A}_{|D_i} \to \mathcal{B}_{|D_i}$ the induced maps of (stratified)
factorization algebras. We also denote  $f_{\partial}:\mathcal{A}_{|U_{\partial}} \to \mathcal{B}_{|U_{\partial}}$
the restriction of $f:\mathcal{A}\to \mathcal{B}$ to $U_{\partial}$.

\begin{lem}\label{L:glueFactmaps}
The family $(g_1,\dots, g_r, f_{\partial})$ of maps of factorization algebras glues together
to define a map \begin{equation*}
 \rho_{U_1,\dots,U_r,D}\,(g_1,\dots, g_r) \in RHom(\mathcal{A}_{|D}, \mathcal{B}_{|D}) \end{equation*}
 which is independent in $\hkmod$ of the choices (of collars) involved.
Further, on global sections, the induced map
$\int_{D}\rho_{U_1,\dots,U_r,D}\,(g_1,\dots, g_r): \int_{D}A\to
\int_{D}B$ is a map of $\mathcal{A}(D)$-$E_n$-modules. The induced
map \begin{multline}\label{eq:StructFactcenter}
 {\rho_{U_1,\dots,U_r,D}}: RHom^{\mathcal{E}_n}_{A} \Big({A}, {B}  \Big)\big(U_1\big)\otimes \cdots
\otimes RHom^{\mathcal{E}_n}_{A} \Big({A}, {B}  \Big)\big(U_r\big)
\\ \longrightarrow  RHom^{\mathcal{E}_n}_{A} \Big({A}, {B}  \Big)\big(D\big)
\end{multline}  is the map given by the factorization algebra structure
of Theorem~\ref{T:EnAlgHoch}.
\end{lem}
\begin{proof}[Proof of Lemma~\ref{L:glueFactmaps}]
Note that all triples intersections in the family $(U_1,\dots, U_r, U_{\partial})$ are empty and
that the only non-empty intersections are those of the form
$U_i\cap U_{\partial}=col_i\cong S^{n-1}\times(0,\epsilon)$. Hence,
by definition of the gluing of factorization algebras, we only have to check that the
maps $g_i$ and $f_{\partial}$ are equivalent on $\mathcal{A}_{|(U_i\cap U_{\partial})}$.
 By assumption, the map $g_i: \int_{U_i} A\to \int_{U_i} B$ is a map of
 $\mathcal{A}(U_i)$-modules.
Then Proposition~\ref{P:EnMod} (and Theorem~\ref{T:Theorem6GTZ2})
implies that the map of factorization algebras
$g_i:\mathcal{A}_{|col_i}\to \mathcal{B}_{|col_i}$ is equivalent to
 the map induced by the $\mathcal{A}(U_i)$-module structure of $\int_{U_i}B\cong B$.
 Since this module structure is given by $f:\mathcal{A}\to \mathcal{B}$, it follows that
$\big(g_i\big)_{|col_i}$
is equivalent to $\big(f_{\partial}\big)_{|col_i}$. Hence the collection
$(g_1,\dots, g_r, f_{\partial})$  assembles to give an object in
$ RHom(\mathcal{A}_{|D}, \mathcal{B}_{|D})$.
Further, we also just proved, that for any choice of collar $col_i\rq{}$  in the disk $U_i$,
 the value of $g_i$ on $\mathcal{A}_{col\rq{}_{i}}$ is given by $f$.
 It is thus independent of the choice of the collar. In order to
 check it induces the same map as Theorem~\ref{T:EnAlgHoch}, it is
 sufficient to check that the underlying maps of factorizations algebras agrees. For this, it is
 further sufficient to do it on the cover of $D$ obtained by taking
 only convex open subsets which are required to belong to either one of the $U_i$ or to $U_\partial$, for which the result follows  by
 definition.
\end{proof}
\end{rem}

\subsubsection{The parametrized factorization algebra structure on Hochschild cohomology of $E_n$-algebras}\label{R:EdHochwithDiskprimeoperad} The $E_n$-algebra structure given in Theorem~\ref{T:EnAlgHoch} is  given by a factorization algebra structure (by Theorem~\ref{P:En=Fact}).
In view of Proposition~\ref{P:AlternativeFacAlg}, it can also be
obtained as a \emph{parametrized} locally constant factorization
algebra (see Definition~\ref{D:AlternativeFacAlg}), that is a
locally constant algebra over the operad  $N(\text{Disk}(M)')$. This
structure is rather easy to describe as we now explain.

Let $f:A\to B$ be an $E_n$-algebra map. As before we may assume that the map is induced by a map  $f:\mathcal{A}\to\mathcal{B}$ of locally constant factorization algebras over $\R^n$.
From Definition~\ref{D:AlternativeFacAlg}, we see  that we need to  associate to any open embedding $\phi: \R^n\to \R^n$ a chain complex $HH_{\mathcal{E}_n}(A,B)(\phi)$. We set
\begin{equation}
HH_{\mathcal{E}_n}(A,B)(\phi):= RHom_{\mathcal{A}(\phi(\R^n))}^{\mathcal{E}_n}(A,B)(U) \cong RHom^{left}_{\int_{\partial U}A}\Big(\int_{U}A, \int_{U}B \Big)
\end{equation}
where $\partial U:=\phi\big( \R^n\setminus D(0,1)\big)$ is the image by $\phi$ of the complement of a closed (bounded) Euclidean disk centered at $0$. Note that any different choice of radius yields canonically equivalent chain complexes since $\mathcal{A}$, $\mathcal{B}$ are locally constant.

Now, we need to define structure maps associated to any open embedding $h: \coprod_{i=1}^r \R^n \to \R^n$ such that $\psi\circ h= \coprod_{i=1}^r \phi_i: \coprod_{i=1}^r \R^n \to M$.
 The structure map $$\rho^h_{\phi_1,\dots, \phi_r,\psi}: HH_{\mathcal{E}_n}(A,B)(\phi_1)\otimes \cdots \otimes HH_{\mathcal{E}_n}(A,B)(\phi_r) \longrightarrow HH_{\mathcal{E}_n}(A,B)(\psi)$$ is defined exactly in the same way as in Remark~\ref{R:FactAlgCollar} and in particular Lemma~\ref{L:glueFactmaps}. Here, we can use the canonical collars given by the complement of a disk centered at $0$ inside each disk $\mathbb{R}^n$.
  Then, using this slight alternative definition of factorization algebras, one proves (in a similar way) the obvious analogue of Theorem~\ref{T:EnAlgHoch}. A proof similar to the one of Theorem~\ref{T:EnAlgHoch} yields
\begin{prop}\label{P:EnAlgHochParam} Let $f:A\to B$ be a map of $E_n$-algebras\footnote{which we may assume to be given by a map $f:\mathcal{A}\to \mathcal{B}$ of factorization algebras, see \S~\ref{S:EnasFact}}.
\begin{enumerate}
\item The structure maps~\eqref{eq:StructFactcenter} $\rho^h_{\phi_1,\dots, \phi_r,\psi}$  make
 $U\phi\mapsto  RHom^{\mathcal{E}_n}_{A} \big({A}, {B}  \big)\big(\phi\big)$ a locally constant parametrized
factorization algebra on $\mathbb{R}^n$ whose global section are naturally equivalent to
$RHom^{\mathcal{E}_n}_{A} \Big({A}, {B}  \Big)$.
\item This parametrized factorisation algebra is equivalent to the one of Theorem~\ref{T:EnAlgHoch} under the equivalence of Proposition~\ref{P:AlternativeFacAlg}.
\item the composition and tensor product of endomorphisms are maps of (locally constant) parametrized factorization algebras.
\end{enumerate}
\end{prop}

\medskip

\subsubsection{The case of $E_\infty$-algebras again.}
 If $f: A\to B$ is a map of $E_\infty$-algebras, then Theorem~\ref{T:EnAlgHoch}
and Proposition~\ref{P:coHH=coTCH} give an $E_n$-algebra structure to $CH^{S^n}(A,B)$. The latter has also
an $E_n$-algebra structure given by Theorem~\ref{T:EdHoch}. The following result shows that these two structures
are the same.
\begin{prop}\label{P:HH(A,B)=CH(A,B)} Let $f:A\to B$ be a map of $E_\infty$-algebra and let
$B$ be endowed with the induced $A$-$E_\infty$-module structure.
 Then the natural equivalence $HH_{\mathcal{E}_n}(A,B)\cong CH^{S^n}(A,B)$ given by
Proposition~\ref{P:coHH=coTCH} is an equivalence of $E_n$-algebras\footnote{where the left hand-side
is the $E_n$-algebra given by Theorem~\ref{T:EnAlgHoch} and the right hand side is the $E_n$-algebra
given by Theorem~\ref{T:EdHoch}}.
\end{prop}
\begin{proof}
The proof of Proposition~\ref{P:coHH=coTCH} shows that
we have equivalences
\begin{multline}\label{eq:equivFactCHTCH}
 HH_{\mathcal{E}_n}(A,B) \cong RHom^{left}_{\int_{S^{n-1}}A}\Big(\int_{D^n}A,B\Big) \cong
RHom^{left}_{CH_{S^{n-1}}(A)}\big(CH_{D^n}(A),B\big)\\
 \cong  Hom^{\mathcal{E}_n}_{CH_{D^n}(A)}\big(CH_{D^n}(A),B\big).
\end{multline}
By Theorem~\ref{T:CH=TCH}, we have natural (in spaces and $E_\infty$-algebras) equivalences
$\int_U A \cong CH_{U}(A)$ and, by the value on a point axiom in Definition~\ref{D:axioms},
 a canonical equivalence $CH_U(B)\cong B$.
We can thus define a rule $U\mapsto Hom^{\mathcal{E}_n}_{CH_{U}(A)}\big(CH_{U}(A),B\big)$
and structure maps (for $U_1,\dots, U_r$ pairwise disjoints convex opens included in a  larger bounded convex open $D$)
\begin{multline}\label{eq:StructFactCH}
 {\rho_{U_1,\dots,U_r,D}}: Hom^{\mathcal{E}_n}_{CH_{U_1}(A)}\big(CH_{U_1}(A),B\big)\otimes \cdots
\otimes Hom^{\mathcal{E}_n}_{CH_{U_r}(A)}\big(CH_{U_r}(A),B\big)
\\ \longrightarrow  Hom^{\mathcal{E}_n}_{CH_{D}(A)}\big(CH_{D}(A),B\big)
\end{multline}
defined exactly as the structure maps~\eqref{eq:structMapsCVX} (in step 2, \S~\ref{SS:step2}) for $RHom^{\mathcal{E}_n}_A(A,B)(U)$.
Then the proof of
Theorem~\ref{T:EnAlgHoch} applies \emph{mutatis mutandis} to prove that
 $U\mapsto Hom^{\mathcal{E}_n}_{CH_{U}(A)}\big(CH_{U}(A),B\big)$
is a \emph{locally constant factorization algebra on $D^n$} and further,
 that the equivalences
$$Hom^{\mathcal{E}_n}_{\mathcal{A}(U)}\Big(\int_{U}A,\int_{U} B\Big)\;\cong \;
Hom^{\mathcal{E}_n}_{CH_{U}(A)}\big(CH_{U}(A),B\big) $$ (induced by Theorem~\ref{T:CH=TCH})
are  equivalences of factorization algebras.

\smallskip

Now, for any collar $\partial U \cong S^{n-1}\times (0,\epsilon)$ inside a disk $U$,
 we have  natural equivalences
\begin{eqnarray*}
Hom^{\mathcal{E}_n}_{CH_{U}(A)}\big(CH_{U}(A),B\big) &\cong &
RHom_{CH_{\partial U}(A)}^{left}\big(CH_{U}(A),B \big) \\&\cong & RHom_A\Big(A, RHom_{CH_{\partial U}(A)}^{left}\big(CH_{U}(A),B\big) \Big) \\
&\cong &  RHom_A\Big(CH_{U}(A) \mathop{\otimes}\limits_{CH_{\partial U}(A)}^{\mathbb{L}} A, \, B\Big) \\
&\cong &  RHom_A\Big( CH_{U/\partial U}(A),B) \cong CH^{U/\partial U}(A,B)
\end{eqnarray*}
where the last equivalences are by the excision axiom (Definition~\ref{D:axioms}) and definition of Hochschild cochains.
Recall from the proof of Proposition~\ref{P:coHH=coTCH} that, for $U=D^n$,
 the above equivalences and the equivalences~\eqref{eq:equivFactCHTCH} are precisely the natural equivalence
$HH_{\mathcal{E}_n}(A,B)\cong CH^{S^n}(A,B)$ of Proposition~\ref{P:coHH=coTCH}.
 Hence, we are left to prove Proposition~\ref{P:HH(A,B)=CH(A,B)} replacing $HH_{\mathcal{E}_n}(A,B)$
with $CH^{S^n}(A,B)$ endowed with the $E_n$-algebra structure
given by the locally constant factorization algebra structure.

\smallskip

By functoriality of Hochschild chains, we have natural maps of $E_\infty$-algebras
\begin{equation}\label{eq:rhoB} \bigotimes_{i=1}^r B\cong \bigotimes_{i=1}^r CH_{U_i}(B)
\cong CH_{\bigcup_{i=1}^r U_i}(B)
\stackrel{(\bigcup_{i=1}^r U_i\hookrightarrow V)_*}\longrightarrow CH_{V}(B) \cong B \end{equation}

Further, we have pinching maps
$D\stackrel{p_{U_1,\dots,U_r,D}}\longrightarrow \bigvee_{i=1}^r \big(U_i/\partial U_i\big)$ obtained by
 collapsing $D\setminus \big(\bigcup_{i=1}^r U_i\setminus \partial U_i\big)$ to a point.
By functoriality of Hochschild cochains, the pinching maps yield
a map
 \begin{equation}\label{eq:rhoA}
 \bigotimes_{i=1}^r CH_{U_i}(A) \cong CH_{\bigcup_{i=1}^r U_i}(A)
\stackrel{(p_{U_1,\dots,U_r,D})_*}\longrightarrow CH_{\bigvee_{i=1}^r \big(U_i/\partial U_i\big)} (A).
\end{equation}
Using the last two maps, we get the composition
\begin{multline}\label{eq:rhovee}
\bigotimes_{i=1}^r RHom^{left}_{CH_{\partial U_i}(A)}\big(CH_{U_i}(A),B\big)
\longrightarrow RHom_{ \bigotimes_{i=1}^r A}\big( \bigotimes_{i=1}^r CH_{U_i}(A),\,  \bigotimes_{i=1}^r B \big)\\
\longrightarrow RHom_{ \bigotimes_{i=1}^r A}\big( \bigotimes_{i=1}^r CH_{U_i}(A),\, B\big)
\cong RHom_{A}\big(CH_{\bigvee_{i=1}^r \big(U_i/\partial U_i\big)} (A),\, B \big)\\
\longrightarrow RHom_{A}(CH_{V}(A), \, B\big)
\end{multline}
where the last maps are respectively induced by the map~\eqref{eq:rhoB},  Lemma~\ref{L:wedge}
 and the map~\eqref{eq:rhoA}.
Using the homotopy invariance of Hochschild cochains and unfolding the definition
 of $\rho_{U_1,\dots,U_r,D}$,  we find that the structure map~\eqref{eq:StructFactCH}
 is transfered to the above map~\eqref{eq:rhovee} under the natural  equivalences
$RHom_{CH_{\partial U}(A)}^{left}\big(CH_{U}(A),B \big)\cong CH^{U/\partial U}(A,B) $
(where $U$ is any disk in $D^n$).
Note that when the $U_i$ are cubes in $\mathcal{C}_n(r)$ and $D$ is $D^n$,
the composition of the map~\eqref{eq:rhovee} with the equivalence
\begin{multline*}\Big(\bigotimes_{i=1}^r CH^{S^n}(A,B) \Big) \;\cong \;  \bigotimes_{i=1}^r
RHom^{left}_{CH_{S^{n-1}}(A,B)}\big(CH_{D^n}(A),B\big) \\ \cong \;
\bigotimes_{i=1}^r RHom^{left}_{CH_{\partial U_i}(A)}\big(CH_{U_i}(A),B\big) \end{multline*}
is the pinching map~\eqref{eq:pinchSr}
$pinch_{S^d,r}^*(c):\left(CH^{S^d}(A,B)\right)^{\otimes r}\longrightarrow CH^{S^d}(A,B)$
(where $c$ is the cube associated to the $U_i$\rq{}s).
 Now, thanks to Theorem~\ref{T:EdHoch} and the definition of the factorization algebra
structure on $U\mapsto Hom^{\mathcal{E}_n}_{CH_{U}(A)}\big(CH_{U}(A),B\big)$, we can apply Lemma~\ref{L:cubetoDisk} below
which implies that the two $E_n$-algebra structure
 on  $CH^{S^n}(A,B)$ (given by Theorem~\ref{T:EdHoch} and the one introduced in this proof by
the structures maps~\eqref{eq:rhovee})
are equivalent.
\end{proof}

\begin{lem}\label{L:cubetoDisk} Let $A\in \hkmod$ and assume that
\begin{enumerate}
\item $A$ has an $\mathcal{C}_n$-algebra structure, \emph{i.e.}, an $E_n$-algebra structure  given
by an action of the chains little $n$-dimensional cube operad.
\item There is  a locally constant
factorization algebra  $\mathcal{A}$  on $\R^n$ (identified with the open unit cube)
 together with an equivalence $\varphi:\mathcal{A}(\R^n)\stackrel{\simeq}\to A$ (in $\hkmod$);
which thus induces another $E_n$-algebra structure on $A$.
\end{enumerate}
Assume further that the two structures given by \emph{(1)} and \emph{(2)} are \emph{compatible} in the following sense:
 for any configuration of cubes $c\in \mathcal{C}_n(r)$, the following diagram
\begin{equation}\label{eq:compatibilitycubetodisk}
 \xymatrix{  A^{\otimes r} \ar[rrr]^{\mu_c} & & & A \\
 \mathcal{A}(c_1) \otimes \cdots \otimes \mathcal{A}(c_r)
\ar[u]^{\bigotimes_{i=1}^r \varphi\circ \rho_{c_i, \R^n}} \ar[rrr]^{\hspace{1pc}\rho_{c_1,\dots,c_r, \R^n}}
&& &\mathcal{A}(\R^n)\ar[u]_{\varphi}  \, ,}
\end{equation}
 (where we denote $c_1,\dots, c_r$ the configuration
of cubes defined by $c=(c_1,\dots, c_r)$, $\mu_c$ is the structure map given by the
operadic structure and $\rho_{U_1,\dots,U_r,V}$ the structure maps of the
factorization algebra structure),
is commutative in $\hkmod$.

\smallskip

Then the two $E_n$-algebras structures on $A$ defined by \emph{(1)} and \emph{(2)} are equivalent
(in $E_n$-Alg).

A similar statement holds with $E_n$-coalgebra structure instead of $E_n$-algebra structures.
\end{lem}
\begin{proof} The statement for coalgebra follows \emph{mutatis mutandis} from the one for algebras; we only prove the last one.

The $E_n$-algebra structure defined by (1), that is by the action of
the little $n$-dimensional cube operad on $A$ yields a locally
factorization algebra $\mathcal{A}'$ on $\R^n$ which is equivalent
to the $E_n$-structure given by (1) and satisfies
$\mathcal{A}'(U)\cong \int_U A$ (see Theorem~\ref{P:En=Fact} and
Theorem~\ref{T:Theorem6GTZ2}).
 Thus
we only have to prove that  $\mathcal{A}'$ is equivalent to
$\mathcal{A}$ as a factorization algebra on  $(0,1)^n$. Let us
analyze further the construction of  $\mathcal{A}'$. The
$\mathcal{C}_n$-action on $A$ gives a structure of
$\mathbb{E}_{\R^n}^{\otimes}$-algebra to $A$, where
$\mathbb{E}_{\R^n}^{\otimes}$ is the $\infty$-operad introduced by
Lurie in~\cite[Section 5.2]{L-HA}, that is, the operad whose
algebras are precisely those given by Definition~\ref{D:EnasDisk}.
  The canonical map of
operad $\mathcal{C}_n\to \mathbb{E}_{\R^n}^{\otimes}$ is an
equivalence by the results of Lurie~\cite[Example 5.2.4.3]{L-HA}.
 By~\cite[Theorem 5.2.4.9]{L-HA} (also see
\S~\ref{S:EnasFact}), the $\infty$-operad map
$N(Disk(\R^n))^{\otimes}\to \mathbb{E}_{\R^n}^{\otimes}$ now yields
the locally constant factorization algebra  structure on
$\mathbb{R}^n$ (denoted $\mathcal{A}'$ above). Let $\mathcal{R}$ be
the factorizing basis  of $(0,1)^n$ given by the open rectangles and
denote $\text{PFac}_{\mathcal{R}}^{lc}$  the category of locally
constant $\mathcal{R}$-prefactorization algebras
(Definition~\ref{D:UPrefactAlg}).
 Then, similarly, we have an equivalence $\mathcal{C}_n\text{-Alg} \stackrel{\simeq}\to\text{PFac}_{\mathcal{R}}^{lc}$ which
  fits into the following commutative diagram:
\begin{equation}\label{eq:Cuben=Fact}\xymatrix{
\mathcal{C}_n\text{-Alg}  \ar[rr]^{\simeq} &&
\text{PFac}_{\mathcal{R}}^{lc} \\
\mathbb{E}_{\R^n}^{\otimes}\text{-Alg}
\ar[u]^{\simeq}\ar[rr]^{\simeq}&& \text{Fac}_{(0,1)^n}^{lc}
\ar@{^{(}->}[u] }
\end{equation}
where the right arrow is given by restrictions to the opens
belonging to $\mathcal{R}$.
  In particular, we
have natural (with respect to the factorization algebra structure)
equivalences $\mathcal{A}'(c)\stackrel{\simeq}\to A$ for any open
any little rectangle $c\in \mathcal{R}$.

\smallskip

Since the family of open cubes inside $\R^n$ forms a factorization
basis of $\R^n$, it is enough, by
Proposition~\ref{P:extensionfrombasis}, to check that the
factorization algebra structures on $\mathcal{A}'$ and $\mathcal{A}$
are equivalent on the bais $\mathcal{R}$ of open rectangles. From
diagram~\ref{eq:Cuben=Fact}, we see that this is precisely the
compatibility condition~\eqref{eq:compatibilitycubetodisk} of the
Lemma. The result follows.
\end{proof}

\begin{rem} It is possible, though more technically involved, to use directly, in the spirit of
Section~\ref{S:Edcochains}, the little cube operad $\mathcal{C}_n$ to
make $RHom^{\mathcal{E}_n} \Big({A}, {B}  \Big)$ an $E_n$-algebra. We now sketch how to do this,
leaving to the interested reader the task to fill in the many details.

\smallskip

We let again $\mathcal{A}$, $\mathcal{B}$ be the factorization algebras corresponding to $A$, $B$.
Recall that we have  factorizations algebras $\mathcal{A}^{\otimes k}$, $\mathcal{B}^{\otimes k}$
on $\coprod_{i=1}^{k} D^n$ and similarly for $\mathcal{B}$.
Let $c\in \mathcal{C}_n(r)$ be a framed embedding $\coprod_{i=1}^r D^n \hookrightarrow D^n$.
Then the little cube $c$ induces a natural (in $\mathcal{A}$ and $c$) equivalence
$\mathcal{A}_{| c^{-1}(D^n)} \cong \mathcal{A}^{\otimes r}$.
We can define a map
\begin{equation*} comp_{r}(f,c): RHom^{\mathcal{E}_n} \Big({A}, {B}  \Big)^{\otimes r}
\longrightarrow RHom^{\mathcal{E}_n} \Big({A}, {B}  \Big)\end{equation*}
similarly to the definition of the structure maps~\eqref{eq:StructFactcenter}. Indeed,
 we first use $c$ to see the factorization algebras $\mathcal{A}$ and $\mathcal{B}$
on $D^n$ as obtained by gluing together  $r+1$-factorization algebras on $D^n$.  The
 image $c(\coprod_{i=1}^r D^n)$ has $r$-open connected components, denoted $D_1,\dots, D_r$.
Choosing small collars $col_1, \dots col_r$  in the neighborhood of
each $D^n\subset \coprod_{i=1}^r D^n$ yield a connected open set
 $U_{\partial}:= D^n \setminus c\big(\coprod_{i=1}^r (D^n-col_i)\big)$.
Since $c$ is an embedding, it induces an identification $\mathcal{A}\cong \mathcal{A}_{|D_i}$
for each $i=1\dots r$. Thus  from any family of maps
$g_1,\dots, g_r \in RHom^{\mathcal{E}_n}\big (\mathcal{A}, \mathcal{B}\big)$,
we get induced maps of factorization algebras $g_i:\mathcal{A}_{|D_i} \to \mathcal{B}_{|D_i}$.
 Further,  we also have, by restriction of $f$ to the open set $U_c$, an induced map
$f_c:\mathcal{A}_{|U_c} \to \mathcal{B}_{|U_c}$.

The argument of the proof of Lemma~\ref{L:glueFactmaps} apply to show
\begin{lem}\label{L:glueFactmapscubes}
The family $(g_1,\dots, g_r, f_c)$ of maps of factorization algebras glues together to define
a map of factorization algebras
$$comp_r(f,c)\,(g_1,\dots, g_r)\in Hom(\mathcal{A}, \mathcal{B}) $$
Further, on global sections, the induced map
$comp_r(f,c)\,(g_1,\dots, g_r)\big( D^n\big): A\to B$ is a map of $A$-$E_n$-module.
\end{lem}

It follows (from the above Lemma~\ref{L:glueFactmaps}) that we have a well-defined  map $(g_1,\dots, g_r)\mapsto comp_r(f,c)(g_1,\dots, g_r)(D^n)$, simply denoted by
\begin{equation}\label{eq:compfc}
 comp_{r}(f,c): RHom^{\mathcal{E}_n} \Big({A}, {B}  \Big)^{\otimes r}
\longrightarrow RHom^{\mathcal{E}_n} \Big({A}, {B}  \Big).
\end{equation}

Recall that the set of $A$-$E_n$-modules homomorphisms is simplicially enriched. Similarly,
there are simplicial sets of maps of factorization algebras, see~\cite{CG}. Equivalently,
 we have topological spaces of such maps.
Using the fact that the factorization
algebras $\mathcal{A}$ and $\mathcal{B}$ are locally constant, one can prove the following
 \begin{lem}\label{L:compContinuity}
  The map $comp_{r}(f,c): RHom^{\mathcal{E}_n} \Big({A}, {B}  \Big)^{\otimes r}
\longrightarrow RHom^{\mathcal{E}_n} \Big({A}, {B}  \Big) $ depends continuously on $c$.
 \end{lem}

The above Lemma~\ref{L:compContinuity} allows us to consider the maps  $comp_{r}(f,c)$ in families over the operad of little cubes and thus one can  let $c$ runs through the operad
$\mathcal{C}_n(r)$ so that we get  the first part of the following result.

\begin{prop}\label{P:comprf} Let $f:A\to B$ be a map of $E_n$-algebras.
\begin{enumerate}\item The maps $comp_r(f,c)$ assembles to give a map
$$comp_r(f): C_{\ast}\big(\mathcal{C}_n(r)\big) \otimes
RHom^{\mathcal{E}_n} \Big({A}, {B}  \Big)^{\otimes r}
\longrightarrow RHom^{\mathcal{E}_n} \Big({A}, {B}  \Big)$$ in $k\text{-}Mod_\infty$.
\item The maps $comp_r(f)$  gives to $ RHom^{\mathcal{E}_n} \Big({A}, {B}  \Big)$ a natural $E_n$-algebra
structure.\end{enumerate}
\end{prop}
The proof of the second assertion of this Proposition is essentially the same as the ones of
Theorem~\ref{T:EnAlgHoch} and Theorem~\ref{T:EdHoch}.
\end{rem}

\subsection{$E_n$-Hochschild cohomology as centralizers}\label{S:maincentralizers}
We will now relate the natural $E_n$-algebra structure of
$RHom_{A}^{\mathcal{E}_n}(A,B)$ (for an $E_n$-algebra
map $f:A\to B$) given in Section~\ref{S:Prelimcentralizers}
with the centralizer $\mathfrak{z}(f)$.
The following definition is due to Lurie~\cite{L-HA, L-VI}
(and generalize the notion of center of a category due to Drinfeld).

\begin{definition}\label{D:center}
The (derived) centralizer of an $E_n$-algebra map $f:A\to B$ is the \emph{universal} $E_n$-algebra $\mathfrak{z}(f)$ endowed with a homomorphism of $E_n$-algebras $e_{\mathfrak{z}(f)}: A\otimes \mathfrak{z}(f)\to B$  making the following diagram
\begin{equation}
\label{eq:Defcenter} \xymatrix{ & A\otimes \mathfrak{z}(f) \ar[rd]^{e_{\mathfrak{z}(f)}} & \\
A \ar[ru]^{id\otimes 1_{\mathfrak{z}(f)}} \ar[rr]^{f} && B  }
\end{equation}
commutative in $E_n$-Alg.
\end{definition}
The existence of the derived centralizer $\mathfrak{z}(f)$ of   an $E_n$-algebra map $f:A\to B$ is a \emph{non-trivial} Theorem of Lurie~\cite{L-HA, L-VI}.
The universal property of the centralizer implies that there are natural
maps of $E_n$-algebras
\begin{equation}
\label{eq:z(circ)} \mathfrak{z}(\circ): \mathfrak{z}(f)\otimes \mathfrak{z}(g) \longrightarrow \mathfrak{z}(g\circ f)
\end{equation}
see~\cite{L-HA, L-VI}.

The $E_n$-algebra structure on the $\mathcal{E}_n$-Hochschild cohomology given by
Theorem~\ref{T:EnAlgHoch} gives an explicit description of the centralizer
$\mathfrak{z}(f)$ (as an $E_n$-algebra):
\begin{prop}\label{P:HHEn=z} Let $f:A\to B$ be an $E_n$-algebra map and endow $HH_{\mathcal{E}_n}(A,B)$ with the $E_n$-algebra structure given by Theorem~\ref{T:EnAlgHoch}.

Then the $\mathcal{E}_n$-Hochschild cohomology
 $HH_{\mathcal{E}_n}(A,B)\cong RHom^{\mathcal{E}_n}_{A} \big({A}, {B}  \big)$ is the centralizer
$\mathfrak{z}(f)$, \emph{i.e.}, there is a natural  equivalence of $E_n$-algebras
$HH_{\mathcal{E}_n}(A,B)\cong \mathfrak{z}(f)$ such that, for any $E_n$-algebra map $g:B\to C$,
the following diagram
$$\xymatrix{RHom^{\mathcal{E}_n}_{A} \big({A}, {B}  \big)
\otimes RHom^{\mathcal{E}_n}_{B} \big({B}, {C}  \big) \ar[d]_{\circ}
\ar[rr]^{\hspace{4pc}\cong\otimes \cong} && \mathfrak{z}(f)\otimes
\mathfrak{z}(g) \ar[d]^{\mathfrak{z}(\circ)} \\
RHom^{\mathcal{E}_n}_{A} \big({A}, {C}  \big) \ar[rr]^{\cong} &
&\mathfrak{z}(g\circ f) } $$ commutes in $E_n\text{-Alg}$.
\end{prop}
\begin{rem}
 Note that in the proof of Proposition~\ref{P:HHEn=z}, we \emph{do not} assume the existence of
centralizers, but actually prove that $HH_{\mathcal{E}_n}(A,B)$ satisfies the universal property
of centralizers. In particular the proof of Proposition~\ref{P:HHEn=z} implies the existence of
centralizers of any map $f:A\to B$ of $E_n$-algebras.
\end{rem}

We first prove a lemma.  Denote $ev: A\otimes RHom_{A}^{\mathcal{E}_n}\big(A, B\big)\to B$ the (derived) evaluation map $(a,f)\mapsto f(a)$.
\begin{lem}\label{L:univpropHHEn}
The evaluation map $ev: A\otimes RHom_{A}^{\mathcal{E}_n}\big(A, B\big)\to B$ is an $E_n$-algebra morphism. Further, the following diagram
$$  \xymatrix{ & A\otimes RHom_{A}^{\mathcal{E}_n}\big(A, B\big) \ar[rd]^{ev} & \\
A \ar[ru]^{id\otimes 1} \ar[rr]^{f} && B  } $$
 is commutative in $E_n$-Alg.
\end{lem}
\begin{proof}
There are canonical equivalences of $E_n$-algebras $RHom_{k}^{\mathcal{E}_n}(k,A) \cong A$ and $RHom_{k}^{\mathcal{E}_n}(k,B) \cong B$ (see Example~\ref{E:A=k}).  Thus, the fact that $ev$ is a map of $E_n$-algebras follows from statement (3) in Theorem~\ref{T:EnAlgHoch}. Further,  the same Theorem implies that the unit of $RHom_{A}^{\mathcal{E}_n}\big(A, B\big)$ is $f:A\to B$. It follows that $ev \circ (id\otimes 1)=f$ which proves the Lemma.
\end{proof}

\begin{proof}[Proof of Proposition~\ref{P:HHEn=z}] By Lemma~\ref{L:univpropHHEn},
 we already know that $HH_{\mathcal{E}_n}^{\bullet}(A,B)\cong RHom_{A}^{\mathcal{E}_n}(A,B)$ fits
into a commutative diagram similar to diagram~\eqref{eq:UnivProp} below.
 We have to prove that for any $E_n$-algebra $\mathfrak{z}$,
endowed with a $E_n$-algebra map $\phi: A\otimes \mathfrak{z}\to B$ fitting in a commutative diagram
\begin{equation}
\label{eq:UnivProp} \xymatrix{ & A\otimes \mathfrak{z} \ar[rd]^{\phi} & \\
A \ar[ru]^{id\otimes 1_{\mathfrak{z}}} \ar[rr]^{f} && B  \, ,}
\end{equation} there exists an $E_n$-algebra map $\mathfrak{z}\to RHom_{A}^{\mathcal{E}_n}$
which makes $A\otimes \mathfrak{z}\stackrel{\phi}\longrightarrow B$
factor through
 $ A\otimes RHom_{A}^{\mathcal{E}_n}\big(A, B\big)\stackrel{ev}\longrightarrow B$.

Let $\theta_{\phi}: \mathfrak{z}\to RHom(A,B)$ be the  map associated to
 $\phi: A\otimes \mathfrak{z}\to B$
under the (derived) adjunction
$RHom(A\otimes \mathfrak{z}, B) \cong RHom\big(\mathfrak{z}, RHom(A,B)\big)$ (in $\hkmod$).

We now prove that $\theta_{\phi}$ takes values in $RHom_{A}^{\mathcal{E}_n}(A,B) $. We use again
the factorization algebra characterization of $E_n$-algebras.
 Let $\mathcal{A}$,  $\mathcal{B}$ and $\mathcal{Z}$ be the locally constant factorization
algebras associated to $A$, $B$ and $\mathfrak{z}$.
For any open sub-disk $D\hookrightarrow D^n$, we get the induced map\footnote{we make, for simplicity, an abuse of notation still denoting by $\phi$ the induced map and similarly with $\theta_{\phi}$ below}
$$\phi: \big(\mathcal{A}\otimes \mathcal{Z}\big)(D)\cong \int_D A\otimes \int_D \mathfrak{z}
\stackrel{\int_D \phi}\longrightarrow \int_D B\cong \mathcal{B}(D)$$
and its (derived) adjoint
$\theta_{\phi}:\mathcal{Z}(D)\longrightarrow
RHom(\mathcal{A}(D),\mathcal{B}(D))$.
We are left to check that this last map is compatible with the factorization algebra structures
 (describing the $A$-$E_n$-module structure of $A$ and $B$).
Let $U_0, U_1,\dots, U_r$ be pairwise disjoints open disks included in a bigger disk $V$,
where we assume that $U_0$ contains the base point of $D^n$.
Also we use the same notation
$$\rho_{U_0,\dots, U_r,V}: \mathcal{F}(U_0)\otimes \cdots \otimes \mathcal{F}(U_r)\longrightarrow \mathcal{F}(V)$$
for the associated structure maps of any one of the factorization algebras
$\mathcal{F}= \mathcal{A}$, $\mathcal{B}$ or  $\mathcal{Z}$ on $D^n$.
 Since $\phi: A\otimes \mathfrak{z}\to B$ is a map of $E_n$-algebras,
 for any $a_i\in \mathcal{A}(U_i)$ ($i=1\dots r$), $x\in \mathcal{A}(U_0)$
and $z\in \mathfrak{z}(U_0)$,  we have
\begin{eqnarray*}
\phi\Big(  \rho_{U_0,\dots, U_r,V}\big(x, a_1,\dots, a_r\big)\!\otimes\! \rho_{U_0,V}\big(z\big)\!\Big) \hspace{-0.7pc}
&= &  \hspace{-0.7pc}\phi\Big(  \rho_{U_0,\dots, U_r,V}\big(x\otimes z, a_1\otimes 1_{\mathfrak{z}},\dots, a_r\otimes 1_{\mathfrak{z}}\big) \!\Big) \\
\hspace{-0.7pc}
&= &  \hspace{-0.7pc} \rho_{U_0,\dots, U_r,V}\Big(\phi(x\otimes z),\phi(a_1\otimes 1_{\mathfrak{z}}),\dots \\
\hspace{-0.7pc}
& & \hspace{7pc}\dots,\phi(a_r\otimes 1_{\mathfrak{z}}) \Big)\\
\hspace{-0.7pc}
&= &  \hspace{-0.7pc}\rho_{U_0,\dots, U_r,V}\big(\phi(x\otimes z),f(a_1),\dots, f(a_r)\big)
\end{eqnarray*} where the last identity follows  from the commutativity of diagram~\eqref{eq:UnivProp}.
 Note that the map $z\mapsto  \rho_{U_0,V}\big(z\big) $ is an equivalence
(since $\mathcal{Z}$ is locally constant).
Since the $A$-$E_n$-module structure on $B$ is given by $f$, the above string of equalities
ensures that
$\theta_{\phi}$ is a map from $\mathfrak{z}$ to $ RHom_{A}^{\mathcal{E}_n}(A,B)$. In particular,
 the map $\theta_{\phi}:\mathfrak{z}\to RHom(A,B)$ factors as
$$\mathfrak{z}\stackrel{\tilde{\theta_{\phi}}}\longrightarrow RHom_{A}^{\mathcal{E}_n}(A,B)
\cong RHom^{left}_{\int_{S^{n-1}} A}\Big(\int_{D^n} A, \int_{D^n} B\Big)
\hookrightarrow RHom(A,B). $$

\smallskip

To finish the proof of Proposition~\ref{P:HHEn=z},
we  need to check that $\tilde{\theta_{\phi}}:\mathfrak{z}\to RHom_{A}^{\mathcal{E}_n}(A,B)$ is a map of $E_n$-algebras.
Recall that there are equivalences $\mathfrak{z}\cong RHom_{k}^{\mathcal{E}_n}(k,\mathfrak{z})$,
$\mathfrak{z}\cong k\cong \mathfrak{z}$
 of $E_n$-algebras (see Example~\ref{E:A=k}).  By definition of the (derived) adjunction,
$\tilde{\theta_{\phi}}:\mathfrak{z}\to RHom_{A}^{\mathcal{E}_n}(A,B)$ factors as the composition
\begin{multline}\label{eq:factorthetaphi}
 \mathfrak{z} \cong k\otimes \mathfrak{z}\hspace{-0.3pc}\stackrel{ 1_{RHom_A^{\mathcal{E}_n}(A,A)}\otimes id}
\longrightarrow \hspace{-0.8pc} RHom_A^{\mathcal{E}_n}(A,A)\otimes \mathfrak{z}
\cong RHom_A^{\mathcal{E}_n}(A,A) \otimes RHom_{k}^{\mathcal{E}_n}(k,\mathfrak{z})\\
\longrightarrow RHom_{A}^{\mathcal{E}_n}(A,A\otimes \mathfrak{z})
\stackrel{\phi_*}\longrightarrow RHom_{A}^{\mathcal{E}_n}(A,B).
\end{multline}
By Theorem~\ref{T:EnAlgHoch}.(2) and (4), the last two maps are an $E_n$-algebra Homomorphisms.
 Thus the composition~\eqref{eq:factorthetaphi} is a composition of $E_n$-algebras maps hence
$\tilde{\theta_{\phi}}:\mathfrak{z}\to RHom_{A}^{\mathcal{E}_n}(A,B)$ itself is
a map of $E_n$-algebras.

\smallskip

Further, by definition of $\theta_{\phi}$,  the identity
$$ev \circ \big(id_A\otimes \theta_{\phi} \big)\;=\; \phi$$ holds. Hence we eventually get
 a commutative  diagram
$$\xymatrix{A \ar[rd]^{id\otimes 1_{\mathfrak{z}}}
\ar@/^/[rrrd]^{id\otimes 1_{RHom_A^{\mathcal{E}_n}(A,B)}} \ar@/_2pc/[rrdd]_{f} & & & \\
 & A\otimes \mathfrak{z}\ar[rd]^{\phi} \ar[rr]^{\hspace{-2.5pc}id\otimes \tilde{\theta_\phi}} & &
A\otimes RHom_{A}^{\mathcal{E}_n}(A,B) \ar[ld]^{ev} \\
& & B  &   } $$ in $E_n$-Alg.

\smallskip

It remains to prove the uniqueness of the map $\mathfrak{z}\to  RHom_{A}^{\mathcal{E}_n}(A,B)$
inducing such a commutative diagram.
Thus assume  that $\alpha: \mathfrak{z}\to  RHom_{A}^{\mathcal{E}_n}(A,B)$ is a map of $E_n$-algebras
such that the following diagram
\begin{equation}\label{eq:univdiagcenter}
 \xymatrix{A \ar[rd]^{id\otimes 1_{\mathfrak{z}}}
\ar@/^/[rrrd]^{id\otimes 1_{RHom_A^{\mathcal{E}_n}(A,B)}} \ar@/_2pc/[rrdd]_{f} & & & \\
 & A\otimes \mathfrak{z}\ar[rd]^{\phi} \ar[rr]^{\hspace{-2.5pc}id\otimes \alpha} & &
A\otimes RHom_{A}^{\mathcal{E}_n}(A,B) \ar[ld]^{ev} \\
& & B  &   }
\end{equation}
is commutative in $E_n$-Alg. Note that the composition
\begin{multline}\label{eq:splitid}
 RHom_{A}^{\mathcal{E}_n}(A,B)\cong
RHom_k^{{\mathcal{E}_n}}\Big(k,RHom_{A}^{\mathcal{E}_n}(A,B)\Big) \\
\stackrel{1_{RHom_A^{\mathcal{E}_n}(A,A)}\otimes id}\longrightarrow
RHom_A^{\mathcal{E}_n}(A,A) \otimes \Big(k,RHom_{A}^{\mathcal{E}_n}(A,B)\Big) \\
\longrightarrow RHom_A^{\mathcal{E}_n}\Big(A,A\otimes RHom_{A}^{\mathcal{E}_n}(A,B)\Big)\\
\stackrel{ev_*}\longrightarrow RHom_{A}^{\mathcal{E}_n}(A,B)
\end{multline}
is the identy map. From the commutativity of Diagram~\eqref{eq:univdiagcenter},
we get the following commutative diagram
\begin{equation}\label{eq:univimplytheta}
 \xymatrix{\mathfrak{z} \cong RHom_k^{{\mathcal{E}_n}}\big(k,\mathfrak{z}\big)
\ar[rr]^{\hspace{-5pc}\alpha_*} \ar[d]  &&
RHom_{A}^{\mathcal{E}_n}(A,B) \cong
RHom_k^{{\mathcal{E}_n}}\big(k,RHom_{A}^{\mathcal{E}_n}(A,B)\big) \ar[d]  \\
 RHom_A^{\mathcal{E}_n}\big(A,A\otimes \mathfrak{z}\big) \ar[rrd]^{\phi_*}
&& RHom_A^{\mathcal{E}_n}\big(A,A\otimes RHom_{A}^{\mathcal{E}_n}(A,B)\big)\ar[d]^{ev_*} \\
&& RHom_{A}^{\mathcal{E}_n}(A,B) }
\end{equation}
in $E_n$-Alg. The composition of the right vertical maps in Diagram~\eqref{eq:univimplytheta}
is the composition~\eqref{eq:splitid}, hence is the identity, and the upper map is
$\alpha:\mathfrak{z}\to RHom_{A}^{\mathcal{E}_n}(A,B)$. It follows that the map $\alpha$ is
equivalent to the map~\eqref{eq:factorthetaphi} hence to $\tilde{\theta_{\phi}}$. This gives the
uniqueness statement and the Proposition will follow once we proved the diagram depicted in Proposition~\ref{P:HHEn=z} is commutative. The latter is an immediate consequence of the universal property of the centralizers (and thus of $HH_{A}^{\mathcal{E}_n}(A,B)$) and of Theorem~\ref{T:EnAlgHoch}.(3).
\end{proof}

\begin{ex}[$E_n$-Koszul duality]\label{E:Bardual}
Assume  $B=k$ so that $f:A\to k$ is an augmentation.
Then by Proposition~\ref{P:HHEn=z} and~\cite[Example 6.1.4.14]{L-HA} and~\cite[Remark 7.13]{L-MP}
 there is an equivalence
of $E_n$-algebras $$ HH_{\mathcal{E}_n}(A,k) \cong RHom(Bar^{(n)}(A),k)$$ where
$Bar^{(n)}(A)$ is the $E_n$-coalgebra given by the iterated Bar construction on $A$, that is,
the (derived) \emph{Koszul dual} of $A$.
Thus Theorem~\ref{T:EnAlgHoch} gives an explicit description of the $E_n$-algebra structure
on the dual of $Bar^{(n)}(A)$. See Section~\ref{S:Barmain} for a more detailed description.
\end{ex}

\smallskip

Combining Proposition~\ref{P:HHEn=z} and Proposition~\ref{P:HH(A,B)=CH(A,B)}, we get
\begin{cor}\label{C:HH(AB)=z}
 let $f:A\to B$ be a map of $E_\infty$-algebras. Then the Hochschild cochains $CH^{S^n}(A,B)$
 over the $n$-sphere (endowed with its $E_n$-algebra structure given by Theorem~\ref{T:EdHoch})
is the centralizer $\mathfrak{z}(f)$ of $f$ viewed as a map of $E_n$-algebras (by restriction).
\end{cor}

\begin{rem}\label{R:derivedcompositionasusualcomposition}
Assume $f: A\to B$ and $g: B\to C$ are maps of CDGA's. Then by the above Corollary~\ref{C:HH(AB)=z} or Proposition~\ref{P:HH(A,B)=CH(A,B)}, there is a composition \begin{equation}\label{eq:defderivedcomposition} CH^{S^n}(A,B)\otimes CH^{S^n}(B,C) \stackrel{\circ}\longrightarrow CH^{S^n}(A,C)\end{equation} (which is a map of $E_n$-algebras) induced by the natural equivalence
$CH^{S^n}(A,B)\cong RHom^{left}_{CH_{S^{n-1}}(A)}(CH_{D^n}(A),B)$ and (derived) compositions of homomorphisms. In the setting of CDGA's, this composition can be represented in an easy way as follows. Let  $I_\bullet$  be the standard simplicial model of the interval (\cite{G, GTZ}); its  boundary $\partial I_\bullet^n$ is a simplicial model for $S^{n-1}$.
Then  the map~\eqref{eq:defderivedcomposition} is represented by the usual composition (of left dg-modules)
\begin{multline*} Hom_{CH_{\partial I_\bullet^n}(A)}^{left}\left(CH_{ I^n_\bullet}(A), CH_{ I^n_\bullet}(B)\right)\otimes Hom_{CH_{\partial I_\bullet^n}(B)}^{left}\left(CH_{ I_\bullet^n}(B), CH_{ I_\bullet^n}(C)\right)\\
 \stackrel{\circ}\longrightarrow  Hom_{CH_{\partial I_\bullet^n}(A)}^{left}\left(CH_{ I_\bullet^n}(A),
CH_{ I_\bullet^n}(C)\right) \end{multline*}
since $CH_{ I^n_\bullet}(A)$ is a (semi-)free $CH_{\partial I_\bullet^n}(A)$-algebra.
\end{rem}

\medskip

\subsection{The higher Deligne conjecture} \label{S:Deligne}
In this section we deal with (some of) the solutions of the higher Deligne conjecture.
That is we specialized the results of the previous sections~\ref{S:centralizers}
and~\ref{S:Edcochains} to the case  $A=B$ and $f=id$.

\smallskip

By Theorem~\ref{T:EnAlgHoch} above,  the composition
of morphisms of $A$-$E_n$-modules
\begin{equation}\label{eq:compEndoA}
RHom_{A}^{\mathcal{E}_n}(A,A)\otimes RHom_{A}^{\mathcal{E}_n}(A,A)
\stackrel{\circ}\longrightarrow  RHom_{A}^{\mathcal{E}_n}(A,A)
\end{equation} is a homomorphism
of $E_n$-algebras (with unit given by the identity map $id:A\to A$).
The composition of morphisms is further (homotopy) associative and unital (with unit $id$);
thus $RHom_{A}^{\mathcal{E}_n}(A,A)$ is actually an $E_1$-algebra
 in the $\infty$-category $E_n$-Alg.

By the $\infty$-category version of Dunn Theorem~\cite{Du, L-HA, L-VI} or see Theorem~\ref{T:Dunn}, there is an equivalence
of $(\infty,1)$-categories $E_1-Alg\big(E_n-Alg\big)\cong E_{n+1}-Alg$. Thus the
multiplication~\eqref{eq:compEndoA} lift the $E_n$-algebra structure of
$HH^\bullet_{\mathcal{E}_n}(A,A)\cong RHom_{A}^{\mathcal{E}_n}(A,A)$ to an $E_{n+1}$-algebra
structure.

In particular we just proved the first part of the following result, which
 has already been given by  Francis~\cite{F} (and Lurie~\cite{L-HA, L-VI}).
\begin{theorem}\label{T:Deligne}  \emph{(Higher Deligne Conjecture)}
\begin{enumerate}\item  Let $A$ be an $E_n$-algebra.
There is a natural $E_{n+1}$-algebra structure on $HH^\bullet_{\mathcal{E}_n}(A,A)$
with underlying $E_n$-algebra
structure  given by Theorem~\ref{T:EnAlgHoch}\footnote{where we take $B=A$ and $f=id$}.
\item Let now $A$ be an $E_\infty$-algebra.
Then there is a natural $E_{n+1}$-algebra structure on $CH^{S^n}(A,A)$ whose underlying
$E_n$-algebra structure is the one given by Theorem~\ref{T:EdHoch}. In particular, the underlying
$E_1$-algebra structure is given by the standard cup-product (see Corollary~\ref{C:cupHoch}
and Example~\ref{E:cupSphere}).
\item For $A$ an $E_\infty$-algebra, the two $E_{n+1}$-structures given by statements (1)  and (2)
are equivalent.
\end{enumerate}
\end{theorem}
\begin{proof}We have already proved the first claim. Note that the underlying $E_n$-algebra structure of an $E_{n+1}$-algebra is induced by the pushforward along the canonical projection $\R^{n} \times \R \to \R^n$, see Theorem~\ref{T:Dunn}.  By Proposition~\ref{P:coHH=coTCH}, $CH^{S^n}(A,A)$ also  inherits  a structure of $E_{n+1}$-algebra whose underlying $E_n$-algebra is the same as the one given by Theorem~\ref{T:EdHoch} thanks to Proposition~\ref{P:HH(A,B)=CH(A,B)}.
 This proves both claims (2) and (3).
\end{proof}

\begin{ex}\label{E:Deligneforn=1}
 In the case $n=1$, Theorem~\ref{T:Deligne} recovers the original Deligne conjecture asserting the existence of a natural $E_2$-algebra structure on the Hochschild cochains lifting the associative algebra structure induced by the cup-product.
It can be proved that this $E_2$-algebra structure induces the usual Gerstenhaber algebra structure (from~\cite{Ge}) on the Hochschild cohomology groups.
\end{ex}
\begin{rem}
Francis~\cite{F} has given a different solution to the higher Deligne conjecture. His solution is directly and explicitly related to the degree $n$ Lie algebra structure on $HH_{\mathcal{E}_n}(A,A)$. However, the underlying cup-product (\emph{i.e.} $E_1$-algebra structure) is more mysterious. This is in contrast to the solution given by Theorem~\ref{T:Deligne}. This latter solution is, by definition, the same as the one of Lurie~\cite{L-HA}. It would therefore be very interesting and useful to relate Francis' construction to ours. Note that the explicit knowledge of the cup-product is useful to us to relate this construction to the higher string topology operations, see~\S~\ref{S:Brane}.
\end{rem}
\subsection{Explicit computations via higher formality}
For the remaining of this section, we work over \emph{a characteristic zero field} and we let $A$ be a commutative differential graded algebra. In that case, we can use the higher formality Theorem~\cite{Toen-Formality, CaWi-Formality} mentioned first in~\cite{PTVV}  to compute more explicitly the $E_{n+1}$-algebra structure on the derived center $HH^\bullet_{\mathcal{E}_n}(A,A)$. 
We will start by identifying our $E_{n+1}$-algebra structure with the one obtained very recently by To\"en~\cite{Toen-Formality}.

\subsubsection{$P_n$-algebras and To\"en Brane operations}\label{SS:Pn}

In characteristic zero, the operad $\mathbb{E}_n$ is formal when $n\geq 2$ (\cite{LV}) so that one can choose\footnote{there are however many possible choices, for instance see~\cite{Tamarkin-GTactsE2}} an equivalence of $\infty$-operads $\mathbb{E}_n \to \mathbb{P}_{n}$ where $\mathbb{P}_{n}$ is the ($\infty$-operad associated to the) operad governing (homotopy) ${P}_n$-algebras (in particular $\mathbb{P}_n\cong H_\bullet(\mathbb{E}_n)$ for $n\geq 2$). 
More precisely, we have the differential graded operad  $\mathbb{P}_n$ which is the (cofibrant) minimal resolution of the  operad  $P_n\cong H_\bullet(\mathbb{P}_n)$ of (strict, differential graded) $P_n$-algebras. 
These two operads yields (using the standard nerve functor from operads to $\infty$-operads) equivalent $\infty$-operads (and dendroidal sets), which we still denote $\mathbb{P}_n$.  

By a $P_n$-algebra\footnote{which are sometimes called $n$-algebras~\cite{GeJo, G} or $e_n$-algebras~\cite{Tamarkin-GTactsE2, CaWi-Formality} in the literature} we mean a differential graded commutative unital algebra $(B,d,\cdot)$ equipped with a (homological) degree $n-1$ bracket which makes the iterated suspension $A[1-n]$ a differential graded Lie algebra. 
The bracket and product are further require to satisfy the graded Leibniz identity $[a\cdot b, c]=\pm a\cdot[b,c] +\pm b\cdot [a,c]$. In other words $P_2$-algebras are the same as (differential graded) Gerstenhaber algebras.

 In particular, if $B$ is a $P_{n+1}$-algebra, then $B[-n]$ is a differential graded Lie algebra. By the above choice of formality  maps $\mathbb{E}_n \to \mathbb{P}_{n}$ of operads, then for any $E_{n+1}$-algebra $H$, $H[-n]$ inherits an (homotopy) dg-Lie algebra structure as well. 

\smallskip

Slightly after a first draft of our paper was written, To\"en used the machinery of $\infty$-operad of configuration type to prove the following result. 
\begin{theorem}[To\"en~\cite{Toen-Formality}] \label{T:ToenFormality} Let $\mathfrak{X}$ be a derived stack and $n\geq 2$. There is a canonical equivalence of  (homotopy) dg-Lie algebras $$\mathbb{R}\Gamma\big(\mathfrak{X}, {\mathrm{Sym}}_{\mathcal{O}_{\mathfrak{X}}}\big(\mathbb{T}_{\mathfrak{X}}[n]\big)\big)[-n] \cong HH^\bullet_{\mathcal{E}_n}(\mathcal{O}_{\mathfrak{X}},\mathcal{O}_{\mathfrak{X}})[-n].$$
Here both sides are given a Lie algebra structure induced by a canonical $P_n$ and $E_n$-algebras structures constructed  in~\cite[{\S~5}]{Toen-Formality}.
\end{theorem}

We are only interested in theorem~\ref{T:ToenFormality} in the case when $\mathfrak{X}=\mathbb{R}\mathrm{Spec}(A)$ is affine, given by a \emph{(differential graded) commutative algebra $A$}. In that case, the tangent complex $\mathbb{T}_{\mathfrak{X}}$ is equivalent to $\mathbb{R}\mathrm{Der}(A,A)$ the right derived functor of derivations of $A$. 
 Hence, theorem~\ref{T:ToenFormality} provides a canonical equivalence of  \emph{(homotopy) dg-Lie algebras}:
\begin{equation}\label{eq:ToenFormalityAffine}
{\mathrm{Sym}}_{A}\big(\mathbb{R}\mathrm{Der}(A,A)[n]\big)[-n] \cong HH^\bullet_{\mathcal{E}_n}(\mathcal{O}_{\mathbb{R}\mathrm{Spec}(A)},\mathcal{O}_{\mathbb{R}\mathrm{Spec}(A)})[-n].
\end{equation}
In fact, To\"en deduced the above equivalence~\eqref{eq:ToenFormalityAffine} from \emph{an equivalence of objects in $E_1\text{-Alg}(E_n\text{-Alg})$} which is a consequence of  the main Theorem of \cite{Toen-Formality}. Note that the left hand side ${\mathrm{Sym}}_{A}\big(\mathbb{R}\mathrm{Der}(A,A)[n]\big)$ of~\eqref{eq:ToenFormalityAffine} is the  $\mathbb{P}_n$-branes cohomology of $\mathbb{R}\mathrm{Spec}(A)$ while the right hand side $HH^\bullet_{\mathcal{E}_n}(\mathcal{O}_{\mathbb{R}\mathrm{Spec}(A)},\mathcal{O}_{\mathbb{R}\mathrm{Spec}(A)})$ is its $\mathbb{E}_n$-branes cohomology.

\smallskip

It is essentially immediate from the definition\footnote{see~\cite[\S 5]{Toen-Formality} or the proof of Proposition~\ref{P:E1Pn} below} that, \emph{as a cochain complex}, $HH^\bullet_{\mathcal{E}_n}(\mathcal{O}_{\mathbb{R}\mathrm{Spec}(A)}, \mathcal{O}_{\mathbb{R}\mathrm{Spec}(A)})$ is equivalent to $HH^\bullet_{\mathcal{E}_n}(A,A)$, the $E_n$-center of $A$ (Definition~\ref{D:EnHoch} and Corollary~\ref{C:HH(AB)=z}) . 
We  prove below that the equivalence is actually an equivalence of $E_{n+1}$-algebras between the right hand side of~\eqref{eq:ToenFormalityAffine} and  our solution to Deligne conjecture from \S~\ref{S:maincentralizers}. 

To do so, for $n\geq 2$, we first choose equivalences of $\infty$-operad  $\mathcal{F}_n:\mathbb{E}_{n} \stackrel{\simeq}\longrightarrow \mathbb{P}_{n}$. 

 Denote $\mathbb{O}\circledcirc \mathbb{P}$  the tensor product of two $\infty$-operads  $\mathbb{O}$ and  $\mathbb{P}$.
It is an $\infty$-operad governing 
$\mathbb{O}$-algebras in $\mathbb{P}$-algebras. If one uses the model given by dendroidal sets for $\infty$-operads, the tensor product is represented by the (derived) tensor product of dendroidal sets (which forms a symmetric monoidal model category) of \cite{CiMo-Dendroidalasinftyoperad, CiMo-DendroidalasinftyoperadII, CiMo-Dendroidalasimplicialoperad}. 

Recall also, that, by Dunn Theorem~\ref{T:Dunn} (see \cite{L-HA}), we  have an equivalence $D_{n+1}:\mathbb{E}_{1}\circledcirc \mathbb{E}_n \stackrel{\simeq}\longrightarrow \mathbb{E}_{n+1} $ of $\infty$-operads.  We sum up these two  facts in the
\begin{prop}
\label{P:TamarkinE1Pn} Let $n\geq 2$. There are equivalences of $\infty$-operads
$$ \xymatrix{  \mathbb{E}_{1}\circledcirc \mathbb{P}_n    &\mathbb{E}_{1}\circledcirc \mathbb{E}_n    \ar[l]^{\simeq}_{id\circledcirc \mathcal{F}_n} \ar[r]_{\simeq}^{D_{n+1}} & \mathbb{E}_{n+1}   \ar[r]_{\simeq}^{\mathcal{F}_{n+1}} & \mathbb{P}_{n+1}}$$
Here $\mathbb{E}_{1}\circledcirc \mathbb{P}_n$ (resp.  $\mathbb{E}_{1}\circledcirc \mathbb{E}_n$) 
are the $\infty$-operad governing $E_1$-algebras in the (symmetric monoidal $\infty$-)category of $P_{n}$-algebras (resp. $E_n$-algebras).
\end{prop}
In particular, we thus get  equivalences $\mathbb{E}_{1}\circledcirc \mathbb{P}_n \cong \mathbb{P}_{n+1}$  (for $n\geq 2$) fitting into a commutative diagram of equivalences of $\infty$-operads:
\begin{equation}\label{eq:En=Pndiagram}
\xymatrix{ \mathbb{E}_{1}\circledcirc \cdots\circledcirc \mathbb{E}_{1}\circledcirc \mathbb{E}_{2}   \ar[d]_{\simeq} \ar[r]^{\qquad \simeq} 
& \cdots \ar[d]_{\simeq} \ar[r]^{\hspace{-2pc}\simeq} &   \mathbb{E}_{1}\circledcirc \mathbb{E}_{1}\circledcirc \mathbb{E}_{n-1}\ar[d]_{\simeq} \ar[r]^{\quad\simeq} &     \mathbb{E}_{1}\circledcirc \mathbb{E}_n\ar[r]^{\simeq} \ar[d]_{\simeq} &  \mathbb{E}_{n+1} \ar[d]_{\simeq}\\
\mathbb{E}_{1}\circledcirc  \cdots\circledcirc \mathbb{E}_{1}\circledcirc \mathbb{P}_{2}  \ar@{.>}[r]^{\qquad\simeq} & \ldots \ar@{.>}[r]^{\hspace{-2pc}\simeq} & \mathbb{E}_{1}\circledcirc  \mathbb{E}_{1}\circledcirc \mathbb{P}_{n-1} \ar@{.>}[r]^{\quad \simeq} & \mathbb{E}_{1}\circledcirc \mathbb{P}_{n} \ar@{.>}[r]^{\simeq}&  \mathbb{P}_{n+1}  }
\end{equation}
where the vertical pointing down arrows are induced by the formality maps $\mathcal{F_i}$ and the horizontal upper arrows by iteration of Dunn Theorem.

We can identify the right hand side of~\eqref{eq:ToenFormalityAffine} with the structure given by the Deligne conjecture (Theorem~\ref{T:Deligne}):
\begin{prop}\label{P:E1Pn} One has a canonical equivalence of $\mathbb{E}_{1}\circledcirc \mathbb{E}_n $-algebras $$HH^\bullet_{\mathcal{E}_n}(\mathcal{O}_{\mathbb{R}\mathrm{Spec}(A)},\mathcal{O}_{\mathbb{R}\mathrm{Spec}(A)})\cong  CH^{S^n}(A,A)$$ where the right hand side is given the structure given by Theorem~\ref{T:Deligne} and the left hand side is given the structure of~\cite[{\S~5}]{Toen-Formality}.
\end{prop}
\begin{proof}
To\"en has proved in~\cite{Toen-Formality} (in particular Corollary 5.1 in \emph{loc. cit.}) that the  $\mathbb{E}_n$-branes cohomology $HH^\bullet_{\mathcal{E}_n}(\mathcal{O}_{\mathbb{R}\mathrm{Spec}(A)},\mathcal{O}_{\mathbb{R}\mathrm{Spec}(A)})$  of $\mathbb{R}\mathrm{Spec}(A)$ is an $E_1\circledcirc E_n$-algebra. It  is given, by definition, by  $$RHom_{D\big(\mathcal{L}^{n-1}(\mathbb{R}\mathrm{Spec}(A))\big)}\Big(\mathcal{O}_{\mathbb{R}\mathrm{Spec}(A)}, \mathcal{O}_{\mathbb{R}\mathrm{Spec}(A)}\Big)$$ where $\mathcal{L}^{n-1}(\mathbb{R}\mathrm{Spec}(A))$ is the  iterated (derived) loop stack $\mathbb{R}Map(S^{k-1}, \mathbb{R}\mathrm{Spec}(A) )$. 
By Corollary~\ref{C:mappingstack}, we get that the derived category $D\big(\mathcal{L}^{n-1}_{f}(\mathbb{R}\mathrm{Spec}(A))\big)$ is equivalent to $CH_{S^{k-1}}(A)\text{-LMod}$. Hence, we get an equivalence of $E_1$-algebras (for the structure given by composition of endomorphisms): 
\begin{equation}\label{eq:Toen=CH} RHom_{D\big(\mathcal{L}^{n-1}(\mathbb{R}\mathrm{Spec}(A))\big)}\Big(\mathcal{O}_{\mathbb{R}\mathrm{Spec}(A)}, \mathcal{O}_{\mathbb{R}\mathrm{Spec}(A)}\Big)\cong RHom^{left}_{CH_{S^{n-1}}(A)}(A,A).\end{equation} 
To\"en has proved that the left hand side has an $E_n$-algebra structure given by a natural map of $\infty$-operads $\big(\mathcal{C}_n(r)\big)_{r}\to End \Big( {D\big(\mathcal{L}^{n-1}(\mathbb{R}\mathrm{Spec}(A))\big)}\Big)$. 
Under the above equivalence,
this action of the little cubes is seen to be transfered to the map~\eqref{eq:rhovee} (for any family of little cubes $U_i$ inside the unit cube $D$) defined in the proof of Proposition~\ref{P:HH(A,B)=CH(A,B)}.
The proof of this Proposition also shows that this  $E_n$-algebra structure on $RHom^{left}_{CH_{S^{n-1}}(A)}(A,A)$ is equivalent to the one  on $CH^{S^n}(A,A)$ given by Theorem~\ref{T:EdHoch} in a natural way. It follows that the equivalence~\eqref{eq:Toen=CH} is an equivalence of $E_1\circledcirc E_n$-algebras.
\end{proof}

\begin{cor}\label{L:ToenHH=CH} For any choice of formality equivalence as in Proposition~\ref{P:E1Pn}, one has canonical equivalences of $E_{n+1}$-algebras (and homotopy $P_{n+1}$-algebras as well)
 $$HH^\bullet_{\mathcal{E}_n}\!(\mathcal{O}_{\mathbb{R}\mathrm{Spec}(A)},\mathcal{O}_{\mathbb{R}\mathrm{Spec}(A)})\cong \mathrm{Sym}_{A}\!\big(\mathbb{R}\mathrm{Der}(A,A)[n]\big)\cong    CH^{S^n}\!(A,A) \cong HH^\bullet_{\mathcal{E}_n}(A,A)$$ where the two right hand sides are given the structure given by Theorem~\ref{T:Deligne} and the left hand sides are given the structure of~\cite[{\S~5}]{Toen-Formality}.
\end{cor}
\begin{proof}
The choice of formality provides an equivalence $(id \circledcirc \mathcal{F}_n)^*:\mathbb{E}_{1}\circledcirc \mathbb{P}_n\text{-Alg} \stackrel{\simeq}\longrightarrow \mathbb{E}_{1}\circledcirc \mathbb{E}_n \text{-Alg}$. Note that    that the map of ($\infty$-)operad $\mathbb{E}_n\stackrel{\mathcal{F}_n}\to \mathbb{P}_n \to  {Comm}$ is the canonical one. 
Further there is a canonical equivalence $\mathbb{E}_1\otimes Comm \stackrel{\simeq}\longrightarrow Comm$ and the following diagram of $\infty$-operads
$$\xymatrix{ \mathbb{E}_{1}\circledcirc \mathbb{P}_n  \ar[rd]   &\mathbb{E}_{1}\circledcirc \mathbb{E}_n   \ar[d]  \ar[l]^{\simeq}_{id\circledcirc \mathcal{F}_n} \ar[r]_{\simeq}^{D_{n+1}} & 
\mathbb{E}_{n+1} \ar[d]    \ar[r]_{\simeq}^{\mathcal{F}_{n+1}} & \mathbb{P}_{n+1}\ar[ld]  \\  & \mathbb{E}_1\otimes Comm \ar[r]^{\simeq}&  Comm &} $$
is commutative. 

Since $A$ is a CDGA seen as both a $P_n$-algebra with trivial bracket and an $E_n$-algebra through the above maps $\mathbb{P}_n\to Comm$  and $\mathbb{E}_n\to Comm$, it follows that the space of $P_n$-branes\footnote{which is $\mathrm{Sym}_{A}\!\big(\mathbb{R}\mathrm{Der}(A,A)[n]\big)$} and the space of $\mathbb{E}_n$-branes\footnote{which is $HH^\bullet_{\mathcal{E}_n}\!(\mathcal{O}_{\mathbb{R}\mathrm{Spec}(A)},\mathcal{O}_{\mathbb{R}\mathrm{Spec}(A)})$} of $A$ (in the sense of  \cite{Toen-Formality}) are equivalent as $\mathbb{E}_{1}\circledcirc \mathbb{E}_n$-algebras. 

Now, in view of   Proposition~\ref{P:TamarkinE1Pn} and Theorem~\ref{T:Deligne}, the result is a corollary of Proposition~\ref{P:E1Pn}.
\end{proof}

\subsubsection{Higher formality and Tamarkin homotopy $P_{n+1}$-structure}

By Corollary~\ref{L:ToenHH=CH}, $\mathrm{Sym}_{A}\big(\mathbb{R}\mathrm{Der}(A,A)[n]\big)$ inherits an homotopy $P_{n+1}$-algebra structure induced by its interpration (due to \cite{Toen-Formality}) as $\mathbb{P}_n$-branes cohomology. 

There is also a canonical (strict) $P_{n+1}$-algebra structure on $\mathrm{Sym}_{A}\big(\mathbb{R}\mathrm{Der}(A,A)[n]\big)$. Indeed, there is a Lie algebra structure on $\mathrm{Sym}_{A}\big(\mathbb{R}\mathrm{Der}(A,A)[n]\big)[-n]$
  given by the Schouten bracket. 
More precisely, $\mathbb{R}\mathrm{Der}(A,A)$ has a canonical  differential graded Lie algebra structure\footnote{given by the standard Lie algebra structure on the derivations $\mathrm{Der}(P_A,P_A)$ where $P_A\to A$ is a resolution of $A$ by semi-free CDGAs}  such that the canonical map $\mathrm{Der}(A,A)\to \mathbb{R}\mathrm{Der}(A,A)$ is a map of (dg-)Lie algebras.  

Then, $\mathrm{Sym}_{A}\big(\mathbb{R}\mathrm{Der}(A,A)[n]\big)$ is made into a $P_{n+1}$-algebra whose underlying CDGA structure is given by the (graded) symmetric algebra construction on the (dg-) $A$-module $ \mathbb{R}\mathrm{Der}(A,A)[n]$. There is a unique extension of the Lie bracket 
on $(\mathbb{R}\mathrm{Der}(A,A)[n])[-n] =\mathbb{R}\mathrm{Der}(A,A)[n]$ satisfying the Leibniz rule, which defines the $P_{n+1}$-algebra structure. 
This (strict)  $P_{n+1}$-structure induces canonically a $\mathbb{P}_{n+1}$-structure (\emph{i.e.} an homotopy $P_{n+1}$-structure) on  $\mathrm{Sym}_{A}\big(\mathbb{R}\mathrm{Der}(A,A)[n]\big)$ and thus,   given any  choice of a formality map,   an $E_{n+1}$-algebra structure as well.

\smallskip

Associated to any $P_n$-algebra $V$, one can define its cohomology complex 
$HH^\bullet_{\mathbb{P}_n}(V,V)$ which by a result of Tamarkin~\cite{Ta-Defofdalgebra} has a canonical homotopy $P_{n+1}$-algebra structure.

Calaque and Willwacher have  recently proved the following higher formality relating the  $P_{n+1}$-structure of ${\mathrm{Sym}}_{A}\big(\mathbb{R}\mathrm{Der}(A,A)[n]\big) $ and the $P_{n+1}$-structure of $HH^\bullet_{\mathbb{P}_n}(A,A)$ for a CDGA $A$, seen as a $P_n$-algebra with trivial bracket.
\begin{theorem}[Calaque Willwacher~\cite{CaWi-Formality}] \label{T:CaWi-Formality}  Let $A$ be a differential graded commutative algebra over a characteristic zero field.  There is a canonical equivalence of  $\mathbb{P}_{n+1}$-algebras: 
\begin{equation}\label{eq:CaWi-Formality}
{\mathrm{Sym}}_{A}\big(\mathbb{R}\mathrm{Der}(A,A)[n]\big) \cong HH^\bullet_{\mathbb{P}_n}(A,A)
\end{equation}
where the right hand side is endowed with the (homotopy) $P_{n+1}$-structure constructed by Tamarkin in~\cite{Ta-Defofdalgebra}.
\end{theorem}
The left hand side of~\eqref{eq:CaWi-Formality} has a very explicit $P_{n+1}$-structure. The same complex has another homotopy $P_{n+1}$-structure given by Corollary~\ref{L:ToenHH=CH}, that is, by our solution to the higher Deligne conjecture for $n\geq 2$.   In order to prove that these two structures are actually the same, now,   we only need to check that Tamarkin  structure on  $HH^\bullet_{\mathbb{P}_n}(A,A)$ is equivalent to the one given by the center (and thus Toen's one as well).

\smallskip

The cochain complex $HH^\bullet_{\mathbb{P}_n}(A,A)$ is defined in~\cite[\S 2]{Ta-Defofdalgebra} as follows.
We denote $\mathbb{P}_n^{\vee}(A)$ the (coproartinian) cofree $P_n$-coalgebra on $A[-n]$ equipped with the differential $\partial_{\mathbb{P}_n^{\vee}(A)}$ corresponding to the (homotopy) $P_n$-algebra structure of $A$ (this is a coderivation of $\mathbb{P}_n^{\vee}(A)$). 
Let $\mathrm{coDer}(\mathbb{P}_n^{\vee}(A), \mathbb{P}_n^{\vee}(A))$ be the vector space
 of coderivations of $\mathbb{P}_n^{\vee}(A)$.
 Then the cochain complex  $HH^\bullet_{\mathbb{P}_n}(A,A)$ is  $\mathrm{coDer}(\mathbb{P}_n^{\vee}(A), \mathbb{P}_n^{\vee}(A))$ equipped with the differential $[\partial_{\mathbb{P}_n^{\vee}(A)},-]$ obtained as the bracket of a coderivation with $\partial_{\mathbb{P}_n^{\vee}(A)}$. 
\begin{lem}\label{L:Tam=Center} Let $A$ be a differential graded commutative algebra over a characteristic zero field.
 There is an natural equivalence of $P_{n+1}$-algebras 
 $$HH^\bullet_{\mathbb{P}_n}(A,A) \cong CH^{S^n}(A,A)$$ where the right hand side is endowed with the structure given by Theorem~\ref{T:Deligne} and the left hand side is endowed with the one constructed by Tamarkin in~\cite{Ta-Defofdalgebra}.
\end{lem}
\begin{proof}
Since $\mathbb{P}_n^{\vee}(A)$ is a coproartinian cofree coalgebra, there is an isomorphism $\mathrm{coDer}(\mathbb{P}_n^{\vee}(A), \mathbb{P}_n^{\vee}(A))\cong Hom(\mathbb{P}_n^{\vee}(A),A[-n])$ from which one deduced an isomorphism of cochain complexes $HH^\bullet_{\mathbb{P}_n}(A,A) \cong Hom(\mathbb{P}_n^{\vee}(A),A)[-n]$ where the right hand side is endowed with the inner differential of $\mathbb{P}_n^{\vee}(A)$ twisted by the canonical map $\mathbb{P}_n^{\vee}(A)\to A[-n]$, see \cite[\S 4]{Ta-Defofdalgebra}.  Further, in \emph{loc. cit.}, Tamarkin 
  proved that $Hom(\mathbb{P}_n^{\vee}(A),A)[-n]$ is an homotopy $P_n$-algebra so that  $\mathbb{P}_n^{\vee}\Big(Hom(\mathbb{P}_n^{\vee}(A),A)[-n]\Big)$ is a $n$-bialgebra in the sense of \cite[\S 4]{Ta-Defofdalgebra}, that is an $E_1$-algebra in the category of homotopy $P_n$-algebras.   
  The latter chain complex is also denoted $\underline{Hom}^{Id}(A,A)$ in \cite[\S 3]{Ta-Defofdalgebra}.
  
 More generally, Tamarkin proved that, associated to any $P_n$-algebra morphism 
 $\phi: A\to B$, one obtains a similar way an homotopy $P_n$-algebra structure on  $Hom(\mathbb{P}_n^{\vee}(A),B)[-n]$, with underlying differential given by the inner differential of $\mathbb{P}_n^{\vee}(A)$ and $B$ and twisted by $\phi$. This structure is equivalent to a differential on the cofree coalgebra    $\mathbb{P}_n^{\vee}\Big(Hom(\mathbb{P}_n^{\vee}(A),B)[-n]\Big)$ which corepresents 
 the canonical functor $F^\phi_{A,B}: d-coart^0\to Sets$ of moduli problems for $P_n$-algebras, see~\cite[\S 2 and 3]{Ta-Defofdalgebra}. It is also denoted $\underline{Hom}^{\phi}(A,B)$ in \emph{loc. cit.} and its universal property induces  an associative map of homotopy $P_n$-algebras  $$\underline{Hom}^{\phi}(A,B) \otimes \underline{Hom}^{\psi}(B,C) \longrightarrow \underline{Hom}^{\psi \circ \phi}(A,C)$$
 which precisely gives the aforementionned $n$-bialgebra structure of $\underline{Hom}^{Id}(A,A)$. 
 
  The (homotopy) $P_{n+1}$-structure on  $HH^\bullet_{\mathbb{P}_n}(A,A)$ is canonically induced by the $n$-bialgebra structure on $\underline{Hom}^{Id}(A,A)$, see \cite[Corollary 4.5 and \S 5]{Ta-Defofdalgebra}, \cite{CaWi-Formality} and the fact it corepresents the functor $F^{id}_{A,A}$. 
  
 Hence in order to prove the lemma we now need to prove that, for any map $f: A\to B$ between CDGAs, $CH^{S^n}(A,B)$ is isomorphic to  $\underline{Hom}^{\phi}(A,B)$ as an homotopy $P_n$-algebra. Since both functors are functorial with respect to CDGA maps, we can further assume that $A$ and $B$ are free graded commutative as algebras. That is $A=(\mathrm{Sym}(V), d)$ and $B=(\mathrm{Sym}(W),b)$.

 By definition, as a coalgebra $\mathbb{P}_n^{\vee}(A) = \mathrm{Sym}\big( \mathrm{CoLie}(A[-1])[1-n]\big)$ where $\mathrm{CoLie}$ is the free Lie coalgebra functor and $\mathrm{Sym}$ is endowed with the cofree coproartinian cocommutative cobracket. Note that $\mathrm{CoLie}(A[-1])$ is simply the underlying vector space of the Harrison chain complex of $A$ (see~\cite{Ta-formality, GiHa, L}). 
 Since $A$ is seen as a $P_n$-algebra with trivial bracket, the differential $\partial_{\mathbb{P}_n^{\vee}(A)}$ boils down to the usual Hochschild/Harisson complex differential. Hence, one has an isomorphism of complexes
 $$\underline{Hom}^{\phi}(A,B)= Hom(\mathbb{P}_n^{\vee}(A),B)[-n] \cong  Hom_{A}\Big(\mathrm{Sym}_{A}\big(\mathrm{Harr}_\ast(A,A)[1-n]\big),B\Big)[-n]$$ where the right hand side is   endowed with the tensor product of Harrison differentials on $A$ and consists of $A$-linear maps; here $B$ is seen as an $A$-module through the map $\phi:A\to B$ and  the $A$-module structure on $\mathrm{Harr}_\ast(A,A)$  is given by the tensor product $A \otimes \mathrm{CoLie}(A[-1])$. The proof is the same as the one in the case $\phi=Id$ in~\cite{Ta-formality, GiHa}.
 
 We recall that the Harrison chain complex $\mathrm{Harr}_\ast(A,A)$ is a sub-complex of the Hochschild chain complex of $A$: it is precisely the weight $1$ part of the Hodge decomposition of the Hochschild complex of the differential graded commutative algebra $A$, see~\cite{L, GiHa}. Its homology is equal to the Andr\'e-Quillen homology since we are in characteristic $0$.
Hence, since $A=(\mathrm{Sym}(V), d)$, by the Hochschild-Kostant-Rosenberg Theorem,  the Harrison chain complex is quasi-isomorphic to $\Omega^1(A)= \mathrm{Sym}(V)\otimes  \mathrm{Sym}(V[-1])$ where the differential is induced by the one on $S(V)$ and $d(v[-1])= -s(d(v))$ where $s$ is the unique derivation extending $v\mapsto v[-1]$ for $v\in V$ (cf. \cite[\S~5]{L}). 

It follows that we have a quasi-isomorphism of $P_n$-algebras 
$$ Hom_{A}\big(\mathrm{Sym}(V\oplus V[-n]), B\big)\stackrel{\simeq}\longrightarrow Hom(\mathbb{P}_n^{\vee}(A),B)[-n]=\underline{Hom}^{\phi}(A,B)$$ given by the convolution product on the left hand side $ Hom_{A}\big(\mathrm{Sym}(V\oplus V[-n]), B\big)\cong Hom(\mathrm{Sym}(V[-n]),B)$ where $\mathrm{Sym}(V[-n])$ is seen as a $P_n$-coalgebra with trivial cobracket and its cofree cocommutative structure.
 Now the result follows from Lemma~\ref{P:HKRrelative}.
\end{proof}

Combining the previous statements, we get:
\begin{cor}\label{C:HigherFormality}
Let $A$ be a differential graded commutative algebra over a characteristic zero field. For  $n\geq i \geq 2$,  choose  formality equivalences of $\infty$-operads  $\mathcal{F}_i:\mathbb{E}_{i} \stackrel{\simeq}\to \mathbb{P}_{i}$. There is a canonical equivalence of $E_{n+1}$-algebras (and thus of $P_{n+1}$-algebras):
\begin{equation}\label{eq:HigherFormality}
{\mathrm{Sym}}_{A}\big(\mathbb{R}\mathrm{Der}(A,A)[n]\big) \cong  HH^\bullet_{\mathcal{E}_n}(A,A)
\end{equation}
 where the right hand side is induced by the $E_{n+1}$-algebra structure given by the Deligne conjecture (Theorem~\ref{T:Deligne}) and the left hand side is given the Schouten structure, with differential induced by the one in $A$.
\end{cor}
\begin{proof}
By Theorem \ref{T:CaWi-Formality} we are left to prove that Tamarkin  $P_{n+1}$-structure on $HH_{\mathbb{P}_n}^\bullet (A,A)$ is equivalent to the (homotopy) $P_{n+1}$-structure on $HH_{\mathcal{E}_n}^\bullet (A,A)\cong CH^{S^n}(A,A)$  provided by Theorem~\ref{T:ToenFormality}. This is precisely the content of Lemma~\ref{L:Tam=Center}.
\end{proof}
\begin{rem}Corollary~\ref{C:HigherFormality} and Corollary~\ref{L:ToenHH=CH} implies in particular that the $P_n$-cohomology $HH_{\mathbb{P}_n}^\bullet (A,A)$  of Tamarkin is equivalent as an $E_1\circledcirc P_n$-algebra to the $\mathbb{P}_n$-branes cohomology of To\"en. 
\end{rem}

\subsubsection{Explicit computations using higher formality} \label{SS:HigherFormalityforFreeAlg}
Now we assume $A= (S(V),d)$ is a \emph{Sullivan algebra}, that is, as an algebra, it is the free graded commutative algebra on a graded vector space $V$ and it is  also equipped with a differential $d$. In that case the canonical map $\mathrm{Der}(A,A)\to \mathbb{R}\mathrm{Der}(A,A)$ is an equivalence of (dg-) Lie algebras. Hence, by Corollary~\ref{C:HigherFormality} we have:
\begin{cor}\label{C:HigherFormalityforFreeAlg} Let $A= (S(V),d)$ be a Sullivan algebra. Under the assumptions of Corollary~\ref{C:HigherFormality}, we have a canonical equivalence of $E_{n+1}$-algebras (and thus $P_{n+1}$-algebras as well):
\begin{equation}\label{eq:HigherFormality2}
{\mathrm{Sym}}_{A}\big(\mathrm{Der}(A,A)[n]\big) \cong HH^\bullet_{\mathcal{E}_n}(A,A).
\end{equation}
Here the right hand side is has the $E_{n+1}$-algebra structure given by the Deligne conjecture (Theorem~\ref{T:Deligne}) and the left hand side is endowed with the structure corresponding to the Schouten algebra structure.
\end{cor} 
The main interest of Corollary~\ref{C:HigherFormalityforFreeAlg} for us, is that the left hand side has a totally explicit and elementary strict $P_{n+1}$-algebra structure. 
It thus gives a \emph{large class of examples of explicit computations} of the $E_{n+1}$-structure of centers of commutative algebras (viewed as $E_n$-algebras).

In the next three examples, we calculate the $P_{n+1}$ structures of the left hand side of equation \eqref{eq:HigherFormality2} for the cases of the Sullivan models of an odd sphere, an even sphere, and for the complex projective space.

\begin{ex}[$P_{n+1}$-structure \eqref{eq:HigherFormality2} for the odd sphere $S^{2k+1}$]
\label{EXA:Pn+1-on-odd-sphere}
We compute the $P_{n+1}$-structure of \eqref{eq:HigherFormality2} for the Sullivan algebra $A$ of the $(2k+1)$-sphere. In this case, we consider the Sullivan algebra $A=(S(V),d)$ given by the free algebra generated by $x$ in degree $|x|=2k+1$ with trivial differential $d=0$. Since $x$ is in odd degree, we have that $x^2=0$, so that $A=span\{1,x\}$ with the trivial algebra structure. Now note, that a (graded) derivation of $A$ is uniquely determined by its value on the generator $x$, and that any derivation maps $1$ to $0$. We have essentially the two derivations of $A$, $\alpha$ and $\beta$, given by
\[
\alpha(1)=0, \quad \alpha(x)=1, \hspace{1cm} \text{and} \hspace{1cm}\beta(1)=0, \quad \beta(x)=x.
\]
Note that $\beta=x.\alpha$, displaying $\Der(A,A)$ as a module over $A$. Furthermore, the Lie-bracket is calculated as the commutator,
\begin{equation}\label{EQU:Sym-bracket-of-odd-sphere}
[\alpha,\beta]=\alpha, \quad\text{ and }\quad [\alpha,\alpha]=[\beta,\beta]=0.
\end{equation}
This induces the bracket of $\Der(A,A)[n]=span\{\alpha,\beta\}$ after a shift by $n$, where $\alpha$ and $\beta$ now have degrees $|\alpha|=n-(2k+1)$ and $|\beta|=n$. Using this, we next calculate the $P_{n+1}$-structure on
\begin{eqnarray*}
\Sym_A(\Der(A,A)[n])&=&A\oplus \Der(A,A)[n]\\
&&\oplus \big(\Der(A,A)[n]\odot_A \Der(A,A)[n]\big)\oplus \dots
\end{eqnarray*}
Here we denote the algebra structure on $\Sym_A(\Der(A,A)[n])$ by ``$\odot$'' or ``$\odot_A$'' to indicate linearity over $A$. Since $\beta=x.\alpha$, note that any element of $\Sym_A(\Der(A,A)[n])$ is a sum of elements of the form $a.\alpha^{\odot p}=a. \alpha\odot_A\dots\odot_A \alpha$ for some $a\in A$ and $\alpha$ is as above. The differential on $\Sym_A(\Der(A,A)[n])$ is zero, since $d=0$. Recall the usual Poisson relation and anti-symmetry for the bracket in a $P_{n+1}$-algebra, \emph{e.g.} from \cite[page 220]{SW},
\begin{eqnarray}
\label{EQU:Poisson-rel-1}
\, [f\odot g, h] &=& f\odot [g, h]+(-1)^{|g|(|h|+n)}[f, h] \odot g \\ 
\label{EQU:Poisson-rel-2}
\, [f,g\odot h] &=& [f,g]\odot h+(-1)^{|g|(|f|+n)}g\odot [f,h], \\ 
\label{EQU:Anti-symmetry-for-n-bracket}
\, [f,g] &=& -(-1)^{(|f|+n)(|g|+n)} [g,f],
\end{eqnarray}
for $f,g,h \in \Sym_A(\Der(A,A)[n])$, as well as the Schouten identities,
\begin{eqnarray}
\label{EQU:Schouten-1}
 \, [\rho,a]&=&\rho(a)\hspace{1.3cm} \text{ for } \rho\in \Der(A,A)[n], \text{ and } a \in A, \\
\label{EQU:Schouten-0}
 \, [a,b]&=&0 \hspace{1.8cm} \text{ for } a,b \in A,
 \end{eqnarray}
which are used in defining the bracket on $\Sym_A(\Der(A,A)[n])$ together with \eqref{EQU:Sym-bracket-of-odd-sphere}. (Note that this gives indeed a well-defined bracket on $\Sym_A(\Der(A,A)[n])$ due to the commutator of derivations giving the consistency relation $[\rho,a.\lambda]=\rho\circ (a\lambda)-(-1)^{|\rho|(|a|+|\lambda|)}(a\lambda)\circ \rho=\rho(a)\lambda+(-1)^{|\rho|\cdot |a|}a[\rho,\lambda]$ for $a\in A$ and $\rho,\lambda\in\Der(A,A)$ before shifting by $n$.)

In the case when $n$ is odd, the degree $|\alpha|=n-(2k+1)$ is even, so that for any $a,b\in A$, and $p,q\in \N_0$,
\begin{eqnarray*}
[a.\alpha^{\odot p},b.\alpha^{\odot q}]&=& a[\alpha^{\odot p},b\alpha^{\odot q}]+[a,b\alpha^{\odot q}]\odot \alpha^{\odot p}  \\
&=& a[\alpha^{\odot p},b]\odot \alpha^{\odot q} +(-1)^{|b|} ab\underbrace{[\alpha^{\odot p},\alpha^{\odot q}}_{=0}] \\
&& +\underbrace{[a,b]}_{=0}\alpha^{\odot p+q}+(-1)^{|b|(|a|+n)}b\underbrace{[a,\alpha^{\odot q}]}_{=-(-1)^{|a|+n}[\alpha^{\odot q},a]}\odot \alpha^{\odot p}\\
&=& \big(p\cdot a\cdot \alpha(b)-(-1)^{(|b|+1)(|a|+1)}q\cdot \alpha(a)\cdot b\big)\alpha^{\odot (p+q-1)}
\end{eqnarray*}
Thus, we obtain the following brackets for $\alpha^{\odot p}$ and $x\alpha^{\odot p}$,
\begin{eqnarray*}
\, [\alpha^{\odot p},\alpha^{\odot q}]&=&0, \\
\, [\alpha^{\odot p},x\alpha^{\odot q}]&=&-[x\alpha^{\odot q},\alpha^{\odot p}]=p\cdot \alpha^{\odot(p+q-1)}, \\
\, [x\alpha^{\odot p},x\alpha^{\odot q}]&=&(p-q)\cdot x\alpha^{\odot (p+q-1)}.
\end{eqnarray*}

In the case where $n$ is even, the degree of $|\alpha|=n-(2k+1)$ is odd, so that $\alpha\odot \alpha=0$ in $\Sym_A(\Der(A,A)[n])$. 
Thus, $\Sym_A(\Der(A,A)[n])=A\oplus \Der(A,A)[n]$ with $n$-bracket given by \eqref{EQU:Schouten-0}, \eqref{EQU:Schouten-1}, and \eqref{EQU:Sym-bracket-of-odd-sphere},
\[ [x,x]=[\alpha,\alpha]=[x\alpha,x\alpha]=0, [\alpha,x]=1, [\alpha,x\alpha]=\alpha, [x\alpha,x]=x\alpha(x)=x. \]
\end{ex}

We note that the above example is consistent with the calculation of the sphere product, see Remark \ref{REM:odd-sphere-consistency} below. We next consider the Sullivan algebra of the even sphere.

\begin{ex}[$P_{n+1}$-structure \eqref{eq:HigherFormality2} for the even sphere $S^{2k}$]\label{EXA:even-sphere-P(n+1)-structure}
For $A=(S(V),d)$ the Sullivan algebra of the even $2k$-sphere, we calculate the left hand side of \eqref{eq:HigherFormality2}. More precisely, let $A$ be the free algebra generated by $x$ and $y$, where $|x|=2k$ and $|y|=4k-1$. Since $y$ is an odd element, it is $y^2=0$. The differential $d$ is given by $d(x)=0$ and $d(y)=x^2$. Any (graded) derivation of $A$ is determined by its action on the generators $x$ and $y$. For $\ell=0,1,2,\dots$, we can define derivations $\alpha_\ell, \beta_\ell, \gamma_\ell, \delta_\ell:A\to A$ of $A$ whose actions on the generators are as follows.
\begin{align}
\label{DERIV:a} & \alpha_\ell(x)=x^\ell, && \alpha_\ell(y)=0, && \\
\label{DERIV:b} & \beta_\ell(x)=x^\ell y, && \beta_\ell(y)=0,&& \\
\label{DERIV:c} & \gamma_\ell(x)=0, && \gamma_\ell(y)=x^\ell,&& \\
\label{DERIV:d} & \delta_\ell(x)=0, && \delta_\ell(y)=x^\ell y.&&
\end{align}
The degrees of these derivations $\alpha_\ell, \beta_\ell, \gamma_\ell,\delta_\ell \in \Der(A,A)[n]$ (after the shift by $n$) are
\begin{align*}
&|\alpha_\ell |=2k(\ell-1)+n, \\
& |\beta_\ell|=2k(\ell-1)+(4k-1)+n=2k(\ell+1)+n-1, \\
&|\gamma_\ell |=2k\ell-(4k-1)+n=2k(\ell-2)+n+1, \\
& |\delta_\ell|= 2k\ell+n.
\end{align*}
Note, that any derivation can be written as a linear combination of the derivations $\alpha_\ell, \beta_\ell, \gamma_\ell$, and $\delta_\ell$. (In particular, $d=\gamma_2$.) 
Furthermore, 
\begin{equation}\label{EQU:module-even-sphere}
 x.\alpha_\ell=\alpha_{\ell+1},\quad\quad
 y.\alpha_\ell=\beta_\ell, \quad\quad
 x.\gamma_\ell=\gamma_{\ell+1}, \quad\quad
 y.\gamma_\ell=\delta_\ell, 
\end{equation}
 showing that $\Der(A,A)[n]$ is freely generated by $\alpha_0$ and $\gamma_0$ as an $A$-module. The commutator in $\Der(A,A)[n]$ for $\alpha_0$ and $\gamma_0$ is easily verified to vanish,

\begin{equation}\label{EQU:even-sphere-gen-bracket}
 [\alpha_0,\alpha_0]=[\alpha_0,\gamma_0]=[\gamma_0,\gamma_0]=0.
 \end{equation}
(However, this does not imply that the bracket vanishes identically, since the bracket is not $A$-linear, but rather satisfies equations \eqref{EQU:Poisson-rel-1}, \eqref{EQU:Poisson-rel-2}, and \eqref{EQU:Schouten-1}. For example, it is $[\alpha_\ell,\alpha_m]=(m-\ell)\alpha_{\ell+m-1}$, etc.) The differential $d$ of $A$ induces a differential $D$ on $\Der(A,A)[n]$ of degree $+1$ given by $D(\rho)=[d,\rho]$ for which we obtain $D(\alpha_\ell)=-2\gamma_{\ell+1}, D(\beta_\ell)=2\delta_{\ell+1}+\alpha_{\ell+2},  D(\gamma_\ell)=0, D(\delta_\ell)=\gamma_{\ell+2}$.

Since $\Der(A,A)[n]$ is generated by $\alpha_0$ and $\gamma_0$ as an $A$-module, we see that any element of 
\begin{equation*}
\Sym_A(\Der(A,A)[n])=A\oplus \Der(A,A)[n]\oplus \big(\Der(A,A)[n]\odot_A \Der(A,A)[n]\big)\oplus \dots
\end{equation*}
can be written as a sum of terms of the form $a.\alpha_0^{\odot p}\odot \gamma_0^{\odot q}$ for $a\in A$ and $p,q\in \N_0$. We thus may obtain a differential $D$ on $\Sym_A(\Der(A,A)[n])$ by taking
\[
D(x)=0,\quad D(y)=x^2, \quad D(\alpha_0)=-2\gamma_1=-2x\gamma_0, \quad D(\gamma_0)=0,
\]
and extending this as a graded derivation (with respect to $\odot$).

Note that $[\alpha_0^{\odot p}\odot \gamma_0^{\odot q}, \alpha_0^{\odot r}\odot \gamma_0^{\odot s}]=0$ by \eqref{EQU:Poisson-rel-1} and \eqref{EQU:Poisson-rel-2} and \eqref{EQU:even-sphere-gen-bracket}. In general, we have (with $|\alpha_0|\equiv n ($mod $2)$ and $|\gamma_0|\equiv n+1 ($mod $2)$ ):
\begin{multline}\label{EQU:n-bracket-in-even-sphere}
[a.\alpha_0^{\odot p}\odot \gamma_0^{\odot q},b.\alpha_0^{\odot r}\odot\gamma_0^{\odot s}]\\
= a[\alpha_0^{\odot p}\odot \gamma_0^{\odot q},b.\alpha_0^{\odot r}\odot\gamma_0^{\odot s}]+(-1)^{\epsilon_1}[a,b.\alpha_0^{\odot r}\odot\gamma_0^{\odot s}] \odot \alpha_0^{\odot p}\odot \gamma_0^{\odot q} \\
= a[\alpha_0^{\odot p}\odot \gamma_0^{\odot q},b]\odot\alpha_0^{\odot r}\odot\gamma_0^{\odot s}+(-1)^{\epsilon_1+\epsilon_2}b[a,\alpha_0^{\odot r}\odot\gamma_0^{\odot s}] \odot \alpha_0^{\odot p}\odot \gamma_0^{\odot q} \\
= a[\alpha_0^{\odot p}\odot \gamma_0^{\odot q},b]\odot\alpha_0^{\odot r}\odot\gamma_0^{\odot s}-(-1)^{\epsilon_1+\epsilon_2+\epsilon_3}b[\alpha_0^{\odot r}\odot\gamma_0^{\odot s},a] \odot \alpha_0^{\odot p}\odot \gamma_0^{\odot q},
\end{multline}
where we used \eqref{EQU:Poisson-rel-1} in the first equality, \eqref{EQU:Poisson-rel-2} with $[\alpha_0^{\odot p}\odot \gamma_0^{\odot q}, \alpha_0^{\odot r}\odot \gamma_0^{\odot s}]=[a,b]=0$ in the second equality, and \eqref{EQU:Anti-symmetry-for-n-bracket} in the third equality. The signs are given as follows,
\begin{eqnarray*}
\epsilon_1&=&(pn+q(n+1))(|b|+rn+s(n+1)+n), \\
\epsilon_2&=& |b|(|a|+n),\\
\epsilon_3&=& (|a|+n)(rn+s(n+1)+n),\\
\text{so that } \epsilon_1+\epsilon_2+\epsilon_3&=& (pn+q(n+1)+|a|+n)(rn+s(n+1)+|b|+n).
\end{eqnarray*}
The right hand side of \eqref{EQU:n-bracket-in-even-sphere} may be evaluated further by evaluating $[\alpha_0^{\odot p}\odot \gamma_0^{\odot q},b]$ and $[\alpha_0^{\odot r}\odot\gamma_0^{\odot s},a]$ using equations \eqref{EQU:Poisson-rel-1} and \eqref{EQU:Schouten-1}.

To be more concrete, we now restrict to the case $n$ being even. In this case $\alpha_0$ is an even element while $\gamma_0$ is an odd element implying that $\gamma_0\odot\gamma_0=0$, so that elements in $\Sym_A(\Der(A,A)[n])$ are either of the form $a.\alpha_0^{\odot p}$ or $a.\alpha_0^{\odot p}\odot \gamma_0$ for some $a\in A$. We obtain
\begin{align*}
& [\alpha_0^{\odot p},x^q]=pq x^{q-1}\alpha_0^{\odot (p-1)}, && [\alpha_0^{\odot p}\odot \gamma_0,x^q]=pq x^{q-1}\alpha_0^{\odot (p-1)}\odot \gamma_0, \\
& [\alpha_0^{\odot p},x^q y]=pq x^{q-1}y\alpha_0^{\odot (p-1)}, && [\alpha_0^{\odot p}\odot \gamma_0,x^q y]=x^q\alpha_0^{\odot p}-pq x^{q-1}y\alpha_0^{\odot (p-1)}\odot \gamma_0.
\end{align*}
This, together with \eqref{EQU:n-bracket-in-even-sphere} gives the following $n$-brackets:
\begin{equation}\label{EQU:n-bracket-for-even-sphere-n-even-1}
\left\{ \begin{array}{cl}
\, [x^r\alpha_0^{\odot p},x^s\alpha_0^{\odot q}]&=(ps-qr)x^{r+s-1}\alpha_0^{\odot(p+q-1)},  \\
\, [x^r y\alpha_0^{\odot p},x^s y\alpha_0^{\odot q}]&=0,  \\
\, [x^r\alpha_0^{\odot p}\odot \gamma_0,x^s\alpha_0^{\odot q}\odot \gamma_0]&=0,  \\
\, [x^r y\alpha_0^{\odot p}\odot \gamma_0,x^s y\alpha_0^{\odot q}\odot \gamma_0]&=0,  
\end{array} \right.
\end{equation}
and
\begin{equation}\label{EQU:n-bracket-for-even-sphere-n-even-2}
\left\{ \begin{array}{cl}
\, [x^r\alpha_0^{\odot p},x^s y\alpha_0^{\odot q}]&=(ps-qr)x^{r+s-1}y\alpha_0^{\odot(p+q-1)},  \\
\, [x^r\alpha_0^{\odot p},x^s\alpha_0^{\odot q}\odot \gamma_0]&=(ps-qr)x^{r+s-1}\alpha_0^{\odot(p+q-1)}\odot \gamma_0,  \\
\, [x^r\alpha_0^{\odot p},x^s y\alpha_0^{\odot q}\odot \gamma_0]&=(ps-qr)x^{r+s-1}y\alpha_0^{\odot(p+q-1)}\odot\gamma_0,  \\
\, [x^r y\alpha_0^{\odot p},x^s\alpha_0^{\odot q}\odot \gamma_0]&=(ps-qr)x^{r+s-1}y\alpha_0^{\odot(p+q-1)}\odot\gamma_0 
+x^{r+s}\alpha_0^{\odot (p+q)},  \\
\, [x^r y\alpha_0^{\odot p},x^s y\alpha_0^{\odot q}\odot \gamma_0]&=-x^{r+s}y\alpha_0^{\odot (p+q)},  \\
\, [x^r \alpha_0^{\odot p}\odot \gamma_0,x^s y\alpha_0^{\odot q}\odot \gamma_0]&=x^{r+s}\alpha_0^{\odot (p+q)}\odot \gamma_0.  
\end{array} \right.
\end{equation}
\end{ex}

\begin{ex}[$P_{n+1}$-structure \eqref{eq:HigherFormality2} for the complex projective space $\C P^m$]
For the complex projective space $\C P^m$, the Sullivan model $A=(S(V),d)$ is generated by $x$ in degree $|x|=2$ and $y$ in degree $|y|=2m+1$ with differential $d(x)=0$ and $d(y)=x^{m+1}$. Note, that in this case we have again two generators $x$ and $y$ which have the same parity as in the generators $x$ and $y$ in the last Example \ref{EXA:even-sphere-P(n+1)-structure} for the even sphere. Therefore, much of the arguments from Example \ref{EXA:even-sphere-P(n+1)-structure} can be repeated with only minor modifications. First, note that $\Der(A,A)[n]$ is generated by the derivations $\alpha_\ell, \beta_\ell, \gamma_\ell, \delta_\ell$ given by formulas \eqref{DERIV:a}-\eqref{DERIV:d}, however with $x$ and $y$ in the new degrees stated above. Thus, the degrees in the (shifted) space of derivations $\Der(A,A)[n]$ are now
\begin{align*}
&|\alpha_\ell |=2(\ell-1)+n, \\
& |\beta_\ell|=2(\ell-1)+(2m+1)+n=2(\ell+m)+n-1, \\
&|\gamma_\ell |=2\ell-(2m+1)+n=2(\ell-m)+n-1, \\
& |\delta_\ell|= 2\ell+n.
\end{align*}
The module relations \eqref{EQU:module-even-sphere} and the basic brackets \eqref{EQU:even-sphere-gen-bracket} remain the same. However, the differential is now $d=\gamma_{m+1}$, so that $D$ on $\Der(A,A)$ becomes $D(\rho)=[\gamma_{m+1},\rho]$, which gives the relations $D(\alpha_\ell)=-(m+1)\gamma_{\ell+m}, D(\beta_\ell)=(m+1)\delta_{\ell+m}+\alpha_{\ell+m+1}, D(\gamma_\ell)=0, D(\delta_\ell)=\gamma_{\ell+m+1}$. Thus $D$ is defined on $\Sym_A(\Der(A,A)[n])$ by taking
\[
D(x)=0,\quad D(y)=x^{m+1}, \quad D(\alpha_0)=-(m+1)\gamma_m=-(m+1)x^m\gamma_0, \quad D(\gamma_0)=0,
\]
and extending this to $\Sym_A(\Der(A,A)[n])$ as a graded derivation. Also, the considerations concerning the $n$-bracket (such as equation \eqref{EQU:n-bracket-in-even-sphere} and, when $n$ is even, equations \eqref{EQU:n-bracket-for-even-sphere-n-even-1} and \eqref{EQU:n-bracket-for-even-sphere-n-even-2}) apply just as in Example \ref{EXA:even-sphere-P(n+1)-structure} for the even sphere.
\end{ex}


\section{Integral chain models for higher string topology operations}\label{S:Brane}

We will use the $E_\infty$-Poincar\'e duality and Hochschild chains to give an algebraic model
for  Brane Topology at the \emph{chain} level, over an arbitrary coefficient ring.

\subsection{Brane operations for $n$-connected Poincar\'e duality space }
Recall that the $n$-dimensional free sphere space is denoted $X^{S^n}=Map(S^n, X)$.
It is the  space of continuous map from $S^n$ to $X$ endowed with the compact-open topology.
Sullivan and Voronov~\cite[Section 5]{CV} have shown that there is a natural graded
commutative algebra structure, called the \emph{sphere product},
 on the shifted homology $H_{\bullet+\dim(M)}(M^{S^n})$ of an oriented closed manifold.
For $n=1$, this structure agrees with the Chas-Sullivan loop product~\cite{CS}.
 This product was extended to all oriented stacks in~\cite{BGNX}.
For $n=2$, the sphere product is a special case of the surface product studied in~\cite{GTZ}.
 Further, it is claimed that $H_{\bullet+\dim(M)}(M^{S^n})$ is an algebra
over the homology $H_{\ast}(\mathcal{E}_{n+1}^{fr})$ of the framed little disk operad
$\mathcal{E}_{n+1}^{fr}$. Below we will forget about the $SO(n+1)$-action and
deal with action of the $\mathcal{E}_{n+1}$-operad at the \emph{chain} (and not homology) level
 and without specific assumptions on the characteristic of the ground ring $k$.

We start by stating one of our main results:
\begin{theorem}\label{T:BraneChain} Let $X$ be an $n$-connected Poincar\'e duality space. Then the shifted chain complex $C_{\bullet+\dim(X)} (X^{S^n})$ has a natural\footnote{with respect to maps of Poincar\'e duality spaces in the sense of Definition~\ref{D:PDmap}}  $E_{n+1}$-algebra structure which induces the sphere product~\cite[Section 5]{CV}
$$H_{p}\big(X^{S^n}\big)\otimes H_{q}\big(X^{S^n}\big) \to H_{p+q-\dim(X)}\big(X^{S^n}\big)$$ in homology when $X$ is an oriented closed manifold.
\end{theorem}

\begin{proof}  Remark~\ref{R:PDimpliesFinGen} implies that the homology groups of $X$ are   finitely generated so that the biduality homomorphism $C_\ast(X)\to (C^\ast(X))^{\vee}$ is a quasi-isomorphism. Since $X$ is a Poincar\'e duality space, it then follows from Corollary~\ref{C:PDmap}
that the Poincar\'e duality  map~\eqref{eq:PDmap}
$$\chi_X: C^{\ast}(X) \to C_{\ast}(X)[\dim(X)] \; \cong\;
\big(C^{\ast}(X)\big)^{\vee}[\dim(X)]$$ is an equivalence of $C^{\ast}(X)$-$E_\infty$-Modules.
 Thus it yields an equivalence
\begin{multline}\label{eq:HHPD}
 CH^{S^n} \left(C^{\ast}(X), C^{\ast}(X)\right)
\cong Hom_{C^{\ast}(X)}\left( CH_{S^n}(C^{\ast}(X)),
 C^{\ast}(X)\right) \\ \stackrel{(\chi_X)\circ -}\longrightarrow
 Hom_{C^{\ast}(X)}\left( CH_{S^n}(C^{\ast}(X)),  \big(C^{\ast}(X)\big)^{\vee}\right)[\dim(X)] \\
 \cong CH^{S^n}\left(C^{\ast}(X)),  \big(C^{\ast}(X)\big)^{\vee}\right)[\dim(X)].
\end{multline}
 Since $X$ is $n$-connected, by Corollary~\ref{C:Itdual}, there is an equivalence
\begin{equation}
 \label{eq:ItPD} CH^{S^n}\left(C^{\ast}(X)),  \big(C^{\ast}(X)\big)^{\vee}\right)\;\cong \; C_{\ast}\big(X^{S^n}\big).
\end{equation}
Combining the equivalences~\eqref{eq:HHPD} and~\eqref{eq:ItPD}, we get
a natural equivalence
\begin{equation}\label{eq:HH=chain}
 CH^{S^n} \left(C^{\ast}(X), C^{\ast}(X)\right) \cong C_{\ast}\big(X^{S^n}\big)[\dim(X)].
\end{equation}
By Theorem~\ref{T:Deligne}, $CH^{S^n} \left(C^{\ast}(X), C^{\ast}(X)\right)$ has
 a natural $E_{n+1}$-algebra structure, whose underlying $E_1$-algebra structure is given
by the cup-product. Hence the equivalence~\eqref{eq:HH=chain} yields a natural
 $E_{n+1}$-structure on $C_{\ast}\big(X^{S^n}\big)[\dim(X)]$. Note that the naturality with
respect to maps $f:X\to Y$ of Poincar\'e duality spaces follows from Theorem~\ref{THM:Ed-on-HC(AM)} below
 since a Poincar\'e duality space yields an object of $\AM$ and a map of Poincar\'e duality space
is  a map in $\AM$, see Example~\ref{EX:examples-for-AM}.(2) below. From this observation follows  the commutativity
of the following diagram (in which $d=\dim(X)=\dim(Y)$)
\[\xymatrix{  \Big(CH^{S^n} \left(C^{\ast}(X), C^{\ast}(X)\right)\Big)^{\otimes 2}
\ar[d]_{{(\chi_X)\circ -}^{\otimes 2}}^{\cong}  \ar[r]^{\circ}
& CH^{S^n} \left(C^{\ast}(X), C^{\ast}(X)\right) \ar[d]^{(\chi_X)\circ -}_{\cong} \\
 \Big(CH^{S^n} \left(C^{\ast}(X), C_{\ast}(X)\right)[d]\Big)^{\otimes 2}
\ar[d]_{(f_*)^{\otimes 2}} & CH^{S^n} \left(C^{\ast}(X), C_{\ast}(X)\right)[d] \ar[d]^{f_*}  \\
\Big(CH^{S^n} \left(C^{\ast}(Y), C_{\ast}(Y)\right)[d]\Big)^{\otimes 2} &
CH^{S^n} \left(C^{\ast}(Y), C_{\ast}(Y)\right)[d]  \\
 \Big(CH^{S^n} \left(C^{\ast}(Y), C^{\ast}(Y)\right)\Big)^{\otimes 2} \ar[r]^{\circ}
\ar[u]^{{(\chi_Y)\circ -}^{\otimes 2}}_{\cong} &  CH^{S^n} \left(C^{\ast}(Y), C^{\ast}(Y)\right)
\ar[u]_{(\chi_Y)\circ -}^{\cong}}
 \]
where the horizontal arrows are given by the composition~\eqref{eq:compEndoA} of (derived)
homomorphisms (and Proposition~\ref{P:coHH=coTCH}). By Theorem~\ref{THM:Ed-on-HC(AM)}, the
vertical maps are maps of $E_n$-algebras. Thus the above diagram shows that a map of
Poincar\'e duality space induces a map of $E_1$-algebras (with respect to the composition~\eqref{eq:compEndoA})
in the (symmetric monoidal) category of $E_n$-algebras and thus induces a map of $E_{n+1}$-algebras by
Dunn Theorem (see~\cite{Du, L-HA} or Theorem~\ref{T:Dunn}):  $E_1-Alg\big(E_n-Alg\big)\cong E_{n+1}-Alg$.

\smallskip

It remains to identify the underlying multiplication in homology with its purely
topological counterpart. This is done in Section~\ref{S:BraneProduct},
 see Proposition~\ref{P:brane=cupsphere}.
\end{proof}
Passing to homology in Theorem~\ref{T:BraneChain},
we recover the following result first stated in~\cite{CV}. 
\begin{cor}\label{C:BraneHomol}
 Let $X$ be a $n$-connected Poincar\'e duality space.
Then the shifted homology $H_{\bullet+\dim(X)} (X^{S^n})$ has a
  natural $P_{n+1}$-algebra\footnote{such that is the induced Lie algebra structure is the one of a restricted Lie algebra} structure which induces the sphere product~\cite[Section 5]{CV}
$$H_{p}\big(X^{S^n}\big)\otimes H_{q}\big(X^{S^n}\big) \to H_{p+q-\dim(X)}\big(X^{S^n}\big)$$ in homology when $X$ is an oriented closed manifold.
\end{cor}

\begin{rem}
Theorem~\ref{T:BraneChain} (as well as Corollary~\ref{C:homologyinvarianceSphereproduct} below) still holds if $X$ is a Poincar\'e duality space which is connected, nilpotent with finite homotopy groups in degree less than or equal to $n$. This is seen by
using Proposition~\ref{P:weakenconnectivity} in addition to Corollary~\ref{C:Itdual} in the proof of the Theorem. 
\end{rem}

\begin{ex}[Explicit computation in characteristic zero]
 Let $n\geq 2$.  In characteristic zero, the singular cochains on $X$ are equivalent, as an $E_\infty$-algebra, 
 to a Sullivan algebra (as in \S~\ref{SS:HigherFormalityforFreeAlg}) and, in particular, one can compute the Brane topology structure given by Theorem~\ref{T:BraneChain} using Corollary~\ref{C:HigherFormalityforFreeAlg} which gives very explicit combinatorial models.
\end{ex}

\begin{ex}
Assume $M$ is a simply connected closed manifold. Then Theorem~\ref{T:BraneChain} yields an $E_2$-structure on the chains $C_\ast(LM)[\dim(M)]$ of the free loop space $LM$, thus string topology operations at the chain level. According to Example~\ref{E:Deligneforn=1} and Proposition~\ref{P:brane=cupsphere} below, the underlying Gerstenhaber structure is the  classical Chas-Sullivan one~\cite{CS}.
\end{ex}

\begin{cor}\label{C:homologyinvarianceSphereproduct} Let $X, Y$ be $n$-connected  ($n\geq 1$) closed manifolds of the same dimension and assume $f: M\to N$ induces an isomorphism in homology such that $f_*([X])= [Y] \in H_\ast(Y,k)$. Then the induced bijection $H_\ast(X^{S^n}) \cong H_{\ast}(Y^{S^n})$ is an algebra isomorphism (with respect to the sphere product).
\end{cor}
 In particular,
the sphere product is an homotopy invariant of $n$-connected manifolds (with respect to orientation preserving maps).
\begin{proof}
By assumption, the induced map $ \cap f_*([X]) : C^\ast(Y) \to C_\ast(Y)[\dim(Y)]$ and $\cap [Y]  C^\ast(Y) \to C_\ast(Y)[\dim(Y)]$
are homotopic. Thus $f$ induces a map of Poincar\'e duality spaces $(X,[X]) \to (Y, [Y))$ which is a quasi-isomorphism.  Then, by Theorem~\ref{T:BraneChain}, $f_*: C_{\ast+\dim(X)}(X^{S^n}) \to C_{\ast+\dim(Y)}(Y^{S^n})$ is an equivalence of $E_n$-algebras. In particular, it is an algebra isomorphism in homology so that the result follows from the identification of the sphere product as asserted in Theorem~\ref{T:BraneChain} (see Proposition~\ref{P:brane=cupsphere}).
\end{proof}

The above brane product fits into a larger setting of setups\footnote{which is useful to study
functoriality of brane operations} to define
 $\mathcal E_n$-actions on $CH^{S^n}(A,M)$. In fact, we start with the following general setup.
\begin{definition}\label{DEF:category-AM}
We define $\AM$ as the following category. The objects of $\AM$ are triples $(A,M,\mu)$, where $A$ is an $E_\infty$-algebra, $M$ is an $E_\infty$-$A$-module, and, considering the $E_\infty$-algebra $A\otimes A$ with canonical $E_\infty$-$(A\otimes A)$-modules $M$ and $M\otimes M$ (induced via the $E_\infty$ structure map $A\otimes A\to A$), we assume that $\mu:M\otimes M\to M$ is an $E_\infty$-$(A\otimes A)$-module map\footnote{said otherwise, the objects of $\AM$ are the objects $N$ of the monoidal $\infty$-category $Mod^{E_n}$ endowed with a structure map $\mu_N:N\otimes N\to N$; the morphisms are however different}. The morphisms of $\AM$ consist of tuples $(f,g):(A,M,\mu)\to (A',M',\mu')$, where $f:A\to A'$ is an $E_\infty$-morphism, thus inducing an $E_\infty$-$A$-module structure on $M'$, and $g:M'\to M$ is an $E_\infty$-$A$-module map, satisfying the compatibility relation,
\begin{equation}\label{EQ:morph-in-AM}
 \xymatrix{M'\otimes M' \ar[r]^{\quad\mu'}\ar[d]_{g\otimes g} & M'\ar[d]^{g}\\ M\otimes M\ar[r]^{\quad\mu} & M  }
\end{equation} in $\hkmod$.
\end{definition}
There are two main examples we have in mind for the above definition.
\begin{ex}\label{EX:examples-for-AM}
\begin{enumerate}
\item
The first example relates to the sphere product as considered in Section \ref{S:Edcochains} and also in \cite{G}. Let $A$ and $B$ be two $E_\infty$-algebras, and let $h:A\to B$ be a morphism of $E_\infty$-algebras. 
Then, $h$ makes $M:=B$ into an $E_\infty$-$A$-module, and the $E_\infty$ structure of $B$ gives a map $B\otimes B\to B$ which is also an $E_\infty$-$(A\otimes A)$-module map. Furthermore, if $h$ factors through an $E_\infty$-algebra $B'$ as a composition of $E_\infty$-algebras maps $h:A\stackrel {h'} \to B'\stackrel {g}\to B$, then this induces a morphism between the spaces $(id_A,g):(A,B,\mu)\to(A,B',\mu')$.
\item
The second example relates to generalizations of sphere topology products as described in Theorem \ref{T:BraneChain} above. Let $A$ be an $E_\infty$-algebra and $M$ be an $E_\infty$-$A$-module and given an $E_\infty$-module map $\rho:M\to A$. We define the induced $E_\infty$-$(A\otimes A)$-module map $\mu:M\otimes M\to M$ as the composition of $\rho$ and the $E_\infty$-$A$-module structure of $M$,
\[
\mu: M\otimes M\stackrel{\rho\otimes id}{\longrightarrow} A\otimes M\longrightarrow M.
\]
Furthermore, any map of two given $E_\infty$-$A$-modules $g:M'\to M$ which commutes with $E_\infty$-$A$-module maps $\rho$ and $\rho'$,
\[
 \xymatrix{M' \ar[rrd]^{\rho'}\ar[dd]_{g} && \\ && A\\ M\ar[rru]^{\rho} && }
\]
also respects the induced relation \eqref{EQ:morph-in-AM}, since $g\circ \mu'(m'_1,m'_2)=g(\rho'(m'_1).m'_2)=\rho'(m'_1).g(m'_2)=\rho(g((m'_1)).g(m'_2)=\mu\circ(g\otimes g)(m'_1,m'_2)$.

For example, consider the setup from
Section \ref{SS:E-inf-PD}: $C_*(X)$ is an $E_\infty$-coalgebra,
$C^{*}(X)=Hom_{k}(C_*(X),k)$ is its linear dual endowed with its canonical
 $E_\infty$-algebra structure, and caping with the fundamental cycle
 $\cap [X]:C^*(X)\to C_*(X)[\dim(X)]$ induces an $E_\infty$-quasi-isomorphism
 of $E_\infty$-$A$-modules. The quasi-inverse of this map is an
 $E_\infty$-$A$-module map $\rho:M:=C_*(X)[\dim(X)]\to A:=C^*(X)$. Moreover, if
$f: (X, [X])\to (Y, [Y])$ is  a map of Poincar\'e duality space (Definition~\ref{D:PDmap}),
 then  the tuple $(f^*:C^{\ast}(Y)\to C^{\ast}(X), f_*:C_{\ast}(X) \to C_{\ast}(Y))$ is a map in
the category $\AM$.
\end{enumerate}
\end{ex}

For any triple $(A,M,\mu)$ which is an object of $\AM$ described in Definition \ref{DEF:category-AM}, we can consider the Hochschild cochains $CH^{S^d}(A,M)$. We claim that there is an $E_d$-algebra structure on $CH^{S^d}(A,M)$, generalizing the $E_d$-algebra structure from Theorem \ref{T:EdHoch}.
\begin{definition}\label{DEF:Ed-on-CHSd}
Using the notation from Section \ref{S:Edcochains}, we define the $E_d$-algebra structure on $CH^{S^d}(A,M)$ by,
\begin{multline*}
 C_{\ast}\big(\mathcal{C}_d(r)\big) \otimes \left(CH^{S^d}(A,M)\right)^{\otimes r}
 \longrightarrow  C_{\ast}\big(\mathcal{C}_d(r)\big)\otimes \left(Hom_A(A^{\otimes S^d},M)\right)^{\otimes r}
 \\
  \longrightarrow  C_{\ast}\big(\mathcal{C}_d(r)\big)\otimes Hom_{A^{\otimes r}}((A^{\otimes S^d})^{\otimes r},M^{\otimes r})
 \\
 \stackrel{id\otimes (\mu^{\circ (r-1)})_*}  \longrightarrow  C_{\ast}\big(\mathcal{C}_d(r)\big)\otimes Hom_{A^{\otimes r}}((A^{\otimes S^d})^{\otimes r},M)
 \\
 \stackrel{\cong}  \longrightarrow  C_{\ast}\big(\mathcal{C}_d(r)\big)\otimes Hom_{A^{\otimes r}}((A^{\otimes S^d})^{\otimes r},Hom_A(A,M))
\\
 \stackrel{\cong}  \longrightarrow  C_{\ast}\big(\mathcal{C}_d(r)\big)\otimes Hom_{A}(A\otimes^{\mathbb{L}}_{A^{\otimes r}}(A^{\otimes S^d})^{\otimes r},M)
\\
 \stackrel{\cong}  \longrightarrow  C_{\ast}\big(\mathcal{C}_d(r)\big)\otimes Hom_{A}(A^{\otimes ( \overbrace{S^d\vee\dots\vee S^d}^{r\text{ times}})},M)
\\
\longrightarrow C_{\ast}\big(\mathcal{C}_d(r)\big)\otimes CH^{S^d\vee\dots\vee S^d}(A,M) \stackrel{pinch^*}  \longrightarrow  CH^{S^d}(A,M).
\end{multline*}
We need to show compatibility of the involved operad action. This is similar to the proof in section \ref{S:Edcochains}.
\end{definition}
In fact, more is true:
\begin{theorem}\label{THM:Ed-on-HC(AM)}
The identification given in the previous Definition \ref{DEF:Ed-on-CHSd} defines a (contravariant) functor $CH^{S^d}:\AM\to  E_d-Alg$.
\end{theorem}
\begin{proof}
It only remains to show that morphisms $(f,g):(A,M,\mu)\to (A',M',\mu')$ in $\AM$ induce maps of $E_d$-algebras. Since $f:A\to A'$ makes $M'$ into an $E_\infty$-$A$-algebra, and with this $\mu':M'\otimes M'\to M'$ into a map of $E_\infty$-$(A\otimes A)$-modules, this follows from the commutativity of the following diagram:
{\small\[
 \xymatrix{
\left(CH^{S^d}(A',M')\right)^{\otimes r} \ar[d]\ar[r]^{(f^*)^{\otimes r}} &
\left(CH^{S^d}(A,M')\right)^{\otimes r} \ar[d]\ar[r]^{(g_*)^{\otimes r}} &
\left(CH^{S^d}(A,M)\right)^{\otimes r} \ar[d]
 \\
Hom_{{A'}^{\otimes r}}(({A'}^{\otimes S^d})^{\otimes r},{M'}^{\otimes r}) \ar[d]_{({\mu'}^{\circ (r-1)})_*} \ar[r]^{(f^{\otimes r})^*} &
Hom_{A^{\otimes r}}((A^{\otimes S^d})^{\otimes r},{M'}^{\otimes r}) \ar[d]^{({\mu'}^{\circ (r-1)})_*}\ar[r]^{(g^{\otimes r})_*} &
Hom_{A^{\otimes r}}((A^{\otimes S^d})^{\otimes r},M^{\otimes r}) \ar[d]^{({\mu}^{\circ (r-1)})_*}
  \\
Hom_{{A'}^{\otimes r}}\left(({A'}^{\otimes S^d})^{\otimes r},M'\right) \ar[d]\ar[r]^{(f^{\otimes r})^*} &
Hom_{A^{\otimes r}}\left((A^{\otimes S^d})^{\otimes r},M'\right) \ar[d]\ar[r]^{g_*} &
Hom_{A^{\otimes r}}\left((A^{\otimes S^d})^{\otimes r},M\right) \ar[d]
 \\
Hom_{A'}\left({A'}^{\otimes ( {S^d\vee\dots\vee S^d})},M'\right) \ar[d]_{{pinch}^*} \ar[r]^{f^*} &
Hom_{A}\left(A^{\otimes ( {S^d\vee\dots\vee S^d})},M'\right)  \ar[d]_{pinch^*} \ar[r]^{g_*}&
Hom_{A}\left(A^{\otimes ( {S^d\vee\dots\vee S^d})},M\right)  \ar[d]_{pinch^*}
   \\
Hom_{A'}({A'}^{\otimes  S^d},M') \ar[r]^{f^*} &
Hom_{A}(A^{\otimes  S^d},M') \ar[r]^{g_*} &
Hom_{A}(A^{\otimes  S^d},M)
   }
\]}
\end{proof}
By the virtue of the previous theorem and Example \ref{EX:examples-for-AM}(2),
we can thus define a family of sphere topology operations,
 one for each $E_\infty$-module map $C_\ast(X)[\dim(X)]\to C^\ast(X)$,
which are related by morphisms of $E_d$-algebras.

In particular, for $d=1$, we can obtain (chain level, characteristic free)
\emph{string topology} operations associated to any $E_\infty$-module map
$C_\ast(M)[\dim(M)]\to C^\ast(M)$.

\subsection{Topological identification of the  brane product} \label{S:BraneProduct}
In this section, we prove that the cup product of Hochschild cochains over spheres
identifies with the usual ``brane product'' in the homology of a free sphere space.
The idea of the proof follows the surface product kind of proof from \cite[Theorem 3.4.2]{GTZ}.

\smallskip

We start by recalling the construction of the sphere product of Sullivan-Voronov~\cite{CV}. Let  $M$ be a manifold equipped with a Riemannian metric and let the sphere  spaces $Map(S^n,M)$ be equipped with  Fr\'echet manifold structures. We further assume that $M$ is closed, oriented.
We have a cartesian square of fibrations
\begin{equation}\label{eq:cartesiansquarebrane}
\xymatrix{
  Map(S^n\vee S^n,M) \ar[r]^{\qquad \rho_{in}\quad\quad\quad\quad} \ar[d]
  & Map(S^n,M)\times Map(S^n,M) \ar[d]^{ev\times ev} \\
 M  \ar[r]^{\text{diagonal}\quad} & M\times M }
\end{equation}
where the evaluation maps on the right are furthermore submersions. We denote  $Tub(M)\subset M\times M$ a tubular neighborhood of the diagonal of $M$, which can be identified to the normal bundle of the diagonal. The pullback $(ev\times ev)^{-1}(Tub(M))$  by the submersion $ev\times ev:Map(S^n,M)\times Map(S^n,M)\to M\times M$ can be identified with a tubular neighborhood $Tub(Map(S^n\vee S^n,M))$ of $\rho_{in}$ and thus with a normal bundle of $\rho_{in}$. One forms the corresponding Thom spaces $M^{-TM}$ and $Map(S^n\vee S^n,M)^{-TM}$  by collapsing all the complements of the tubular neighborhood to a point. These Thom spaces are spheres (of dimension $\dim(M)$)  bundles over, respectively $M$, and $Map (S^n\vee S^n,M)$. Hence, we have a diagram of pullback squares
\begin{equation*}
\xymatrix{   Map(S^n\coprod S^n,M) \ar[r]^{\hspace{-1pc}\text{collapse}}\ar[d]_{ev\times ev} & Map(S^n\vee S^n,M)^{-TM} \ar[r]^{\pi} \ar[d]^{ev} & Map(S^n\vee S^n,M)  \ar[d]^{ev} \\
  M\times M   \ar[r]^{\text{collapse}\quad} &  M^{-TM} \ar[r]^{\pi}& M}
\end{equation*}
where the vertical arrows are  fibrations. In particular, the Thom class of $\rho_{in}$  is the pullback $(ev^*)(th(M))\in H^{\dim(M)}(Map(S^n\vee S^n,M)^{-TM})$ of the Thom class $th(M)\in H^{\dim(M)}(M^{-TM})$ of $M\to M\times M$.

The above setup allows us to define a Gysin map
$$(\rho_{in})_!: H_\ast\Big(M^{S^n\coprod S^n} \Big) \longrightarrow H_{\ast -\dim(M)}\Big( M^{S^n\vee S^n}\Big) $$
 as the composition
 \begin{equation} \label{eq:defnrhoin}(\rho_{in})_! = \pi_*\circ (-\cap ev^*(th(M))) \circ(\text{collapse})_*.\end{equation}
\begin{definition}[Sullivan-Voronov~\cite{CV}] \label{D:sphereproduct}
The sphere product is the composition
\begin{multline*}\star_{S^n}: H_{\ast+\dim(M)}\Big(M^{S^n}\Big)^{\otimes 2}\to H_{\ast+2\dim(M)}\Big(M^{S^n\coprod S^n}\Big)\\
\stackrel{(\rho_{in})_!}\longrightarrow H_{\ast+\dim(M)}\Big(M^{S^n\vee S^n}\Big)\stackrel{(\delta_{S^n}^*)_*}\longrightarrow H_{\ast+\dim(M)}\Big(M^{S^n}\Big) \end{multline*} where $\delta_{S^n}: S^n\to S^n\vee S^n$ is the pinching map.
\end{definition}

Note that the Thom class $th(M)$ can be represented by any cocycle
 $t(M)$ which is Poincar\'e dual to the pushforward of the fundamental cycle $[M]$ of $M$,
\emph{i.e.}, $\chi_{M\times M}(\text{collapse}_*(t(M))=\big(\text{diagonal}_*([M]) \big)$
or, equivalently,
$$\chi_{M^{-TM}}(t(M))=\big(\text{collapse} \circ \text{diagonal}_*([M]) \big).$$

By Corollary~\ref{C:PD}, we get  maps of $E_\infty$-modules
$$\rho_{th(M)}: C_{\ast}(M^{-TM}) \longrightarrow C_{\ast-\dim(M)}(M^{-TM}), $$
 $$\rho_{ev^*(th(M))}:C_{\ast}\big(Map(S^n\vee S^n,M)^{-TM}\big)
\longrightarrow C_{\ast-\dim(M)}\big(Map(S^n\vee S^n,M)^{-TM}\big)$$ lifting
the cap-products $-\cap t(M)$ and $-\cap ev^*(t(M))$.
Thus we obtain the following chain level interpretation of the sphere product.
\begin{lem}\label{L:Defnsphereproduct} The sphere product (Definition~\ref{D:sphereproduct}) is induced by passing to the homology
groups in the following composition
\begin{multline}\label{eq:Dsphereproduct}\star_{S^n}: \big(C_{\ast}\big(M^{S^n}\big)[\dim(M)]\big)^{\otimes 2}
\to C_{\ast}\big(M^{S^n\coprod S^n}\big)[2\dim(M)]\\
\stackrel{\text{collapse}_*}\longrightarrow C_{\ast}\big(\big(M^{S^n\vee S^n}\big)^{-TM}\big)[2\dim(M)]
\stackrel{\rho_{ev^*(th(M))}}\longrightarrow C_{\ast}\big(\big(M^{S^n\vee S^n}\big)^{-TM}\big)[\dim(M)]\\
\stackrel{\pi_*}\longrightarrow C_{\ast}\big(M^{S^n\vee S^n}\big)[\dim(M)]
\stackrel{(\delta_{S^n}^*)_*}\longrightarrow C_{\ast}\big(M^{S^n}\big)[\dim(M)].
 \end{multline}
\end{lem}
\begin{rem}
In this section we only identify the sphere product which is the degree $0$-component of a higher framed $E_{n+1}$-structure claimed in~\cite[Section 5]{CV}. The reason is that we do not know higher degree representative of this operations (in a way similar to the map~\eqref{eq:Dsphereproduct}) since such higher operations would involve a careful analysis of Gysin maps associated to higher cacti \emph{in families}. However, it is possible that the new operads introduced by Bargheer in~\cite{Ba} could lead in a near future to explicit representatives of the degree $n$ Lie Bracket in homology.
\end{rem}

\medskip

We now further assume $X$ is a general Poincar\'e duality space (see Definition~\ref{D:PDspace}).

\smallskip

Recall that by Corollary~\ref{C:Itdual} and Corollary~\ref{C:PD}, we have the
equivalence~\eqref{eq:HH=chain}:
$$CH^{S^n} \left(C^{\ast}(X), C^{\ast}(X)\right)\; \cong\; C_{\ast}\big(X^{S^n}\big)[\dim(X)].$$
The cup-product can be thus transfered (through the above equivalence) to give a multiplication
$\Big(C_{\ast}\big(X^{S^n}\big)[\dim(X)]\Big)^{\otimes 2} \to C_{\ast}\big(X^{S^n}\big)[\dim(X)]$.
We first wish to give another chain level representative for this multiplication, which is essentially the content of Lemma~\ref{L:cupproductdual} below. We will then compare it with the sphere product $\star_{S^n}$ given by the composition~\eqref{eq:Dsphereproduct}.

The $E_\infty$-algebra map $C^{\ast}(X\times X) \stackrel{diag^*}\longrightarrow C^{\ast}(X)$
induced by the diagonal $X\to X\times X$ makes $C_{\ast}(X)$ an
$E_\infty\text{-}C^{\ast}(X\times X)$-module. By functoriality of the cap-product,
the  diagonal $C_{\ast}(X)\to C_{\ast}(X\times X)$ is a map of left
$(C^{\ast}(X\times X), \cup)$-module.

By Theorem~\ref{T:lifttoEinfty}, we thus get a unique
lift $C_{\ast}(X)\stackrel{diag_*}\longrightarrow C_{\ast}(X\times X)$ of the diagonal map in
$C_{\ast}(X\times X)\text{-}Mod^{E_\infty}$.
By Lemma~\ref{L:AWEinfty}, there is an equivalence of $E_\infty$-algebras
$C^{\ast}(X\times X) \cong C^{\ast}(X)\otimes C^{\ast}(X)$. Further,
Poincar\'e duality (Corollary~\ref{C:PDmap})
gives  equivalences of $E_\infty$-$C^{\ast}(Y)$-modules
$\chi_X: C^{\ast}(Y)\stackrel{\simeq}\to C_{\ast}(Y)[\dim(Y)]$ for any Poincar\'e
duality space $Y$.

Putting together the last three statements we obtain the first assertion in
\begin{lem}\label{L:diagasmodmap} Let $X$ be a Poincar\'e duality space. 
 There is a   map in $C^{\ast}(X)\otimes C^{\ast}(X)-Mod^{E_\infty}$ given by the following
composition:
\begin{multline*}
 \nabla_X: C^{\ast}(X)\stackrel{\simeq}\to C_{\ast}(X)[\dim(X)]
\stackrel{diag_*}\longrightarrow C_{\ast}(X\times X)[\dim(X)] \\ \stackrel{\simeq}\to
 C_{\ast}(X\times X)[\dim(X)] \stackrel{\simeq}\leftarrow  C^{\ast}(X\times X)[-\dim(X)]
\\ \cong C^{\ast}(X)\otimes C^{\ast}(X)[-\dim(X)].
\end{multline*}
Further, for any closed oriented manifold $M$, the following diagram is commutative
\[ \xymatrix{C^{\ast}(M) \ar@/_/[rrrd]_{\nabla_M}\ar[r]^{\pi^*} & C^{\ast}(M^{-TM})
\ar[r]^{\hspace{-2pc}\rho_{th(M)}^{\vee}} &
C^{\ast}(M^{-TM})[-\dim(M)] \ar[r]^{\text{collapse}^*}
& C^{\ast}(M\times M) [-\dim(M)]\ar[d]^{\cong}\\
& &  & C^{\ast}(M)\otimes C^{\ast}(M)[-\dim(M)]   }
\]
in  $C^{\ast}(X)\otimes C^{\ast}(X)-Mod^{E_\infty}$.
\end{lem}
\begin{proof}
 The second assertion follows from the identity
\begin{eqnarray*}
 \text{collapse}^*\big(\pi^*(x)\big)\cap \big(\text{collapse}^*\big(t(M) \big)
\cap [M\times M]\big)&\!=&\!
\text{collapse}^*\big(\pi^*(x)\big)\cap \text{diagonal}_*([M])\\
&\!=&\! \text{diagonal}_*(x\cap [M])
\end{eqnarray*}
which follows from $\pi\circ \text{collapse}\circ \text{diagonal} =id$ and the definition
of the Thom class.
\end{proof}

It follows that the map $\nabla_X$ yields a map of $C^{\ast}(X)$-$E_\infty$-modules
\begin{multline}\label{eq:defnabla}
 {\nabla_X}_* : CH_{S^n\vee S^n}(C^{\ast}(X)) \cong CH_{S^n\coprod S^n}(C^{\ast}(X))
\mathop{\otimes}^{\mathbb{L}}_{C^{\ast}(X)^{\otimes 2}} C^{\ast}(X) \\
\stackrel{1\otimes \nabla_X}\longrightarrow CH_{S^n\coprod S^n}(C^{\ast}(X))
\mathop{\otimes}^{\mathbb{L}}_{C^{\ast}(X)^{\otimes 2}} C^{\ast}(X)^{\otimes 2}[-\dim(X)]
 \\ \cong CH_{S^n\coprod S^n}(C^{\ast}(X))[-\dim(X)].
\end{multline}
Thus, dualizing, we get a map
\begin{multline}\label{eq:nabladual}
 \nabla^!:CH^{S^n\coprod S^n}\Big(C^{\ast}(X), (C^{\ast}(X))^{\vee}\Big)\cong
Hom_{C^{\ast}(X)}\Big(CH_{S^n\coprod S^n}(C^{\ast}(X)), (C^{\ast}(X))^{\vee} \Big) \\
\stackrel{{\nabla_X}_*^*}\longrightarrow
Hom_{C^{\ast}(X)}\Big(CH_{S^n\vee S^n}(C^{\ast}(X)), (C^{\ast}(X))^{\vee} \Big)[-\dim(X)]\\
\cong
CH^{S^n\vee S^n}\Big(C^{\ast}(X), (C^{\ast}(X))^{\vee}\Big)[-\dim(X)].
\end{multline}
Recall that we have a pinching map $\delta_{S^n}: S^n\to S^n\vee S^n$ induced by collapsing
the equator of $S^n$ to a point. This gives us a multiplication
\begin{multline}\label{eq:muSn}
\mu_{S^n}: CH^{S^n}\Big(C^{\ast}(X), (C^{\ast}(X))^{\vee}\Big)^{\otimes 2} \cong
Hom_{C^{\ast}(X)}\Big(CH_{S^n}(C^{\ast}(X)), (C^{\ast}(X))^{\vee} \Big)^{\otimes 2}
\\ \longrightarrow
 Hom_{C^{\ast}(X)^{\otimes 2}}\Big(CH_{S^n\coprod S^n}(C^{\ast}(X)), (C^{\ast}(X))^{\vee}\otimes
(C^{\ast}(X))^{\vee}  \Big) \\
\cong Hom_{C^{\ast}(X)}\Big(CH_{S^n\coprod S^n}(C^{\ast}(X)), (C^{\ast}(X))^{\vee} \Big)
\\ \stackrel{\nabla^!}\longrightarrow
CH^{S^n\vee S^n}\Big(C^{\ast}(X), (C^{\ast}(X))^{\vee}\Big)[-\dim(X)] \\
\stackrel{\delta_{S^n}^*}\longrightarrow CH^{S^n}\Big(C^{\ast}(X),
(C^{\ast}(X))^{\vee}\Big) [-\dim(X)].
\end{multline}

\begin{lem}\label{L:cupproductdual} Let $X$ be a Poincar\'e duality space.
There  is a commutative (in $\hkmod$) diagram
\[ \xymatrix{CH^{S^n} \left(C^{\ast}(X), C^{\ast}(X)\right)^{\otimes 2} \ar[rr]^{\cup_{S^n}}
\ar[d]_{\cong}  & &
   CH^{S^n} \left(C^{\ast}(X), C^{\ast}(X)\right) \ar[d]^{\cong} \\
\left(CH^{S^n} \left(C^{\ast}(X), (C^{\ast}(X))^{\vee}\right)[\dim(X)]\right)^{\otimes 2} \ar[rr]^{\mu_{S^n}}
 & &
   CH^{S^n} \left(C^{\ast}(X), (C^{\ast}(X))^{\vee}\right)[\dim(X)]
 }
\]
where the top arrow is the sphere cup-product of Corollary~\ref{C:cupHoch} and the vertical  arrows are  induced
by the Poincar\'e duality map $\chi_X: C^{\ast}(X)\to C_{\ast}(X)[\dim(X)]\to
\big(C^{\ast}(X)\big)^{\vee}[\dim(X)]$.
\end{lem}
\begin{proof}
 By Lemma~\ref{L:AWEinfty}, the $E_\infty$-algebra map
$m_X:C^{\ast}(X)\otimes C^{\ast}(X)\to C^{\ast}(X)$ is the composition
$$m_X:C^{\ast}(X)\otimes C^{\ast}(X)\stackrel{AW^{\vee}}\longrightarrow C^{\ast}(X\times X)
\stackrel{diag^*}\longrightarrow C^{\ast}(X).$$ It follows that the map $\nabla_X$ defined in
 Lemma~\ref{L:diagasmodmap} sits inside a commutative diagram
\begin{equation*}
 \xymatrix{ C^{\ast}(X)\otimes C^{\ast}(X)\ar[rr]^{m_X} \ar[d]_{\chi_X\otimes \chi_X}
&& C^{\ast}(X)
 \ar[d]^{\chi_X}\\ \big(C^{\ast}(X)\big)^{\vee}\otimes \big(C^{\ast}(X)\big)^{\vee}[2\dim(X)]
\ar[rr]^{\quad (\nabla_X)^{\vee}} && \big(C^{\ast}(X)\big)^{\vee}[\dim(X)]
}
\end{equation*}
in $C^{\ast}(X)\otimes C^{\ast}(X)-Mod^{E_\infty}$. It follows that we get a commutative
diagram
{\small \begin{equation}\label{eq:cupproductdual}
 \xymatrix{
Hom_{C^{\ast}(X)^{\otimes 2}}\Big(CH_{S^n\coprod S^n}(C^{\ast}(X)),
 C^{\ast}(X)^{\otimes 2} \Big) \ar[r]^{(m_X)_*} \ar[d]_{(\chi_X)_*^{\otimes 2}}
& Hom_{C^{\ast}(X)^{\otimes 2}}\Big(CH_{S^n\coprod S^n}(C^{\ast}(X)),C^{\ast}(X)\Big)
\ar[d]^{(\chi_X)_*} \\
Hom_{C^{\ast}(X)^{\otimes 2}}\Big(CH_{S^n\coprod S^n}(C^{\ast}(X)),
 (C^{\ast}(X))^{\vee})^{\otimes 2} \Big) \ar[r]^{(\nabla_X)^{\vee}_*} \ar[d]_{\cong}
& Hom_{C^{\ast}(X)^{\otimes 2}}\Big(CH_{S^n\coprod S^n}(C^{\ast}(X)),(C^{\ast}(X))^{\vee}\Big)
\ar[d]^{\cong} \\
Hom_{C^{\ast}(X)}\Big(CH_{S^n\coprod S^n}(C^{\ast}(X)),
 (C^{\ast}(X))^{\vee} \Big) \ar[r]^{\nabla^{!}}
& Hom_{C^{\ast}(X)}\Big(CH_{S^n\vee S^n}(C^{\ast}(X)),(C^{\ast}(X))^{\vee}\Big)
}
\end{equation}} in $\hkmod$ (note that we have suppress the degree shifting in the diagram
for simplicity). By functoriality, we have a commutative diagram
{\small \begin{equation*}
 \xymatrix{ Hom_{C^{\ast}(X)}\Big(CH_{S^n\vee S^n}(C^{\ast}(X)),C^{\ast}(X)\Big)
\ar[r]^{\delta_{S^n}^*} \ar[d]_{(\chi_X)_*} &
Hom_{C^{\ast}(X)}\Big(CH_{S^n}(C^{\ast}(X)),C^{\ast}(X)\Big)
\ar[d]^{(\chi_X)_*}  \\
Hom_{C^{\ast}(X)}\Big(CH_{S^n\vee S^n}(C^{\ast}(X)),(C^{\ast}(X))^{\vee}\Big)
\ar[r]^{\delta_{S^n}^*}
& Hom_{C^{\ast}(X)}\Big(CH_{S^n}(C^{\ast}(X)),(C^{\ast}(X))^{\vee}\Big)
}
\end{equation*}}
which, together with the previous diagram~\eqref{eq:cupproductdual} and the definition
of the map $\cup_{S^n}$ (see Corollary~\ref{C:cupHoch}) and the map~\eqref{eq:muSn}, implies the Lemma.
\end{proof}

The cartesian square of fibrations~\eqref{eq:cartesiansquarebrane} shows that,
 when $M$ is $n$-connected (and thus $M^{S^n}$ is path connected),
 there is a
quasi-isomorphism
\begin{equation}\label{eq:decspherespace}C^\ast\big(M^{S^n\vee S^n} \big) \cong  C^\ast\big(M^{S^n\coprod S^n}\big)
\mathop{\otimes}\limits^{\mathbb{L}}_{C^{\ast}\big(M\times M \big) } C^{\ast}(M) \end{equation}
so that the map $\nabla_M$ of Lemma~\ref{L:diagasmodmap} yields a map
\begin{multline*}
 id\otimes \nabla_M: C^\ast\big(M^{S^n\vee S^n} \big) \cong  C^{\ast}\big(M\times M \big)
\mathop{\otimes}\limits^{\mathbb{L}}_{C^{\ast}\big(M\times M \big) } C^{\ast}(M)
\\
 \stackrel{id\otimes_{C^{\ast}\big(M\times M \big)}\nabla_M}\longrightarrow
 C^{\ast}\big(M\times M \big)[-\dim(M)] \cong C^\ast\big(M^{S^n}\big)^{\otimes 2}[-\dim(M)].
\end{multline*}

\begin{lem}\label{L:sphereproduct=Hoch}
 Let $X$ be a $n$-connected Poincar\'e duality space. The following diagram
\[
\xymatrix{CH_{S^n}\big(C^{\ast}(X) \big)    \ar[r]^{\big(\delta_{S^n}\big)_*}
 \ar[d]_{\mathcal{I}t}  & CH_{S^n\vee S^n}\big(C^{\ast}(X) \big) \ar[d]_{\mathcal{I}t}
\ar[r]^{\hspace{-2pc}{\nabla_X}_*}
&  CH_{S^n}\big(C^{\ast}(X) \big)^{\otimes 2}
\ar[d]^{\mathcal{I}t^{\otimes 2}}[-\dim(X)] \\
 C^\ast\big(X^{S^n}\big) \ar[r]^{\delta_{S^n}^*} & C^{\ast}\big(X^{S^n\vee S^n} \big)
\ar[r]^{id\otimes \nabla_X}& C^\ast\big(X^{S^n}\big)^{\otimes 2}[-\dim(X)]
} \]
is commutative in $\hkmod$ (here the map ${\nabla_X}_*$ is the map~\eqref{eq:defnabla}).
\end{lem}
\begin{proof}
 This is a consequence of the naturality of the map $\mathcal{I}t:CH_X(C^\ast(Y))\to C^{\ast}(Y^{X})$, see
Corollary~\ref{C:Einftymapping}.
\end{proof}

\begin{prop} \label{P:brane=cupsphere}Let $M$ be an $n$-connected  oriented closed manifold. 
Then the following diagram
 \[ \xymatrix{CH^{S^n} \left(C^{\ast}(X), C^{\ast}(X)\right)^{\otimes 2} \ar[rr]^{\cup_{S^n}}
\ar[d]_{\cong}  & &
   CH^{S^n} \left(C^{\ast}(X), C^{\ast}(X)\right) \ar[d]^{\cong} \\
\left(C_\ast\big(X^{S^n}\big)[\dim(X)]\right)^{\otimes 2} \ar[rr]^{\star_{S^n}}
 & &
   C_\ast\big(X^{S^n}\big)[\dim(X)]
 }
\] is commutative in $\hkmod$.
Here the horizontal arrows are the sphere cup-product of Corollary~\ref{C:cupHoch} and the sphere
product~\eqref{eq:Dsphereproduct};
  the vertical arrows are given by the equivalences~\eqref{eq:HH=chain} (induced
by the Poincar\'e duality map and Corollary~\ref{C:Itdual}).
\end{prop}
Since the vertical arrows are the maps defining the $E_{n+1}$-structure given by Theorem~\ref{T:BraneChain}
 on $C^\ast\big(X^{S^n}\big)[\dim(X)]$; it follows that the underlying commutative algebra structure
on homology agrees with the sphere product.

\begin{proof}
Recall that there
 is a canonical isomorphism $CH^{S^n}(A,A^{\vee})\cong \Big(CH_{S^n}(A,A^{\vee})\Big)^{\vee}$.
By assumption $M$ is a Poincar\'e duality space, hence it has finitely generated homology groups (Remark~\ref{R:PDimpliesFinGen}) and the canonical biduality map
$C_{\ast}(X)\to \big(C^{\ast}(X)\big)^{\vee}$ is a quasi-isomorphism and it is sufficient
to prove that the dual of the diagram depicted in Proposition~\ref{P:brane=cupsphere} is
commutative.

By definition~\eqref{eq:muSn} of $\mu_{S^n}$, lemma~\ref{L:sphereproduct=Hoch} and lemma~\ref{L:cupproductdual}
we are left to prove that  the map
$ id\otimes \nabla_M: C^\ast\big(M^{S^n\vee S^n} \big) \to C^\ast\big(M^{S^n}\big)^{\otimes 2}[-\dim(M)]$
sits inside a commutative diagram
{\small \begin{equation}\label{eq:dualsphereproduct}\xymatrix{C^{\ast}\big(M^{S^n\vee S^n}\big) \
\ar@/_/[rrrdd]_{id\otimes \nabla_M}\ar[r]^{\pi^*} & C^{\ast}\big(\big(M^{S^n\vee S^n}\big)^{-TM}\big)
\ar[rr]^{\hspace{-2pc}\rho_{ev^*(th(M))}^{\vee}} & &
C^{\ast}\big(\big(M^{S^n\vee S^n}\big)^{-TM}\big)[-\dim(M)] \ar[d]^{\text{collapse}^*} \\
&& &C^{\ast}\big(M^{S^n\coprod S^n}\big) [-\dim(M)]\ar[d]^{\cong}\\
& &  & C^{\ast}\big(M^{S^n}\big)^{\otimes 2}[-\dim(M)]   }\end{equation}}
in $\hkmod$.

Recall the equivalence~\eqref{eq:decspherespace} above. Under this equivalence, the cup-product
by the pullback $ev^*(t(M))$ is given by
\begin{multline*}
C^\ast\big(M^{S^n\vee S^n} \big) \cong  C^\ast\big(M^{S^n\coprod S^n}\big)
\mathop{\otimes}\limits^{\mathbb{L}}_{C^{\ast}\big(M\times M \big) } C^{\ast}(M)
\\\stackrel{id\mathop{\otimes}^{\mathbb{L}}_{C^{\ast}\big(M\times M \big) }\cup t(M)} \longrightarrow
C^\ast\big(M^{S^n\coprod S^n}\big)
\mathop{\otimes}\limits^{\mathbb{L}}_{C^{\ast}\big(M\times M \big) } C^{\ast}(M) [-\dim(M)]\\
\cong C^\ast\big(M^{S^n\vee S^n} \big)[-\dim(M)].
\end{multline*}
Now, the commutativity of diagram~\eqref{eq:dualsphereproduct} follows from
Lemma~\ref{L:diagasmodmap}.
\end{proof}

\subsection{A spectral sequence to compute the  brane topology product}\label{SS:SpecSeqBrane}

In \cite{CJY}, the Chas-Sullivan product of spheres and projective spaces was computed using a spectral sequence of algebras. There is a similar approach for the Brane product. Indeed,
following arguments similar to those of \cite{CJY}, there exists a spectral sequence of algebras which converges (as algebras) to the homology $H_*\big((M)^{S^n}\big)$ with product being the sphere product from \eqref{eq:Dsphereproduct}. It is essentially given by the Serre spectral sequence applied to the fibration $\Omega^n M \to Map(S^n, M)\to M$.

\begin{prop}\label{P:spec-seq-for-M}Let $M$ be closed oriented connected manifold. 
There exists a second quadrant spectral sequence of algebras $\{E^r_{p,q}, d^r:E^r_{p,q}\to E^r_{p-r,q+r-1}\}_{r\geq 1}$, which converges (as algebras) to the homology $H_{*+\dim(M)}\big((M)^{S^n}\big)$ with product being the sphere product from \eqref{eq:Dsphereproduct}.

 Furthermore, the $E^2$ page of the spectral sequence is given by
\begin{equation}\label{equ:spec-seq-for-M}
E^2_{-p,q}\cong H^p(M;H_q(\Omega^n M)) \cong H^p(M) \otimes H_q(\Omega^n M).
\end{equation}
Here, the algebra structure on $H^p(M)$ is given by the usual cup product on cohomology, and the product on $H_q(\Omega^n M)$ is the Pontrjagin product on the based loop space.
\end{prop}

\begin{proof}The  spectral sequence  is essentially given by the Serre spectral sequence applied to the fibration $\Omega^n M \to Map(S^n, M)\to M$. The construction is similar to the calculation of the loop product in \cite[section 2]{CJY} where we now consider the sphere product as defined by the sequence of maps in \eqref{eq:Dsphereproduct}. We now give some of the details of this rather lengthy construction, and leave it to the interested reader to give the full details which are completely analogous to those in \cite{CJY}.

For a fibration $F\hookrightarrow E\stackrel \pi \to B$, consider the singular simplicial set with $r$-simplicies $S_r (E)=\{\sigma:\Delta^r\to E\}$. This simplicial set has a filtration obtained by taking simplicies whose induced simplicies on $B$ are essentially of simplicial degree at most $p$,
\begin{multline*}
 F_p(S_r(E))=\big\{\sigma:\Delta^r\to E: \pi\circ \sigma=\rho\circ f_\sharp, \text{ where } \rho\in S_q(B), q\leq p, \text{ and } \\
f_\sharp:\Delta^r\to \Delta^p \text{ is induced by some non-decreasing } f:\{0,\dots , r\} \to \{0,\dots, p\} \big\}
\end{multline*}
$F_p(S_\bullet(E))$ forms a subsimplicial set of $S_\bullet(E)$, and taking the associated chain complex, which is denoted by $F_p(C_*(E)):=C_*(F_p(S_\bullet(E)))$, induces the filtration
\[ \{0\}\to \dots\to F_{p-1}(C_*(E))\to F_p(C_*(E))\to\dots\to C_*(E) \]
of $C_*(E)$. This is the filtration which in turn induces the Serre spectral sequence converging to $H_*(E)$.

Just as in \cite{CJY}, we may analyze the involved maps in the sphere product \eqref{eq:Dsphereproduct} and their behavior under the above filtration. For the space $(M^{S^n\vee S^n})^{-TM}=ev^*(M^{-TM})$ appearing in \eqref{eq:Dsphereproduct}, which involves the Thom construction, we may in turn use the filtration of pairs,
\[
F_p(C_*(ev^*(D_{TM}), ev^*(S_{TM}))):= F_p(C_*(ev^*(D_{TM}))) / F_p(C_*(ev^*(S_{TM}))),
\]
where $M^{-TM}=D_{TM}/S_{TM}$ is given by a collapsing the tubular disk $D_{TM}$ of $M$ in $TM$ by the corresponding boundary sphere $S_{TM}$. This relative Serre spectral sequence converges to $H_*((M^{S^n\vee S^n})^{-TM})$. All the terms in \eqref{eq:Dsphereproduct} respect the filtration in the following sense:
\begin{eqnarray*}
F_p\big(C_{\ast}(M^{S^n})\big)\otimes F_q\big(C_{\ast}(M^{S^n})\big) &\longrightarrow & F_{p+q}\big(C_{\ast}\big(M^{S^n\coprod S^n}\big)\big) \\
F_{p}\big(C_{\ast}\big(M^{S^n\coprod S^n}\big)\big) &\stackrel{\text{collapse}_*}{\longrightarrow} & F_p(C_*(ev^*(D_{TM}), ev^*(S_{TM}))) \\
F_p(C_*(ev^*(D_{TM}), ev^*(S_{TM}))) &\stackrel{\rho_{ev^*(th(M))}}{\longrightarrow} & F_{p-\dim(M)}(C_*(ev^*(D_{TM}), ev^*(S_{TM}))) \\
F_p(C_*(ev^*(D_{TM}), ev^*(S_{TM}))) &\stackrel{\pi_*}{\longrightarrow} & F_p\big( C_{\ast}\big(M^{S^n\vee S^n}\big)\big)\\
F_p\big( C_{\ast}\big(M^{S^n\vee S^n}\big)\big) & \stackrel{(\delta_{S^n}^*)_*}\longrightarrow  & F_p\big( C_{\ast}\big(M^{S^n}\big)\big).
\end{eqnarray*}
The above is obvious for most of the stated maps, except for the map $\rho_{ev^*(th(M))}$ involving capping with the Thom class, which can be proved just as in \cite[Theorem 8]{CJY}. Thus, we arrive at a map of filtered chain complexes
\[ F_p(C_*(M^{S^n}))\otimes F_q(C_*(M^{S^n}))\to F_{p+q-\dim(M)}(C_*(M^{S^n})) \]
inducing a map of spectral sequences
\[ \tilde E^r_{p,s}\otimes \tilde E^r_{q,t}\to \tilde E^r_{p+q-\dim(M),s+t} \]
such that on the $\tilde E^2$ page this is given by a map
\[ H_p(M; H_s(\Omega^n M))\otimes H_q(M; H_t(\Omega^n M))\to H_{p+q-\dim(M)}(M; H_{s+t}(\Omega^n M)). \]
Checking the maps in the sphere product according to their component on the base $M$ and the fiber $\Omega^n M$ shows that this product is just the intersection product on $H_*(M)$, and the Pontrjagin product on $H_*(\Omega^n M)$. Now, shifting the spectral sequence to the left by $\dim(M)$, we obtain another spectral sequence $ E^r_{p,t}:=\tilde E^r_{p+\dim(M),t}$, which now lives completely in the second quadrant, and for which the sphere product induces a map $ E^r_{p,s}\otimes  E^r_{q,t}\to  E^r_{p+q,s+t}$ converging to the sphere product on homology,
\[ H_{p+\dim(M)}(M^{S^n})\otimes H_{q+\dim(M)}(M^{S^n})\to H_{p+q+\dim(M)}(M^{S^n}). \]
The $E^2$ page of this new spectral sequence is given by $E^2_{p,s}=H_{p+\dim(M)}(M;H_s(\Omega^n M))$, so that, as a last step for obtaining the claimed spectral sequence in \eqref{equ:spec-seq-for-M}, we use Poincar\'e duality to write $E^2_{p,s}=H^{-p}(M;H_s(\Omega^n M))$. As the the intersection product becomes the cup product in cohomology under Poincar\'e duality, we finally arrive at the claimed spectral sequence.\end{proof}

We end this section with a computation of the sphere product in the case where $M=S^k$ is itself a sphere of a certain dimension $k$.
\begin{ex}\label{EXA:sphere-product-of-odd-sphere}
Let $S^n$ denote the $n$-sphere, let $n>k$, and denote by $A=C^\bullet(S^k)$ a cochain model for the $k$-sphere. By formality of the sphere, we may assume that $A=H^\bullet(S^k)$. 
By Proposition~\ref{P:spec-seq-for-M}, 
there exists  a second quadrant spectral sequence of algebras $\{E^r_{p,q}, d^r:E^r_{p,q}\to E^r_{p-r,q+r-1}\}_{r\geq 1}$, which converges (as algebras) to the Brane homology $H_{*+k}\big((S^k)^{S^n}\big)$; the product being the sphere product from \eqref{eq:Dsphereproduct}. Furthermore, the $E^2$ page of the spectral sequence is given by
\begin{equation}\label{equ:spec-seq-for-Sk}
E^2_{-p,q}\cong H^p(S^k;H_q(\Omega^n S^k)) \cong H^p(S^k) \otimes H_q(\Omega^n S^k).
\end{equation}
Now we use the fact that $H^\bullet(S^k)=span\{1,a\}$ is a two dimensional space with $a$ in degree $k$ with the obvious cup product. To determine $H_\bullet(\Omega^n S^k)$, we distinguish between the cases where $k$ is odd or even.

First, assume that $k=2m+1$ is odd. One calculates the homology with Pontrjagin product as $H_\bullet(\Omega^n S^k)\cong \Lambda(x)$ as the free algebra in one generator $x$ in degree $k-n$. Thus, we can distinguish two cases, where $n$ is even or odd. Figure \ref{fig:odd-sphere} displays a picture of the generators of $H^\bullet(S^k) \otimes H_\bullet(\Omega^n S^k)$.
\begin{figure}
\includegraphics[scale=.4]{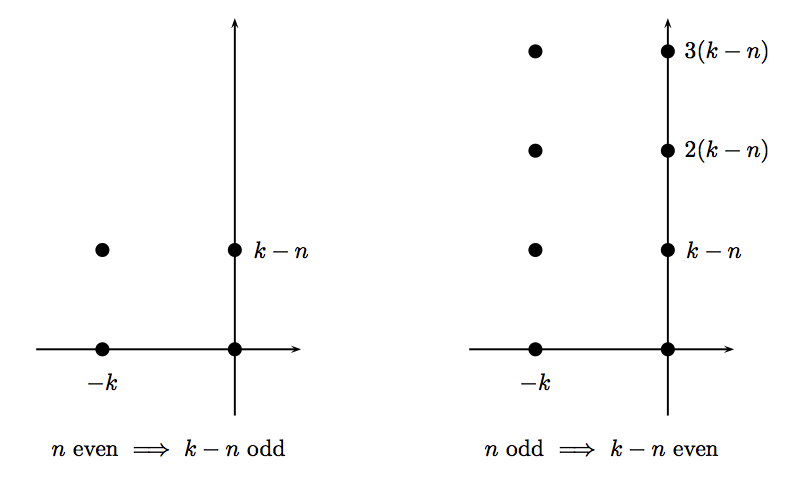}
\caption{The generators of $H^\bullet(S^k) \otimes H_\bullet(\Omega^n S^k)$}\label{fig:odd-sphere}
\end{figure}

In the case where $n$ is even, $x$ is of odd degree $k-n$, so that $x^2=0$. Since the differential $d^r$ on the $r$th page is of bi-degree $(-r, r-1)$, all differentials have to be zero in this case. (More precisely, the $r$th differential would map bi-degree $(0,0)$ to $(-r,r-1)\neq (-k,k-n)$ for any $r$.)

In the case where $n$ is odd, $x$ is of even degree $k-n$, so that $x^2\neq 0$. We see that the $r$th differential $d^r$ could possibly nonzero, namely in the case where $(-r,r-1)=(-k,p(k-n))$. Solving $r=k$ and thus $k-1=p(k-n)$, we see that the differential is always zero in the case when $k-1$ is not a factor of $k-n$. (In the case where $k-1$ is a factor of $k-n$ further analysis needs to be applied. For example, the above calculation applies when calculating the sphere product of $H_{*+13}\big((S^{13})^{S^5}\big)$ but it would not apply to $H_{*+13}\big((S^{13})^{S^7}\big)$ since $13-7=6$ is a factor of $13-1=12$.)

We thus obtained the following result:
\begin{multline}
\text{Let $k>n> 1$ with $k$ odd.} \\
\text{If ($n$ is even), or if ($n$ is odd, and $k-n$ is not a factor of $k-1$), then:}\\
H_{*+k}\big((S^k)^{S^n}\big)=\Lambda[a]\otimes \Lambda[x], \quad\quad \text{with $|a|=-k$ and $|x|=k-n$}
\end{multline}

One can ask for a similar analysis of the brane product when the underlying manifold is an even sphere $S^k$ with $k=2m$. One calculates the Pontrjagin ring as $H_\bullet(\Omega^n S^k)\cong \Lambda[x,y]$, where $x$ is of degree $k-n$, and $y$ is of degree $2k-1-n$. Unfortunately, the simple analysis of the involved degrees, similar to the one done for odd spheres, fails in this case, since one can check that the differential will never be zero purely by degree reasons. Thus, a computation of the brane product involves identifying the differentials of the spectral sequence, which are in general non trivial. Indeed, in \cite{CJY}, the differential $d^k$ in the spectral sequence for $k$ even and $n=1$ was shown to be non-zero.
\end {ex}

\begin{rem}\label{REM:odd-sphere-consistency}
For an odd sphere $S^k$ of sufficiently high connectivity, we have identified the sphere product with the cup product in Proposition \ref{P:brane=cupsphere}, $C_*((S^k)^{S^n})[k]\cong CH^{S^n}(A,A)$, where we may take $A$ to be the Sullivan model of the $k$-sphere. By Proposition \ref{P:HH(A,B)=CH(A,B)}, this structure is identified with the algebra structure on $CH^{S^n}(A,A)\cong HH_{\mathcal E_n}(A,A)$, which by Corollary \ref{C:HigherFormalityforFreeAlg}, was also identified with $HH_{\mathcal E_n}(A,A)\cong \Sym_A(\Der(A,A)[n])$. 
This, in turn, was explicitly calculated in Example \ref{EXA:Pn+1-on-odd-sphere}. Note that the two calculations in Examples \ref{EXA:sphere-product-of-odd-sphere} and \ref{EXA:Pn+1-on-odd-sphere} for the product on $C_*((S^k)^{S^n})[k]$ do indeed produce the same result.
\end{rem}

\section{Iterated bar constructions}\label{S:Barmain}
\subsection{Iterated loop spaces and iterated bar constructions for $E_\infty$-algebras}\label{S:Bar}
In this section we study the case of spaces of pointed maps from spheres to $X$,
\emph{i.e.} iterated loop spaces. The idea is to apply the formalism of the higher Hochschild functor to the Bar construction of augmented $E_\infty$-algebras.

\smallskip

Let $(A,d)$ be a differential graded unital associative algebra (DGA for short)
which is equipped with an augmentation $\epsilon: A\to k$.
Denote $\overline{A}=ker(A\stackrel{\epsilon}\to k)$ the augmentation ideal of $A$.
 The standard bar construction on $A$
is the chain complex $(Bar^{std}(A),b)$ defined by
 $$Bar^{std}(A)=\bigoplus_{n\geq 1} {\overline{A}}^{\otimes n} $$
with differential given by
\begin{eqnarray*}
 b(a_1\otimes \cdots a_n)     & = &   \sum_{i=1}^n \pm a_1\otimes \cdots \otimes d(a_i)\otimes \cdots
\otimes a_n\\
& & +  \sum_{i=1}^{n-1} \pm a_1\otimes \cdots \otimes (a_i\cdot a_{i+1})\otimes \cdots \otimes a_n
\end{eqnarray*}
see~\cite{FHT, Fre2, KM} for details (and signs). Further, if $A$ is a commutative differential graded algebra (CDGA for short), then the shuffle product makes  the Bar construction $Bar^{std}(A)$ a CDGA as well.
\begin{rem}Note that, by our convention on $\otimes_k$, if $A$ is not flat over
$k$, we replace it by a flat resolution. In particular $Bar^{std}: E_1-Alg\to \hkmod$ preserves weak equivalences. This definition of $Bar^{std}(A)$ thus agrees with the classical one as soon as the underlying
chain complex of
$A$ is flat over $k$.
\end{rem}

\begin{rem}
 The standard bar construction above extends naturally to $A_\infty$-algebras. It also extends
to any augmented $E_1$-algebra. Indeed, one can prove a Lemma similar to Lemma~\ref{L:Bar=HochI}
below
with factorization homology $\int_{I} (A,k)$ (of the $E_1$-algebra)
instead of Hochschild chains over $I$ (see~\cite{F, F2, L-HA, AFT}). Note that, there is a natural
equivalence $\int_{I} (A,k)\cong k\mathop{\otimes}\limits^{\mathbb{L}}_{A} k$.
In particular, if $B$ is any  DGA equivalent to
$A$, then the standard bar construction of $A$ is naturally equivalent to
$Bar^{std}(B) $ in
$\hkmod$.
\end{rem}

Let $A$ be an $E_\infty$-algebra and let $\epsilon: A\to k$ be an augmentation.
In particular, we can see $k$ as an $A$-module thanks to the augmentation $\epsilon$.
 In particular
$A$ is an $E_1$-algebra so that we can choose a  DGA $B$ and a quasi-isomorphism $f:B\to A$ of
$E_1$-algebras. Then we  define $Bar^{std}(A):=Bar^{std}(B)$. The fact that this construction is
well-defined\footnote{i.e. independent of the choice of $B$}
 indeed follows from the following lemma:
\begin{lem}\label{L:Bar=HochI}  Let $I=[0,1]$ be the closed interval. There is a natural equivalence (in $\hkmod$)
\begin{equation} \label{eq:Bar=HochI} CH_{I}(A) \mathop{\otimes}\limits_{CH_{S^0}(A)}^{\mathbb{L}} k
\;\cong \; Bar^{std}(A)\end{equation} where $Bar^{std}(A)$ is
the
standard Bar construction. Further if
$A$ is a CDGA, the equivalence~\eqref{eq:Bar=HochI} is an equivalence of $E_\infty$-algebras
(where $Bar^{std}(A)$ is endowed with its CDGA-structure induced by the shuffle product).
\end{lem}
\begin{proof}

Let $I^{std}_\bullet$ be the standard simplicial set model of the interval
(viewed as a CW-complex with two vertices and one non-degenerate 1-cell). More precisely,
$I^{std}_k=\{0,\dots,k+1\}$ with face maps $d_i$ given, for $i=0,\dots,k$, by $d_i(j)$
equal to $j$ or $j-1$ depending on $j\leq i$ or $j>i$.
For any differential graded associative algebra $B$, one can form the simplicial dg-algebra
 $$\widetilde{C_{I^{std}_{\bullet}}}(B):=(B^{\otimes I^{std}_{k}})_{k\geq 0}
= (B\otimes B^{k}\otimes B)_{k\geq 0}$$ where the simplicial
structure is defined as for Hochschild chains of a CDGA (it is immediate to check, and well known, that the
commutativity is not necessary to check the simplicial identities in that case). Further the
associated\footnote{via the usual Dold-Kan construction} differential graded module $DK(\widetilde{C_{I^{std}_{\bullet}}}(B))$ is the two-sided Bar construction $Bar^{std}(B,B,B)$
of $B$ (see~\cite[Example 2.3.4]{GTZ}).
  In particular,
if $f:B\stackrel{\simeq}\to A$ is an equivalence of $E_1$-algebras, with $B$ a DGA, then
$DK(\widetilde{C_{I^{std}_{\bullet}}}(B))\stackrel{f_*}\to Bar^{std}(A)$ is an equivalence
(natural in $A$, $B$).

Now, forgetting the $E_\infty$-structure of the Hochschild chain complex
 $CH_{I^{std}_\bullet}(A)$, we get a quasi-isomorphism of simplicial chain complexes
$$f:\widetilde{C_{I^{std}_{\bullet}}}(B)\stackrel{\simeq}\longrightarrow
\widetilde{CH_{I^{std}_\bullet}(A)}  $$ and thus after taking the Dold-Kan $\infty$-functor
$DK:sk\text{-}Mod_\infty \to \hkmod$,
 we see that  $CH_{I^{std}_\bullet} (A) \cong DK(\widetilde{C_{I^{std}_{\bullet}}}(B))$.
Now the result
follows since $Bar^{std}(B)\cong \cong \bigoplus_{n\geq 1} {\overline{B}}^{\otimes n}$ is the
normalized chain complex associated to the simplicial chain complex
$\widetilde{C_{I^{std}_{\bullet}}}(B)$, thus is quasi-isomorphic to $DK(\widetilde{C_{I^{std}_{\bullet}}}(B)) $.

When $A$ is a CDGA, the result follows from Corollary~\ref{C:cdgaE}
 and \cite[Section 2]{GTZ}.
\end{proof}

In particular, we get an $E_\infty$-lifting of the Bar construction of an $E_\infty$-algebra
 that we denote \begin{equation}\label{eq:NBar}
Bar(A):= CH_{I}(A) \mathop{\otimes}\limits_{CH_{S^0}(A)}^{\mathbb{L}} k.\end{equation}
Note that the augmentation $\epsilon:A\to k$ induces  augmentations $\epsilon_*: CH_{I}(A)\to CH_{I}(k)\cong k$,   $\epsilon_*: CH_{S^0}(A)\to CH_{S^0}(k)\cong k$ and thus an augmentation $Bar(A)\to k$ as well.

 Since $Bar(A)$ is an augmented $E_\infty$-algebra, we can take its Bar construction again.
\begin{definition}\label{D:Bar}
The $n^{th}$-iterated Bar construction of an augmented $E_\infty$-algebra $A$ is the $E_\infty$-algebra $Bar^{(n)}(A)= Bar(\cdots (Bar(A))\cdots)$.
\end{definition}

Summing up the above results we have:
\begin{prop}\label{P:LiftBar}
The $n^{th}$-iterated Bar construction Bar is an $\infty$-functor
$$Bar^{(n)}: E_\infty\text{-Alg}\to E_\infty\text{-Alg}.$$ Further,
there is a natural equivalence in $E_\infty$-Alg between
$B^{(n)}(A)$ and the $n^{th}$-iterated Bar construction defined by
B. Fresse~\cite{Fre2}.
\end{prop}
\begin{proof}
Since $A\mapsto CH_{I}(A)$ is an $\infty$-endofunctor of $E_\infty$-Alg, the same follows for $Bar$ (and its iteration).
By Lemma~\ref{L:Bar=HochI}, the $Bar(A)$ is equivalent (in $\hkmod$) to $Bar^{std}(A)$ and, further, this   equivalence is an equivalence in $E_\infty$-Alg if $A$ is a CDGA and $Bar^{std}(A)$ is endowed with the CDGA structure given by the shuffle product.
Thus the uniqueness of the Bar construction in $E_\infty$-Alg obtained in~\cite{Fre2} shows that $Bar^{(n)}$ is the correct  $n^{th}$-iterated Bar construction.
\end{proof}
\begin{rem}
Since the canonical map $CH_{I}(A)\to CH_{pt}(A)\cong A$ is an equivalence,
 we recover immediately
from the excision axiom
$$Bar(A) \cong A \mathop{\otimes}\limits_{CH_{S^0}(A)}^{\mathbb{L}} k
\cong k \mathop{\otimes}\limits_{A}^{\mathbb{L}} k.$$
\end{rem}
\begin{rem}In terms of factorization algebras, one has the following definition. Considered the unit interval with two stratified points given by its endpoints. Then, the analogue of Proposition~\ref{P:EnMod=Fact} in that case is that a locally constant (stratified) factorization algebra on $I$ is the same as the data of an $E_1$-algebra $A$ and a pair of left $A$-module $M$ and a right $A$-module $N$. In particular taking the factorization algebra $\mathcal{A}$ for which $A$ is augmented and $M=N=k$, we obtain that the factorization homology $\int_I \mathcal{A}$ (denoted $\int_I (A,k)$ in~\cite{F}) is equivalent to the Bar construction, see~\cite{F} for details.
\end{rem}

There is an easy interpretation of the iterated Bar construction in terms
 of higher Hochschild chains. Note that, since $k$ is an $A$-algebra (via the augmentation),
$CH_{S^n}(A,k) \cong CH_{S^n}(A) \mathop{\otimes}\limits_{A}^{\mathbb{L}} k $ is an $E_\infty$-algebra.
\begin{prop}\label{P:nBarisCHIn}
There are natural equivalences of $E_\infty$-algebras $$CH_{S^n}(A,k) \cong Bar^{(n)}(A) .$$
\end{prop}
\begin{proof}
Since $S^n\cong D^n\cup_{S^{n-1}}^{h}{pt}$, the homotopy invariance and excision axiom for Hochschild chains implies the following sequence of natural (in $A$)  equivalences of $E_\infty$-algebras
\begin{eqnarray*}
CH_{S^n}(A,k)\;\;\cong\;\; CH_{S^n}(A)\mathop{\otimes}\limits_{A}^{\mathbb{L}} k & \cong & CH_{I^n}(A)
\mathop{\otimes}\limits_{CH_{S^{n-1}}(A)}^{\mathbb{L}} k
\end{eqnarray*}
Thus, for $n=1$, the Lemma is proved (by Definition~\eqref{eq:NBar}).

 Since $CH_{X}(k)\cong k$ for all $X\in \hTop$, by Corollary~\ref{C:properties}.(3), there are equivalences of $E_\infty$-algebras
\begin{eqnarray*}
 CH_{I}\Big(CH_{S^{n-1}}(A) \mathop{\otimes}\limits_{A}^{\mathbb{L}} k \Big)
&\cong & CH_{I}\big(CH_{S^{n-1}}(A) \big)\mathop{\otimes}\limits_{CH_{I}(A)}^{\mathbb{L}} k\\
&\cong & CH_{(I\times S^{n-1})/I\times\{1\}}(A).
\end{eqnarray*}
where the last equivalence follows from Corollary~\ref{C:properties}.(4) and the excision axiom.
Tensoring the last equivalence by
$\mathop{\otimes}\limits_{CH_{S^0}\big(CH_{S^{n-1}}(A,k)\big)}^{\mathbb{L}} \hspace{-1pc} k \hspace{0.5pc}$
and applying the excision axiom again, we get
\begin{eqnarray*} CH_{I}\Big(CH_{S^{n-1}}(A) \mathop{\otimes}\limits_{A}^{\mathbb{L}} k \Big) \mathop{\otimes}\limits_{CH_{S^0}\big(CH_{S^{n-1}}(A,k)\big)}^{\mathbb{L}} k
\;\;\cong \;\;CH_{S^n}(A)\mathop{\otimes}^{\mathbb{L}}_{A} k.
\end{eqnarray*}
Since the left hand side is $Bar\Big(CH_{S^{n-1}}(A) \mathop{\otimes}\limits_{A}^{\mathbb{L}} k  \Big)$, the Lemma now follows by induction.
\end{proof}

We now study the coalgebra structure carried by the iterated Bar construction.
Recall that the standard Bar construction of a DGA carries a natural associative
coalgebra structure. We wish to apply the results of Section~\ref{S:Edcochains}
to study the same result for $E_n$-coalgebras structures.

Recall the continuous map~\eqref{eq:pinchcube}
$ pinch: \mathcal{C}_n(r) \times S^n \longrightarrow  \bigvee_{i=1\dots r}\, S^n$.
Similarly to the definition of the map~\eqref{eq:pinchSr}, applying the singular set
functor to the map~\eqref{eq:pinchcube}  we get a morphism
\begin{multline}\label{eq:copinchSn}
pinch^{S^n,r}_*: C_{\ast}\big(\mathcal{C}_n(r)\big) \otimes CH_{S^n}(A)  \mathop{\otimes}\limits_{A}^{\mathbb{L}} k \\
\stackrel{pinch_*\otimes_{A}^{\mathbb{L}} id} \longrightarrow  CH_{\bigvee_{i=1}^r S^n}(A)  \mathop{\otimes}\limits_{A}^{\mathbb{L}} k
\cong \Big( CH_{\coprod_{i=1}^r S^n}(A)  \mathop{\otimes}\limits_{A^{\otimes r}}^{\mathbb{L}} A\Big) \mathop{\otimes}\limits_{A}^{\mathbb{L}} k \\
\cong \Big( CH_{\coprod_{i=1}^r S^n}(A) \Big)  \mathop{\otimes}\limits_{A^{\otimes r}}^{\mathbb{L}} k
\cong \Big( CH_{S^n}(A,k) \Big)^{\otimes r}
\end{multline}
where the last equivalences follows from the excision axiom, the coproduct axiom and the
 definition of $CH_{S^n}(A,k)$.

Note that there is a canonical equivalence
\begin{equation}\label{eq:coBarHH} Hom_{k}\Big(CH_{S^n}(A)\otimes_{A}^{\mathbb{L}} k, k\Big)\cong RHom_A\Big(CH_{S^n}(A),k\Big)\cong CH^{S^n}(A,k).\end{equation} Under this identification, the dual of the map~\eqref{eq:copinchSn} is the pinching map~\eqref{eq:pinchSr} from Section~\ref{S:wedge}.
\begin{theorem}\label{P:EncoAlgBar}Let $A$ be an $E_\infty$-algebra and $\epsilon: A\to k$ an augmentation.
\begin{enumerate}
\item The maps~\eqref{eq:copinchSn} $pinch^{S^n,r}_*: C_{\ast}\big(\mathcal{C}_n(r)\big) \otimes CH_{S^n}(A,k) \to   \Big( CH_{S^n}(A,k) \Big)^{\otimes r}$ makes  the iterated Bar construction $Bar^{(n)}(A)\cong  CH_{S^n}(A,k)$ a natural $E_n$-coalgebra (in the $(\infty,1)$-category of $E_\infty$-algebras)
\item The dual $E_n$-algebra $RHom(Bar^{(n)}(A),k)$ is naturally equivalent to $CH^{S^n}(A,k)$ in $E_n$-Alg and thus to the centralizer $\mathfrak{z}(\epsilon)$ of the augmentation (viewed as a map of $E_n$-algebra by restriction).
\end{enumerate}
\end{theorem}
\begin{proof}
The proof of the first statement is similar to the proof of Theorem~\ref{T:EdHoch} (except that we take the predual of it). Fixing $c\in C_{\ast}(\mathcal{C}_n(r))$,   all
maps involved in the composition~\eqref{eq:copinchSn} defining $pinch^{S^n,r}_*(c,-)$ are maps of $E_\infty$-algebras.  Hence the structure maps of the $E_n$-coalgebra structures are compatible with the $E_\infty$-structure.

Further, since the linear dual of the map~\eqref{eq:copinchSn} is the pinching map~\eqref{eq:pinchSr}, statement (2) follows from Theorem~\ref{T:EdHoch}, the equivalence
$$RHom(Bar^{(n)}(A),k) \cong RHom_A\Big(CH_{S^n}(A),k\Big)\cong  CH^{S^n}(A,k)$$
and Corollary~\ref{C:HH(AB)=z}
\end{proof}

If $Y$ is a pointed space, its $E_\infty$-algebra of cochains $C^\ast(Y)$ has a canonical augmentation $C^\ast(Y)\to C^\ast(pt)\cong k$ induced by the base point $pt\to Y$. Tensoring the map
$\mathcal{I}t: CH_{S^n}(C^\ast(Y)) \longrightarrow C^{\ast}\big( Y^{S^n} \big)$
 (given by Theorem~\ref{T:Einftymapping}) with $\otimes_{C^{\ast}(Y)}^{\mathbb{L}}k$ yields a natural $E_\infty$-algebra morphism
\begin{multline}
\label{eq:iteratedlooptoBar}
 \mathcal{I}t^{\Omega^n}: Bar^{(n)}(C^{\ast}(Y))\cong CH_{S^n}(C^{\ast}(Y),k)
 \\  \stackrel{\mathcal{I}t\otimes_{C^{\ast}(Y)}^{\mathbb{L}} k}\longrightarrow
C^{\ast}\big(Y^{S^n}\big)\otimes_{C^{\ast}(Y)}^{\mathbb{L}} k
 \longrightarrow C^{\ast}\big(\Omega^n(Y) \big)
\end{multline}
where the last map is induced by applying the singular cochain functor
 to $\Omega^n (Y)\cong Y^{S^n}\times^{h}_{Y} pt$.

Further, using the equivalence~\eqref{eq:coBarHH}, the linear dual of this map (composed with the canonical
 biduality morphism) yields a map
\begin{equation}
\label{eq:iteratedlooptoBarchain} \mathcal{I}t_{\Omega^n}:  C_{\ast}\big(\Omega^n(Y) \big)
 \longrightarrow  C^{\ast}\big(\Omega^n(Y)\big)^{\vee} \longrightarrow
 CH^{S^n}(C^{\ast}(Y),k)\cong  \Big(Bar^{(n)}(C^{\ast}(Y))\Big)^{\vee}
\end{equation}

We can now state our main application to iterated loop spaces,
 generalizing classical results in algebraic topology.
Since the iterated loop space $\Omega^n(Y)$ are $E_n$-algebras
in spaces,
 $C^{\ast}(\Omega^n(Y))$ is an $E_n-coalgebra$ in
$E_\infty$-Alg and $C_{\ast}(\Omega^n(Y))$ an  $E_n$-algebra (in $E_\infty$-coAlg, the
$(\infty,1)$-category of $E_\infty$-coalgebras).
\begin{cor}\label{C:iteratedlooptoBar} Let $Y$ be a pointed topological space.
\begin{enumerate}
\item  The map~\eqref{eq:iteratedlooptoBar}
$ \mathcal{I}t^{\Omega^n}: Bar^{(n)}(C^{\ast}(Y))
\to C^{\ast}\big(\Omega^n(Y)\big) $
is an $E_n$-coalgebra morphism in the category of $E_\infty$-algebras.
 It is further an equivalence
if $Y$ is $n$-connected.
\item Dually, the map~\eqref{eq:iteratedlooptoBarchain}
$\mathcal{I}t_{\Omega^n}:  C_{\ast}\big(\Omega^n(Y) \big) \longrightarrow
\Big(Bar^{(n)}(C^{\ast}(Y))\Big)^{\vee}$ is an $E_n$-algebra morphism (in $\hkmod$).
Further, if $k$ is a field, $Y$ is $n$-connected
and has finite dimensional homology groups, then $\Big(Bar^{(n)}(C^{\ast}(Y))\Big)^{\vee}$
is an $E_\infty$-coalgebra and the map~\eqref{eq:iteratedlooptoBarchain}
$\mathcal{I}t_{\Omega^n}$ is an equivalence of $E_n$-algebras in $E_\infty$-coAlg.
\end{enumerate}\end{cor}
In particular, the Hochschild chains over the spheres is a model for the natural $E_n$-algebra
 structure on $C_\ast(\Omega^n Y)$.
\begin{proof} By Theorem~\ref{T:Einftymapping}, the map
$\mathcal{I}t: CH_{S^n}(C^\ast(Y)) \longrightarrow C^{\ast}\big( Y^{S^n} \big)$ is an $E_\infty$-algebra
map and thus so is $\mathcal{I}t_{\Omega^n}$. Further, Theorem~\ref{T:Einftymapping}
gives a natural transformation
$$\mathcal{I}t: CH_{X}^{\ast}\big(C^{\ast}(Y)\big) \longrightarrow C^{\ast}\big(Y^X\big)$$
from which we deduce a commutative diagram
\begin{equation}\label{eq:ItfromCHtoBarisEinfty}
 \xymatrix{
CH_{I^n}\big(C^{\ast}(Y)\big)\hspace{-1pc} \mathop{\otimes}\limits_{CH_{S^{n-1}}
\big(C^{\ast}(Y)\big)}^{\mathbb{L}} \hspace{-1pc} C^{\ast}(Y)
\mathop{\otimes}\limits_{C^{\ast}(Y)}^{\mathbb{L}} k
 \ar[d]_{\mathcal{I}t\otimes_{\mathcal{I}t}^{\mathbb{L}}\,id }
\ar[r]^{\hspace{3pc}\simeq} &  CH_{S^n}\big(C^{\ast}(Y)\big) \mathop{\otimes}\limits_{C^{\ast}(Y)}^{\mathbb{L}} k
\ar[d]^{\mathcal{I}t\otimes_{C^{\ast}(Y)}^{\mathbb{L}} id}\\
C^{\ast}\big(Y^{I^n} \big)\hspace{-1pc}
\mathop{\otimes}\limits_{C^{\ast}\big(Y^{S^{n-1}}\big)}^{\mathbb{L}}
 \hspace{-1pc} C^{\ast}(Y)
\mathop{\otimes}\limits_{C^{\ast}(Y)}^{\mathbb{L}} k
 \ar[r] & C^{\ast}\big(Y^{S^n}\big)\mathop{\otimes}\limits_{C^{\ast}(Y)}^{\mathbb{L}} k }
\end{equation}
in $E_\infty$-Alg in which the horizontal arrows are induced by the homotopy pushout
 $\Omega^n X\cong X^{I^n}\cup_{X^{S^{n-1}}} pt$. The lower horizontal arrow is an equivalence when
$Y$ is $n$-connected. Further, the map $\mathcal{I}t:CH_{S^{n-1}}
\big(C^{\ast}(Y)\big)\to C^{\ast}\big(Y^{S^{n-1}}\big)$ is an equivalence when $Y$ is $n-1$-connected
by Theorem~\ref{T:Einftymapping}. Since the map induced by the base point
$C^{\ast}(Y)\to CH_{I^n}(C^{\ast}(Y))$ is an equivalence, the map
$\mathcal{I}t:CH_{I^n}(C^{\ast}(Y)) \to C^{\ast}\big(Y^{I^n}\big)$ is an equivalence when $Y$
is connected.
Thus, we deduce from the commutativity of diagram~\eqref{eq:ItfromCHtoBarisEinfty}
 that the map $\mathcal{I}t^{\Omega^n}: Bar^{(n)}(C^{\ast}(Y))
\to C^{\ast}\big(\Omega^n(Y)\big)$ is an equivalence when $Y$ is $n$-connected.

In order to finish the proof of Assertion 1 in Corollary~\ref{C:iteratedlooptoBar}, it remains
to check that $\mathcal{I}t^{\Omega^n}$ is a map of $E_n$-coalgebras. By definition, the
$E_n$-coalgebra structure of $C^{\ast}\big(\Omega^n (Y)\big)$ is induced by taking the singular
cochains functor (from $\hTop$ to $E_\infty$-Alg) to the $E_n$-algebra structure of
$\Omega^n(Y)$ which is the (homotopy pullback) $\Omega^n(Y)\cong \big(Y^{S^n}\times_{Y} pt\big)$.
By definition the $E_n$-algebra structure of $\Omega^n(Y)$ is induced by the pinching
map~\eqref{eq:pinchSr} $\mathcal{C}_n(r)\times S^n\to \bigvee_{i=1\dots r} S^n$. Indeed, since
the pinching map preserves the base point of $S^n$, we have the following composition
\begin{multline}\label{eq:pinchOmegan}
 \mathcal{C}_n(r) \times \big(Y^{S^n}\times^{h}_{Y} pt\big)^{r}
  \stackrel{\cong}\longrightarrow \mathcal{C}_n(r) \times
\big(Y^{\coprod_{i=1\dots r} S^n}\big) \times_{Y^{r}} pt
\\ \stackrel{\cong}\longrightarrow \big(Y^{\bigvee_{i=1\dots r} S^n}\big)\times_{Y} pt
\stackrel{pinch^*} \longrightarrow  Y^{S^n}\times_{Y} pt.
\end{multline}
By naturality of $\mathcal{I}t$, we have a commutative diagram
\[\xymatrix{ CH_{S^n}\big(C^{\ast}(Y)\big) \mathop{\otimes}\limits_{C^{\ast}(Y)}^{\mathbb{L}} k
\ar[rr]^{\hspace{-0.5pc}pinch_*\mathop{\otimes}\limits_{C^{\ast}(Y)}^{\mathbb{L}} id}
\ar[d]_{\mathcal{I}t\otimes_{C^{\ast}(Y)}^{\mathbb{L}} id}& &
CH_{\bigvee_{i=1\dots r}S^n}\big(C^{\ast}(Y)\big)
 \mathop{\otimes}\limits_{C^{\ast}(Y)}^{\mathbb{L}} k \ar[d]^{\mathcal{I}t\otimes_{C^{\ast}(Y)}^{\mathbb{L}} id}\\
C^{\ast}\big(Y^{S^n}\big)\mathop{\otimes}\limits_{C^{\ast}(Y)}^{\mathbb{L}} k
 \ar[rr]^{C^{\ast}(pinch^*)} &   &
C^{\ast}\big(Y^{\bigvee_{i=1\dots r} S^n}\big)\mathop{\otimes}\limits_{C^{\ast}(Y)}^{\mathbb{L}} k }.
 \]
The commutativity of this diagram, together with the  definition of the
map~\eqref{eq:copinchSn} $pinch_*^{S^n,r}:C_{\ast}\big(\mathcal{C}_n(r)\big)
 \otimes CH_{S^n}(A)  \mathop{\otimes}\limits_{A}^{\mathbb{L}} k \to
\Big(CH_{S^n}(A)  \mathop{\otimes}\limits_{A}^{\mathbb{L}} k\Big)^{\otimes r}$ giving the $E_n$-coalgebra
 structure of $Bar^{(n)}(C^{\ast}(Y))$, and the fact that the $E_n$-coalgeba structure of
$C^{\ast}(\Omega^n(Y))$ is given by applying the functor $C^{\ast}(-)$ to the
composition~\eqref{eq:pinchOmegan} show that $\mathcal{I}t^{\Omega^n}$ is an $E_n$-algebra map.

\smallskip

The proof of the fact that $\mathcal{I}t_{\Omega^n}$ is a map of $E_n$-algebra is similar, using
in addition the naturality of the biduality morphism $C\to C^{\vee \vee}$ and
Corollary~\ref{C:Itdual}. Further, when $k$ is a field and the groups $H_{\ell}(Y)$ are finitely
generated, then $C_{\ast}(Y)\to (C^{\ast}(Y))^{\vee}$ is an equivalence. Further, if $Y$ is
$n$-connected, it follows from the Eilenberg-Moore spectral sequence  that $Bar^{(n)}(C^{\ast}(Y))$ has finite
dimensional homology groups. Hence, the dual $\Big(Bar^{(n)}(C^{\ast}(Y))\Big)^{\vee}$ inherits an
natural $E_\infty$-coalgebra structure
(dual of the $E_\infty$-algebra structure of $Bar^{(n)}(C^{\ast}(Y))$). It is then immediate to
check that the arguments to prove Statement (1) above can be dualized to prove that
$\mathcal{I}t_{\Omega^n}$ is also an equivalence of $E_\infty$-coalgebras.
\end{proof}
\begin{rem}\label{R:EilenbergMooreconvergenceisenough}
 A careful analysis of the proof of Corollary~\ref{C:iteratedlooptoBar} shows that the assumption that $Y$ is $n$-connected can be replaced by the assumption that the cohomological Eilenberg-Moore spectral sequence of the path space fibration is strongly convergent for all $\Omega^{i}Y$ ($i\leq n$).
\end{rem}

\begin{rem}\label{R:CobarandHHforCoAlg}
Statement (2) in Corollary~\ref{C:iteratedlooptoBar} is somehow unsatisfying since one recovers an $E_\infty$-coalgebra structure on the right hand side $\Big(Bar^{(n)}(C^{\ast}(Y))\Big)^{\vee}$ only when the biduality morphism $Bar^{(n)}(C^{\ast}(Y))\to \Big(Bar^{(n)}(C^{\ast}(Y))\Big)^{\vee \vee}$ is an equivalence (while the left hand side has always such a structure). The reason for it is that this statement is in fact the bidual of a statement involving iterated \emph{coBar} construction of $E_\infty$-coalgebras.

Indeed, one can define Hochschild cochains over spaces for
$E_\infty$-coalgebras in a similar way to what we do in
Section~\ref{S:HHforEinftyAlg} getting an $\infty$-functor
$CH:\hTop^{op}\times E_\infty\text{-coAlg} \to
E_\infty\text{-coAlg}$ ($(X,C)\mapsto CH^X(C)$). For instance, one
has a natural equivalence $CH^{X}(C) \cong C
\mathop{\otimes}\limits^{\mathbb{L}}_{\mathbb{E}_\infty^{\otimes}}
C^\ast(X)$ similar to Proposition~\ref{P-CH-coeq}.

All results of Section~\ref{S:HHforEinftyAlg},
Section~\ref{S:EinftyModel} and Section~\ref{S:Operation} have
``dual'' counterparts which can be proved similarly. We claim that
there is an iterated cobar construction $coBar^{(n)}:
E_\infty\text{-coAlg}\to E_n\text{-Alg}(E_\infty\text{-coAlg})$
defined similarly to this Section~\ref{S:Bar} and that further there
is a natural  $E_n$-algebra map $coBar^{(n)}(C_{\ast}(Y)) \to
C_{\ast}(\Omega^n(Y))$ in $E_\infty$-coAlg which is an equivalence
when $Y$ is $n$-connected. We leave the many details to be filled in
to the interested reader.
\end{rem}

\subsection{Iterated Bar constructions of augmented $E_n$-algebras} \label{SS:BarforEn}
In this section we explain how to generalize the iterated Bar construction for $E_\infty$-algebras in \S~\ref{S:Bar}
 to $E_n$-algebras. In particular we describe the $E_n$-coalgebra structure of the $n$-times iterated Bar construction.
  Our definition and study of the Bar construction follows the ones given by Francis~\cite{F} and Lurie~\cite{L-HA}.

\smallskip

\subsubsection{Definition of the iterated Bar construction for $E_n$-algebras}\label{SS:DefBarIteratedforEn}
In this section we assume $A$ is an augmented $E_n$-algebra and we
denote $\epsilon: A\to k$ the augmentation (which is a map of
$E_n$-algebras). In particular, we endow $k$ with its structure of
$A$-$E_n$-module given by the augmentation. We denote
\emph{$E_n\text{-Alg}^{aug}$ the $(\infty,1)$-category of augmented
$E_n$-algebras}.

 For an augmented $E_n$-algebra, Definition~\eqref{eq:NBar} and Lemma~\ref{L:Bar=HochI} suggest to define
\begin{equation} \label{eq:DefnBarEm} Bar(A) := \int_{D^1 \times \mathbb{R}^{n-1}}A
\mathop{\otimes}^{\mathbb{L}}_{\int_{S^0\times \mathbb{R}^{n-1}} A} k\end{equation}
where $k\cong \int_{I\times \mathbb{R}^{n-1}} k$ is endowed with its natural structure of $A$-$E_1$-module.
 This definition agrees with the usual one:
\begin{lem}[Francis~\cite{F}] \label{L:Bar=BarstdEm}
 There is a natural equivalence (in $\hkmod$)
\begin{equation} \label{eq:Bar=HochIforEn} Bar(A)
\;\cong \; Bar^{std}(A) \; \cong \; k\otimes_{A}^{\mathbb{L}} k\end{equation} where $Bar^{std}(A)$ is
the
standard Bar construction as in \S~\ref{S:Bar}.
\end{lem}
When $X$ be a manifold of dimension $d$ equipped with a framing of
$X\times \mathbb{R}^k$, then for any $E_{d+k}$-algebra $B$,
$\int_{X\times \mathbb{R}^k} B$ is canonically an $E_k$-algebra,
see~\cite{L-HA, F} for details. Note that this follows from
Theorem~\ref{T:Theorem6GTZ2} and the fact that  factorization
algebras on $X\times \R^{k}$ are the same as factorization algebras
on $X$ with values in $E_{d}\text{-Alg}$ (see
Theorem~\ref{P:En=Fact} or~\cite{GTZ2}).   Applying this observation
to $X=I$ or $X=S^0$ we get the following result which is also proved
in~\cite{F, L-HA}.

\begin{prop}\label{P:BarEkisEkminus1} The Bar construction~\eqref{eq:Bar=HochIforEn} for augmented $E_m$-algebras
 ($m\geq 1$)
has a canonical lift  $$Bar: E_m\text{-Alg}^{aug} \to
E_{m-1}\text{-Alg}^{aug}$$ which coincides for $E_\infty$-algebras
with the one given in \S~\ref{S:Bar} and further sits into a
commutative diagram
\begin{equation*}
 \xymatrix{E_1\text{-Alg} \ar[d]^{Bar} & E_2\text{-Alg}  \ar[d]^{Bar}\ar[l]& \cdots \ar[l]&
 E_m\text{-Alg}\ar[l] \ar[d]^{Bar}  & \cdots \ar[l]& E_\infty\text{-Alg} \ar[l]  \ar[d]^{Bar} \\
 \hkmod  & E_1\text{-Alg} \ar[l]& \cdots\ar[l] & E_{m-1}\text{-Alg} \ar[l] & \cdots \ar[l]& E_\infty\text{-Alg} \ar[l]}
\end{equation*}
where the horizontal arrows are the canonical forget functors induced by the tower of maps
of operads~\eqref{eq:towerofEnoperad}.
\end{prop}
\begin{proof}
 By  Theorem~\ref{T:Theorem6GTZ2} (and the above observation which is a special case of  of the Fubini formula for
  factorization homology~\cite[Corollary 17]{GTZ2} ), we have that $\int_{S^0\times \mathbb{R}^{m-1}} A $,
 $\int_{I\times \mathbb{R}^{m-1}} A$ and $k \cong \int_{I\times \mathbb{R}^{m-1}} A$ are (global sections of)
  locally constant factorizations algebras
 over $\mathbb{R}^{m-1}$. Here, we see $S^0$ as being the boundary of
 the closed interval $I=[-1,1]$ which is framed (we choose the framing so that the induced orientation is the natural
 one); in particular $S^0\cong \{-1,1\}$ inherits a framing as well (note that the two points in $S^0$ get opposite
 orientation this way).
 In particular  $S^0\times \mathbb{R}^{m-1}$ is equipped with the product
 framing.
Since $S^0\times \mathbb{R}^{m-1}$ is a framed $m-1$ dimensional manifold and $A$ an $E_m$-algebra,  its factorization
homology with value in $A$ is the one of   the product of framed manifolds $S^0\times \mathbb{R}^{m-1} \times \mathbb{R}$.
Hence, we have that
   $\int_{S^0\times \mathbb{R}^{m-1}} A $ is in fact an $E_m$-algebra, that is a locally constant factorization algebra
    over
  $ \mathbb{R}^{m-1}\times \mathbb{R}$.

  In particular, using Theorem~\ref{T:Dunn},  $\int_{S^0\times \mathbb{R}^{m-1}} A
  $ is naturally an $E_1$-algebra in the symmetric monoidal category of $E_{m-1}$-algebras,
  \emph{i.e.}, an $E_1$-algebra in the category of
  locally constant factorizations algebras over $\mathbb{R}^{m-1}$.

  Similarly  $\int_{I\times \mathbb{R}^{m-1}} A$ is a left module over $\int_{S^0\times \mathbb{R}^{m-1}} A $
 in the symmetric monoidal category  locally constant factorizations algebras over $\mathbb{R}^{m-1}$
 (or equivalently of $E_{m-1}$-algebras).  In other words,
  it belongs to
  $\big(\int_{S^0\times \mathbb{R}^{m-1}} A\big)\text{-}Mod^{E_1}\big(\text{Fac}^{lc}_{\mathbb{R}^{m-1}}\big)$
  which is equivalent to  $\big(\int_{S^0\times \mathbb{R}^{m-1}} A\big)\text{-}Mod^{E_1}\big(E_{m-1}\text{-Alg}\big)$.
  Since the same holds for $k$, we obtain that the Bar construction is an object in $\text{Fac}^{lc}_{\mathbb{R}^{m-1}}$,
  hence
 inherits a structure of $E_{m-1}$-algebra.

 Further,   the augmentation
 $\epsilon: A\to k$ induces a maps  $\int_{I\times \mathbb{R}^{m-1}} \epsilon: \int_{I\times \mathbb{R}^{m-1}} A\to k$
 which is a map of locally constant factorization algebras on $\mathbb{R}^{m-1}$ hence of $E_{m-1}$-algebras.
 Similarly $\int_{S^0\times \mathbb{R}^{m-1}} A \to k$ is a map of $E_{m}$-algebras; hence $\epsilon$
 induces an augmentation $Bar(A)\to k$ in $E_{m-1}\text{-Alg}$. The equivalence of the two definitions for
  $E_\infty$-algebras is an immediate consequence of Theorem~\ref{T:CH=TCH} or~\cite[Theorem 5]{GTZ2}.

  The commutativity of the diagram follows from the fact that $E_m\text{-Alg} \to E_{m-1}\text{-Alg}$ is induced
   by the map of $\infty$-operad $\mathbb{E}_{m-1}^{\otimes} \to  \mathbb{E}_{m}^{\otimes}$ induced
   by taking the product of $m-1$-dimensional disks with the interval $\mathbb{R}$, \emph{i.e},
   it is induced by the pushforward of factorization algebras along the projection
    $\mathbb{R}^{m-1}\times \mathbb{R}\to \mathbb{R}^{m-1}$.
 \end{proof}
 By Proposition~\ref{P:BarEkisEkminus1}, we can iterate (up to $m$-times) the Bar constructions of an $E_m$-algebra.
 \begin{definition}\label{D:BariteratedforEm} Let $0\leq n\leq m$.
The $n^{th}$-iterated bar construction of an augmented $E_m$-algebra $A$ is the $E_{m-n}$-algebra
 (given by Proposition~\ref{P:BarEkisEkminus1}) $$Bar^{(n)}(A) := Bar(\cdots (Bar(A))\cdots)$$
which is the value on $A$ of the ($n$-fold iterated Bar) functor:
$Bar^{\circ n}: E_m\text{-Alg}^{aug} \longrightarrow
E_{m-n}\text{-Alg}^{aug}$.
\end{definition}
Proposition~\ref{P:BarEkisEkminus1} implies that Definition~\ref{D:BariteratedforEm}
agrees with Definition~\ref{D:Bar} for $E_\infty$-algebras.
\begin{rem}
The iterated Bar construction given in Definition~\ref{D:BariteratedforEm} should be closely related to the one
(obtained at the level of model categories) by Fresse~\cite{FresseBarEn}.
\end{rem}

 The following result, due to Francis~\cite[Lemma 2.44]{F}, identifies the iterated Bar construction
  in terms of factorization homology
 \begin{lem}[Francis]\label{L:iteratedBarforEnisCH} Let $A$ be an $E_m$-algebra and $0\leq n\leq m$.
 There is a natural equivalence of $E_{m-n}$-algebras
 $$Bar^{(n)}(A) \cong
 \int_{D^m\times \mathbb{R}^{m-n}} A\mathop{\otimes}^{\mathbb{L}}_{\int_{S^{m-1}
 \times\mathbb{R}^{m-n+1}} A} k$$
 \end{lem}
\begin{proof}
 This is essentially Lemma 2.44 together with Corollary 3.32 in~\cite{F}.
  Alternatively, on can use a proof similar to the one of Proposition~\ref{P:nBarisCHIn}
  replacing $CH_{I^n}(A)$ with the $E_{m-n}$-algebra $\int_{I^m\times \mathbb{R}^{m-n}} A$
  using excision for factorization homology (see~\cite{F, AFT, GTZ2}),
  and the Fubini theorem for factorization homology~\cite[Corollary 17]{GTZ2}
  instead of Corollary~\ref{C:properties}.(4).
\end{proof}

\medskip

\subsubsection{Factorization algebra models for iterated Bar
construction}\label{SS:BarasFactAlg}
 In this section, we show that the iterated
Bar construction can be   computed as factorization homology of a
stratified factorization algebra on the closed $n$-disk and the
$n$-sphere as well.

\smallskip

Identify $I^n=[-1,1]^n$ with the closed unit disk in $\mathbb{R}^n$
and let $D^n =I^n\setminus \partial I_n$ be its interior.  We
consider $I^n$ as a stratified space with two strata, one  of which
is its boundary $\partial I^n$ (of codimension 1) and the remaining
open strata being $D^n$. A factorization algebra $\mathcal{F}$ on
the stratified disk is thus locally constant if for any inclusion of
disks
 $U\subset V\subset D^n$ and for any inclusion of half-disks $U\subset V$ where $V\not\subset D^n$,
 the structure map $\mathcal{F}(U)\to \mathcal{F}(V)$ is a quasi-isomorphism.

Let $\epsilon :A\to k$ be an augmented $E_n$-algebra which we may assume to be given by a map
$\epsilon: \mathcal{A}\to {k}$ of factorization algebras.
 By~\cite[Proposition 30 and Proposition 31]{G-Houches}\footnote{the reader shall be aware that the notation
  $D^n$ in \cite{G-Houches} is what we denote $I^n$ in the present paper},
  any map of $E_n$-algebra $f:A\to B$  defines a locally constant stratified  factorization algebra on
  $I^n$. Indeed, by \emph{loc. cit.} we have a faithful functor
  \begin{equation}\label{eq:DefUpsilon}\Upsilon: {Hom}_{E_n\text{-Alg}} \longrightarrow
  \text{Fac}^{lc}_{I^n}
\end{equation}
between the $(\infty,1)$-categories of $E_n$-algebras morphisms and
stratified locally constant factorization algebras on $I^n$.
\begin{definition}\label{D:BarasFactAlg} We let  $(A,k)$  be the locally constant stratified factorization algebra
on $I^n$ defined by the  augmentation $\epsilon: \mathcal{A}\to {k}$
that is $(A,k)=\Upsilon (\epsilon)$.
\end{definition}
Since $\Upsilon$ is a functor, then $(A,k)$ is functorial with
respect to maps of augmented $E_n$-algebras. More precisely, we have
the faithful functor
\begin{equation}\label{eq:DefUpsilontilde}
\tilde{\Upsilon}: E_n\text{-Alg}^{aug}\longrightarrow
{Hom}_{E_n\text{-Alg}} \longrightarrow
  \text{Fac}^{lc}_{I^n}.
\end{equation}
\begin{rem}\label{R:BarasFactAlgebraHomology}The factorization algebra  $(A,k)$  is explicitly described as follows.
 Let $\mathcal{U}_{I^n}$ be the (factorizing) basis of opens consisting of all  open subset $U\subset D^n\subset I^n$,
 and all open half-disk $D$; recall that
  we call a half-disk of $I^n$ an open $D\subset I^n$ such that is there is an homeomorphism
  $\theta:D\stackrel{\cong}\to \tilde{D}\times [0,1)$ with
  $D\cap \partial I^n=\theta^{-1}(\tilde{D}\times \{0\})$.
\begin{lem}\label{L:StructureofAk}
The factorization algebra $(A,k)$ of Definition~\ref{D:BarasFactAlg}
satisfies that:
\begin{enumerate}
\item for any $U\subset D^n\subset I^n$, one has $(A,k)(U)=\mathcal{A}(U)$;
\item for any half-disk $D\in \mathcal{U}_{I^n}$, one has $(A,k)(D)=
k(D)=k$.
\item The restriction to $\mathcal{U}_{I^n}$ of the structure maps $\rho_{U_1,\dots, U_{r}, V}:
\bigotimes_{i=1}^{r} (A,k)(U_i)\to (A,k)(V)$  is given as
follows: if $U_1,\dots, U_r\in \mathcal{U}_{I^n}$ are pairwise
disjoint open  lying in $V\in \mathcal{U}_{I^n}$, then
\begin{itemize}
\item if $V\subset D^n$, one has $$\rho_{U_1,\dots, U_{r}, V}=
\bigotimes_{i=1}^{r} (A,k)(U_i) \cong \bigotimes_{i=1}^{r}
\mathcal{A}(U_i)\stackrel{\rho^{\mathcal{A}}_{U_1,\dots, U_{r},
V}}\longrightarrow (A,k)(V)$$  where the last map is the structure
map of the factorization algebra $\mathcal{A}$;
\item if $V$ is a half-disk, $U_1,\dots, U_i \in D^n$ ($0\leq i\leq r$) and $U_{i+1},\dots, U_{r}$ are
half-disks, one has
\begin{multline*} \rho_{U_1,\dots, U_{r},
V}=\bigotimes_{j=1}^{r} (A,k)(U_j)\cong
\Big(\bigotimes_{j=1}^i\mathcal{A}(U_j)\Big) \otimes
\Big(\bigotimes_{j=i+1}^{r} k(U_j)\Big)\\
\stackrel{(\bigotimes_{j=1}^i\epsilon)\otimes id}\longrightarrow
\bigotimes_{j=1}^{r} k(U_i)\stackrel{\rho^{k}_{U_1,\dots, U_{r}, V}}
\longrightarrow k
\end{multline*} where the last map is the structure
maps of the factorization algebra associated to $k$
(Example~\ref{E:kasFact}).
\end{itemize}
\end{enumerate}
Further, any $\mathcal{U}_{I^n}$-prefactorization algebra (in
particular any factorization algebra on on $I^n$)  satisfying (1),
(2) and (3) above is a $\mathcal{U}_{I^n}$-factorization algebra and
equivalent to $(A,k)$.
\end{lem}
\begin{proof}
Since $\mathcal{U}_{I^n}$ is a factorizing, stable by finite
intersection, basis of opens, the uniqueness in the last claim
follows from Proposition~\ref{P:extensionfrombasis}. Since $(A,k)$
is a factorization algebra, then it is a
$\mathcal{U}_{I^n}$-factorization algebra and the last claim now
follows from the first one.

The restriction to $D^n=I^n\setminus \partial I^n$ of $(A,k)$ is
just the factorization algebra $\mathcal{A}$ by~\cite[Proposition
30]{G-Houches}. Now,  by~\cite[Example 37 6]{G-Houches}, the
structure maps when $V$ is a half-disk is given by the left module
structure of $\int_{(\partial I^n)\times (-\infty,0]} k$ over
$int_{(\partial I^n)\times (-\infty,0]}A$. The formula for the
structure maps now follows from~\cite[Corollary 6]{G-Houches} and
the factorization algebra formula for left modules as
in~\cite[\S~6.1]{G-Houches}.
\end{proof}
\end{rem}
The ground ring $k$ is trivially augmented and the augmentation
$\epsilon:A\to k$ is a map of augmented algebras. Applying the
functor $\tilde{\Upsilon}$ defined by the
composition~\eqref{eq:DefUpsilontilde} above, we get the following
map of factorization algebras:
\begin{equation}
\tilde{\Upsilon}(\epsilon): (A,k) \longrightarrow (k,k).
\end{equation}
An immediate consequence of Lemma~\ref{L:BarasFactIn} and
Example~\ref{E:kasFact} is that$(k,k)=k$. We thus obtain
\begin{lem}\label{L:Akisaugmented} Let $\epsilon:A\to k$ be an
augmented $E_n$-algebra. \begin{itemize} \item The factorization
algebra $(A,k)$ has a canonical augmentation
$(A,k) \stackrel{\tilde{\Upsilon}(\epsilon):}
\longrightarrow (k,k)=k$.
\item If $D\in\mathcal{U}_{I^n}$, then $\underline{\epsilon}(D): (A,k)(D) \to
k(D)=k$ is equal to $\epsilon(D)(D):\mathcal{A}(D)\to k$ if $D\subset D^n$ and to the identity
otherwise.
\end{itemize}\end{lem}
\begin{proof}
The first claim was established right before the lemma. The second claim is a direct derivation of the construction of  the functor~\eqref{eq:DefUpsilon}: $\Upsilon: {Hom}_{E_n\text{-Alg}} \longrightarrow
  \text{Fac}^{lc}_{I^n}$ in \cite{G-Houches} and Lemma~\ref{L:BarasFactIn}.
\end{proof}
\medskip

Assume now that $\epsilon: A\to k$ is an augmented $E_m$-algebra,
with $m\geq n$ and let again $\epsilon: \mathcal{A}\to k$ be a map
of locally constant factorization algebras representing it. By
Theorem~\ref{P:En=Fact}, then $\epsilon:\mathcal{A}\to k$ can be
seen as a map in
$E_{n}\text{-Alg}(\text{Fac}^{lc}_{\mathbb{R}^{m-n}})\cong
E_{n}\text{-Alg}(E_{m-n}\text{-Alg})$. From
Definition~\ref{D:BarasFactAlg}, the factorization algebra $(A,k)$
then belongs to $\text{Fac}^{lc}_{I^n}(E_{m-n}\txt{-Alg})$.
\begin{lem}\label{L:BarasFactIn} Let $\epsilon:A\to k$ be an augmented $E_m$-algebra and
$(A,k)\in \text{Fac}^{lc}_{I^n}(E_{m-n}\text{-Alg})$ be the
associated factorization algebra (Definition~\ref{D:BarasFactAlg}).
The factorization homology of $(A,k)$ is equivalent (naturally with
respect to maps of augmented $E_m$-algebras) in $E_{m-n}\text{-Alg}$
to the iterated bar construction of $A$:
$$ p_*(A,k) \stackrel{\simeq}\longleftarrow \int_{I^{n}\times \mathbb{R}^{m-n}} (A,k) \cong Bar^{(n)}(A).
$$
\end{lem}
\begin{proof}
A similar result can be found in~\cite{F, AFT}. Let $q: I^n \to [0,
1]$ be the supremum norm map: $(x_1,\dots,x_n)\mapsto \max(|x_i|)$.
We thus have the factorization algebra $q_*(A,k)\in
\text{Fac}_{[0,1]}(E_{m-n}\text{-Alg}) \cong
\text{Fac}_{[0,1]}(\text{Fac}^{lc}_{\mathbb{R}^{m-n}})$. By
\cite[\S~6.1]{G-Houches}, $q_*(A,k)$ is a stratified locally
constant. Here, we see $[0,1]$ as being stratified with two
$0$-dimensional strata given by the points $\{0\}$ and $\{1\}$. The algebra
corresponding to the open dimension 1 stratum (as in
\cite[Proposition 26]{G-Houches}) is $\int_{(\partial I^n)\times
(-1,1)\times \mathbb{R}^{m-n}} A$, while the right module
corresponding to the stratum $\{0\}$ is $A\cong \int_{D^n \times
\mathbb{R}^{m-n}}$ and the left module corresponding to $\{1\}$ is
$k=\int_{(\partial I^n)\times \mathbb{R}^{m-n}} k$ (by
Example~\ref{E:kasFact}).

 By definition, the factorization homology of $(A,k)$
is the same as the factorization homology of $q_*(A,k)\in
\text{Fac}^{lc}_{[0,1]}(E_{m-n}\text{-Alg})$: \begin{eqnarray*}
p_*(A,k) \cong \int_{I^n\times \mathbb{R}^{m-n}} (A,k) &\cong&
\int_{[0,1]} q_*(A,k) \\ &\cong& \int_{D^m\times \mathbb{R}^{m-n}}
A\mathop{\otimes}^{\mathbb{L}}_{\int_{S^{m-1}
 \times\mathbb{R}^{m-n+1}} A} k \end{eqnarray*}
 where the last line comes from~\cite[Proposition 26]{G-Houches}.
 All the equivalences are further natural with respect to augmented
 $E_m$-algebras maps since $p_*$, $q_*$ are functors and by
 \emph{loc. cit.}.
 The result now
 follows from Lemma~\ref{L:iteratedBarforEnisCH}.
\end{proof}

\medskip

We now derive another factorization algebra model for the Bar
construction. Let $\widehat{D^n}=S^n$ be the one point
compactification of $D^n$ and let $\kappa:I^n \to S^n =
\widehat{D^n}$ be the canonical projection collapsing the boundary
$\partial I^n$ to a point. We endow $\widehat{D^n}=S^n$ with the
stratification with one dimension $0$ stratum given by the point at
infinity and one dimension $n$ stratum. This way,  $\kappa:I^n \to
S^n = \widehat{D^n}$ is a map of stratified spaces (that is maps
strata onto strata).

\begin{definition}\label{D:BarasFactAlgonSn}
We let $\widehat{\mathcal{A}}\in \text{Fac}_{S^n}$ be the
factorization algebra $\kappa_*((A,k))$ obtained by pushforward
along $\kappa:I^n \to S^n = \widehat{D^n}$  of the factorization
algebra $(A,k)$ of Definition~\ref{D:BarasFactAlg}.
\end{definition}
The pushforward  $\kappa_*(k)$ of the trivial factorization algebra  $k$ on $I^n$ is equal to the trivial factorization algebra $k$ on $\widehat{D^n}$. Hence the pushforward along $\kappa$ of the augmentation of $(A,k)$ (Lemma~\ref{L:Akisaugmented}), that is the  map
\begin{equation}\label{eq:AugmentationHat}
\widehat{\epsilon}: \widehat{A}=\kappa_*((A,k)) \stackrel{\kappa_*\big(\tilde{\Upsilon}(\epsilon)\big)}\longrightarrow
k
\end{equation}
is an augmentation for $\widehat{\mathcal{A}}$.

\medskip

Assume now that $\epsilon: A\to k$ is an augmented $E_m$-algebra,
with $m\geq n$ and let again $\epsilon: \mathcal{A}\to k$ be a map
of locally constant factorization algebras representing it. We then
get that the factorization algebra $\widehat{\mathcal{A}}$
belongs\footnote{by abuse of notation we keep the notation
$\widehat{\mathcal{A}}$ for the factorization algebra
$\pi_*(\widehat{\mathcal{A}})\in
\text{Fac}_{S^n}(\text{Fac}_{\mathbb{R}^{m-n}})$} to
$\text{Fac}_{S^n}(\text{Fac}_{\mathbb{R}^{m-n}})$.
\begin{lem}\label{L:BarasFactSn}   Let $\epsilon:A\to k$ be an augmented $E_m$-algebra ($m\geq n$) represented by a map
 $\epsilon:\mathcal{A}\to k$ of factorization algebras and
 $\widehat{\mathcal{A}}$ be given by
 Definition~\ref{D:BarasFactAlgonSn}.

The factorization homology of $\widehat{\mathcal{A}}$ is equivalent
in $E_{m-n}\text{-Alg}$ to the iterated bar construction of $A$:
$$ p_*(\widehat{\mathcal{A}}) \cong \int_{I^{n}\times \mathbb{R}^{m-n}} (A,k) \stackrel{\simeq}\longleftarrow Bar^{(n)}(A).
$$This equivalence is natural with respect to maps of augmented
$E_m$-algebras.
\end{lem}
\begin{proof}
Since $p_*$, $\kappa_*$ are functors, by Lemma~\ref{L:BarasFactIn}
we have an natural equivalence
$$ p_*(\widehat{\mathcal{A}})= p_*(\kappa_*((A,k)))=
p_*((A,k))\stackrel{\simeq}\longleftarrow \int_{I^{n}\times \mathbb{R}^{m-n}} (A,k)
  \cong Bar^{(n)}(A).$$
\end{proof}

We now describe in more details $\widehat{\mathcal{A}}$. Recall that
we see $S^n=\widehat{D^n}=D^n \cup \{\infty\}$ as a stratified space
with one stratum being $D^n\subset \widehat{D^n}$ and the other one
being its point at infinity. A basis of neighborhood of $\infty$ is
given by the complements of closed \emph{Euclidean} disk centered at
$0$ in $D^n$. In particular, we have a factorizing and stable by
finite intersection basis  $\mathcal{U}_{\widehat{D^n}}$ consisting
of all opens $U\subset D^n\subset \widehat{D^n}$ and all
 opens which are the complement $\widehat{D^n}\setminus \overline{D}$
 of a non-empty Euclidean disk $\overline{D}\subset D^n$ whose center is $0$.
\begin{prop}\label{P:BarasFactonSn}
The factorization algebra $\widehat{\mathcal{A}}$ of
Definition~\ref{D:BarasFactAlgonSn} satisfies that:
\begin{enumerate}
\item for any $U\subset D^n\subset \widehat{D^n}=S^n$, one has $\widehat{\mathcal{A}}(U)=\mathcal{A}(U)$;
\item for any compact disk\footnote{by a compact disk of $D^n$, we mean the image  in $D^n$ of an embedding of the closed unit Euclidean disk} $\overline{D}\subset D^n$,  one has as an natural equivalence
$$\widehat{\mathcal{A}}\big(\widehat{D^n}\setminus \overline{D}\big) \cong
k\big(\widehat{D^n}\setminus \overline{D}\big)=k.$$
\item The restriction to $\mathcal{U}_{\widehat{D^n}}$ of the structure maps $\rho_{U_1,\dots, U_{r}, V}:
\bigotimes_{i=1}^{r} \widehat{\mathcal{A}}(U_i)\to
\widehat{\mathcal{A}}(V)$ is given as follows: if $U_1,\dots, U_r\in
\mathcal{U}_{\widehat{D^n}}$ are pairwise disjoint open lying in
$V\in \mathcal{U}_{\widehat{D^n}}$, then
\begin{itemize}
\item if $V\subset D^n$, one has $$\rho_{U_1,\dots, U_{r}, V}=
\bigotimes_{i=1}^{r} \widehat{\mathcal{A}}(U_i) \cong
\bigotimes_{i=1}^{r}
\mathcal{A}(U_i)\stackrel{\rho^{\mathcal{A}}_{U_1,\dots, U_{r},
V}}\longrightarrow \widehat{\mathcal{A}}(V)$$  the last map being
the structure map of the factorization algebra $\mathcal{A}$;
\item if $V=\widehat{D^n}\setminus \overline{D}$, where $D$ is an Euclidean disk centered at $0$,
 $U_1,\dots, U_i \in D^n$ ($0\leq i\leq r$) and $U_{i+1},\dots, U_{r}$ are
complements of Euclidean disks centered at $0$\footnote{note that the $U_i$\rq{}s being disjoint implies $i=r$ or $r-1$},
 one has
\begin{multline*} \rho_{U_1,\dots, U_{r},
V}=\bigotimes_{j=1}^{r} \widehat{\mathcal{A}}(U_j)\cong
\Big(\bigotimes_{j=1}^i\mathcal{A}(U_j)\Big) \otimes
\Big(\bigotimes_{j=i+1}^{r} k(U_j)\Big)\\
\stackrel{(\bigotimes_{j=1}^i\epsilon)\otimes id}\longrightarrow
\bigotimes_{j=1}^{r} k(U_i)\stackrel{\rho^{k}_{U_1,\dots, U_{r}, V}}
\longrightarrow k
\end{multline*} where the last map is the structure
map of the factorization algebra associated to $k$.
\end{itemize}
\item $ \widehat{\mathcal{A}}$ is stratified locally constant on $\widehat{D^n}=S^n$
(stratified by $\{\infty\}\subset D^n\cup \{\infty\}= S^n$ as above).
\item Any $\mathcal{U}_{\widehat{D^n}}$-prefactorization algebra  satisfying (1),
(2) and (3) above is a $\mathcal{U}_{\widehat{D^n}}$-factorization algebra whose
unique\footnote{by Proposition~\ref{P:extensionfrombasis}} extension as a factorization algebra is further equivalent
 to $ \widehat{\mathcal{A}}$.
\end{enumerate}
\end{prop}
Point~\textbf{(5)} implies that $\widehat{\mathcal{A}}$ is the unique factorization algebra on
 $\widehat{D^n}$ satisfying the properties \textbf{(1)}, \textbf{(2)} and \textbf{(3)}
 of Proposition~\ref{P:BarasFactonSn}.
 
 \begin{rem}\label{R:naturalityinProp9.25}
 Note also that by Point.\textbf{(3)}, the equivalence, for any compact sub-disk $\overline{D}$, $\widehat{\mathcal{A}}\big(\widehat{D^n}\setminus \overline{D}\big) \cong
k$ of Point.\textbf{(2)}, is induced by the augmentation map~\eqref{eq:AugmentationHat}. Namely, let  $\mathcal{V}$ 
be a factorizing cover of $\widehat{D^n}\setminus \overline{D}$. 
For any open subset $V \in \mathcal{V}$, we have the augmentation   
$\widehat{\epsilon}: \widehat{\mathcal{A}}(V) \to k(V)=k$ which is a map of factorization algebras, hence induces a map of \v{C}ech complexex: $\check{C}(\widehat{\mathcal{A}}, \mathcal{V}) \to \check{C}(k, \mathcal{V})$. 
The following diagram, in which the lower arrow is the equivalence of Proposition~\ref{P:BarasFactonSn}.\textbf{(2)},
\begin{equation}
\label{eq:naturalityinProp9.25}
\xymatrix{\check{C}(\widehat{\mathcal{A}}, \mathcal{V}) \ar[rr]^{\widehat{\epsilon}} \ar[d]
& &  \check{C}(k, \mathcal{V}) \ar[d] \\ 
\widehat{\mathcal{A}}\big(\widehat{D^n}\setminus \overline{D}\big) \ar[rr]^{\simeq}
&& k}
\end{equation}
is commutative in $\hkmod$ (as proved in the proof of Proposition~\ref{P:BarasFactonSn}). In particular, the lower map in the diagram is just $\widehat{\epsilon}(\widehat{D^n}\setminus \overline{D}\big) :\widehat{\mathcal{A}}\big(\widehat{D^n}\setminus \overline{D}\big) \longrightarrow k\big(\widehat{D^n}\setminus \overline{D}\big)=k$.
\end{rem}
\begin{proof}[Proof of Proposition~\ref{P:BarasFactonSn}]
Since $\widehat{\mathcal{A}}(U)=(A,k)\big(\kappa^{(-1)}(U)\big)$,
point~\textbf{(1)} is immediate from Lemma~\ref{L:BarasFactIn}. 

By the generalized Schoenflies Theorem, the complement $\widehat{D^n}\setminus \overline{D}$ is homeomorphic to a disk. Consequently,
claim \textbf{(2)}  boils down to proving that $(A,k)\big((\partial
I^n)\times [0,r)\big)\cong k$ (that is, we are left to the case where  $\overline{D}$ is a compact Euclidean disk).
 The open set $(\partial I^n)\times [0,r)$ has a factorizing cover, stable by finite intersection,
 which consists of open half-disks of the form $C\times [0,r)$ (where, for instance $C$ is a small convex sub-disk of $\partial I^n$); denote $\mathcal{H}$ such a cover.  

Since $(A,k)$ is a factorization
  algebra, we have that  $(A,k)\big((\partial I^n)\times [0,r)\big)$ is computed by the \v{C}ech complex $\check{C}\big(\mathcal{H},(A,k)\big)$ of this cover
consisting of open half-disks.
     By  Lemma~\ref{L:StructureofAk}.\textbf{(2)}, the value of $(A,k)$ on any half-disk is just $k$; further the structure maps
  are those of the trivial factorization algebra $k$ (see
  Example~\ref{E:kasFact}). 
  Thus the  \v{C}ech complex  $\check{C}\big(\mathcal{H},(A,k)\big)$ is
  the same as the one of the factorization algebra $k$, that is, $\check{C}\big(\mathcal{H},(A,k)\big)= \check{C}\big(\mathcal{H},k\big)$.
Now,  Lemma~\ref{L:kasFact}  implies that
   $$(A,k)\big((\partial I^n)\times [0,r)\big)\cong k \big((\partial I^n)\times [0,r)\big)
   \cong CH_{(\partial I^n)\times [0,r)}(k) \cong k,$$
 and,  combining this with the definition of the augmentation map~\eqref{eq:AugmentationHat} and Lemma~\ref{L:Akisaugmented}, we further obtain a commutative diagram 
   $$\xymatrix{\check{C}\big(\mathcal{H},(A,k)\big) \ar[rr]^{\cong} \ar[d]
& &  \check{C}(k, \mathcal{H}) \ar[d] \\ 
\widehat{\mathcal{A}}\big(\widehat{D^n}\setminus \overline{D}\big) \ar[rr]^{\widehat{\epsilon}}
&& k(\widehat{D^n}\setminus \overline{D})=k} $$
   which finishes to prove claim \textbf{(2)}. In particular we see that the natural equivalence in claim \textbf{(2)} is induced by the augmentation. This implies in particular the commutativity of diagram~\eqref{eq:naturalityinProp9.25} above and thus Remark~\ref{R:naturalityinProp9.25} as well.

\smallskip

Claim~\textbf{(3)} is proved as in Lemma~\ref{L:StructureofAk}.

\smallskip

 Let
$V$ be an open disk containing  $\infty$. Then
$V\setminus\{\infty\}$ is homeomorphic to
 $\mathbb{R}^n\setminus \{0\}$ and $\kappa^{-1}(V)$ is homeomorphic to $(\partial I^n) \times [0,1)$.
  Thus statement~\textbf{(4)} reduces to statement~\textbf{(2)}.

\smallskip

To prove claim~\textbf{(5)}, note that  $\mathcal{U}_{\widehat{D^n}}$ is a factorizing, stable by finite
intersection, basis of opens and  $\widehat{A}$ is a prefactorization algebra  satisfying claims \textbf{(1)},
 \textbf{(2)} and \textbf{(3)}. Since we already know that  $\widehat{A}$ is a factorization algebra,
  we know that the data given by the claims \textbf{(1)}, \textbf{(2)} and \textbf{(3)} does define a
  $\mathcal{U}_{\widehat{D^n}}$-factorization algebra.    Proposition~\ref{P:extensionfrombasis} implies that
   any factorization algebra  whose value on $\mathcal{U}_{\widehat{D^n}}$ agrees with the one of $\widehat{A}$
   (which is given by claims \textbf{(1)}, \textbf{(2)} and \textbf{(3)})  is equivalent to $\widehat{A}$
   which terminates the proof of statement \textbf{(5)}.
\end{proof}

\medskip

\subsubsection{The $E_n$-coalgebra structure of the iterated Bar
construction}\label{SS:EnCoalgforBar}

In this section we prove that the iterated Bar construction
$Bar^{(n)}(A)$ of an augmented $E_n$-algebra has an $E_n$-coalgebra
structure. In view of Theorem~\ref{P:En=Fact}, it is equivalent to
prove that there exists a locally constant factorization algebra on
$\mathbb{R}^n$ whose global section is the iterated Bar
construction $Bar^{(n)}(A)$. This is the
approach we follow here. 

\smallskip

\begin{rem}[\emph{sketch of the construction}] The result of Lemma~\ref{L:BarasFactSn} and
Proposition~\ref{P:BarasFactonSn} is that the Bar construction of an
augmented $E_n$-algebra $\epsilon: A\to k$  is the global section
(i.e. factorization homology) of the stratified locally constant
factorization algebra on the Alexandroff compactification $D^n\cup
\{\infty\}=\widehat{D^n}=S^n$ of $D^n$ whose value on $D^n$ is just
the one of $A$ and whose value in a disk centered at $\infty$ is
just $k$. For any disk $D$ inside $D^n$, we can also form its
Alexandroff one-point compactification $\widehat{D}=\{\infty\}\cup
D$ and by restriction to $D$, the factorization algebra associated
to $A$ will give rise to a stratified locally constant factorization
algebra on $D^n$. The procedure can be done simultaneously for
pairwise disjoint opens in $D^n$; this suggest how the iterated bar
construction gives rise to a factorization coalgebra on $D^n$. We
now make this scheme precise.
\end{rem}

\medskip

Let $\epsilon: A\to k$ be a map of $E_m$-algebras (with $m\geq n$) and which we assume to be represented by a map
 $\epsilon: \mathcal{A}\to k$
 of factorization algebras over $\R^m$.
In other words, $\mathcal{A}(D)\cong \int_D A$ for any disk
$D\subset \mathbb{R}^m$. Recall from Theorem~\ref{T:Dunn} that
$\pi_*(\mathcal{A})\in
\text{Fac}^{lc}_{\R^n}(\text{Fac}^{lc}_{\R^{m-n}})$. In the
following, in order to shorten notations, we will simply write
$\mathcal{A}$ for  $\pi_*(\mathcal{A})$.
In particular for any $U$
open subset of $\R^n$, $\mathcal{A}(U\times \R^n)=\int_{U\times
\R^{m-n}} A$ inherits an $E_{m-n}$-algebra structure (canonically induced by a locally constant factorization algebra structure on $\R^{m-n}$).

\smallskip

The \emph{restriction of $\epsilon:\mathcal{A}\to k$ to $U\times \R^{m-n}$ is an augmentation for
  $\mathcal{A}(U\times \R^{m-n})$}.

\begin{rem} \label{R:PhiisotoU}
Let  $\phi: \mathbb{R}^n \stackrel{\simeq}\to U \subset\mathbb{R}^n$ be an embedding of a disk in $\mathbb{R}^n$.
  Then we have an homeomorphism $\phi\times id:  \mathbb{R}^n\times \R^{m-n} \stackrel{\simeq}\to U\times \R^{m-n}$
   which makes
$\mathcal{A}_{|U\times \R^{m-n}}$ into a locally constant factorization algebra over $\R^{m}$ hence
  $\mathcal{A}(U\times \R^{m-n})$ is an $E_m$-algebra.
  \begin{definition}\label{D:Barphi}  We denote \emph{$\mathcal{A}_{\phi}$ the augmented  $E_m$-algebra
$\mathcal{A}(U\times \R^{m-n})\cong \int_{U\times \R^m}A$}.

We can thus define \emph{$Bar^{(n)}(A_\phi)$ the $n$-fold iterated
Bar construction of $A_\phi$} and
 \emph{$\widehat{\mathcal{A}_\phi}\in \text{Fac}^{lc}_{\widehat{U}}(E_{m-n}\text{-Alg})$ the stratified locally constant
 factorization  algebra} of Definition~\ref{D:BarasFactAlgonSn}.
\end{definition}
\end{rem}

We wish to define a factorization algebra $\widehat{\mathcal{A}_U}$ on $\widehat{U}\cong U\cup \{\infty\}$ the Alexandroff compactification of $U$. We essentially proceed as for Definition~\ref{D:BarasFactAlgonSn} above: 
\begin{definition}\label{Def:AU} Let $\mathcal{W}_{\widehat{U}}$ be the open cover\footnote{which is not stable under intersection} of $\widehat{U}$ consisting of all opens $W$ such that either  $W\subset U\subset \widehat{U}$ or else $W=\widehat{U}\setminus \overline{D}$ where $\overline{D}\subset U$  is any  compact disk\footnote{by a compact disk in $U$, we mean the image  in $U$ of an embedding of the closed unit Euclidean disk} \begin{enumerate}
\item \label{eq:valueofAU1}for any $W\subset U\subset \widehat{U}$ and for any compact disk $\overline{D}\subset U$ ,    set  $$\widehat{\mathcal{A}_U}(W)=\mathcal{A}(W), \qquad
\widehat{\mathcal{A}_U}\big(\widehat{U}\setminus \overline{D}\big) =k\big(\widehat{U}\setminus \overline{D}\big)=k.$$
\item  \label{eq:valueofAU2}Let $W_1,\dots, W_r\subset \widehat{U}$ are pairwise disjoint opens  lying in
$V\in \widehat{U}$. Assume in addition that 
\begin{enumerate}
 \item \label{eq:case1} either $V$ is in  $U\subset \widehat{U}$ (and thus so are all $W_i$);
 \item \label{eq:case2} or $V=\widehat{U}\setminus \overline{D}$, where $D$ is a compact disk in $U$,
 $W_1,\dots, W_i \in U$  and $W_{i+1},\dots, W_{r}$ are
complements of compact disks (the $W_i$\rq{}s being disjoint implies  $i=r$ or $r-1$).
\end{enumerate}
We  define \lq\lq{}structure maps\rq\rq{} $\rho_{W_1,\dots, W_{r}, V}:
\bigotimes_{i=1}^{r} \widehat{\mathcal{A}_U}(W_i)\longrightarrow
\widehat{\mathcal{A}_U}(V)$  as follows: 
\begin{itemize}
\item in case~\eqref{eq:case1}, we set  $$\rho_{W_1,\dots, W_{r}, V}=
\bigotimes_{i=1}^{r} \widehat{\mathcal{A}}(W_i) \cong
\bigotimes_{i=1}^{r}
\mathcal{A}(W_i)\stackrel{\rho^{\mathcal{A}}_{W_1,\dots, W_{r},
V}}\longrightarrow \widehat{\mathcal{A}}(V)$$  the last map being
the structure map of the factorization algebra $\mathcal{A}$;
\item in case~\eqref{eq:case2}, we set
\begin{multline*} \rho_{W_1,\dots, W_{r},
V}=\bigotimes_{j=1}^{r} \widehat{\mathcal{A}}(W_j)=
\Big(\bigotimes_{j=1}^i\mathcal{A}(W_j)\Big) \otimes
\Big(\bigotimes_{j=i+1}^{r} k(W_j)\Big)\\
\stackrel{(\bigotimes_{j=1}^i\epsilon)\otimes id}\longrightarrow
\bigotimes_{j=1}^{r} k(W_i)\stackrel{\rho^{k}_{W_1,\dots, W_{r}, V}}
\longrightarrow k
\end{multline*} where the last map is the structure
map of the factorization algebra associated to $k$.
\end{itemize}
\end{enumerate}
\end{definition}
\begin{lem}\label{L:BarUisaFacAlg}
 \begin{itemize}\item There is a unique\footnote{up to a contractible family of choices} factorization algebra $\widehat{\mathcal{A}_U}$ on $\widehat{U}$ which takes the values given by Definition~\ref{Def:AU}.\eqref{eq:valueofAU1} and with structure maps specified by Definition~\ref{Def:AU}.~\eqref{eq:valueofAU2} above on the relevant opens.
 \item Further, $\widehat{\mathcal{A}_U}$ is stratified locally constant on $\widehat{U}$, which is stratified with one  dimensions $0$ stratum given by the point at $\infty$ and one dimension $n$ stratum given by$U$.
 \end{itemize}
\end{lem}
\begin{proof} Since $U$ is a disk, we can find an embedding
 $\phi: \mathbb{R}^n \stackrel{\simeq}\to U \subset\mathbb{R}^n$, which induces an homeomorphism $\widehat{D^n} \cong \widehat{U}$. As in Definition~\ref{D:Barphi}, we have the $E_m$ algebra  $\mathcal{A}_{\phi}$ and a \emph{stratified locally constant} factorization algebra on $\widehat{D^n} \cong \widehat{U}$. By Proposition~\ref{P:BarasFactonSn}, we see that the factorization algebra  $\mathcal{A}_{\phi}$  takes the same value as $\widehat{A}_{U}$ on the opens specified in point~\eqref{eq:valueofAU1}.
 Further,  it has the same structure maps as those given  by Definition~\ref{Def:AU}~\eqref{eq:valueofAU2} on the basis of opens $\Big(\phi(V), V\in \mathcal{U}_{\widehat{D^n}}\Big)$. 
 
 Thus by Proposition~\ref{P:BarasFactonSn}.(5), we see that $\mathcal{A}_{\phi}$  is the unique factorization algebra structure on $\widehat{U}$ taking these values.

 It only remains to prove that $\widehat{\mathcal{A}_U}$ does has the structure maps claimed by Definition~\ref{Def:AU}~\eqref{eq:valueofAU2} in case~\eqref{eq:case2} for arbitrary compact disk $\overline{D}$. The proof is similar to the proof of the commutativity of Diagram~\eqref{eq:naturalityinProp9.25} obtained in the proof of Proposition~\ref{P:BarasFactonSn}; we used the generalized Schoenflies theorem to restrict to a cover by half disks and use the \v{C}ech complexes of this cover to deduce the result. 
\end{proof}

\begin{definition}\label{D:BarU}   We denote $\widehat{\mathcal{A}_U}$ the stratified locally constant
 factorization  algebra on $\widehat{U}$ defined by Lemma~\ref{L:BarUisaFacAlg}. It is augmented $\widehat{\epsilon}:\widehat{\mathcal{A}_U}\to k$.
\end{definition}
\begin{rem}\label{R:PhiisotoU2}
 Let  $\phi: \mathbb{R}^m \stackrel{\simeq}\to U$ be an homeomorphism so that we have the $E_n$-algebra $\mathcal{A}_{\phi}$ from Definition~\ref{D:Barphi}. 
 By Lemma \ref{L:BarasFactSn} and Lemma~\ref{L:BarUisaFacAlg},  we have an natural (with respect to maps of augmented
 $E_m$-algebras) equivalence of $E_{m-n}$-algebras:
\begin{equation}\label{eq:BarphiasFact}
Bar^{(n)}(A_\phi)\stackrel{\simeq}\longrightarrow
p_*\big(\widehat{\mathcal{A}_\phi}\big) \cong
\widehat{\mathcal{A}_U}(\widehat{U}).
\end{equation}
\end{rem}

\smallskip

The \emph{one point compactification is contravariant} with respect to open inclusions: if $U\subset V$  are open subsets of
$\R^n$, we have the continuous map 
\begin{equation}\label{eq:DefiotaUV}\iota_U^V: \widehat{V} \longrightarrow \widehat{U}\end{equation} which is the identity on
 $U\subset \widehat{V}$ and collapses the complement $\widehat{V}\setminus U$ to the point at $\infty$ of
 $\widehat{U}=U\cup \{\infty\}$.
 
 By pushing forward along $\iota_U^V$, we
get the factorization algebra
$${\iota_U^V}_*(\widehat{\mathcal{A}_V})\in
\text{Fac}_{\widehat{U}}(\text{Fac}^{lc}_{\R^{m-n}}).$$ Note that by
definition of factorization homology for factorization algebras we
have:
\begin{equation}\label{eq:HomofIota=HomV}
{\iota_U^V}_*(\widehat{\mathcal{A}_V})(\widehat{U})\cong
p_*\circ {\iota_U^V}_*(\widehat{\mathcal{A}_V}) \cong
p_*(\widehat{\mathcal{A}_V}) \cong
\widehat{\mathcal{A}_V}(\widehat{V})
\end{equation}
We wish to define a quasi-isomorphism 
$\gamma^V_U:{\iota_U^V}_*(\widehat{\mathcal{A}_V}) \longrightarrow
\widehat{\mathcal{A}_U}$ of factorization algebras over
$\widehat{U}$. 

To do this, we consider the cover of $\widehat{U}$ given consisting of all
opens $W\subset U\subset \widehat{U}$ and all
 opens which are the complement $\widehat{U}\setminus \overline{D}$
 of a compact disk $\overline{D}$. We note

\begin{lem} \label{L:valueofiota} Let $D$ be an open subset of $\widehat{U}$. 
\begin{enumerate}
\item If   $D\subset U$, then
${\iota_U^V}_*(\widehat{\mathcal{A}_V}) (D) =
\widehat{\mathcal{A}_V}(D) = \mathcal{A}(D) =
\widehat{\mathcal{A}_V}(D)$;
\item if  $D=\widehat{U}\setminus \overline{D}$, then
${\iota_U^V}_*(\widehat{\mathcal{A}_V}) (D) =
\widehat{\mathcal{A}_V}\big(\widehat{V}\setminus
\overline{D}\big)\cong k\big(\widehat{V}\setminus
\overline{D}\big)=k$.
\end{enumerate}
\end{lem}
\begin{proof}
Choose homeomorphisms $\phi: \mathbb{R}^n \stackrel{\simeq}\to U$ and
$\psi: \R^n \stackrel{\simeq}\to V$ so that $\psi$ identifies $V$ with $D^n$ and we are left to the case where $\phi: \R^n\to U$ is a sub-disk of $D^n$. 
We have a factorizing and stable by finite
intersections basis $\mathcal{U}_{\widehat{U}}$ (as in Proposition~\ref{P:BarasFactonSn}).  The basis
$\mathcal{U}_{\widehat{U}}$ is defined as the set consisting of all
opens $W\subset U\subset \widehat{U}$ and all
 opens which are the complement $\widehat{U}\setminus \phi(\overline{D})$
 of the image by $\phi$ of a non-empty Euclidean compact disk $\overline{D}\subset \R^n$ whose center is $0$.

Now, the lemma  is a consequence of the Definition of the pushforward
${\iota_U^V}_*(\widehat{\mathcal{A}_\psi})$,
Definition~\ref{D:BarasFactAlgonSn} and
Proposition~\ref{P:BarasFactonSn}.(1) and (2).
\end{proof}
We now define the aforementioned map $\gamma^V_U$.
\begin{lem}\label{L:StructureMapsBar1} Let $W$ be in the cover $\mathcal{W}_{\widehat{U}}$  as defined in Definition~\ref{Def:AU}.~\eqref{eq:valueofAU1}; that is either $W\subset U\subset \widehat{U}$ or $W$ is the complement $\widehat{U}\setminus \overline{D}$ of 
 a compact disk\footnote{by a compact disk in $U$, we mean the image  in $U$ of an embedding of the closed unit Euclidean disk} $\overline{D}\subset U$. 
 Let
$$\gamma^V_U(W):{\iota_U^V}_*(\widehat{\mathcal{A}_V})(W) \longrightarrow
\widehat{\mathcal{A}_U}(W) $$ be the augmented $E_{m-n}$-algebra
map defined (using the identifications provided by
Lemma~\ref{L:valueofiota}),
\begin{itemize}
\item  as the identity map
 $$\gamma^V_U(W):{\iota_U^V}_*(\widehat{\mathcal{A}_V})(W) = \mathcal{A}(W)\stackrel{id}\longrightarrow
\mathcal{A}(W)=\widehat{\mathcal{A}_U}(W)$$  if $W\subset U$;
\item and,   if $W= \widehat{U}\setminus \overline{D}$, $ \overline{D}\subset U$ a compact disk, as the
restriction of the augmentation of $\widehat{\mathcal{A}_V}$:
\begin{equation*}\gamma^V_U(W):{\iota_U^V}_*(\widehat{\mathcal{A}_V})(W) =
\widehat{\mathcal{A}_V}\big(\widehat{V}\setminus
 \overline{D}\big)
\xrightarrow{\widehat{\epsilon}\big(\widehat{V}\setminus
 \overline{D}\big)} k\big(\widehat{V}\setminus
 \overline{D}\big)=k = \widehat{A_{U}}(W)\end{equation*} (where the last equality
follows from Proposition~\ref{P:BarasFactonSn} and $\widehat{\epsilon}$ is the augmentation map~\eqref{eq:AugmentationHat}).
\end{itemize}
\begin{enumerate} \item The collection $\big(\gamma^V_U(W)\big)_{W\in
\mathcal{U}_{\widehat{U}}}$ is a map of
$\mathcal{W}_{\widehat{U}}$-factorization algebras.
\item 
The collection $\big(\gamma^V_U(W)\big)_{W\in
\mathcal{U}_{\widehat{U}}}$ has an \emph{unique}\footnote{up to a contractible family of choices} extension into \emph{a map
$\gamma_{U}^{V}: {\iota_U^V}_*(\widehat{\mathcal{A}_V})
\longrightarrow \widehat{\mathcal{A}_U}$ of factorization
algebras over $\widehat{U}$}. 
\item $\gamma_{U}^{V}$ is further a map of augmented factorization algebras (with respect to the augmentation~\eqref{eq:AugmentationHat}).
\item \label{eq:iotaisquis}The  map $ {\iota_U^V}_*(\widehat{\mathcal{A}_V})
\stackrel{\gamma_{U}^{V}}\longrightarrow \widehat{\mathcal{A}_U}$ is an equivalence of factorization algebras.
\end{enumerate}
\end{lem}
\begin{proof}
We need to prove that, for any pairwise disjoint open  $W_1,\dots, W_r\in
\mathcal{W}_{\widehat{U}}$   lying in
$Z\in \mathcal{W}_{\widehat{U}}$, the following diagram
\begin{equation}\label{eq:StructureMapsBar1}
\xymatrix{\bigotimes\limits_{i=1}^{r} {\iota_U^V}_*(\widehat{\mathcal{A}_{V}})(W_i)   \ar[d]_{\bigotimes\limits_{i=1}^{r} \gamma^{V}_{U}(W_i)}
\ar[rrr]^{\rho_{W_1,\dots, W_r, Z}} &&&  {\iota_U^V}_*(\widehat{\mathcal{A}_{V}})(Z)   \ar[d]^{\bigotimes\limits_{i=1}^{r} \gamma^{V}_{U}(Z)} \\
\bigotimes\limits_{i=1}^{r} \widehat{\mathcal{A}_{U}}(W_i)\ar[rrr]^{\rho_{W_1,\dots, W_r, Z}} &&& \widehat{\mathcal{A}_{U}}(Z) }
\end{equation}
is commutative.

If $Z \subset U$ (and consequently all the $W_i\subset U$ as well), then this is a trivial consequence of Lemma~\ref{L:valueofiota}.(1).
Let $Z=\widehat{U}\setminus \overline{K}$ with $\overline{K}$ a compact disk in $U$. We may assume $W_1,\dots, W_j \in U$ with $j=r-1$ or $j=r$, and the  remaining  $W_{j+1},\dots, W_r$ (note that there may be only one or zero such $W_\ell$) to be of the form $\widehat{U}\setminus \overline{T}$ where $\overline{T}$ is a compact  disk in $U$.
Unfolding the definition of $\gamma^V_U$ using   statement (3) in  Proposition~\ref{P:BarasFactonSn} and Lemma~\ref{L:valueofiota}, we obtain   that the diagram~\eqref{eq:StructureMapsBar1} can be rewritten as the following diagram
\begin{equation}\label{eq:StructureMapsBar2}
\xymatrix{ \big(\bigotimes\limits_{i=1}^{j} \widehat{\mathcal{A}_{V}}(W_i)\big)\otimes   \big(\bigotimes\limits_{i=j+1}^{r}\widehat{\mathcal{A}_V}(\widehat{V}\setminus
\overline{T}\big) \ar[d]_{\big(\bigotimes\limits_{i=1}^{j} id\big) \otimes \big(\bigotimes\limits_{i=j+1}^r \widehat{\epsilon}(\widehat{V}\setminus \overline{T})\big)}
\ar[rrrr]^{\quad \rho_{W_1,\dots, W_r, Z}} \ar[rrrd]^{\qquad\bigotimes\limits_{i=1}^r \widehat{\epsilon}(W_i)}&&&&  \widehat{\mathcal{A}_V}\big(\widehat{V}\setminus \overline{T}\big)  \ar[d]^{ \widehat{\epsilon}(\widehat{V}\setminus \overline{K})} \\
\big(\bigotimes\limits_{i=1}^{j} \widehat{\mathcal{A}_{V}}(W_i)\big)\otimes   \big(\bigotimes\limits_{i=j+1}^{r} k\big) \ar[rrr]_{\quad\qquad \big(\bigotimes\limits_{i=1}^{j} \widehat{\epsilon}(W_i)\big) \otimes \big(\bigotimes\limits_{i=j+1}^r id\big)} &&& \bigotimes\limits_{i=1}^r k \ar[r]^{\;\;\;\rho^k_{W_1,\dots, W_r, Z}} &  k }
\end{equation}
and that further the lower left triangle in diagram~\eqref{eq:StructureMapsBar2} is commutative.
 The commutativity of the upper right part of diagram~\eqref{eq:StructureMapsBar2} is given by the fact that $\widehat{\epsilon}: \widehat{\mathcal{A}_V}\to k$ is a map of factorization algebras. This proves that the $(\gamma^V_U(W))$ forms a map of $\mathcal{W}_{\widehat{U}}$-factorization algebras.

\smallskip

Now, note that $\mathcal{W}_{\widehat{U}}$ contains a factorizing, stable by finite intersections, basis  $\mathcal{U}_{\widehat{U}}$ of opens. Indeed, let $\phi: \R^n\to U$ be an homeomorphism. Then the cover $\mathcal{U}_{\widehat{U}}$ consists of all opens $W\subset U$ and all opens $\widehat{U}\setminus \phi(K)$ where $K$ is a compact Euclidean ball  centered at $0$ in $\R^n$. 
The fact that the  collection $(\gamma^V_U(W))$   extends uniquely to a map of factorization algebras is hence a consequence of Proposition~\ref{P:extensionfrombasis}.

 Indeed, if   $\mathcal{F}\in \text{Fac}_{\widehat{U}}$ and $D\subset \widehat{U}$ is an open set, then the \v{C}ech complex $\check{C}(D_{\mathcal{U}_{\widehat{U}}}, \mathcal{F}) \cong \mathcal{F}(D)$ where $D_{\mathcal{U}_{\widehat{U}}}$ is the cover of $D$ consisting of all opens of $\mathcal{U}_{\widehat{U}}$ which lies in $D$.
 In particular any map of    $\mathcal{U}_{\widehat{U}}$-factorization algebras defines a map between the associated \v{C}ech complexes. This construction is the inverse of the restriction functors from factorization algebras to   $\mathcal{U}_{\widehat{U}}$-factorization algebras.

\smallskip

From above, to prove statement.(3), it is  sufficient to check that $\gamma^V_U$ is a map of augmented factorization algebras on the opens of $\mathcal{U}_{\widehat{U}}$. If $W\subset U$, then there is nothing to prove since $\gamma^V_U(W)$ is the identity. If  $W= \widehat{U}\setminus \phi(\overline{K})$ then $k = \widehat{A_{\phi}}(W)$ and the augmentation map~\eqref{eq:AugmentationHat} is the identity $k\to k$ and there is nothing left to prove.

\smallskip

Finally, again by Proposition~\ref{P:extensionfrombasis}, to prove that $ {\iota_U^V}_*(\widehat{\mathcal{A}_\psi})
\stackrel{\gamma_{U}^{V}}\longrightarrow \widehat{\mathcal{A}_\phi}$ is an equivalence of factorization algebras, it is sufficient to prove that its restriction $\gamma_{U}^{V}(W)$ on any open $W \in \mathcal{U}_{\widehat{U}}$ of the above basis is a quasi-isomorphism. The only case which needs a proof is when
  $W= \widehat{U}\setminus \overline{D}$ with $ \overline{D}=\phi(K)$ where $K$   is a compact Euclidean ball  centered at $0$ in $\R^n$.  
By Proposition~\ref{P:BarasFactonSn}.(2) and  diagram~\eqref{eq:naturalityinProp9.25}, we have a commutative diagram 
\begin{equation*}\xymatrix{   {\iota_U^V}_*(\widehat{\mathcal{A}_V})(W) =
\widehat{\mathcal{A}_V}\big(\widehat{V}\setminus
 \overline{D}\big) \ar[rr]^{\qquad \gamma^V_U(W)} \ar[d]_{\widehat{\epsilon}\big(\widehat{V}\setminus \overline{D}\big)}^{\cong} &&  k\big(\widehat{V}\setminus
 \overline{D}\big)=k \\ k \ar[rru]_{id} && }\end{equation*}
from which we deduce that $\gamma^V_U(W)$ is a quasi-isomorphism. Hence Claim.(4) of the Lemma holds. 
\end{proof}

\smallskip

Passing to factorization homology, \emph{i.e.} evaluating on
$\widehat{U}$, the factorization algebra map
$\gamma^V_U:{\iota_U^V}_*(\widehat{\mathcal{A}_V})
\longrightarrow \widehat{\mathcal{A}_U}$ induces a map
$\gamma^V_U(\widehat{U}):
{\iota_U^V}_*(\widehat{\mathcal{A}_V})(\widehat{U})
\longrightarrow  \widehat{\mathcal{A}_U} (\widehat{U})$.
Composing this map with the string of
equivalences~\eqref{eq:HomofIota=HomV}, we get the following  map of
(augmented) $E_{m-n}$-algebras
\begin{equation}\label{eq:DefGamma}
\tilde{\gamma}^V_U: p_*(\widehat{\mathcal{A}_V}) \cong
\widehat{\mathcal{A}_\psi}(\widehat{V}) \cong
{\iota_U^V}_*(\widehat{\mathcal{A}_V})
\stackrel{\gamma^V_U(\widehat{U})}\longrightarrow
\widehat{\mathcal{A}_U}\cong p_*(\widehat{\mathcal{A}_U})
\end{equation}
between the factorization homology of $\widehat{\mathcal{A}_V}$
and the factorization homology of $\widehat{\mathcal{A}_U}$.

\medskip

We now define the  factorization \emph{coalgebra}
(Definition~\ref{D:FacAlg}) $U\mapsto
Bar^{(n)}(A)(U)$ we have been seeking for.

\begin{definition}\label{D:BarisaParamFact} Let  $\epsilon: \mathcal{A}\to k$ be a map
 of locally constant factorization algebras over $\R^m$.
\begin{itemize}
\item Let $U\subset \R^n$ be a disk. We define $Bar^{(n)}(\mathcal{A})(U):=
p_*(\widehat{\mathcal{A}_U})\in E_{m-n}\text{-Alg}^{aug}$ the factorization homology of the Factorization algebra

$\mathcal{A}_{U}$ on $\widehat{U}$ from Definition~\ref{D:BarU}.
\item Let $U_1,\dots, U_r$ be a family pairwise disjoint open sub-disks of an open disk $V\subset \R^n$. 
We define the structure map $\delta_{U_1,\dots,U_r,V}:
Bar^{(n)}(A)(V) \longrightarrow \bigotimes_{i=1}^{r}
Bar^{(n)}(A)(U_i) $ to be the following maps in
$E_{m-n}\text{-Alg}^{aug}$:
\begin{multline*}
\delta_{U_1,\dots,U_r,V}: Bar^{(n)}(A)(V) =
p_*(\widehat{\mathcal{A}_V})\\
\stackrel{\bigotimes\limits_{i=1}^r \tilde{\gamma}^V_{U_i} }
\longrightarrow \bigotimes\limits_{i=1}^r
p_*(\widehat{\mathcal{A}_{U_i}}) = Bar^{(n)}(A)(U_1) \otimes
\cdots \otimes Bar^{(n)}(A)(U_r).
\end{multline*}Here the maps
$\tilde{\gamma}^V_{U_i}$ are the compositions~\eqref{eq:DefGamma}.
\end{itemize}
Let $\epsilon: A\to k$ be a map of $E_m$-algebras (with
$m\geq n$) and assume it is represented\footnote{in other words $\mathcal{A}(W)\cong \int_W A$ for any open subset $W\subset
 \R^n$} by the factorization algebra map $\epsilon: \mathcal{A}\to k$.
 In that case, \emph{we also denote} $$Bar^{(n)}(A)(U):= Bar^{(n)}(\mathcal{A})(U).$$
\end{definition}
Unfolding the definition, the map $\delta_{U_1,\dots,U_r,V}$ is essentially  the map of factorization algebra given by the identity on each $U_i$ and the augmentation in their complement as is pictured 
 in
Figure~\ref{fig:coalgebra} (in the case $r=3$).
\begin{figure}
\includegraphics[scale=0.4]{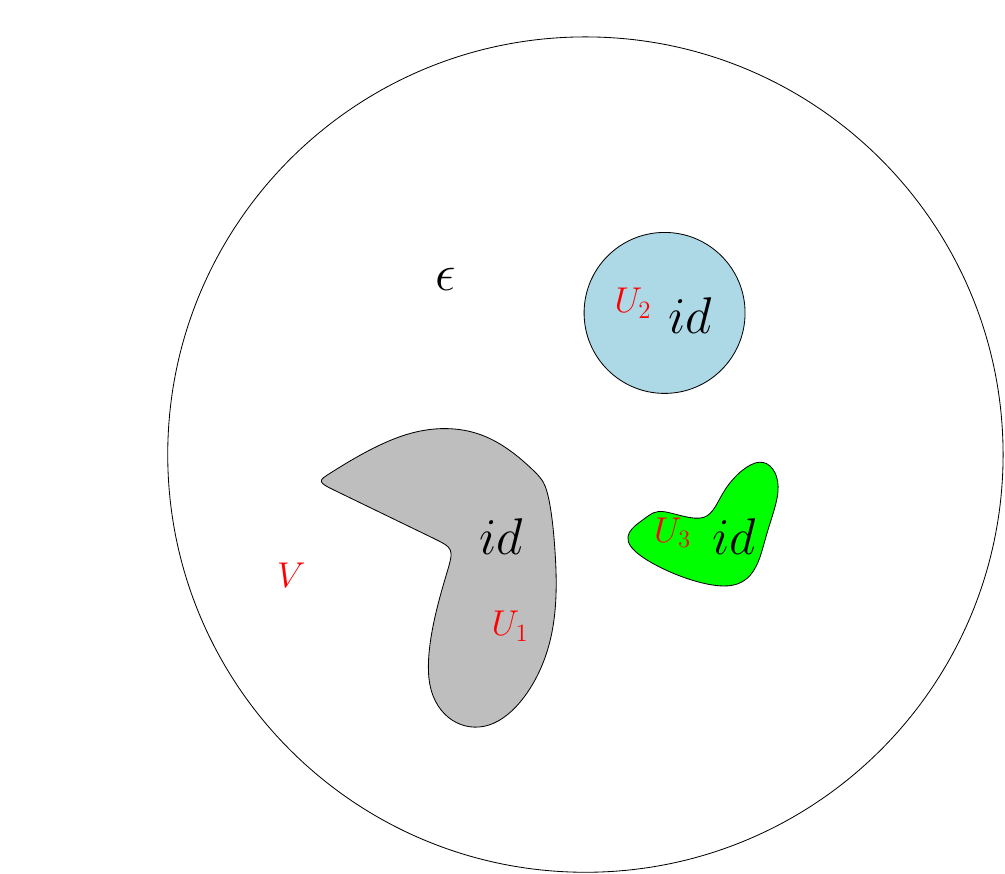}
\caption{The map $Bar^{(n)}(A)(V) \to Bar^{(n)}(A)(U_1)\otimes
Bar^{(n)}(A)(U_2)\otimes  Bar^{(n)}(A)(U_3)$}\label{fig:coalgebra}
\end{figure}

\begin{theorem}\label{P:BarforEn}
Let $0\leq n\leq
m$. 
\begin{enumerate}
\item There is an $\infty$-functor 
$$\mathcal{B}\textit{ar}^{(n)}: \text{Fac}^{lc, aug}_{\mathbb{R}^n} \longrightarrow
\text{coFac}^{lc}_{\R^n}\Big(\text{Fac}_{\R^{m-n}}^{lc, aug}\Big)$$
from $(\infty,1)$-category of locally constant augmented factorization algebras over $\R^m$ to the $(\infty,1)$-category of locally constant cofactorization algebras\footnote{in the sense of
Definition~\ref{D:FacAlg}, that is a locally constant
$N(\text{Disk})(\mathbb{R}^n)$-coalgebra} over $\R^n$ with values in locally constant augmented factorization algebras over $\R^{m-n}$.

\smallskip

The functor $\mathcal{B}\textit{ar}^{(n)}$ is given by the 
rule $\phi\mapsto Bar^{(n)}(A)(U)$ together with the structure maps  $\delta_{U_1,\dots,U_r,V}$
of Definition~\ref{D:BarisaParamFact}.
\item let $\epsilon:A\to k$ be an augmented $E_m$-algebra. 
There is an natural equivalence
$Bar^{(n)}(A) \stackrel{\simeq}\longrightarrow Bar^{(n)}(A)(\R^n)$ between the iterated Bar construction of $A$ (in the sense of Definition~\ref{D:BariteratedforEm})
and the cofactorization homology of
$\mathcal{B}\textit{ar}^{(n)}(\mathcal{A})$.

In particular, the iterated Bar construction $Bar^{(n)}(A)$ has an natural structure of
$E_n$-coalgebra in $E_{m-n}\text{-Alg}^{aug}$ and the iterated Bar
construction functor  (Definition~\ref{D:BariteratedforEm})  lifts
as a functor of $(\infty,1)$-categories
  $$Bar^{(n)}: E_m\text{-Alg}^{aug} \longrightarrow E_n\text{-coAlg}\Big(E_{m-n}\text{-Alg}^{aug}\Big).$$
\end{enumerate}
\end{theorem}

\begin{proof}
First note that $Bar^{(n)}(A)(\R^n )=
p_*(\widehat{\mathcal{A}})$ where $\widehat{\mathcal{A}}$ is the factorization algebra
on $\R^n$ of Definition~\ref{D:BarasFactAlgonSn}. Thus, by
Lemma~\ref{L:BarasFactSn}, we have an natural (with respect to maps
of augmented $E_m$-algebras) equivalence
$$Bar^{(n)}(A) \stackrel{\simeq}\longrightarrow p_*(\widehat{A})=Bar^{(n)}(A)(\R^n).$$
 Hence, part \textbf{(2)} in the Theorem is a corollary of part \textbf{(1)} and
 the relationship between factorization (co)-algebras an $E_\ell$-(co)-algebras, namely
Theorem~\ref{P:En=Fact} and Proposition~\ref{P:AlternativeFacAlg}. 

\smallskip

We now prove part (1).
The functoriality of $U\mapsto Bar^{(n)}(\mathcal{A})(U)$ is a straightforward consequence of the functoriality of $p_*(\widehat{\mathcal{A}_{U}})$ and of the transformations $\gamma^V_U$  of Lemma~\ref{L:StructureMapsBar1}.

Now, recall that each of the maps $\tilde{\gamma}^V_{U_i} $ (defined as the
composition~\eqref{eq:DefGamma}) are augmented $E_{m-n}$-algebras
maps. Hence so is the map 
 $$
\delta_{U_1,\dots,U_r,V}:  Bar^{(n)}(\mathcal{A})(V)   \longrightarrow  Bar^{(n)}(\mathcal{A})(U_1) \otimes \cdots \otimes  Bar^{(n)}(\mathcal{A})(U_r)
 $$ from
Definition~\ref{D:BarisaParamFact}.

The invariance under the symmetric group action of the structure map follows right away from its definition. We also need to check the  naturality of the structure maps with respect to inclusions of disks, \emph{i.e.}, the identity:
 \begin{equation} \label{eq:idnatFactCobar0}
 \Big( \delta_{W_{1}^1,\dots, W_{i_1}^1, U_1}\otimes \cdots \otimes \delta_{W^r_{1},\dots, W^r_{i_r}, U_r} \Big) \circ \delta_{U_1,\dots, U_r, V}
 =
\delta_{W_1^1,\dots, W_{i_1}^1,\dots, W_1^r,\dots,
W_{i_r}^r,V },\end{equation} which has to hold for
  any families of pairwise disjoint open sub-disks
$W_i^j\subset
U_j$
(where $j=1\dots r$).
Unfolding Definition~\ref{D:BarisaParamFact},  we see that the identity~\eqref{eq:idnatFactCobar0} follows from the following identity 
 \begin{equation}\label{eq:identityUiWijV}
  \tilde{\gamma}^{U_j}_{W_i^j}   \circ {\iota^V_{U_j}}_*\big(\tilde{\gamma}^V_{U_j}\big)= \tilde{\gamma}^V_{W_i^j} 
\end{equation}
if it holds for all inclusions $W^j_i\subset U_j\subset V$ of open subsets.  It is enough to check this identity for the underlying factorization algebras maps, that is too prove:
 \begin{equation} \label{eq:idnatFactCobar}
  {\gamma}^{U_j}_{W_i^j}   \circ {\iota^V_{U_j}}_*\big({\gamma}^V_{U_j}\big)= {\gamma}^V_{W_i^j} .
\end{equation}
Let  $\theta_i^j: \R^n \to W_i^j$  and  $\phi_j:\mathbb{R}^n\stackrel{\simeq}\to U_j\subset V$ be homeomorphisms.
In view of Lemma~\ref{L:StructureMapsBar1}, it is sufficient  to check the above identity~\eqref{eq:idnatFactCobar} on the factorizing cover
$\mathcal{U}_{\widehat{W^j_i}}$  consisting of all opens in $W^j_i$ and all complements of $\theta^j_i(\overline{D})$ where $D$ is a compact Euclidean disk.   Both sides of  identity~\eqref{eq:idnatFactCobar} are equal  to the identity when restricted to an open subsets of   $W^j_i$ and to the (restriction of the) augmentation in the second case since ${\gamma}^V_{U_j}$  is a map of augmented algebras (Lemma~\ref{L:StructureMapsBar1}).

\smallskip

It remains to prove the locally constant condition. That is we need to see that for an open sub-disk $U\hookrightarrow V$ of a disk $V$, the map~\eqref{eq:DefGamma}
\begin{multline*}\tilde{\gamma}^V_U: Bar^{(n)}(\mathcal{A})(V)=p_*(\widehat{\mathcal{A}_V}) \cong
\widehat{\mathcal{A}_\psi}(\widehat{V}) \cong
{\iota_U^V}_*(\widehat{\mathcal{A}_V})
\\ \xrightarrow{\gamma^V_U(\widehat{U})}
\widehat{\mathcal{A}_U}\cong p_*(\widehat{\mathcal{A}_U})=Bar^{(n)}(A)(U)\end{multline*} is a quasi-isomorphism.  That $\gamma^V_U(\widehat{U})$ is a quasi-isomorphism is given by Lemma~\ref{L:StructureMapsBar1}.\eqref{eq:iotaisquis}. Hence, so is   $\tilde{\gamma}^V_U$.

 We have proved that  
 the rule $\phi\mapsto Bar^{(n)}(\mathcal{A})(\phi)$ (see construction~\eqref{D:BarisaParamFact}) is a locally constant $N(\text{Disk}(\mathbb{R}^n))$-coalgebra object in $E_{m-n}$-algebras.  Consequently,  the iterated Bar construction given by Definition~\ref{D:BariteratedforEm}
 is a functor from augmented $E_m$-algebras to  $E_n\text{-coAlg}\Big(E_{m-n}\text{-Alg}^{aug}\Big)$.
 By Proposition~\ref{P:BarEkisEkminus1}, this functor agrees (in the $(\infty,1)$-category  $E_{m-n}\txt{-Alg}^{aug}$) with the one given in Section~\ref{S:Bar}. The identification of the two  $E_n$-coalgebras structure is done as in the proof of Proposition~\ref{P:HH(A,B)=CH(A,B)}.
\end{proof}

\begin{rem}[\emph{sketch of a variant}]\label{R:Encoalgsketch} Let $A$ be an $E_n$-algebra induced by a factorization algebra $\mathcal{A}$ on $\R^n$.  
The factorization algebra $\widehat{\mathcal{A}}$ on $S^n$ (from Definition~\ref{D:BarasFactAlgonSn}) is obtained by pushing forward the factorization algebra $(A,k)$ on the stratified closed disk $I^n$ from Definition~\ref{D:BarasFactAlg}.
Further, for convex bounded open subsets of $\R^n$,   we can think of the iterated Bar construction of $A$, restricted on $V$ as a stratified factorization algebra $D\mapsto Bar^{(n)}(A)(D)$ on the closure $\overline{V}$ of $V$ (which assigns the $A$-$E_n$-module $k$ to balls in a neighborhood of the boundary $\overline{V}\setminus V$). 
In fact, for any disk $V$ and any homeomorphism $\psi: \R^n \stackrel{\simeq}\to  V$, we can construct a factorization algebra $(A_{\psi},k)$ on the stratified closed disk $I^n$ and the global section of this factorization algebra is quasi-isomorphic to $Bar^{(n)}(A)$.  It is possible to define this way a locally constant \emph{parametrized} factorization algebra on $\R^n$ (Definition~\ref{D:AlternativeFacAlg}) which is equivalent as the one we construct using $\widehat{\mathcal{A}}$ in Theorem~\ref{P:BarforEn}.
 
The basic idea is that, given sub-disks $U_1,\dots, U_r$ in $V$ with homeomorphisms $\phi_i:\R^n \stackrel{\simeq}\to U_i$ and an embedding  $h: \coprod_{i=1}^r \R^n \to \R^n$ such that $\psi \circ h= \coprod_{i=1}^{r} \phi_i$, we can construct a locally constant stratified factorization algebra $\mathcal{F}$ on $I^n$ which is stratified  with one open strata given by the union of  the disks
$h(\bigcup_{i=1}^r \R^n)$ and one closed stratum given by their
complement $I^n\setminus h\big(\bigcup_{i=1}^r \R^n\big)$.
 Then $\mathcal{F}$ is roughly defined as the rule which to each ball $D$  inside $\phi_i^{-1}(U_i)$ associates $\int_{\phi_{i} (D)} A$,  and which associates $\mathcal{F}(D)=k$ on the closed strata.
 The factorization algebra structure is given by the $A$-$E_n$-module structure of $k$. 
The map which is the identity on  each disk $D$ inside the preimage of a $U_i$
and is the augmentation $\epsilon: A\to k$ on each disk in a small
neighborhood of the closed strata defines a map of factorization
algebra $(A_{\psi},k) \to \mathcal{F}$, which on the global section is a map from $Bar^{(n)}(A_\psi) \to \bigotimes_{i=1}^{r} Bar^{(n)}(A_{\phi_i})$. 
\end{rem}
\medskip

Let $\epsilon:\mathcal{A}\to k$ be a map of augmented locally constant factorization algebras over $\R^{m}$ and   $1\leq i, j$ be such that $i+j\leq m$. 
By Theorem~\ref{P:BarforEn}, we have the $i^{\mbox{th}}$ Bar construction  $\mathcal{B}\textit{ar}^{(i)}(\mathcal{A})\in \text{coFac}_{\R^i}^{lc, aug}\Big(\text{Fac}_{\R^{m-i}}^{lc, aug}\Big).$
In particular for every open set $U\in \R^i$, we get an augmented factorization algebra $\mathcal{B}\textit{ar}^{(i)}(\mathcal{A})(U)\in \text{Fac}_{\R^{m-i}}^{lc, aug}$ from which, by Theorem~\ref{P:BarforEn} again, we get  
$$\mathcal{B}\textit{ar}^{(j)}\big(\mathcal{B}\textit{ar}^{(i)}(\mathcal{A})(U)\big)\in \text{coFac}_{\R^j}^{lc}\Big(\text{Fac}_{\R^{m-i-j}}^{lc, aug}\Big).$$
Recall that the structure maps $\delta_{U_1\dots, U_r, V}$ from Definition~\ref{D:BarisaParamFact} (associated to the functor $\mathcal{B}\textit{ar}^{(i)}$ and sub-disks $U_i$, $V$) are as the tensor product $\bigotimes_{i=1}^r \tilde{\gamma}_{U_i}^{V}$ where the $\tilde{\gamma}_{U_i}^{V}$ are maps of augmented factorization algebras over $\R^{m-i}$. 
Hence we get a map 
\begin{multline}\label{eq:structuremapsBariBarj}\bigotimes_{i=1}^r \mathcal{B}ar^{(j)}\big(\tilde{\gamma}_{U_i}^{V}\big):\mathcal{B}\textit{ar}^{(j)}\big(\mathcal{B}\textit{ar}^{(i)}(\mathcal{A})(V)\big) \\ \longrightarrow \mathcal{B}\textit{ar}^{(j)}\big(\mathcal{B}\textit{ar}^{(i)}(\mathcal{A})(U_1)\big)\otimes \cdots \otimes \mathcal{B}\textit{ar}^{(j)}\big(\mathcal{B}\textit{ar}^{(i)}(\mathcal{A})(U_r)\big)\end{multline} in $\text{coFac}_{\R^j}^{lc}\Big(\text{Fac}_{\R^{m-i-j}}^{lc, aug}\Big)$. 
The proof of Theorem~\ref{P:BarforEn} and the proof of Lemma~\ref{L:iteratedBarforEnisCH} shows that
\begin{prop} Let $\epsilon:\mathcal{A}\to k$ be a map of augmented locally constant factorization algebras over $\R^{m}$ and   $1\leq i, j$ be such that $i+j\leq m$. 
\begin{enumerate} \item The structure maps~\eqref{eq:structuremapsBariBarj} make 
 $\mathcal{B}\textit{ar}^{(j)}\big(\mathcal{B}\textit{ar}^{(i)}(\mathcal{A})\big)$ an object of  the $(\infty,1)$-category $\text{coFac}_{\R^i}^{lc}\Big(\text{coFac}_{\R^j}^{lc}\Big(\text{Fac}_{\R^{m-i-j}}^{lc, aug}\Big)\Big)$, functorially in $\mathcal{A}$: in other words we have a functor $$\mathcal{B}\textit{ar}^{(j)}\circ \mathcal{B}\textit{ar}^{(i)}:\text{Fac}_{\R^{m}}^{lc, aug} \longrightarrow  \text{coFac}_{\R^i}^{lc}\Big(\text{coFac}_{\R^j}^{lc}\Big(\text{Fac}_{\R^{m-i-j}}^{lc, aug}\Big)\Big).$$
 \item There is a commutative diagram of functors:$$\xymatrix{ \text{Fac}_{\R^{m}}^{lc, aug} \ar[rrd]_{\mathcal{B}\textit{ar}^{(j)}\circ \mathcal{B}\textit{ar}^{(i)}\quad}
 \ar[rr]^{\mathcal{B}\textit{ar}^{(i+j)}}& &\text{coFac}_{\R^{i+j}}^{lc}\Big(\text{Fac}_{\R^{m-i-j}}^{lc, aug}\Big) \ar[d]_{\simeq}^{\pi_*}\\ 
 &&\text{coFac}_{\R^i}^{lc}\Big(\text{coFac}_{\R^j}^{lc}\Big(\text{Fac}_{\R^{m-i-j}}^{lc, aug}\Big)\Big) } $$ where the left vertical arrow is the pushforward.
 \end{enumerate}
\end{prop}

In other words, through Dunn isomorphism, the proposition states that the functor $\mathcal{B}\textit{ar}^{(n)}$ is the same as the $n$-times iterated Bar construction $\mathcal{B}\textit{ar}^{(1)}\circ \cdots \circ \mathcal{B}\textit{ar}^{(1)} $.

\medskip
We finish this section by comparing the iterated Bar construction of Theorem~\ref{P:BarforEn} with centralizers and the construction of \S~\ref{S:Bar}.
 \begin{prop}\label{P:IdentificationBarforEn} Let $\epsilon:A\to k$ be an augmented $E_m$-algebra and $0\leq n\leq
m$.
\begin{enumerate}
\item \label{eq:ClaimdualBar}The dual  $RHom(Bar^{(m)}(A), k)$,
endowed with the $E_{m}$-algebra structure dual to the
$E_{m}$-coalgebra structure of $Bar^{(m)}(A)$ (given by
Theorem~\ref{P:BarforEn}.(2)), is the centralizer
$\mathfrak{z}(A\stackrel{\epsilon}\to k)$ of the augmentation (see
\S~\ref{S:maincentralizers}).
\item $Bar^{(1)}(A)$ is equivalent as an $E_1$-coalgebra to the standard (\S~\ref{S:Bar})
 Bar construction $Bar^{std}(A)$ and $Bar^{(n)}(A)$ is equivalent to the iterated Bar constructions of~\cite{F}
  (in the $\infty$-category $E_n\text{-coAlg}\Big(E_{m-n}\text{-Alg}^{aug}\Big)$).
  \item If $m=\infty$, the iterated Bar functor $$Bar^{(n)}: E_\infty\text{-Alg}^{aug} \longrightarrow
E_n\text{-coAlg}\Big(E_{\infty}\text{-Alg}^{aug}\Big)$$ given by Theorem~\ref{P:BarforEn} is naturally
equivalent to the one obtained in \S~\ref{S:Bar} (and in particular
Theorem~\ref{P:EncoAlgBar}). \label{eq:ClaimBarEinfty}
\end{enumerate}
\end{prop}
\begin{proof}
 Dualizing the construction of the locally constant $N(\text{Disk}(\mathbb{R}^n))$-coalgebra structure shows that the dual  $RHom(Bar^{(n)}(A), k)$ of the Bar construction has a locally constant $N(\text{Disk}(\mathbb{R}^n))$-algebra structure whose global section gives us the $E_{m}$-algebra structure on $RHom(Bar^{(m)}(A), k)$ asserted in Claim~\eqref{eq:ClaimdualBar}
 
 By Proposition~\ref{P:extensionfrombasis}, it is enough to check that this dual structure coincides with the one given in Theorem~\ref{T:EnAlgHoch} on the factorizing basis $\mathcal{CVX}$ of bounded convex open subsets of $\R^n$. 
 Recall $\int_{V} k \cong k(V) =k$ for any open $V$.
 For $U\in \mathcal{CVX}$ with center $*_U$, we have have an natural equivalence
 \begin{equation}\label{eq:Bar(n)Uexcision}
  Bar^{(m)}(A)(U)\; \cong \; p_*(\widehat{\mathcal{A}_U}) \;\;\cong \;\;       \int_{U} A \;\mathop{\otimes}\limits_{\int_{U\setminus \{*_U\}}\!\!A}^{\mathbb{L}}\; \int_{\widehat{U}\setminus \{*_U\}} k \;\;\cong \;\;\int_{U} A \;\mathop{\otimes}\limits_{\int_{U\setminus \{*_U\}}\!\!A}^{\mathbb{L}}\; k
 \end{equation}
given by Lemma~\ref{L:BarasFactSn}, Proposition~\ref{P:BarasFactonSn} and Lemma~\ref{L:BarasFactIn} (this also follows from Remark~\ref{R:PhiisotoU2} applied to any $\phi: \R^n \stackrel{\simeq}\longrightarrow U$ such that $\phi(0)=*_U$). 
 It follows that
 \begin{eqnarray}\label{eq:dualofBariscentralizer}
  RHom(Bar^{(m)}(A)(U), k(U)) &\cong & RHom^{left}_{\int_{U\setminus \{*_U\}}\!\!A}\Big( \int_{U} A , \int_U k\Big)\\
  &\cong & RHom^{\mathcal{E}_n}_{A}\big(A, k\big)(U) \nonumber
 \end{eqnarray}
where the last line is from Step 2, \S~\ref{SS:step2} and the $A$-$E_n$-module structure on $k$ is given by the augmentation $\epsilon: A\to k$.
To conclude that 
 the dual of $U\mapsto Bar^{(m)}(A)(U)$ is the factorization algebra of Theorem~\ref{T:EnAlgHoch}, it remains to the compare the dual of the structure maps of Definition~\ref{D:BarisaParamFact} with the ones in \S~\ref{SS:step2}.
 
 Let  $U_1,\dots, U_r$ be convex open sets lying inside a bounded convex open set $V$.  By Lemma~\ref{L:StructureMapsBar1}, the dual 
 $RHom(\gamma^{V}_{U_i},k)$ is a factorization algebra map on $U_i$ which is given by the augmentation $\widehat{\epsilon}$ on every open subset $W=\widehat{U}\setminus \overline{D}$ which is the complement of a compact Euclidean disks containing the $*_i$. 
 Further, on any open subset $W'$ of such a $W$, the image  under the equivalence~\eqref{eq:dualofBariscentralizer} of $RHom(\gamma^{V}_{U_i},k)$ evaluated on $W'$,  is a section in $\text{Map}_{\text{Fac}^{lc}_{{U_i}_{*_i}}}(\mathcal{A}_{|U_i}, \mathcal{B}_{|U_i})$ (see \S~\ref{SS:step2}) which, again, is simply given by the  augmentation.
 
 \emph{A contrario}, on any open set $W$ lying inside $U_i$, the dual 
 $RHom(\gamma^{V}_{U_i},k)$ is the identity. Thus, its image under the equivalence~\eqref{eq:dualofBariscentralizer} on any open subset $*_i\subset W\subset U_i$, is  just the map taking a global section $f\in RHom^{\mathcal{E}_n}_{A}\big(A, k\big)(U_i) $ to its restriction on $W$. 
 
 Since $\delta_{U_1,\dots, U_r}$ is obtained by  tensor product of the $\gamma^{V}_{U_i}$ (Definition~\ref{D:BarisaParamFact}), it follows that the dual of $\delta_{U_1,\dots, U_r}$ coincides with the structure maps $\rho_{U_1,d\dots,U_r, V}$ given by Formula~\eqref{eq:FormularhoonV} on the cover $\mathcal{U}_{U_1,\dots, U_r,V}$.   This proves Claim~\eqref{eq:ClaimdualBar}.

\smallskip

That the algebraic of $Bar^{(n)}(A)$ agrees with the one in~\cite{F} follows from Dunn Theorem (see~\cite{L-HA, F} or
 Theorem~\ref{T:Dunn}) once we know that $Bar^{(1)}(A)$ is equivalent, as an $E_1$-coalgebra,  to the standard
 Bar construction $Bar^{std}(A)$. By homotopy invariance, we may assume that $A$ is a differential
  graded associative algebra. By Lemma~\ref{L:Bar=BarstdEm}, we have a natural equivalence $Bar(A)\cong Bar^{std}(A)$
  and further the (two constructions) of the Bar construction computes the derived functor
$k\otimes_{A}^\mathbb{L} k$.
 The coalgebra structure of  $Bar^{std}(A)$ is induced by the comultiplication $\delta:   Bar^{std}(A)\to
 Bar^{std}(A)\otimes Bar^{std}(A)$ which realized the following map of derived functors (in $\hkmod$):
 \begin{equation}\label{eq:deltaderived}\delta: k\otimes_{A}^\mathbb{L} k \cong k\otimes_{A}^\mathbb{L} A
 \otimes_{A}^\mathbb{L} k \stackrel{id\otimes_A^{\mathbb{L}}\epsilon \otimes_A^{\mathbb{L}}id} \longrightarrow
  k\otimes_{A}^\mathbb{L} k\otimes_{A}^\mathbb{L} k \cong \Big(k\otimes_{A}^\mathbb{L} k\Big)^{\otimes 2}.\end{equation}
 The construction~\eqref{eq:DefnBarEm} can be rewritten as
 $$ Bar(A) \;\cong \;k\mathop{\otimes}_A^{\mathbb{L}} \int_{I} A \mathop{\otimes}_A^{\mathbb{L}} k$$ using the natural
 $A\otimes A^{op}\cong \int_{S^0}A$-module structure of $\int_I A$.
 Now the $E_1$-coalgebra structure of $Bar(A)$ is given by the inclusion of two disjoint open intervals $I_1$ and
 $I_2$ inside $I$. We denote $J_1, J_2, J_3$ the three disjoint intervals
whose union is the complement $I\setminus (I_1\cup I_2)$.
Unfolding the definition of the map $\delta_{I_1, I_2, I}$ given by Definition~\ref{D:BarisaParamFact} and Lemma~\ref{L:StructureMapsBar1},
 using excision for factorization homology (see~\cite{L-HA, F, GTZ2, AFT}), we find that,  $\delta_{I_1, I_2, I}$
is the composition
\begin{multline}\label{eq:checkBarconstruction}
 Bar(A) \;\cong \;k\mathop{\otimes}_A^{\mathbb{L}} \int_{I} A \mathop{\otimes}_A^{\mathbb{L}} k \longrightarrow
 k\mathop{\otimes}_A^{\mathbb{L}} \int_{J_1} A \mathop{\otimes}_A^{\mathbb{L}}
  \int_{I_1} A \mathop{\otimes}_A^{\mathbb{L}} \int_{J_2} A \mathop{\otimes}_A^{\mathbb{L}}
  \int_{I_2} A \mathop{\otimes}_A^{\mathbb{L}} \int_{J_3} A\mathop{\otimes}_A^{\mathbb{L}} k \\
 \stackrel{id\otimes_A^{\mathbb{L}}\epsilon \otimes_A^{\mathbb{L}}id\otimes_A^{\mathbb{L}}\epsilon \otimes_A^{\mathbb{L}}
 id\otimes_A^{\mathbb{L}}\epsilon \otimes_A^{\mathbb{L}}
 id} \longrightarrow k\mathop{\otimes}_A^{\mathbb{L}}   \int_{I_1} A \mathop{\otimes}_A^{\mathbb{L}}k
 \mathop{\otimes}_A^{\mathbb{L}}  \int_{I_2} A \mathop{\otimes}_A^{\mathbb{L}}  k \cong Bar(A) \otimes Bar(A).
\end{multline}
Hence, the underlying coproducts of the  $E_1$-coalgebra structure on $Bar(A)$ realize the map~\eqref{eq:deltaderived}.
Thus, they induce the $E_1$-coalgebra structure of $Bar^{std}(A)$ under the equivalence given by Lemma~\ref{L:Bar=BarstdEm}.

\smallskip

We are left to prove Claim~\eqref{eq:ClaimBarEinfty}. By Proposition~\ref{P:BarEkisEkminus1} and Lemma~\ref{L:iteratedBarforEnisCH}, we know that the  iterated bar functor $Bar^{(n)}$ from Theorem~\ref{P:BarforEn} coincides in $E_\infty\text{-Alg}^{aug}$  with the the one obtained in \S~\ref{S:Bar}. 

We  need to compare the $E_n$-coalgebra structures.
By Lemma~\ref{L:cubetoDisk}, we are left to compare the structure maps $\delta_{c_1,\dots, c_r,\R^n}: \bigotimes_{i=1}^{r} Bar^{(n)}(A)(c_i) \to Bar^{(n)}(A)(\R^n)$ with the maps~\eqref{eq:copinchSn} giving rise to the structure in Theorem~\ref{P:EncoAlgBar}. 

Further,  from equivalence~\eqref{eq:Bar(n)Uexcision} above, Proposition~\ref{P:nBarisCHIn} and its proof we obtain a commutative diagram of equivalences
\begin{equation}\label{eq:BarnandCHinfty0}
 \xymatrix{ p_*(\widehat{\mathcal{A}_U})= Bar^{(n)}(A)(U)  \ar[rr]^{\cong} && CH_{\widehat{U}}(A)\mathop{\otimes}\limits_{A}^{\mathbb{L}} k\cong CH_{\widehat{U}}(A,k)\\ 
\int_{U} A \;\mathop{\otimes}\limits_{\int_{U\setminus \{*_U\}}\!\!A}^{\mathbb{L}}\; k \ar[rr]^{\cong} \ar[u]^{\cong}&& CH_{U}(A) \mathop{\otimes}\limits_{CH_{U\setminus \{*_U\}}(A)}^{\mathbb{L}} k \ar[u]_{\cong}} 
\end{equation}
for every convex open set (in particular cube) $U\subset \R^n$. The lower arrow of the diagram is the tensor product of the equivalences between factorization and Hochschild  homology given by Theorem~\ref{T:CH=TCH}. 

We wish to analyze the structure maps $\delta_{U_1,\dots, U_r, V}$ (where all the sets $U_i$'s, $V$ are convex) under this equivalence. 
For any $i=1\dots r$, from the above diagram~\eqref{eq:BarnandCHinfty0} and the definition of the map~\eqref{eq:DefGamma}, we get the commutative diagrams
\begin{equation}\label{eq:BarnandCHinfty}
 \xymatrix{ p_*(\widehat{\mathcal{A}_V}) \ar[d]_{\cong}  \ar[rr]^{\tilde{\gamma}^V_{U_i}} && p_*(\widehat{\mathcal{A}_{U_i}})  \ar[d]^{\cong} \\ CH_{\widehat{V}}(A)\mathop{\otimes}\limits_{A}^{\mathbb{L}} k \ar[rr]^{\big(\iota^{V}_{U_i}\big)_*\otimes id} && CH_{\widehat{U_i}}(A)\mathop{\otimes}\limits_{A}^{\mathbb{L}} k} 
\end{equation}
where $\iota^{V}_{U_i}: \widehat{V}\to \widehat{U_i}$ is the map~\eqref{eq:DefiotaUV} which collapses the complement of $U_i$ in $\widehat{V}$ to a point.

Recall that $\delta_{U_1,\dots, U_r, V}$ is the tensor product $\bigotimes \tilde{\gamma}^V_{U_i}$ (Definition~\ref{D:BarisaParamFact}), tensoring the  commutative diagrams~\eqref{eq:BarnandCHinfty} applied to cubes $U_1,\dots, U_r$ inside $V=\R^n$, we get the commutative diagram
\begin{equation*}
 \xymatrix{ Bar^{(n)}(A)(\mathbb{R}^n)=p_*(\widehat{\mathcal{A}_{\mathbb{R}^{n}}}) \ar[d]_{\cong}  \ar[rrr]^{\hspace{-1pc}\delta_{U_1,\dots, U_r, \R^n}} &&& \bigotimes\limits_{i=1}^r p_*(\widehat{\mathcal{A}_{U_i}}) =  \bigotimes\limits_{i=1}^r Bar^{(n)}(A)(U_i)    \ar[d]^{\cong}\\ CH_{S^n}(A)\mathop{\otimes}\limits_{A}^{\mathbb{L}} k \ar[rrr]^{\hspace{-5pc}pinch^{S^n,r}_*(U_1,\dots,U_r)} &&& \bigotimes\limits_{i=1}^r \Big(CH_{\widehat{U_i}}(A)\mathop{\otimes}\limits_{A}^{\mathbb{L}} k\Big) \cong \Big(CH_{S^n}(A,k)\Big)^{\otimes r}} 
\end{equation*}
where the lower map is the pinching map~\eqref{eq:copinchSn} applied to the cubes $U_1,\dots, U_r$. Together with Lemma~\ref{L:cubetoDisk}, this proves that the $E_n$-coalgebra structure given by Theorem~\ref{P:EncoAlgBar} is the same as the one from Theorem~\ref{P:BarforEn}.
\end{proof}

\begin{rem}[\textbf{$E_n$-analogues of (homotopy) bialgebras}]
The category of (differential graded) bialgebras
 is the same as the category $\text{coAlg}(\text{Alg})$ of (differential graded) coalgebra objects in the category of
 (differential graded) algebras.
 
 In particular, the $(\infty,1)$-category $E_1\text{-coAlg}\big(E_{1}\text{-Alg}\big)$ is equivalent to the
$(\infty,1)$-category of  (differential graded) bialgebras in
 $\hkmod$.
 
 We thus think of $ E_p\text{-coAlg}\big(E_{q}\text{-Alg}\big)$ as analogues of  bialgebras with some commutativity
  and cocommutativity conditions lying in between (dg-)bialgebras and (dg-)commutative and cocommutative bialgebras. 

Note that in characteristic zero, by choice of a formality isomorphism $\mathbb{P}_d\cong \mathbb{E}_d$, a model for the $\infty$-category  $E_d\text{-coAlg}\big(E_{1}\text{-Alg}\big)$ is given by the $\infty$-category of (homotopy) $d$-bialgebras considered by Tamarkin~\cite{Ta-Defofdalgebra}.

Also,
Proposition~\ref{P:BarforEn} implies that the Bar construction of an
$E_2$-algebra is naturally a (homotopy) bialgebra. It would be
interesting to relate this result with a "somehow dual" result  of
Kadeishvili~\cite{Kad} stating that the cobar construction of a
(dg-)bialgebra has an natural structure of  homotopy Gerstenhaber
algebra structure, hence of $E_2$-algebras in characteristic zero.
\end{rem}


\begin{thebibliography}{1111}
\bibitem[AFT]{AFT} D. Ayala, J. Francis, H.-L. Tanaka, \textit{Factorization homology for stratified manifolds}, preprint
\bibitem[Ba]{Ba} T. Bargheer, \textit{The Cleavage Operad and String Topology of Higher Dimension}, preprint arXiv:1012.4839
\bibitem[BGNX]{BGNX} K. Behrend, G. Ginot, B. Noohi, P. Xu, \textit{String topology for stacks}, Ast\'erisque, No. 343 (2012), pp. xiv+169
\bibitem[BD]{BD} A. Beilinson, V. Drinfeld, \textit{Chiral algebras}, American Mathematical Society Colloquium Publications, 51. American Mathematical Society, Providence, RI, 2004
\bibitem[BF]{BF} C. Berger, B. Fresse, \textit{Combinatorial operad actions on cochains}, Math. Proc. Cambridge Philos. Soc. 137 (2004), 135-174.
\bibitem[BS]{BoSe} R. Bott, \ G. Segal, \textit{The cohomology of the vector fields on a manifold}, Topology {\bf 16} (1977), no.~4, 285--298.
\bibitem[CW]{CaWi-Formality} D. Calaque, T. Willwacher, \textit{Triviality of the higher Formality Theorem}, Preprint arXiv:1310.4605 
\bibitem[CS]{CS} M. Chas, D. Sullivan, \textit{String Topology}, arXiv:math/9911159
\bibitem[C]{Chataur} D. Chataur, \textit{A bordism approach to string topology}, Int. Math. Res. Not. {\bf 2005}, no.~46, 2829--2875.
\bibitem[Ch]{Ch}
K.-T. Chen, \textit{Iterated integrals of differential forms and loop space homology},
 Ann. of Math. (2) 97 (1973), 217--246.
 \bibitem[CiMo]{CiMo-Dendroidalasinftyoperad} D.-C. Cisinski, I. Moerdijk, \textit{Dendroidal sets as models for homotopy operads}, J. Topol. {\bf 4} (2011), no.~2, 257--299.
\bibitem[CiMo2] {CiMo-DendroidalasinftyoperadII} D.-C. Cisinski, I. Moerdijk, \textit{Dendroidal spaces and $\infty$-operads},  J. Topol. {\bf 6} (2013), no.~3, 675--704.
\bibitem[CiMo3]{CiMo-Dendroidalasimplicialoperad} D.-C. Cisinski, I. Moerdijk, \textit{Dendroidal sets and simplicial operads},  J. Topol. {\bf 6} (2013), no.~3, 705--756.
\bibitem[CV]{CV}
R. Cohen, A. Voronov, \textit{Notes on String topology},  String topology and cyclic homology,  1--95, Adv. Courses Math. CRM Barcelona, Birkh\"auser, Basel, 2006.
\bibitem[CJY]{CJY} R. Cohen, J.D.S. Jones, J. Yan, \textit{The loop homology algebra of spheres and projective spaces}, Categorical decomposition techniques in algebraic topology (Isle of Skye, 2001), 77--92, Progr. Math., 215, Birkh\"auser, Basel, 2004. 
\bibitem[CG]{CG} K. Costello, O. Gwilliam, \textit{Factorization algebras in perturbative quantum field theory},
 \\available at
http://math.northwestern.edu/$\sim$costello/factorization$\_$public.html
\bibitem[C]{Co} K. Costello, \textit{A geometric construction of Witten Genus, I},  Proceedings of the International Congress of Mathematicians. Volume II, 942--959, Hindustan Book Agency, New Delhi.
\bibitem[Du]{Du} G.~Dunn, \textit{Tensor product of operads and iterated loop spaces},
 J. Pure Appl. Algebra 50 (1988), no. 3, 237--258.
\bibitem[DK]{DK} W.G. Dwyer, D.M. Kan, \textit{Calculating simplicial localizations}, J. Pure Appl. Alg. 18 (1980), 17--35.
\bibitem[EKMM]{EKMM} A. D. Elmendorf, M. Mandell, I. Kriz, J. P. May, \textit{Rings, modules, and algebras in stable homotopy theory}, Mathematical Surveys and Monographs, 47, Amer. Math. Soc., Providence, RI, 1997.
\bibitem[FHT]{FHT} Y. F\'elix, S. Halperin, J.-C. Thomas, \textit{Rational homotopy theory}, Graduate Texts in Mathematics, 205. Springer-Verlag. (2009)
\bibitem[FT]{FT} Y. F\'elix, J.-C.~Thomas, \textit{Rational BV-algebra in string topology},
Bull. Soc. Math. France {\bf 136} (2008), no.~2, 311--327
\bibitem[FTV]{FTV} Y.~F\'elix, J.-C. Thomas, M. Vigu\'e, \textit{The Hochschild cohomology of a closed manifold}, Publ. Math. Inst. Hautes \'Etudes Sci. No. 99 (2004), 235--252.
\bibitem[F1]{F} J. Francis. \textit{The tangent complex and Hochschild cohomology of  $E_n$-rings},  Compos. Math. 149 (2013), no. 3, 430--480.
\bibitem[F2]{F2} J. Francis, \textit{Factorization homology of topological manifolds},  	arXiv:1206.5522
\bibitem[Fre]{Fre-Mod} B. Fresse, \textit{Modules over Operads and Functors}, Lect. Notes in Maths. vol. 1967, Springer verlag (2009).
\bibitem[Fre2]{Fre2}  B. Fresse, \textit{The bar complex of an $E_\infty$-algebra}, Adv. Math. 223 (2010), pp. 2049--2096.
\bibitem[Fre3]{FresseBarEn} B. Fresse, \textit{Iterated bar complexes of E-infinity algebras and homology theories}, Alg. Geom. Topol. {\bf 11} (2011), pp. 747--838
\bibitem[Ge]{Ge} M. Gerstenhaber, \textit{The Cohomology Structure Of An
 Associative ring} Ann. Maths. 78(2) (1963).
\bibitem[GeJo]{GeJo} E. Getzler, J.D.S.~Jones, \textit{Operads, homotopy algebra and iterated integrals for double loop spaces}, arXiv: hep-th/9403055.
\bibitem[G1]{G} G. Ginot, \textit{Higher order Hochschild Cohomology},  C. R. Math. Acad. Sci. Paris  346  (2008),  no. 1-2, 5--10.
\bibitem[G2]{G-Houches} G. Ginot, \textit{Notes on factorization algebras, factorization homology and applications}, arXiv:1307.5213.
\bibitem[GiHa]{GiHa} G. Ginot, G. Halbout, \textit{A formality theorem for Poisson manifolds}, Lett. Math. Phys. {\bf 66} (2003), no.~1-2, 37--64.
\bibitem[GTZ]{GTZ} G. Ginot, T. Tradler, M. Zeinalian, \textit{A Chen model for mapping spaces and the surface product}, Ann. Sc. de l'\'Ec. Norm. Sup., 4e s\'erie, t. 43 (2010), p. 811-881.
\bibitem[GTZ2]{GTZ2} G. Ginot, T. Tradler, M. Zeinalian, \textit{Higher Hochschild Homology, Topological Chiral Homology and Factorization Algebras}, Comm. Math. Phys. 326 (2014), no. 3, 635--686. 
\bibitem[GJ]{GoJa} P. Goerss, J. Jardine, \textit{Simplicial Homotopy Theory}, Modern Birkh\"auser Classics, first ed. (2009), Birkh\"auser Basel.
\bibitem[Hu]{Hu} P. Hu, \textit{Higher string topology on general spaces}, Proc. London Math. Soc. \textbf{93}
(2006), 515--544.
\bibitem[HKV]{HKV}  P. Hu, I. Kriz, A. Voronov, \textit{On Kontsevich's Hochschild cohomology conjecture}, Compos. Math.
\textbf{142} (2006), no. 1, 143--168.
\bibitem[H1]{H} G. Hochschild, \textit{On the cohomology groups of an associative algebra}, Ann. of Math. (2) 46, (1945). 58--67.
\bibitem[H2]{H2} G. Hochschild, \textit{On the cohomology theory for associative algebras}, Ann. of Math. (2) 47, (1946). 568--579.
\bibitem[Ho]{Ho} M. Hovey, \textit{Model Categories}, Mathematical Surveys and Monographs, 63. American Mathematical Society, Providence, RI, 1999. xii+209 pp.
\bibitem[Jo]{Jo} J.D.S. Jones \textit{Cyclic homology and equivariant homology},
 Inv. Math. 87, no.2 (1987), 403--423
 \bibitem[Ka]{Kad} T. Kadeishvili, \textit{On the cobar construction of a bialgebra},  Homology
Homotopy Appl. 7 (2005), no. 2, 109--122.
 \bibitem[Ke]{Ke} G. M. Kelly, \textit{Basic concepts of enriched category theory}, London Mathematical Society Lecture Note Series, 64, Cambridge Univ. Press, Cambridge, 1982.
 \bibitem[KS]{KS} M. Konstevich, I. Soibelman, \textit{Deformation Theory}, Vol. 1, Unpublished book draft.
\bibitem[KM]{KM} I. K{r}{i}{z},  J. P. May, \textit{Operads, algebras, modules and motives}, Ast\'erisque No. 233 (1995), {\rm iv}+145pp.
\bibitem[LV]{LV} P. Lambrechts, I. Volic, \textit{Formality of the little N-discs operad}, arXiv:0808.0457,  to appear in Memoirs of the AMS.
\bibitem[L]{L} J.-L. Loday, \textit{Cyclic Homology}, Grundlehren der mathematischen Wissenschaften 301 (1992), Springer Verlag.
\bibitem[L-HTT]{Lu11} J. Lurie, \textit{Higher Topos Theory}, Annals of Mathematics Studies, 170. Princeton University Press, Princeton, NJ, 2009. xviii+925 pp.
\bibitem[Lu1]{L-II} J. Lurie, \textit{Derived Algebraic Geometry II: Noncommutative Algebra}, preprint, arXiv:math/0702299.
\bibitem[Lu2]{L-III} J. Lurie, \textit{Derived Algebraic Geometry III: Commutative Algebra}, preprint, arXiv:math/0703204v4.
\bibitem[Lu3]{L-VI} J. Lurie, \textit{Derived Algebraic Geometry VI: $E_k$ Algebras}, preprint, arXiv:0911.0018v1.
\bibitem[Lu4]{L-TFT} J. Lurie, \textit{On the Classification of Topological Field Theories}, preprint, arXiv:0905.0465v1.
\bibitem[Lu-MP]{L-MP} J. Lurie, \textit{Moduli Problem for ring spectra}, preprint, \\
available at http://www.math.harvard.edu/~lurie/
\bibitem[L-HA]{L-HA} J. Lurie, \textit{Higher Algebra}, book, available at http://www.math.harvard.edu/~lurie/
\bibitem[McC]{McC} R.~McCarthy, \textit{On operations for Hochschild homology},  Comm. Algebra  21  (1993),  no. 8, 2947--2965.
\bibitem[MCSV]{MCSV} J. McClure, R. Schw\"anzl\ and\ R. Vogt, \textit{$THH(R)\cong R\otimes S\sp 1$ for $E\sb \infty$ ring spectra}, J. Pure Appl. Algebra {\bf 121} (1997), no.~2, 137--159.
\bibitem[M1]{M} M. Mandell, \textit{Topological Andr\'e-Quillen cohomology and $E_\infty$ Andr\'e-Quillen cohomology}, Adv. Math. 177 (2003), no. 2, 227--279.
\bibitem[M2]{M2}  M. Mandell, \textit{Cochains and homotopy type}, Publ. Math. Inst. Hautes \'Etudes Sci. No. 103 (2006), 213--246.
\bibitem[Ma]{Ma}  J. P. May, \textit{Simplicial Objects in Algebraic Topology}, University of Chicago Press, 1992
\bibitem[PTVV]{PTVV} T. Pantev, B. To\"en, M. Vaqui\'e, G. Vezzosi, \textit{Shifted symplectic structures},  Publ. Math. Inst. Hautes ƒtudes Sci. 117 (2013), 271--328.
\bibitem[PT]{PT} F. Patras, J.-C. Thomas, \textit{Cochain algebras of mapping spaces and finite group actions}, Topology and its Applications 128 (2003), pp. 189--207.
\bibitem[P]{P} T. Pirashvili, \textit{Hodge Decomposition for higher order Hochschild Homology},  Ann. Sci. \'Ecole Norm. Sup. (4)  33  (2000),  no. 2, 151--179.
\bibitem[R]{Re} C. Rezk, \textit{A model for the homotopy theory of homotopy theory}, Trans. Amer. Math. Soc. 353 (3) (2001), 937--1007.
\bibitem[SW]{SW} P.~Salvatore, N.~Wahl, \textit{Framed discs operads and Batalin Vilkovisky algebras}, Q. J. Math. {\bf 54} (2003), no. 2, 213--231.
\bibitem[Ta]{Ta-formality} D. Tamarkin, \textit{Another proof of M. Kontsevich formality theorem}, math.QA/9803025.
\bibitem[Ta2]{Ta-Defofdalgebra} D. Tamarkin, \textit{Deformation complex of a d-algebra is a (d+1)-algebra}, preprint arXiv:math/0010072.
\bibitem[Ta3]{Tamarkin-GTactsE2} D. Tamarkin, \textit{Action of the Grothendieck-Teichmueller group on the operad of Gerstenhaber algebras}, preprint arXiv: math/0202039.
\bibitem[To]{Toen-Formality} B. To\"en, \textit{Operations on derived moduli spaces of branes}, Preprint arXiv:1307.0405
\bibitem[TV]{ToVe} B. To\"en, G. Vezzosi, \textit{Homotopical Algebraic Geometry II: geometric stacks and applications}, Mem. Amer. Math. Soc. 193 (2008), no. 902.
\bibitem[TV2]{ToVe2} B. To\"en, G. Vezzosi, \textit{A note on  Chern character, loop spaces and derived algebraic geometry}, Abel Symposium, Oslo (2007), Volume 4, 331-354.
\end{thebibliography}
\end{document}